\documentclass[10pt,reqno]{amsart}
\usepackage{amsfonts,amsrefs,latexsym,amsmath, amssymb, mathrsfs, verbatim}
\usepackage{url,color}
\usepackage{bm}
\usepackage{pifont}
\usepackage{upgreek}
\usepackage{fancyhdr}
\usepackage{hyperref}
\usepackage{calligra}
\usepackage{marvosym}
\usepackage[percent]{overpic}
\usepackage{pict2e}
\usepackage{float}

\usepackage[hypcap=false]{caption}
\usepackage{xcolor}
\usepackage{wasysym}
\usepackage{accents}

\topmargin - .4 in

\newtheorem{theorem}{Theorem}[section]
\newtheorem{proposition}{Proposition}[section]
\newtheorem{lemma}[proposition]{Lemma}
\newtheorem{corollary}[proposition]{Corollary}
\newtheorem{conjecture}{Conjecture}

\theoremstyle{definition}
\newtheorem{definition}{Definition}[section]
\newtheorem{remark}{Remark}[section]
\newtheorem{convention}{Convention}[section]

\DeclareMathAlphabet{\mathcalligra}{T1}{calligra}{m}{n}
\DeclareFontShape{T1}{calligra}{m}{n}{<->s*[2.2]callig15}{}

\newcommand{\gfour}{\mathbf{g}}

\newcommand{\hfour}{\mathbf{h}}

\newcommand{\Transport}{\mathbf{B}}
\newcommand{\Sigmatnormal}{\mathbf{N}}

\newcommand{\fourvelocity}{\mathbf{u}}

\newcommand{\normalizer}{\mathfrak{n}}

\newcommand{\Ent}{s}
\newcommand{\Enth}{H}
\newcommand{\Temp}{\uptheta}
\newcommand{\speed}{c}
\newcommand{\Lnenth}{h}
\newcommand{\TempoverEnth}{q}

\newcommand{\eik}{U}

\newcommand{\wavearray}{\vec{\Psi}}

\newcommand{\GradEnt}{S}
\newcommand{\modDivGradEnt}{\mathcal{D}}

\newcommand{\uperpvort}{\mbox{\upshape vort}}

\newcommand{\vort}{\varpi}
\newcommand{\modVortVort}{\mathcal{C}}

\newcommand{\smoothfunction}{\mathrm{f}}

\newcommand{\nullform}{\mathfrak{Q}}

\newcommand{\Dfour}{\mathbf{D}}

\newcommand{\muX}{\breve X}
\newcommand{\Yvf}[1]{Y_{(#1)}}

\newcommand{\Fullset}{\mathscr{Z}}
\newcommand{\Tanset}{\mathscr{P}}
\newcommand{\Angularset}{\mathscr{Y}}

\newcommand{\Singletan}{P}

\newcommand{\mytr}{{\mbox{\upshape tr}}}
\newcommand{\gtorus}{g \mkern-8mu / }

\newcommand{\Ricfour}{\mathbf{Ric}}

\newcommand{\angLap}{ {\Delta \mkern-12mu / \, } }
\newcommand{\angD}{ {\nabla \mkern-14mu / \,} }
\newcommand{\angDarg}[1]{{\angD_{\mkern-3mu #1}}}

\newcommand{\RRiemann}{\mathcal{R}_{(+)}}
\newcommand{\LRiemann}{\mathcal{R}_{(-)}}

\newcommand{\rmd}{\mathrm{d}}

\newcommand{\voltorus}{{\rmd \varpi_{\gtorus} \, }}

\newcommand{\enmomem}{\mathbf{Q}}

\newcommand{\multipliervectorfield}{\breve{T}}

\newcommand{\deformarg}[3]{{^{(#1)} {\pmb{\pi}_{#2 #3}}}}

\newcommand{\spacetimeintegralcontrolwave}{\mathbb{K}}

\newcommand{\fullymodquant}[1]{{ ^{(#1)} \mkern-3mu \mathscr{X}}}

\newcommand{\Ntop}{N_{\textnormal{top}}}

\newcommand{\Tboot}{T_{\textnormal{Boot}}}

\newcommand{\Lunit}{L}
\newcommand{\uLunit}{\underline{L}}
\newcommand{\Lgeo}{L_{(\textnormal{geo})}}

\newcommand{\newuL}{\breve{\uLunit}}
\newcommand{\intcurvenewuL}{\breve{\underline{\gamma}}}

\newcommand{\RRiemannPS}{\mathcal{R}_{(+)}^{\textnormal{PS}}}
\newcommand{\LRiemannPS}{\mathcal{R}_{(-)}^{\textnormal{PS}}}
\newcommand{\dataRRiemannPS}{\mathring{\mathcal{R}}_{(+)}^{\textnormal{PS}}}
\newcommand{\LRiemannPSdata}{\mathring{\mathcal{R}}_{(-)}^{\textnormal{PS}}}
\newcommand{\muPS}{\upmu^{\textnormal{PS}}}
\newcommand{\PSLmusourcetermfunction}{G}
\newcommand{\antiderivativePSLmusourcetermfunction}{\mathfrak{A}}
\newcommand{\blowupdelta}{\mathring{\updelta}_*}
\newcommand{\blowuptime}{T_{\textnormal{Shock}}}
\newcommand{\minimumspeed}{\mathfrak{c}}
\newcommand{\maximumspeed}{\mathfrak{C}}
\newcommand{\timeisafunctionofeikonalinsingularcurve}{\mathfrak{t}_{\textnormal{Sing}}}
\newcommand{\timeisafunctionofeikonalcauchyhorizon}{\mathfrak{t}_{\textnormal{CH}}}

\newcommand{\rightu}{\eik_1}
\newcommand{\leftu}{\eik_2}
\newcommand{\compactsupportu}{\eik_3}
\newcommand{\interestingu}{\eik_{\mbox{\tiny \Radioactivity}}}
\newcommand{\PSdatamuHessianTaylorcoefficient}{\mathfrak{b}}
\newcommand{\PSLmunottoonegativeparameter}{\mathfrak{p}}

\newcommand{\PSthirdorderTaylorremaindercoefficientfunction}{\mathscr{R}}
\newcommand{\PSinitialkeymufunctioncoefficients}{\Lambda}

\newcommand{\nullhyparg}[1]{\mathcal{P}_{#1}}
\newcommand{\nullhyptwoarg}[2]{\mathcal{P}_{#1}^{#2}}
\newcommand{\singularcurve}{\mathcal{B}}
\newcommand{\cauchyhor}{\underline{\mathcal{C}}}
\newcommand{\crease}{\partial_-\mathcal{B}}
\newcommand{\classicaldev}{\mathcal{M}_*}
\newcommand{\classicaldevsingular}{\mathcal{M}_{\textnormal{Sing}}}
\newcommand{\classicaldevregular}{\mathcal{M}_{\textnormal{Reg}}}

\newcommand{\p}{\partial}
\newcommand{\R}{\mathbb{R}}
\newcommand{\geop}[1]{\frac{\p}{\p #1}}
\newcommand{\positivebigO}{\mathcal{O}_{\textnormal{pos}}}

\newcommand{\seed}{\mathring{\varphi}}

\newcommand{\singcurvevectorfield}{Q}

\newcommand{\spacetimemanifold}{\mathcal{M}}

\newcommand{\mink}{\bm{\mathfrak{m}}}

\newcommand{\solutionarray}{\mathbf{V}}
\newcommand{\variations}{\dot{\mathbf{V}}}

\newcommand{\genericvf}{\mathbf{Z}}

\newcommand{\RIfunction}{F}

\begin{document}

\title{The relativistic Euler equations: ESI notes on their geo-analytic structures and implications for shocks in $1D$ and multi-dimensions}
\author[LA,JS]{Leonardo Abbrescia$^{* \dagger}$ and Jared Speck$^{** \dagger\dagger}$}

\thanks{$^{*}$Vanderbilt University, Nashville, TN, USA.
\texttt{leonardo.abbrescia@vanderbilt.edu}}

\thanks{$^{**}$Vanderbilt University, Nashville, TN, USA.
\texttt{jared.speck@vanderbilt.edu}}

\thanks{$^\dagger$ LA gratefully acknowledges support from an NSF Postdoctoral Fellowship.}

\thanks{$^{\dagger\dagger}$JS gratefully acknowledges support from NSF grant \# DMS-2054184 and NSF CAREER grant \# DMS-1914537.
}

\begin{abstract}
In this article, we provide notes that complement the lectures on
the relativistic Euler equations and shocks that were given by the second author at
the program \emph{Mathematical Perspectives of Gravitation Beyond the Vacuum Regime},
which was hosted by the Erwin Schr\"{o}dinger International Institute for Mathematics and Physics  
in Vienna in February, 2022.
We set the stage by introducing a standard first-order formulation of the relativistic Euler equations and 
providing a brief overview of local well-posedness in Sobolev spaces. 
Then, using Riemann invariants, we provide
the first detailed construction of a localized subset of
the maximal globally hyperbolic developments 
of an open set of initially smooth, shock-forming isentropic solutions in $1D$, 
with a focus on describing the singular boundary and 
the Cauchy horizon that emerges from the singularity.
Next, we provide an overview of the new second-order formulation of the $3D$ relativistic Euler equations 
derived in \cite{mDjS2019}, its rich geometric and analytic structures, 
their implications for the mathematical theory of shock waves,
and their connection to the setup we use in our $1D$ analysis of shocks.
We then highlight some key prior results on the study of shock formation and related problems.
Furthermore, we provide an overview of how the formulation of the flow derived in \cite{mDjS2019} 
can be used to study shock formation in multiple spatial dimensions.
Finally, we discuss various open problems tied to shocks.

\bigskip


\bigskip

\end{abstract}

\maketitle

\centerline{\today}

\tableofcontents
\setcounter{tocdepth}{1}

\newpage

\section{Introduction}
\label{S:INTRO}
In this article, we provide notes that complement the lectures on
the relativistic Euler equations and shocks that were given by the second author at
the program \emph{Mathematical Perspectives of Gravitation Beyond the Vacuum Regime},
which was hosted by the Erwin Schr\"{o}dinger International Institute for Mathematics and Physics  
in Vienna in February, 2022. Broadly speaking, our main goals are the following: 
\begin{enumerate}
\item To briefly introduce the equations.
 \item To provide an overview of the new formulation of the equations
	derived in \cite{mDjS2019}.
	\item To describe the implications of the formulation from \cite{mDjS2019} 
		for the study of multi-dimensional shocks.
\end{enumerate}

Although our main interest is solutions in $3D$ \emph{without} symmetry assumptions,
we highlight Sect.\,\ref{S:1DMAXIMALDEVELOPMENT}, in which we rigorously study a family of initially smooth, simple isentropic shock-forming 
\emph{plane-symmetric} solutions and give a complete description of a portion of their  maximal (classical) globally hyperbolic developments 
(MGHD for short), up to the boundary. 
To the best of our knowledge,
this is the first article 
on \emph{any} quasilinear hyperbolic system in $1D$
that provides the detailed construction 
of the MGHD within the vicinity of a shock singularity.
Sect.\,\ref{S:1DMAXIMALDEVELOPMENT} in particular
provides an introduction to some of the main ideas behind the study of shock formation
in a gentle, semi-explicit $1D$ setting, where energy estimates are not needed.
Roughly, the MGHD is the largest possible classical solution + globally hyperbolic 
region that is launched by the initial data,
and when it exists and is unique
(for some solutions to some hyperbolic PDEs, it is not unique -- see Sect.\,\ref{SS:MAXIMALDEVELOPMENT}!), 
it is the holy grail object, at least from the point of view of classical solution.

We stress outright that, although the relativistic Euler equations are posed on an ambient Lorentzian manifold $(\spacetimemanifold,\gfour)$, from the point of view of the causal structure of the fluid,
the correct notion of ``globally hyperbolic''
is not with respect to the spacetime metric $\gfour$, but rather 
with respect to the acoustical metric $\hfour$, introduced in
Def.\,\ref{D:ACOUSTICALMETRIC}. To avoid confusion, we will refer to this as $\hfour$-global hyperbolicity and the corresponding $\hfour$-MGHD.

Many of the techniques and geometric insights behind the modern approach to studying
shocks have roots in mathematical General Relativity,
notably the celebrated proof by Christodoulou--Klainerman
\cite{dCsK1993} of the nonlinear stability of Minkowski spacetime as a solution to Einstein's equations.
A key unifying theme between the global existence result \cite{dCsK1993} and the results on shocks that we discuss here is that 
\emph{nonlinear geometric optics}, implemented via an \emph{eikonal function} $\eik$, plays a central role in the study of the flow.
We will discuss eikonal functions in detail
in Sects.\,\ref{S:1DMAXIMALDEVELOPMENT} and \ref{S:SHOCKFORMATIONAWAYFROMSYMMETRY}.

We highlight a key advantage of the formulation of the flow derived in \cite{mDjS2019}.
\begin{quote}
	For general solutions (in $3D$, i.e., three spatial dimensions), 
	it allows one to implement a sharp version of nonlinear geometric optics.
	Nonlinear geometric optics is crucial for the study of multi-dimensional shocks,
	and it has other applications, such as low-regularity well-posedness 
	(see Sect.\,\ref{SSS:LWPFIRSTORDER}).
\end{quote}

In line with points (1)--(3) above, this article is meant to give an introduction to the mathematical techniques and methods needed for the PDE analysis through the lens of nonlinear geometric optics. A comprehensive topical review of the rich history of shock singularities in the framework of hyperbolic conservation laws would encompass a work of strictly larger order of magnitude. We refer the reader to 
\cites{cD2010, bressan2011nonlinear,gHsKjSwW2016} and the references therein for a good start.

\subsection{Motivation, context, and the structure of the article}
\label{SS:MOTIVATIONCONTEXTSTRUCTURE}

\subsubsection{Mathematical and physical motivations} \label{SSS:MATHEMATICALANDPHYSICALMOTIVATIONS}
The relativistic Euler equations are an important matter model in mathematical General Relativity. 
For example, relativistic fluids are often used in cosmology to model the average energy-matter
content of spacetime, and they in particular play a central role in the Standard Model of cosmology. 

For a non-exhaustive yet comprehensive overview of these physical applications, we refer the reader to the texts 
\cites{bO1983, rW1984,sW2008}, and the references therein. There has also been exciting rigorous mathematical progress on the
global structure of solutions when the spacetime is expanding
\cites{iRjS2013,oliynyk2016future,fajman2021stabilizing,fajman2023stability,jS2012}. We also highlight 
\cite{dC2007b} and \cite{dC2007}*{Chapter~1} for a modern introduction to 
the relativistic Euler equations and their connections to the laws of mechanics and thermodynamics.

Despite the above remarks, 
the rigorous mathematical theory of multi-dimensional solutions is far from complete. Even if one considers the relativistic Euler equations 
on a given spacetime background (i.e., without coupling to Einstein's equations),
there are many mathematically rich phenomena that have not been fully understood.
Chief among these is the fundamental issue, 
going back to Riemann's foundational work \cite{bR1860}
on non-relativistic compressible fluids in $1D$,
that initially smooth solutions can develop shock singularities in finite time.
Shocks are singularities such that various fluid variables' gradients blow up in finite time
(in a precise, controlled fashion, as it turns out), though the fluid variables themselves (such as the velocity and density)
remain bounded. This phenomenon is also known as \emph{wave breaking}.
The relatively mild nature of the singularity gives rise to the hope that one might 
-- at least for short times --
be able to uniquely continue the solution weakly
past the ``first singularity,\footnote{In multi-dimensions, the correct notion of
``first singularity'' is not a point in spacetime, but rather a co-dimension $2$ submanifold of points;
see Sect.\,\ref{SS:MAXIMALDEVELOPMENT}. \label{FN:FIRSTSINGULARITY}}'' 
subject to suitable selection criteria, typically 
in the form of Rankine--Hugoniot-type jump conditions across a shock hypersurface 
(which is not known in advance) and an entropy-type condition. 
This weak continuation problem is known as the \emph{shock development problem}, and in full generality, it remains open, for both the $3D$ 
relativistic Euler equations and their non-relativistic analog, i.e., the $3D$ compressible Euler equations;
see Sect.\,\ref{SS:SHOCKDEVELOPMENT} for further discussion. The long-term goal is certainly the following:
\begin{quote}
	Develop a rigorous \emph{global-in-time-and-space} 
	theory of existence and uniqueness for open sets of multi-dimensional solutions
	that are allowed to transition from classical to weak due to shock formation,
	and describe the interactions of different 
	shock hypersurfaces, Cauchy horizons (see Sect.\,\ref{SS:1DBOUNDARYOFMAXIMALDEVELOPMENT}),
	and other crucial qualitative features of the flow.
\end{quote}
This goal is far out of reach as of present. If one could accomplish it even in a single perturbative regime, 
that would represent truly remarkable progress.

While there have been many works on shock formation for the multi-dimensional non-relativistic compressible
Euler equations, 
such as \cites{dCsM2014,jS2016b,jSgHjLwW2016,jLjS2018,jLjS2021,tBsSvV2022,tBsSvV2020,tBsSvV2019a,tBsI2022,lAjS2020,lAjS20XX,lAjS2022},
there have been relatively few works on multi-dimensional shock formation for the relativistic Euler equations. 
The most notable work in the relativistic case is Christodoulou's breakthrough monograph \cite{dC2007}, 
in which he proved shock formation and studied some aspects of the $\hfour$-MGHD for 
open sets of irrotational and isentropic solutions
in $3D$ that are compactly supported perturbations of non-vacuum constant fluid states. 
Unlike the non-relativistic case, there are currently no rigorous results that prove multi-dimensional shock formation
for the relativistic Euler equations in the presence of vorticity and entropy.
However, in Sect.\,\ref{S:SHOCKFORMATIONAWAYFROMSYMMETRY}, we provide a blueprint for how one can use the equations of
\cite{mDjS2019} and the analytic framework of \cites{dC2007,jS2016b,jSgHjLwW2016,jLjS2018,jLjS2021,lAjS2020,lAjS2022}
to prove stable shock formation in the $3D$ relativistic case for open sets of solutions with vorticity and entropy.

The remainder of the article is organized as follows:
\begin{itemize}
	\item In Sect.\,\ref{S:EQUATIONSANDLOCALWELLPOSEDNESS}, we set up the study of the relativistic Euler equations and provide an overview
		of basic ingredients that play a role in the proof of local well-posedness in Sobolev spaces.
		In particular, we introduce the \emph{acoustical metric} $\hfour$, which is the solution-dependent
		Lorentzian metric that drives the propagation of sound waves;
		the acoustical metric is fundamental for implementing nonlinear geometric optics and studying multi-dimensional shocks.
		While we state the equations in the case of an arbitrary spacetime background $(\spacetimemanifold,\gfour)$,
		the vast majority of the article concerns the case in which $\spacetimemanifold = \mathbb{R}^{1+3}$ 
		and $\gfour = \mink$ is the Minkowski metric on $\spacetimemanifold$. 
		We again point out that nearly all notions of hyperbolicity will refer to the acoustical metric $\hfour$ and not $\gfour$ or $\mink$;
		see also Sect.\,\ref{SS:FIRSTCOMMENTSONMOREGENERALSPACETIMES} and Remark~\ref{R:SECONDCOMMENTONMOREGENERALSPACETIMES}. 
		
	\item In Sect.\,\ref{S:1DMAXIMALDEVELOPMENT}, to help readers gain intuition for the subsequent discussion, 
		we study the flow in one spatial dimension, which is equivalent to studying plane-symmetric solutions in $3D$. 
		In particular, in Theorem~\ref{T:MAINTHEOREM1DSINGULARBOUNDARYANDCREASE},
		we provide a detailed analysis of the $\hfour$-MGHD for a large set of shock-forming relativistic Euler solutions in $1D$.
		Some aspects of the theorem, notably the results concerning the formation of Cauchy horizons, 
		have not been proved in detail elsewhere in the literature. Although 
		Theorem~\ref{T:MAINTHEOREM1DSINGULARBOUNDARYANDCREASE} specifically concerns the relativistic Euler equations,
		the techniques could be applied to a wide variety of strictly hyperbolic genuinely nonlinear
		PDEs in one spatial dimension.
	\item In Sect.\,\ref{S:NEWFORMULATIONOFFLOW}, we provide an overview of 
		the new formulation of the $3D$ relativistic Euler equations derived in \cite{mDjS2019}.
	\item In Sect.\,\ref{S:PRIORWORKSSHOCKFORMATION}, we describe some previous works on shock formation and related problems.
	\item In Sect.\,\ref{S:SHOCKFORMATIONAWAYFROMSYMMETRY}, 
		based in part on our prior experience \cites{jS2016b,jSgHjLwW2016,jLjS2018,jLjS2021,lAjS2020,lAjS2022}
		in studying shock formation in multi-dimensional non-relativistic compressible fluids,
		we provide an overview of how the new formulation of the flow from \cite{mDjS2019}
		can be used to study shock formation for open sets of solutions to the $3D$ 
		relativistic Euler equations.
		As of present, rigorous fluid shock formation 
		results in multi-dimensions with vorticity and entropy 
		have been proved only for the non-relativistic compressible Euler equations.
		Nonetheless, from the perspective of shock formation, the relativistic Euler equations and 
		the non-relativistic compressible Euler equations enjoy many structural commonalities,
		and we expect that the overview we provide in Sect.\,\ref{S:SHOCKFORMATIONAWAYFROMSYMMETRY} could be turned
		(with substantial effort) into a complete proof.
\item In Sect.\,\ref{S:OPENPROBLEMS}, we describe various open problems.
\end{itemize}

\section{Standard formulations of the relativistic Euler equations and local well-posedness}
\label{S:EQUATIONSANDLOCALWELLPOSEDNESS}
In this section, we introduce some standard first-order formulations of the relativistic Euler equations.
We also introduce the acoustical metric $\hfour$ and set up a corresponding 
geometric version of the energy method that applies to a first-order formulation of the flow.
Finally, we state Prop.\,\ref{P:LWPFIRSTORDER}, which is a standard result on local well-posedness,
and briefly discuss how its proof is connected to the energy method.

\subsection{The basic setup}
\label{SS:BASICSETUP}
We start by discussing standard formulations of the relativistic Euler equations on
an arbitrary four-dimensional Lorentzian manifold $(\spacetimemanifold,\gfour)$, 
where $\gfour$ is the spacetime metric of signature $(-,+,+,+)$.
Our discussion in this section is motivated by Christodoulou's
presentation in \cite{dC2007}*{Chapter~1}.
We are somewhat terse in our presentation here, and thus we refer readers to \cite{dC2007}*{Chapter~1} for additional details.

\begin{convention}[Moving indices with $\gfour$ and $\mink$]
	\label{C:MOVEINDICESWITHSPACETIMEMETRIC}
	In this section, we will lower and raise indices with the spacetime metric $\gfour$ 
	and its inverse $\gfour^{-1}$
	respectively, i.e., $\fourvelocity_{\alpha} := \gfour_{\alpha \beta} \fourvelocity^{\beta}$
	and $\upxi^{\alpha} := (\gfour^{-1})^{\alpha \beta} \upxi_{\beta}$. Similar comments apply to the bulk of the article, 
	once we fix $\spacetimemanifold = \R^{1+3}$ and $\gfour = \mink$ as the Minkowski metric on $\R^{1+3}$; 
	see Sect.\,\ref{SSS:INDEXCONVENTIONS}. 
	This will become especially important later on, when we introduce the acoustical metric $\hfour$, i.e., 
	we do \underline{not} raise or lower indices with $\hfour$ or $\hfour^{-1}$.
\end{convention}

\subsubsection{The basic fluid variables and relations}
\label{SSS:INTROBASICFLUIDVARIABLESANDRELATIONS}
We now introduce the basic fluid variables that will play a role in our study of the relativistic Euler equations:
$\uprho : \spacetimemanifold \rightarrow [0,\infty)$ denotes the fluid's (proper) energy density,
$p : \spacetimemanifold \rightarrow [0,\infty)$ denotes the fluid pressure,
$\Ent : \spacetimemanifold \rightarrow [0,\infty)$ denotes the entropy per particle (``entropy'' for short from now on),
$n : \spacetimemanifold \rightarrow [0,\infty)$ denotes the proper number density,
$\Temp : \spacetimemanifold \rightarrow [0,\infty)$ denotes the temperature,
\begin{align} \label{E:ENTHDEF}
\Enth 
	& := \frac{\uprho + p}{n} 
\end{align}
denotes the enthalpy per particle,
and $\fourvelocity^{\alpha}$ denotes the fluid's four-velocity, 
which is a future-directed vectorfield on $\spacetimemanifold$ normalized by:
\begin{align} \label{E:FLUIDFOURVELOCITYNORMALIZED}
	\gfour(\fourvelocity,\fourvelocity)
	& = - 1.
\end{align}

In the rest of the paper,
\begin{align} \label{E:FIXEDPOSITIVEENTHALPYCONSTANT}
		\overline{\Enth} & > 0
\end{align}
denotes an arbitrary fixed, positive value of $\Enth$; we find it convenient to use $\Enth$ to normalize
various constructions.

We find it convenient to work with the natural log of the enthalpy.
\begin{definition}[Logarithmic enthalpy]
	\label{D:LOGENTHALPY}
	We define
	the (dimensionless) logarithmic enthalpy $\Lnenth$ as follows:
	\begin{align} \label{E:LOGENTHALPY}
		\Lnenth
		&:= \ln \left(\Enth/\overline{\Enth} \right).
	\end{align}
\end{definition}

\begin{definition}[The quantity $\TempoverEnth$]
	\label{D:TEMPDIVIDEDBYENTH}
	We define the quantity $\TempoverEnth$ as follows:
	\begin{align} \label{E:TEMPDIVIDEDBYENTH}
		\TempoverEnth
		&:= \frac{\Temp}{\Enth}.
	\end{align}
\end{definition}

\subsubsection{A first statement of the relativistic Euler equations}
\label{SSS:FIRSTSTATEMENTOFEQUATIONS}
Relative to arbitrary coordinates, 
the relativistic Euler equations on $(\spacetimemanifold,\gfour)$ can be expressed as
the following system of \emph{quasilinear hyperbolic conservation laws}:
\begin{subequations}
\begin{align} \label{E:INTRODIVOFREENERGYMOMENTUMTENSORIS0}
	\nabla_{\kappa} \mathbf{T}^{\alpha \kappa}
	& = 0,
	& (\alpha & = 0,1,2,3),
		\\
	\nabla_{\kappa} (n \fourvelocity^{\kappa})
	& = 0,
	&&
	\label{E:INTROCONSERVATIONOFPARTICLENUMBER}
\end{align}
\end{subequations}	
where $\nabla$ is the Levi-Civita connection of the spacetime metric $\gfour$ and:
\begin{align} \label{E:FLUIDENERGYMOMENTUMTENSOR}
	\mathbf{T}^{\alpha \beta}
	& := (\uprho + p) \fourvelocity^{\alpha} \fourvelocity^{\beta}
			+
			p (\gfour^{-1})^{\alpha \beta},
	& (\alpha, \beta & = 0,1,2,3),
\end{align}
is the energy-momentum tensor of the fluid.
It turns out that under additional assumptions discussed in Sect.\,\ref{SSS:EOSANDTHERMODYNAMICS},
equation \eqref{E:FLUIDFOURVELOCITYNORMALIZED} can be viewed as a constraint that is preserved by
the flow of the PDEs~\eqref{E:INTRODIVOFREENERGYMOMENTUMTENSORIS0}; see
equation~\eqref{E:EVOUTIONOFMINKOWSKINORMOFFOURVELCOITY} and the discussion surrounding it.
We also emphasize that later on, we will discuss several other formulations of the flow.

\subsubsection{Equation of state and thermodynamic relations}
\label{SSS:EOSANDTHERMODYNAMICS}
The system \eqref{E:FLUIDFOURVELOCITYNORMALIZED} 
+ 
\eqref{E:INTRODIVOFREENERGYMOMENTUMTENSORIS0}--\eqref{E:INTROCONSERVATIONOFPARTICLENUMBER} 
+
\eqref{E:FLUIDENERGYMOMENTUMTENSOR}
is not
closed because there are too many fluid variables and not enough equations. 
The standard path to closing the system
begins with an \emph{equation of state}, that is, an assumed
functional relationship of the form $p = p(\uprho,\Ent)$.
The basic hyperbolicity of the equations will require that
\emph{speed of sound} $\speed$, 
defined by:
\begin{align} \label{E:SPEEDOFSOUND}
	\speed 
	& := 
	\sqrt{\frac{\partial p}{\partial \uprho} \left|\right._{\Ent}},
\end{align}
should be real and non-negative,
where $\frac{\partial p}{\partial \uprho} \left|\right._{\Ent}$
denotes the partial derivative of $p$ (i.e., of the equation of state) 
with respect to $\uprho$ at fixed $\Ent$.
The vanishing of $\speed$ causes a severe degeneracy in the system,
and we will therefore restrict our attention to 
solution regimes in which:
\begin{align} \label{E:SPEEDOFSOUNDASSUMEDBOUNDS}
	0 & < \speed \leq 1,
\end{align}
where we have normalized our setup so that the speed of light is unity
(i.e., the second inequality in \eqref{E:SPEEDOFSOUNDASSUMEDBOUNDS} implies that the speed
of sound is less than or equal to the speed of light).
We also restrict our attention to solutions such that:
\begin{align} \label{E:POSITIVITYOFSOMEFLUIDVARIABLES}
	\uprho 
	& > 0,
		\,
	p > 0,
		\,
	n > 0,
		\,
	\Temp > 0,
		\,
	\Enth > 0.
\end{align}

The laws of thermodynamics demand that the fluid variables satisfy the following functional relations:
\begin{align} \label{E:THERMOLAWS}
\Enth
& =
\frac{\partial \uprho}{\partial n}\left|_{\Ent}\right.,
&
\Temp 
& = \frac{1}{n} \frac{\partial \uprho}{\partial \Ent} \left|\right._n,
&
\mathrm{d} \Enth 
& = \frac{\mathrm{d} p}{n} + \Temp \mathrm{d} \Ent,
\end{align}
where 
$\frac{\partial}{\partial n}\left|_{\Ent}\right.$
denotes partial differentiation with respect to $n$ at fixed $\Ent$
and
$\frac{\partial}{\partial \Ent} \left|\right._n$
denotes partial differentiation with respect to $\Ent$ at fixed $n$.

\subsection{Minkowski metric assumption and Minkowski-rectangular coordinates}
\label{SS:MINKOWSKIASSUMPTIONS}
In the rest of the paper, unless we explicitly state otherwise,
it should be understood that the spacetime manifold $\spacetimemanifold$ is equal to $\mathbb{R}^{1+3}$
and that $\gfour$ is equal to the Minkowski metric, which for clarity we denote by $\mink$.
We fix a standard global Minkowski-rectangular coordinate system $\lbrace x^{\alpha} \rbrace_{\alpha=0,1,2,3}$ on $\spacetimemanifold$,
relative to which $\mink = \mbox{\upshape diag}(-1,1,1,1)$. 
We use the notation $\lbrace \partial_{\alpha} \rbrace_{\alpha=0,1,2,3}$ to denote the partial derivative vectorfields in this
coordinate system. We also use the alternate notation $t := x^0$ and $\partial_t := \partial_0$.

\subsubsection{Index conventions}
\label{SSS:INDEXCONVENTIONS}
	From now until the end of the paper, we use the following conventions for indices.
	\begin{itemize}
	\item (\textbf{Lowercase Greek index conventions}) Lowercase Greek spacetime indices 
	$\alpha$, $\beta$, etc.\ correspond to the Minkowski-rectangular coordinates
	and vary over $0,1,2,3$. All lowercase Greek indices are lowered and raised with the Minkowski metric $\mink$ 
	and its inverse $\mink^{-1}$.
	Throughout the article, if $\upxi$ is a type $\binom{m}{n}$ spacetime tensorfield,
	then unless we indicate otherwise, in our identities and estimates,
	\textbf{$\left\lbrace \upxi_{\beta_1 \cdots \beta_n}^{\alpha_1 \cdots \alpha_m} \right\rbrace_{
	\alpha_1, \cdots \alpha_m, \beta_1, \cdots, \beta_n = 0,1,2,3}$ denotes its 
	components with respect to the Minkowski-rectangular coordinates}. 
	\item (\textbf{Lowercase Latin index conventions})
	Lowercase Latin spatial indices $a$, $b$, etc.\ correspond to the Minkowski-rectangular spatial coordinates and vary over $1,2,3$.
	Much like in the previous point, 
	if $\upxi$ is a type $\binom{m}{n}$ $\Sigma_t$-tangent tensorfield,
	then \textbf{$\left\lbrace \upxi_{b_1 \cdots b_n}^{a_1 \cdots a_m} \right\rbrace_{a_1,\cdots,a_m,b_1,\cdots,b_n = 1,2,3}$ 
	denotes its components with respect to the Minkowski-rectangular spatial coordinates}. 
	\item (\textbf{Uppercase Latin index conventions})
	Uppercase Latin spatial indices $A,B$, etc.\ correspond to the spatial coordinates 
	$(x^2,x^3)$ and vary over $2,3$. 
	This will be important in Sect.\,\ref{S:SHOCKFORMATIONAWAYFROMSYMMETRY}, where
	$(x^2,x^3)$ will correspond to perturbations away from plane-symmetry.
\item (\textbf{Einstein summation}) 
	We use Einstein's summation convention in that repeated indices are summed,
	e.g., $\Lunit^A X^A := \Lunit^2 X^2 + \Lunit^3 X^3$.
	\end{itemize}

\subsection{A well-posed first-order formulation}
\label{SS:WELLPOSEDFIRSTORDERFORMULATION}
In this section, we provide a first-order formulation of the flow in
which the unknowns are taken to be
$\left(\Lnenth,\fourvelocity^0,\fourvelocity^1,\fourvelocity^2,\fourvelocity^3,\Ent \right)$.
Thanks to the assumptions state in Sect.\,\ref{SSS:FIRSTSTATEMENTOFEQUATIONS}, 
the remaining fluid variables can be expressed\footnote{By \eqref{E:FLUIDFOURVELOCITYNORMALIZED},
one can also express $\fourvelocity^0 = \sqrt{1 + \fourvelocity_a \fourvelocity^a}$. \label{FN:U0REDUNDANT}} 
as functions of the unknowns.

\begin{definition}[Derivative of a scalar function with respect to a vectorfield]
\label{D:DIRECTIONALDERIVATIVEOFSCALARFUNCTION}
	Throughout the paper, if $\genericvf$ is a vectorfield and $f$ is a scalar function,
	then $\genericvf f := \genericvf^{\kappa} \partial_{\kappa}f$ denotes the derivative of $f$
	in the direction of $\genericvf$.
\end{definition}

With the help of \eqref{E:THERMOLAWS},
one can compute that for $C^1$ solutions, 
\eqref{E:INTRODIVOFREENERGYMOMENTUMTENSORIS0}--\eqref{E:INTROCONSERVATIONOFPARTICLENUMBER}
is equivalent to the following first-order formulation of the flow:
\begin{subequations}
\begin{align}
	\fourvelocity^{\kappa} \partial_{\kappa} \Lnenth
	+
	\speed^2 \partial_{\kappa} \fourvelocity^{\kappa}
	& = 0,
	&&
		\label{E:ENTHALPYEVOLUTION} 
			\\
	\fourvelocity^{\kappa} \partial_{\kappa} \fourvelocity^{\alpha}
	+ 
	\Pi^{\alpha \kappa} \partial_{\kappa} \Lnenth
	-
	\TempoverEnth
	(\pmb{\mathfrak{m}}^{-1})^{\alpha \kappa} \partial_{\kappa} \Ent
	& = 0,
	&
	(\alpha & = 0,1,2,3),
	\label{E:VELOCITYEVOLUTIONWITHPROJECTION}
		\\
	\fourvelocity^{\kappa} \partial_{\kappa} \Ent
	& = 0,
	&&
	\label{E:ENTROPYEVOLUTION}
		\\
	\mink(\fourvelocity,\fourvelocity) 
	& = - 1,
	&&
	\label{E:AGAINFLUIDFOURVELOCITYNORMALIZED}
\end{align}
\end{subequations}
where:
\begin{align} \label{E:PROJECTTIONONTOORTHOGONALCOMPLEMENTOFFOURVELOCITY}
\Pi^{\alpha \beta} 
& 
:= (\mink^{-1})^{\alpha \beta} 
+ 
\fourvelocity^{\alpha} \fourvelocity^{\beta},
& 
(\alpha, \beta & = 0,1,2,3),
\end{align}
denotes projection onto the $\mink$-orthogonal complement of $\fourvelocity$. 
In particular,
\begin{align} \label{E:PROJECTIONANNIHILATESFOURVELOCITY}
		\Pi^{\alpha \beta} \fourvelocity_{\beta}
		& = 0,
		&
		(\alpha & = 0,1,2,3).
\end{align}
Equations \eqref{E:VELOCITYEVOLUTIONWITHPROJECTION}--\eqref{E:AGAINFLUIDFOURVELOCITYNORMALIZED}
+
\eqref{E:PROJECTTIONONTOORTHOGONALCOMPLEMENTOFFOURVELOCITY}
are a first-order quasilinear hyperbolic system and are locally well-posed in suitable Sobolev spaces;
see Prop.\,\ref{P:LWPFIRSTORDER}.
using \eqref{E:ENTROPYEVOLUTION}
We also note by contracting \eqref{E:VELOCITYEVOLUTIONWITHPROJECTION} against $\fourvelocity_{\beta}$ and
using \eqref{E:ENTROPYEVOLUTION}, we deduce -- without using \eqref{E:AGAINFLUIDFOURVELOCITYNORMALIZED} --
the following identity:
\begin{align} 
\begin{split} \label{E:EVOUTIONOFMINKOWSKINORMOFFOURVELCOITY}
	\fourvelocity^{\kappa} \partial_{\kappa}
	\left\lbrace
		\mink(\fourvelocity,\fourvelocity)
		+
		1
	\right\rbrace
	& = \fourvelocity^{\kappa} (\partial_{\kappa} \fourvelocity_{\lambda}) \fourvelocity^{\lambda}
		\\
	& =
	-
	\left\lbrace
		\mink(\fourvelocity,\fourvelocity)
		+
		1
	\right\rbrace
	\fourvelocity^{\kappa} \partial_{\kappa} \Lnenth.
\end{split}
\end{align}
Using \eqref{E:EVOUTIONOFMINKOWSKINORMOFFOURVELCOITY}, one can easily show via Gr\"{o}nwall's inequality
that if $\mink(\fourvelocity,\fourvelocity) + 1$ vanishes at $t=0$, then
$\mink(\fourvelocity,\fourvelocity) + 1$ vanishes on any slab $[0,T] \times \mathbb{R}^3$
of classical existence on which the solution is $C^1$.
Hence, \eqref{E:AGAINFLUIDFOURVELOCITYNORMALIZED} can be viewed as a constraint on the initial data that
is preserved by the flow.

\begin{remark}[Additional terms in the context of curved spacetime backgrounds]
\label{R:ADDITIONALTERMS}
When posed on a general smooth spacetime background $(\spacetimemanifold,\gfour)$,
there are additional terms on
RHS~\eqref{E:ENTHALPYEVOLUTION}--\eqref{E:VELOCITYEVOLUTIONWITHPROJECTION}
of the schematic form $\bf{\Gamma} \cdot \smoothfunction(\solutionarray)$,
where $\bf{\Gamma}$ schematically denotes Christoffel symbols of
$\gfour$, 
$\solutionarray$ denotes the array of fluid variable unknowns (see definition~\ref{E:SOLUTIONARRAY}),
and $\smoothfunction$ schematically denotes a smooth function.
In the context of fluid shock formation, the dominant terms in the flow
are Riccati-type nonlinearities in the fluid obtained after differentiating the equations one time,
i.e., they are quadratic in the first derivatives of the fluid;
see Remark~\ref{SS:1DRICCATIBLOWUP}.
Note that upon differentiating the term
$\bf{\Gamma} \cdot \smoothfunction(\solutionarray)$,
one obtains a term that is \emph{linear} in the derivatives of the fluid variables.
This means, in particular, that the for large-gradient solutions,
the differentiated terms $\partial_{\alpha} \left(\bf{\Gamma} \cdot \smoothfunction(\solutionarray) \right)$
will be small relative to the shock-driving Riccati-type fluid terms,
i.e., the ``new terms'' coming from the non-flat ambient geometry are \emph{relatively} small. 	
This observation is a key reason behind our expectation that one can use the methods
described later in this article to prove shock formation
for the relativistic Euler equations on a general spacetime
$(\spacetimemanifold,\gfour)$,
at least for initial data with suitable large gradients.
\end{remark}

\subsubsection{Equations of variation}
\label{SSS:EQUATIONSOFVARIATION}
As a preliminary step in studying the well-posedness of 
\eqref{E:VELOCITYEVOLUTIONWITHPROJECTION}--\eqref{E:AGAINFLUIDFOURVELOCITYNORMALIZED},
we will discuss the equations satisfied by the derivatives of solutions. This is important in the 
sense that proofs of local well-posedness rely on differentiating the equations to obtain energy estimates
for the solution's higher derivatives.
Specifically, by differentiating the equations, it is straightforward to see that
given a solution $(\Lnenth,\fourvelocity^{\alpha},\Ent)$ to \eqref{E:ENTHALPYEVOLUTION}--\eqref{E:AGAINFLUIDFOURVELOCITYNORMALIZED},
any derivative (of any order), which we denote by $(\dot{\Lnenth},\dot{\fourvelocity}^{\alpha},\dot{\Ent})$,
satisfies an inhomogeneous system, known as the \emph{equations of variation}, of the following form:
\begin{subequations}
\begin{align}
	\fourvelocity^{\kappa} \partial_{\kappa} \dot{\Lnenth}
	+
	\speed^2 \partial_{\kappa} \dot{\fourvelocity}^{\kappa}
	& = \dot{\mathfrak{F}},
	&&
		\label{E:EOVENTHALPYEVOLUTION} 
			\\
	\fourvelocity^{\kappa} \partial_{\kappa} \dot{\fourvelocity}^{\alpha}
	+ 
	\Pi^{\alpha \kappa} \partial_{\kappa} \dot{\Lnenth}
	-
	\TempoverEnth
	(\pmb{\mathfrak{m}}^{-1})^{\alpha \kappa} \partial_{\kappa} \dot{\Ent}
	& = \dot{\mathfrak{G}}^{\alpha},
	&
	(\alpha & = 0,1,2,3),
	\label{E:EOVVELOCITYEVOLUTIONWITHPROJECTION}
		\\
	\fourvelocity^{\kappa} \partial_{\kappa} \dot{\Ent}
	& = \dot{\mathfrak{H}},
	&&
	\label{E:EOVENTROPYEVOLUTION}
		\\
	\mink(\fourvelocity,\dot{\fourvelocity}) 
	& = \dot{\mathfrak{I}}.
	&&
	\label{E:EOVFLUIDFOURVELOCITYNORMALIZED}
\end{align}
\end{subequations}
Although here we do not provide a detailed formula for the structure of 
the inhomogeneous terms $\dot{\mathfrak{F}}$, $\cdots$, $\dot{\mathfrak{I}}$ 
on RHSs~\eqref{E:EOVENTHALPYEVOLUTION}--\eqref{E:EOVFLUIDFOURVELOCITYNORMALIZED},
it is easy to see that they feature terms involving $\leq$ the number of derivatives represented by the 
variations $(\dot{\Lnenth},\dot{\fourvelocity}^{\alpha},\dot{\Ent})$,
i.e., fewer derivatives than the principal terms on LHSs~\eqref{E:EOVENTHALPYEVOLUTION}--\eqref{E:EOVFLUIDFOURVELOCITYNORMALIZED}.

\subsubsection{The acoustical metric}
\label{SSS:ACOUSTICALMETRIC}
A basic question about the equations of variation is: what kinds of energy estimates are available for solutions?
It turns out that this question is connected to the intrinsic geometry of the flow via the \emph{acoustical metric},
which plays a fundamental role throughout the article.

\begin{definition}[The acoustical metric and its inverse] 
\label{D:ACOUSTICALMETRIC}
\begin{subequations}
We define:
\begin{align} \label{E:ACOUSTICALMETRIC}
		\hfour_{\alpha \beta}
	& : = 
		\normalizer
		\left\lbrace
		\speed^{-2} \mink_{\alpha \beta} 
		+ 
		(\speed^{-2} - 1) \fourvelocity_{\alpha} \fourvelocity_{\beta}
		\right
		\rbrace,
					\\
	(\hfour^{-1})^{\alpha \beta}
	& :
		= 
		\normalizer^{-1}
		\left\lbrace
		\speed^2 (\mink^{-1})^{\alpha \beta} 
			+ 
			(\speed^2 - 1)
			\fourvelocity^{\alpha} \fourvelocity^{\beta}
		\right\rbrace,
			\label{E:INVERSEACOUSTICALMETRIC}
				\\
	\normalizer
	& := 
		\speed^2
		+
		(1 - \speed^2) (\fourvelocity^0)^2 \geq 1.
		\label{E:ACOUSTICALMETRICNORMALIZER}
\end{align}
\end{subequations}
We refer to $\hfour$ as the \emph{acoustical metric}.
\end{definition}

With the help of \eqref{E:AGAINFLUIDFOURVELOCITYNORMALIZED}, one can compute that:
\begin{align} \label{E:INVERSEISTHEINVERSE}
	\hfour_{\alpha \gamma} (\hfour^{-1})^{\gamma \beta}
	& = \updelta_{\alpha}^{\beta},
\end{align}
where $\updelta_{\alpha}^{\beta}$ is the Kronecker delta. That is,
\eqref{E:INVERSEISTHEINVERSE} shows that $\hfour^{-1}$ is indeed the inverse of $\hfour$.
Moreover, by construction, we have:
\begin{align} \label{E:GINVERSE00ISMINUSONE}
	(\hfour^{-1})^{00}
	& = -1.
\end{align}
 
\begin{remark}[Normalization condition for $\hfour$]
\label{R:NORMALIZATIONCONDITIONFORGFOUR}
\eqref{E:GINVERSE00ISMINUSONE} is a convenient normalization condition.
Other works (e.g.\ \cite{dC2007}) rely on a different normalization condition for the acoustical metric.
That represents a minor difference that typically has no important implications for the study of solutions. 
\end{remark}

\begin{remark}[Subluminal speeds]
\label{R:SUBLUMINAL}
Note that by \eqref{E:SPEEDOFSOUNDASSUMEDBOUNDS},
\eqref{E:ACOUSTICALMETRIC},
and \eqref{E:ACOUSTICALMETRICNORMALIZER},
$\hfour(\genericvf,\genericvf) < 0 \implies \mink(\genericvf,\genericvf) < 0$,
i.e., the $\hfour$-null cones are inside the $\mink$-null cones.
Note also that by \eqref{E:AGAINFLUIDFOURVELOCITYNORMALIZED},
\eqref{E:ACOUSTICALMETRIC}, 
and \eqref{E:ACOUSTICALMETRICNORMALIZER},
we have:
\begin{align} \label{E:GFOURLENGTHOFFOURVELOCITY}
	\hfour_{\kappa \lambda} 
	\fourvelocity^{\kappa}
	\fourvelocity^{\lambda}
	& = - \normalizer < 0.
\end{align}
In particular, $\fourvelocity$ is $\hfour$-timelike (and thus $\mink$-timelike, as one can explicitly see
via \eqref{E:AGAINFLUIDFOURVELOCITYNORMALIZED}).
\end{remark}

\begin{remark}[Acoustical metric in a general ambient spacetime $(\spacetimemanifold,\gfour)$]
\label{R:ACOUSTICALMETRICGENERALAMBIENTSPACETIME}
When studying the relativistic Euler equations in a general ambient spacetime $(\spacetimemanifold,\gfour)$
equipped with coordinates $(x^0,x^1,x^2,x^3)$, 
where $t := x^0$ is a smooth time function
(see Sect.\,\ref{SS:FIRSTCOMMENTSONMOREGENERALSPACETIMES}),
the acoustical metric is (modulo the comments of Remark~\ref{R:NORMALIZATIONCONDITIONFORGFOUR}):
\begin{align} \label{E:ACOUSTICALMETRICGENERALAMBINETSPACETIME}
\hfour_{\alpha \beta}
	: = 
	\normalizer
		\left\lbrace
		\speed^{-2} \gfour_{\alpha \beta} 
		+ 
		(\speed^{-2} - 1) \fourvelocity_{\alpha} \fourvelocity_{\beta}
		\right
		\rbrace,
\end{align}
where
\begin{align} \label{E:GENERALAMBIENTSPACETIMEACOUSTICALMETRICNORMALIZER}
\normalizer
	& := 
		-
		(\gfour^{-1})^{00}
		\speed^2
		+
		(1 - \speed^2) (\fourvelocity^0)^2 > 0,
\end{align}
where the positivity on RHS~\eqref{E:GENERALAMBIENTSPACETIMEACOUSTICALMETRICNORMALIZER}
stems from \eqref{E:SPEEDOFSOUNDASSUMEDBOUNDS} and the assumption that $t$ is a time function for $\gfour$, 
which in particular implies that relative to the coordinates $(x^0,x^1,x^2,x^3)$,
we have
$(\gfour^{-1})^{00} 
= 
(\gfour^{-1})^{\alpha \beta} 
(\partial_{\alpha} t) 
\partial_{\beta} t
< 0
$.
\end{remark}

\subsubsection{The arrays $\solutionarray$ and $\variations$}
\label{SSS:SOLUTIONARRYANDVARIATIONARRAY}
For notational convenience, we define the \emph{solution array}
$\solutionarray$ and the \emph{variation array} $\variations$ as follows:
\begin{align}
	\solutionarray
	& := \left(\Lnenth,\fourvelocity^0,\fourvelocity^1,\fourvelocity^2,\fourvelocity^3,\Ent \right),
		\label{E:SOLUTIONARRAY} 
		\\
	\variations
	& := \left(\dot{\Lnenth},\dot{\fourvelocity}^0,\dot{\fourvelocity}^1,\dot{\fourvelocity}^2,\dot{\fourvelocity}^3,\dot{\Ent} \right).
		\label{E:VARIATIONARRAY}
\end{align}

\subsubsection{Energy currents and energy identity in differential form}
\label{SSS:ENERGYCURRENTSANDIDENTITY}
We now introduce energy currents, which form the basis for energy identities for the equations of variation.

\begin{definition}[Energy currents]
\label{D:ENERGYCURRENTS}
Given a solution $\solutionarray$ to
\eqref{E:VELOCITYEVOLUTIONWITHPROJECTION}--\eqref{E:AGAINFLUIDFOURVELOCITYNORMALIZED}
+
\eqref{E:PROJECTTIONONTOORTHOGONALCOMPLEMENTOFFOURVELOCITY},
variations $\variations = (\dot{\Lnenth},\dot{\fourvelocity}^{\alpha},\dot{\Ent})$
that satisfy \eqref{E:EOVENTHALPYEVOLUTION}--\eqref{E:EOVFLUIDFOURVELOCITYNORMALIZED}, 
and scalar functions $f_1$ and $f_2$, 
we define the corresponding \emph{energy current} vectorfield
$\dot{\mathbf{J}}_{(\solutionarray,f_1,f_2)}$ 
as follows:
	\begin{align} 
	\begin{split}
	\label{E:ENERGYCURRENTVECTORFIELD}
		\dot{\mathbf{J}}_{(\solutionarray,f_1,f_2)}^{\alpha}[\variations,\variations]
		& :=	 
			f_1
			\fourvelocity^{\alpha} \dot{\Ent}^2
			-
			2 \speed^2 \TempoverEnth \dot{\fourvelocity}^{\alpha} \dot{\Ent}
			- 
			2 \TempoverEnth \fourvelocity^{\alpha} \dot{\Ent} \dot{\Lnenth}
			+
			\fourvelocity^{\alpha} \dot{\Lnenth}^2
			+
			2 \speed^2 \dot{\fourvelocity}^{\alpha} \dot{\Lnenth}
			+
			\speed^2 \fourvelocity^{\alpha} \dot{\fourvelocity}_{\kappa} \dot{\fourvelocity}^{\kappa}
				\\
		& \ \
			+
			2 \speed^2 \fourvelocity^{\alpha} \fourvelocity_{\kappa} \dot{\fourvelocity}^{\kappa} \dot{\Lnenth}
			+
			f_2 \fourvelocity^{\alpha} (\fourvelocity_{\kappa} \dot{\fourvelocity}^{\kappa})^2.
	\end{split}
\end{align}
\end{definition}

Note that the components of $\dot{\mathbf{J}}$ can be viewed as quadratic forms in
$\variations$ with coefficients that depend on the solution $\solutionarray$.
Related energy currents were used by Christodoulou in \cite{dC2007}. A general framework for 
constructing energy currents in the context of \emph{regularly hyperbolic} PDEs was
developed by Christodoulou in \cite{dC2000}.

The following lemma provides a differential version of the energy identities
afforded by the energy current. The proof is a tedious but straightforward computation, omitted here
(c.f.\ \cite{dC2007}*{Equation~(1.41)}).

\begin{lemma}[Energy identity in differential form]
\label{L:ENERGYIDENTITYINDIFFERENTIALFORM}
Under the assumptions of Def.\,\ref{D:ENERGYCURRENTS}, with
$\dot{\mathbf{J}}_{(\solutionarray,f_1,f_2)}^{\alpha}[\variations,\variations]$ 
denoting the energy current defined in \eqref{E:ENERGYCURRENTVECTORFIELD},
the following identity holds relative to the Minkowski-rectangular coordinates:
\begin{align} 
\begin{split} \label{E:ENERGYCURRENTDIVERGENCEIDENTITY}
	\partial_{\kappa} \dot{\mathbf{J}}_{(\solutionarray,f_1,f_2)}^{\kappa}[\variations,\variations]
	& = 2 f_1 \dot{\Ent} \dot{\mathfrak{H}} 
			-
			2 \TempoverEnth \dot{\Lnenth} \dot{\mathfrak{H}}
			-
			2 \TempoverEnth \dot{\Ent} \dot{\mathfrak{F}}
			+
			2 \dot{\Lnenth} \dot{\mathfrak{F}}
			+
			2 \speed^2 \dot{\fourvelocity}_{\kappa} \dot{\mathfrak{G}}^{\kappa}
			+ 
			2 \speed^2 \dot{\Lnenth} \fourvelocity_{\kappa} \dot{\mathfrak{G}}^{\kappa}
				\\
		& \ \
			+
			2 f_2 \TempoverEnth \fourvelocity_{\kappa} \dot{\fourvelocity}^{\kappa} \dot{\mathfrak{H}}
			+
			2 f_2 \fourvelocity_{\kappa} \dot{\fourvelocity}^{\kappa} \fourvelocity_{\lambda} \dot{\mathfrak{G}}^{\lambda}
			+
			\mathscr{Q}_{(\solutionarray,f_1,f_2,\bm{\partial} \solutionarray,\bm{\partial} f_1, \bm{\partial} f_2)}[\variations,\variations],
\end{split}
\end{align}
where 
$\mathscr{Q}_{(\solutionarray,f_1,f_2,\bm{\partial} \solutionarray,\bm{\partial} f_1, \bm{\partial} f_2)}[\variations,\variations]$ 
is a quadratic form in $\variations$ with coefficients that depend on
$\solutionarray$, $f_1$, $f_2$, and their first-order partial derivatives; 
it is straightforward to compute $\mathscr{Q}$
in detail, though we do not do so here.
\end{lemma}

By integrating \eqref{E:ENERGYCURRENTDIVERGENCEIDENTITY} over a spacetime region and applying the divergence theorem,
one obtains an energy identity, where the energies are defined along the boundary of the region. A fundamental question is:
when are the energies obtained in this fashion positive definite\footnote{In general, energies are not useful if they are indefinite.
\label{FN:INDEFINITEENERGIESNOTUSEFUL}}? 
The answer to this question, provided by Lemma~\ref{L:POSITIVEDEFINITENESSOFENERGYCURRENT}, 
is tied to the acoustic geometry of the system, captured by the acoustical metric.
In short, the energies are positive definite whenever the boundary of the region 
has $\hfour$-timelike normals, i.e., whenever the boundary is spacelike with respect to $\hfour$.

Before proving the lemma, we first introduce a subset of state-space in which the 
relativistic Euler equations are hyperbolic.

\begin{definition}[Regime of hyperbolicity]
	\label{D:REGIMEOFHYPERBOLICITY}
	We define the \emph{regime of hyperbolicity} $\mathcal{H}$ to be the following subset of solution variable space:
	\begin{align} \label{E:REGIMEOFHYPERBOLICITY}
		\mathcal{H}
		& := \left\lbrace
					\left(\Lnenth,\fourvelocity^0,\fourvelocity^1,\fourvelocity^2,\fourvelocity^3,\Ent \right) \in \mathbb{R}^6
					\ | \ 0 < \speed(\Lnenth,\Ent) \leq 1
			 \right\rbrace.
	\end{align}
\end{definition}

\begin{lemma}[Positive definiteness of {$\dot{\mathbf{J}}_{(\solutionarray,f_1,f_2)}[\variations,\variations]$}]
\label{L:POSITIVEDEFINITENESSOFENERGYCURRENT}
Let $\upxi$ be a past-directed $\hfour$-timelike one-form, 
i.e., assume that $(\hfour^{-1})^{\kappa \lambda} \upxi_{\kappa} \upxi_{\lambda} < 0$ and $\upxi_0 := \upxi(\partial_t) > 0$.
Let $\mathfrak{K}$ be a compact subset of the set $\mathcal{H}$ defined in \eqref{E:REGIMEOFHYPERBOLICITY}, 
and assume that the solution array $\solutionarray$ defined in \eqref{E:SOLUTIONARRAY} satisfies 
$\solutionarray \in \mathfrak{K}$.
Then if $f_1$ and $f_2$ are sufficiently large (in a manner that depends on $\upxi$ and $\mathfrak{K}$), 
there exists a $C = C(\upxi,\mathfrak{K},f_1,f_2) > 0$ such that
the quadratic form $\upxi_{\alpha} \dot{\mathbf{J}}_{(\solutionarray,f_1,f_2)}^{\alpha}[\variations,\variations]$
is positive definite in the following sense:
\begin{align} \label{E:POSITIVEDEFINITENESSOFENERGYCURRENT}
	\upxi_{\kappa} \dot{\mathbf{J}}_{(\solutionarray,f_1,f_2)}^{\kappa}[\variations,\variations]
	& \geq 
		C |\variations|^2.
\end{align}
\end{lemma}

\begin{remark}[Motivation for Lemmas \ref{L:ENERGYIDENTITYINDIFFERENTIALFORM}--\ref{L:POSITIVEDEFINITENESSOFENERGYCURRENT}]
Although Lemmas~\ref{L:ENERGYIDENTITYINDIFFERENTIALFORM}--\ref{L:POSITIVEDEFINITENESSOFENERGYCURRENT} might seem at first glance like unmotivated technical results, they sit within the general framework of \emph{multiplier methods} for hyperbolic equations. To give a rough idea of the multiplier method, consider a Schwartz-class solution $\phi$ to the linear wave equation 
$-\p_t^2\phi + \Delta_x \phi = 0$ on Minkowski spacetime $(\R^{1+3},\mink)$.
We define the \emph{energy current} vectorfield $\mathbf{J}$ to be the vectorfield with the following components relative
the standard Minkowski-rectangular coordinates:
\begin{subequations}
\begin{align}
\mathbf{J}^0 
& := \frac{1}{2} (\partial_t \phi)^2 + \frac{1}{2} |\nabla_x \phi|^2,
	&&
	\label{E:WAVEENERGYCURRENT0COMPONENT}
	\\
\mathbf{J}^i
& := - (\p_i \phi) \p_t \phi,
&& (i=1,2,3).
\label{E:WAVEENERGYCURRENTSPACECOMPONENTS}
\end{align}
\end{subequations}
In \eqref{E:WAVEENERGYCURRENT0COMPONENT}, $\nabla_x \phi := (\p_1 \phi,\phi_2,\phi_3)$ denotes the spatial gradient of $\phi$
and $| \cdot |$ is the standard Euclidean norm on $\R^3$. 
Then straightforward calculations yield that $\partial_{\alpha} \mathbf{J}^\alpha = 0$ and that $(\p_t)_\alpha \mathbf{J}^\alpha = \tfrac{1}{2} (\p_t \phi)^2 + \tfrac{1}{2} |\nabla_x \phi|^2$. These two identities are analogs of \eqref{E:ENERGYCURRENTDIVERGENCEIDENTITY} and 
\eqref{E:POSITIVEDEFINITENESSOFENERGYCURRENT} respectively. 
The upshot is that by integrating $\partial_{\alpha} \mathbf{J}^\alpha = 0$ 
over $[0,t] \times \mathbb{R}^3$, one obtains the celebrated conservation of energy formula for solutions
to the linear wave equation:
\begin{align} \label{E:CONSERVATIONOFENERGYINLINEARWAVEEQ}
	\frac{1}{2}
	\int_{\{t\} \times \R^3}  
	\left\{(\p_t \phi)^2 + |\nabla_x \phi|^2
	\right\}\mathrm{d} x 
	& =  
	\frac{1}{2}
	\int_{\{0\} \times \R^3}  \left\{(\p_t\phi)^2 + |\nabla_x \phi|^2 \right\} \mathrm{d} x. 
\end{align}
The approach is called the ``multiplier method'' because it is a geometric version of the more standard proof of 
\eqref{E:CONSERVATIONOFENERGYINLINEARWAVEEQ},
in which one multiplies the wave equation by $\p_t \phi$ and integrates by parts over $[0,t] \times \mathbb{R}^3$.

\end{remark}

\begin{proof}
	Throughout, $C =C(\upxi,\mathfrak{K},f_1,f_2)$ 
	denotes a positive constant (often smaller than $1$) that can vary from line to line.
	
	We start by defining the scalar function $\upxi_{\parallel}$ by:
	\begin{align} \label{E:XICONTRACTIWITHFOURVELOCITY}
		\upxi^{\parallel}
		& := \upxi_{\kappa} \fourvelocity^{\kappa}.
	\end{align}
	Then, using \eqref{E:AGAINFLUIDFOURVELOCITYNORMALIZED}, we decompose:
	\begin{align} \label{E:DEOCMPOSEXIINTOPARALLELANDPERPENDICULAR}
		\upxi_{\alpha}
		& = 
			- 
			\upxi^{\parallel} \fourvelocity_{\alpha}
			+
			\upxi_{\alpha}^{\perp},
	\end{align}
	where the one-form $\upxi_{\alpha}^{\perp}$ satisfies
	$\upxi_{\kappa}^{\perp} \fourvelocity^{\kappa} = 0$.
	In particular, $\upxi_{\alpha}^{\perp}$ is $\mink$-orthogonal to the $\mink$-timelike vectorfield
	$\fourvelocity^{\alpha}$ and thus $\upxi_{\alpha}^{\perp}$ is $\mink$-spacelike whenever it is non-zero,
	i.e., $(\mink)^{-1} \upxi_{\kappa}^{\perp} \upxi_{\lambda}^{\perp} \geq 0$.
	Since $(\hfour^{-1})^{\kappa \lambda} \upxi_{\kappa} \upxi_{\lambda} < 0$ and 
	$\hfour_{\kappa \lambda} \fourvelocity^{\kappa} \fourvelocity^{\lambda} < 0$
	(see \eqref{E:GFOURLENGTHOFFOURVELOCITY}),
	and since $\upxi_{\alpha}$ is past-directed while $\fourvelocity^{\alpha}$ is future-directed,
	it follows that there is a $C$ with $0 < C < 1$ such that:
	\begin{align} \label{E:BOUNDFORXIPARALLEL}
		C
		& \leq
		\upxi^{\parallel}
		\leq \frac{1}{C}.
	\end{align}
	
	We can similarly decompose $\dot{\fourvelocity}^{\alpha}$ into a variation that is parallel to
	$\fourvelocity^{\alpha}$ and one that vanishes upon contraction with $\fourvelocity_{\alpha}$.
	We claim that prove \eqref{E:POSITIVEDEFINITENESSOFENERGYCURRENT},
	it suffices to show that when $f_1 = f_2 = 0$, $\dot{\Ent} = 0$,
	and $\fourvelocity_{\kappa} \dot{\fourvelocity}^{\kappa} = 0$,
	we have:
	\begin{align} \label{E:REDUCEDPOSITIVEDEFINITENESSSTATEMENT}
		\upxi_{\kappa} \dot{\mathbf{J}}^{\kappa}[\variations,\variations]
		& \geq C (\dot{\Lnenth}^2 + \dot{\fourvelocity}_{\kappa} \dot{\fourvelocity}^{\kappa}).
	\end{align}
	The estimate \eqref{E:REDUCEDPOSITIVEDEFINITENESSSTATEMENT} would yield control over $\dot{\Lnenth}^2$ and variations such that
	$\dot{\Ent} = \fourvelocity_{\kappa} \dot{\fourvelocity}^{\kappa} = 0$
	(note that $\dot{\fourvelocity}^{\alpha}$ is $\mink$-spacelike for such variations and thus the
	term $\dot{\fourvelocity}_{\kappa} \dot{\fourvelocity}^{\kappa}$ on RHS~\eqref{E:REDUCEDPOSITIVEDEFINITENESSSTATEMENT}
	is positive-definite in such variations).
	Then
	to handle general variations,
	since the first and last
	terms on RHS~\eqref{E:ENERGYCURRENTVECTORFIELD} 
	yield the two coercive (in view of \eqref{E:BOUNDFORXIPARALLEL}) terms 
	$f_1 \upxi^{\parallel} \dot{\Ent}^2$ and $f_2 \upxi^{\parallel} (\fourvelocity_{\kappa} \dot{\fourvelocity}^{\kappa})^2$,
	we can choose $f_1$ and $f_2$ to be sufficiently large 
	(for example, we can choose them to be large, positive constants)
	to obtain control over
	$\dot{\Ent}^2$ and $(\fourvelocity_{\kappa} \dot{\fourvelocity}^{\kappa})^2$, and large enough so that these
	two coercive terms 
	and the coercivity \eqref{E:REDUCEDPOSITIVEDEFINITENESSSTATEMENT}
	can collectively be used to absorb (via Young's inequality) 
	all the remaining terms on RHS~\eqref{E:ENERGYCURRENTVECTORFIELD}. 
	In total, this would yield \eqref{E:POSITIVEDEFINITENESSOFENERGYCURRENT}.
	
	Hence, to complete the proof of \eqref{E:POSITIVEDEFINITENESSOFENERGYCURRENT}, we can 
	assume that $f_1 = f_2 = 0$, $\dot{\Ent} = 0$,
	and $\fourvelocity_{\kappa} \dot{\fourvelocity}^{\kappa} = 0$.
	This implies that $\dot{\fourvelocity}^{\alpha}$ is $\mink$-spacelike.
	We use \eqref{E:AGAINFLUIDFOURVELOCITYNORMALIZED},
	\eqref{E:ENERGYCURRENTVECTORFIELD}, 
	\eqref{E:XICONTRACTIWITHFOURVELOCITY},
	and
	\eqref{E:DEOCMPOSEXIINTOPARALLELANDPERPENDICULAR}
	to compute that:
	\begin{align} \label{E:REDUCEDPOSITIVEDEFINITENESSEQUALITY}
		\upxi_{\kappa} \dot{\mathbf{J}}^{\kappa}[\variations,\variations]
		 & = 
			\upxi^{\parallel} \dot{\Lnenth}^2 
			+
			2 \speed^2 \upxi_{\kappa}^{\perp} \dot{\fourvelocity}^{\kappa} \dot{\Lnenth}
			+
			\speed^2 \upxi^{\parallel} \dot{\fourvelocity}_{\kappa} \dot{\fourvelocity}^{\kappa}.
	\end{align}
	Since $(\hfour^{-1})^{\kappa \lambda} \upxi_{\kappa} \upxi_{\lambda} < - C < 0$,
	$C < \normalizer < \frac{1}{C}$,
	and $\fourvelocity_{\kappa} \dot{\fourvelocity}^{\kappa} = 0$,
	we compute using 
	\eqref{E:AGAINFLUIDFOURVELOCITYNORMALIZED},
	\eqref{E:INVERSEACOUSTICALMETRIC}, 
	\eqref{E:XICONTRACTIWITHFOURVELOCITY},
	and \eqref{E:DEOCMPOSEXIINTOPARALLELANDPERPENDICULAR}
	that there is a $C' > 0$ such that:
	\begin{align} \label{E:KEYBOUNDFOR}
		\normalizer
		(\hfour^{-1})^{\kappa \lambda} \upxi_{\kappa} \upxi_{\lambda}
		& = 
			- 
			(\upxi^{\parallel})^2
			+ 
			\speed^2
			(\mink^{-1})^{\kappa \lambda} 
			\upxi_{\kappa}^{\perp}
			\upxi_{\lambda}^{\perp}
			< - C'.
	\end{align}
	We are denoting the positive constant on RHS~\eqref{E:KEYBOUNDFOR} by ``$C'$'' (rather than ``$C$'') for clarity.
	Since $\upxi_{\alpha}^{\perp}$ is $\mink$-spacelike, we have
	$\speed^2
			(\mink^{-1})^{\kappa \lambda} 
			\upxi_{\kappa}^{\perp}
			\upxi_{\lambda}^{\perp} \geq 0$
	and thus $C' < (\upxi^{\parallel})^2$.
	Since $\upxi_{\alpha}^{\perp}$ and $\dot{\fourvelocity}^{\alpha}$ are
	$\mink$-spacelike,
	for any $\upsigma > 0$,
	we can use 
	\eqref{E:KEYBOUNDFOR},
	the Cauchy--Schwarz inequality with respect to $\mink$,
	and Young's inequality 
	to bound the cross term on RHS~\eqref{E:REDUCEDPOSITIVEDEFINITENESSEQUALITY} as follows:
	\begin{align} 
	\begin{split} \label{E:KEYBOUNDFORCROSSTERM}
		|2 \speed^2 \upxi_{\kappa}^{\perp} \dot{\fourvelocity}^{\kappa} \dot{\Lnenth}|
		& \leq 
			\frac{1}{\upsigma}
			\speed^4
			(\mink^{-1})^{\alpha \beta} 
			\upxi_{\alpha}^{\perp}
			\upxi_{\beta}^{\perp}
			\dot{\fourvelocity}_{\kappa} \dot{\fourvelocity}^{\kappa}
			+
			\upsigma
			\dot{\Lnenth}^2
				\\
		& 
		\leq
		\frac{1}{\upsigma}
		\speed^2
		\left\lbrace
			(\upxi^{\parallel})^2
			-
			C'
		\right\rbrace
		\dot{\fourvelocity}_{\kappa} \dot{\fourvelocity}^{\kappa}
		+
		\upsigma
		\dot{\Lnenth}^2
			\\
	& 
		\leq
		\frac{1}{\upsigma}
		\speed^2
		(\upxi^{\parallel})^2
		\left\lbrace
			1
			-
			C' C^2 
		\right\rbrace
		\dot{\fourvelocity}_{\kappa} \dot{\fourvelocity}^{\kappa}
		+
		\upsigma
		\dot{\Lnenth}^2,
	\end{split}
	\end{align}
	where the last $C$ on RHS~\eqref{E:KEYBOUNDFORCROSSTERM} is from \eqref{E:BOUNDFORXIPARALLEL}.
	We now set $\upsigma := \upxi^{\parallel} 
	\left\lbrace
			1
			-
			C' C^2 
		\right\rbrace^{1/2}$. It follows that:
	\begin{align} \label{E:FINALKEYBOUNDFORCROSSTERM}
		|2 \speed^2 \upxi_{\kappa}^{\perp} \dot{\fourvelocity}^{\kappa} \dot{\Lnenth}|
		& \leq 
		\speed^2
		\upxi^{\parallel}
		\left\lbrace
			1
			-
			C' C^2 
		\right\rbrace^{1/2}
		\dot{\fourvelocity}_{\kappa} \dot{\fourvelocity}^{\kappa}
		+
		\upxi^{\parallel}
		\left\lbrace
			1
			-
			C' C^2 
		\right\rbrace^{1/2}
		\dot{\Lnenth}^2.
	\end{align}
	Thus,
	since $1
			-
			C' C^2 < 1$,
	we deduce from \eqref{E:FINALKEYBOUNDFORCROSSTERM} 
	that the cross term on RHS~\eqref{E:REDUCEDPOSITIVEDEFINITENESSEQUALITY} can be absorbed
	into the two positive definite terms with room to spare.
	We have therefore proved \eqref{E:REDUCEDPOSITIVEDEFINITENESSSTATEMENT}
	assuming $f_1 = f_2 = 0$, $\dot{\Ent} = 0$,
	and
	$\fourvelocity_{\kappa} \dot{\fourvelocity}^{\kappa} = 0$,
	which completes the proof of the lemma.
\end{proof}

\subsubsection{Local well-posedness and continuation principle in Sobolev spaces}
\label{SSS:LWPFIRSTORDER}
In Prop.\,\ref{P:LWPFIRSTORDER}, we discuss local well-posedness for the relativistic Euler equations in Sobolev spaces.

\begin{proposition}[Local well-posedness and continuation principle in Sobolev spaces]
	\label{P:LWPFIRSTORDER}
	Recall the notation $\solutionarray$ introduced in \eqref{E:SOLUTIONARRAY}.
	Let 
	$\mathring{\solutionarray}
	=
		(\mathring{\Lnenth},\mathring{\fourvelocity}^0,\mathring{\fourvelocity}^1,\mathring{\fourvelocity}^2,
		\mathring{\fourvelocity}^3,\mathring{\Ent})
	:=\solutionarray|_{t=0}$
	be initial data for the relativistic Euler equations \eqref{E:ENTHALPYEVOLUTION}--\eqref{E:AGAINFLUIDFOURVELOCITYNORMALIZED}
	(in particular, assume that $\mathring{\fourvelocity}$ satisfies \eqref{E:AGAINFLUIDFOURVELOCITYNORMALIZED}).
	Assume that there is an integer $N \geq 3$ such that 
	$
	\mathring{\Lnenth},\mathring{\fourvelocity}^1,\mathring{\fourvelocity}^2,\mathring{\fourvelocity}^3, \mathring{\Ent}
	\in H^N(\mathbb{R}^3)
	$ 
	(and thus \eqref{E:AGAINFLUIDFOURVELOCITYNORMALIZED} and the standard Sobolev--Moser calculus imply that
	$\mathring{\fourvelocity}^0 - 1 \in H^N(\mathbb{R}^3)$)
	and that $\mathring{\solutionarray}(\mathbb{R}^3)$ is contained in a compact subset
	of the interior of set $\mathcal{H}$ defined in \eqref{E:REGIMEOFHYPERBOLICITY}.
	Then there is a maximal $T_{\textnormal{Lifespan}} > 0$ such that a unique solution to 
	\eqref{E:ENTHALPYEVOLUTION}--\eqref{E:AGAINFLUIDFOURVELOCITYNORMALIZED}
	exists and satisfies 
	$\Lnenth, \fourvelocity^0 - 1, \fourvelocity^1, \fourvelocity^2, \fourvelocity^3, \Ent 
	\in 
	C^0\left([0,T_{\textnormal{Lifespan}}),H^N(\mathbb{R}^3) \right) 
	\cap 
	C^1\left([0,T_{\textnormal{Lifespan}}), H^{N-1}(\mathbb{R}^3) \right)
	$. 
	Moreover, either $T_{\textnormal{Lifespan}} = \infty$, or one of the following two breakdown scenarios occurs:
	\begin{enumerate}
		\item  (\textbf{Blowup}). The solution's $C^1$ norm blows up:
			\begin{align} \label{E:C1BLOWUPSCENARIO}
				\lim_{t \uparrow T_{\textnormal{Lifespan}}}
				\| \solutionarray \|_{C^1([0,t] \times \mathbb{R}^3)}
				= \infty.
			\end{align}
			\item  (\textbf{Exiting the regime of hyperbolicity}).
				There exists a sequence of points 
				$\lbrace (t_m,x_m) \rbrace_{m \in \mathbb{N}} \subset [0,T_{\textnormal{Lifespan}} ) \times \mathbb{R}^3$
				such that for every compact subset $\mathfrak{K}$ of the interior of $\mathcal{H}$,
				there exists an $m_{\mathfrak{K}} \in \mathbb{N}$ such that
				$\solutionarray(t_m,x_m)
				\notin \mathfrak{K}
				$
				whenever $m \geq m_{\mathfrak{K}}$.
	\end{enumerate}
\end{proposition}

We make the following remarks.

\begin{itemize}
	\item In Sects.\,\ref{S:1DMAXIMALDEVELOPMENT} and \ref{S:SHOCKFORMATIONAWAYFROMSYMMETRY},  we will study the formation
		of shock singularities, in which the solution remains in the interior of the regime of hyperbolicity and the breakdown is precisely 
		due to the scenario \eqref{E:C1BLOWUPSCENARIO}.
	\item The techniques of \cite{tHtKjM1976}, which rely only on energy estimates, Sobolev embedding, and the fractional
		Sobolev--Moser calculus, can be used to save half a derivative compared to Prop.\,\ref{P:LWPFIRSTORDER}, that is, to
		prove local well-posedness for the $3D$ relativistic Euler equations whenever
		$
		\mathring{\Lnenth},\mathring{\fourvelocity}^1,\mathring{\fourvelocity}^2,\mathring{\fourvelocity}^3, \mathring{\Ent}
		\in H^{{5/2}^+}(\mathbb{R}^3)
		$.
	\item In \cite{sY2022}, a low-regularity
			local well-posedness result for the $3D$ relativistic Euler equations was proved. The low-regularity assumption
			was on the ``wave-part'' of the initial data, which was assumed to belong to $H^{2^+}(\mathbb{R}^3)$, 
			while the ``transport-part'' (including the vorticity and entropy) was assumed to satisfy some
		additional smoothness assumptions, which were shown to be propagated by the flow;
		the terminology ``wave-part'' and ``transport'' 
		is tied to the formulation of the flow provided by Theorem~\ref{T:GEOMETRICWAVETRANSPORTDIVCURLFORMULATION}
		and is explained in \cite{sY2022}.
		Lindblad's work \cite{hL1998} implies that the assumptions on the wave-part
		of the initial data are \emph{optimal} in the scale of Sobolev spaces due to the
		phenomenon of instantaneous shock formation for initial data in $H^2(\mathbb{R}^3)$.
		The analysis in \cite{sY2022}
		relies on the new formulation of the flow presented in Theorem~\ref{T:GEOMETRICWAVETRANSPORTDIVCURLFORMULATION},
		which allows one to employ analytical techniques based on nonlinear geometric optics and Strichartz estimates.
		The techniques used in \cite{sY2022} have roots in the works
		\cites{mDcLgMjS2022,qW2022} (see also \cites{zH2020,zHlA2021,zHlA2022})
		on low-regularity well-posedness for the non-relativistic compressible Euler equations,
		which in turn have roots in the works
		\cites{sKiR2003,sKiR2005d,sKiR2006a,sKiR2005b,sKiR2005c,sKiR2008,hSdT2005,qW2017}
		on low-regularity well-posedness for quasilinear wave equations.
\end{itemize}

\begin{proof}[Discussion of proof of Prop.\,\ref{P:LWPFIRSTORDER}]
This is a standard result that can be proved by linearizing the equations,
deriving a priori estimates for solutions to the linearized equations, and
then invoking a standard iteration scheme.
We refer to \cites{jS2008a,jS2008b,jS2008c} for the main ideas, where an analog of Prop.\,\ref{P:LWPFIRSTORDER} 
was proved for the coupling of the relativistic Euler equations to
Nordstr\"{o}m's theory of gravity. The main idea behind
the a priori estimates (in $H^N$) is to construct energy currents for the linearized systems
such that analogs of \eqref{E:ENERGYCURRENTDIVERGENCEIDENTITY}
and \eqref{E:POSITIVEDEFINITENESSOFENERGYCURRENT} hold,
and to control RHS~\eqref{E:ENERGYCURRENTDIVERGENCEIDENTITY}
in $L^2$ using the Sobolev--Moser calculus (which is standard in $3D$ when $N \geq 3$).
\end{proof}

\section{A study of the maximal development in simple isentropic plane-symmetry}
\label{S:1DMAXIMALDEVELOPMENT}    
Our goal in this section is to prove Theorem~\ref{T:MAINTHEOREM1DSINGULARBOUNDARYANDCREASE}, which yields
a complete description of a localized portion of the boundary of the maximal (classical) $\hfour$-globally hyperbolic development ($\hfour$-MGHD for short) 
for open sets of shock-forming solutions to the relativistic Euler equations \eqref{E:ENTHALPYEVOLUTION}--\eqref{E:ENTROPYEVOLUTION}
in simple isentropic plane-symmetry; see Fig.\,\ref{F:1DPLANESYMMETRYCAUCHYHORANDSINGULARCURVE}.
Roughly, the $\hfour$-MGHD is the largest possible classical solution + $\hfour$-globally hyperbolic region that is launched by the initial data.
By a $\hfour$-globally hyperbolic region $\mathfrak{R}$ of classical existence launched by the initial data, we mean that every inextendible
$\hfour$-causal curve in $\mathfrak{R}$ intersects $\Sigma_0 := \lbrace t = 0 \rbrace$,
where $\hfour$ is the acoustical metric from Def.\,\ref{D:ACOUSTICALMETRIC}. 
In particular, the $\hfour$-MGHD cannot contain any points lying in the $\hfour$-causal future of any singularity,
where the $\hfour$-causal future of a point is effectively determined by integral curves of
the two transversal null vectorfields $\Lunit$ and $\uLunit$ from
Def.\,\ref{D:NULLVECTORFIELDSINPLANESYMMETRY}.

The portion of the boundary that we study includes a singular boundary, where the fluid's gradient blows up in a shock singularity, 
and a Cauchy horizon, along which the solution remains smooth.  
For the solutions under study, the past boundary of the Cauchy horizon intersects the past boundary of the singular boundary 
in a distinguished point that we refer to as ``the crease;'' see Fig.\,\ref{F:MINKOWSKIRECTANGULAR1DPLANESYMMETRYCAUCHYHORANDSINGULARCURVE}. 
The crease plays the role of the true initial singularity in
a shock-forming relativistic Euler solution. 

Sect.\,\ref{S:1DMAXIMALDEVELOPMENT} is intended to provide a gentle introduction to the main ideas behind the study of the $\hfour$-MGHD
in a simplified setting in which the geometry is easy to study and one can avoid the burden of energy estimates.
It will also help set the stage for Sect.\,\ref{S:SHOCKFORMATIONAWAYFROMSYMMETRY}, in which we 
outline some of the main ideas needed to study shock formation in $3D$.

\subsection{A description of the symmetry class}
\label{SS:SIMPLEISENTROPICPLANESYMMETRY}
As we mentioned, we will study simple isentropic plane-symmetric solutions, which we now describe.
First, by ``isentropic solutions,'' we mean that $\Ent \equiv 0$;
for any constant $\Ent_0$, we could treat the case $\Ent \equiv \Ent_0$ using the same arguments.
From \eqref{E:ENTROPYEVOLUTION}, we infer that for $C^1$ isentropic solutions, 
the condition $\Ent = \Ent_0$
is preserved by the flow if it is satisfied at $t = 0$. By ``plane-symmetry,'' we mean that
$\fourvelocity^2 = \fourvelocity^3 \equiv 0$ and that $\Lnenth = \Lnenth(t,x^1)$, $\fourvelocity^1 = \fourvelocity^1(t,x^1)$, and $\fourvelocity^0 = \fourvelocity^0(t,x^1)$. 
Such solutions can be viewed as solutions in $1D$ in which $\Ent \equiv 0$, and throughout
the rest of the paper,
we make that identification. By ``simple,'' we mean that only one Riemann invariant 
(see Sect.\,\ref{SS:RIEMANNINVARIANTS}) is non-vanishing.

\subsection{The basic setup and conventions in plane-symmetry}
\label{SS:SETUPANDCONVENTIONSINPLANESYMMETRY}
In Sect.\,\ref{S:1DMAXIMALDEVELOPMENT}, since we are studying isentropic plane-symmetric solutions, 
we will slightly abuse notation and use $\hfour$ to denote the restriction of
the acoustical metric \eqref{E:ACOUSTICALMETRIC}--\eqref{E:INVERSEACOUSTICALMETRIC}, 
to the $(t,x^1)$-plane. That is, in Sect.\,\ref{S:1DMAXIMALDEVELOPMENT}, relative to the $(t,x^1)$-coordinates,
we have:
\begin{subequations}
\begin{align} \label{E:1DACOUSTICALMETRIC}
		\hfour_{\alpha \beta}
	& : = 
		\normalizer
		\left\lbrace
		\speed^{-2} \mink_{\alpha \beta} 
		+ 
		(\speed^{-2} - 1) \fourvelocity_{\alpha} \fourvelocity_{\beta}
		\right
		\rbrace,
		&
		(\alpha, \beta 
		& = 0,1),
					\\
	(\hfour^{-1})^{\alpha \beta}
	& :
		= 
		\normalizer^{-1}
		\left\lbrace
		\speed^2 (\mink^{-1})^{\alpha \beta} 
			+ 
			(\speed^2 - 1)
			\fourvelocity^{\alpha} \fourvelocity^{\beta}
		\right\rbrace,
		&
		(\alpha, \beta 
		& = 0,1),
			\label{E:1DINVERSEACOUSTICALMETRIC}
\end{align}
\end{subequations}
where $\normalizer$ is defined by \eqref{E:ACOUSTICALMETRICNORMALIZER},
and we recall that \eqref{E:GINVERSE00ISMINUSONE} holds.

For future use, we note that for $C^1$  isentropic plane-symmetric solutions,
the relativistic Euler equations \eqref{E:ENTHALPYEVOLUTION}--\eqref{E:ENTROPYEVOLUTION} 
in Minkowski spacetime are equivalent to:
\begin{subequations}
\begin{align}
	\fourvelocity^{\kappa} \partial_{\kappa} \Lnenth
	+
	\speed^2 \partial_{\kappa} \fourvelocity^{\kappa}
	& = 0,
		\label{E:1DENTHALPYEVOLUTION} 
			\\
	\fourvelocity^{\kappa} \partial_{\kappa} \fourvelocity^1
	+ 
	\Pi^{1 \kappa} \partial_{\kappa} \Lnenth
	& = 0,
	\label{E:1DVELOCITYEVOLUTIONWITHPROJECTION}
		\\
	 \fourvelocity^0 
	& = \sqrt{1 + (\fourvelocity^1)^2}.
	\label{E:1DFLUIDFOURVELOCITYNORMALIZED}
\end{align}
\end{subequations}
For future use, we also note the following algebraic identity, which holds in plane-symmetry due to 
\eqref{E:ACOUSTICALMETRICNORMALIZER}
and
\eqref{E:1DFLUIDFOURVELOCITYNORMALIZED}:
\begin{align} \label{E:NORMAIZERFACTORSIN1D}
\normalizer 
& = (\fourvelocity^0 + \fourvelocity^1 \speed)(\fourvelocity^0 - \fourvelocity^1 \speed).
\end{align}

To ensure ensure that shocks can form in solutions near constant fluid states with constant enthalpy $\overline{\Enth} > 0$, 
we assume the following non-degeneracy assumption: 
\begin{align} \label{E:NONDEGENARCYASSUMPTION}
	1 - (\overline{\speed})^2 + \frac{\overline{\speed}'}{\overline{\speed} \overline{\Enth}}
	& \neq 0,
\end{align}
where in \eqref{E:NONDEGENARCYASSUMPTION}, we are viewing $\speed = \speed(\Enth)$,
$\speed' := \frac{\mathrm{d}}{\mathrm{d} \Enth} \speed$, 
and LHS \eqref{E:NONDEGENARCYASSUMPTION} is defined to be the constant obtained by evaluating 
$1 - \speed^2 + \frac{\speed'}{\Enth \speed}$ at $\Enth = \overline{\Enth}$. 
The condition \eqref{E:NONDEGENARCYASSUMPTION} implies in particular that
the isentropic plane-symmetric relativistic Euler equations are genuinely nonlinear near the
constant-state solution $(\Enth,\fourvelocity^1) \equiv (\overline{\Enth},0)$; 
see also Sect.\,\ref{SS:1DRICCATIBLOWUP}.
Christodoulou showed \cite{dC2007}*{Chapter~1} that irrotational relativistic fluids 
(of which isentropic plane-symmetric fluids are a special case)
admit a Lagrangian formulation for a potential function, 
and the only equation of state leading to
$1 - \speed^2 + \frac{\speed'}{\Enth \speed} \equiv 0$ 
is such that the potential function solves the
timelike minimal surface equation.\footnote{For the minimal surface equation, 
$\speed \equiv 1$ corresponds to the trivial embedding $\R^{1+d} \hookrightarrow \R^{1 + (d+1)}$ with respect to the Minkowski metric.} 
For that equation, which satisfies a nonlinear version of the null condition that is stronger than the 
classic null condition identified by Klainerman \cite{sK1984},
the first author proved that perturbations of plane-symmetric data 
(which are similar to the data we study here in Sect.\,\ref{S:1DMAXIMALDEVELOPMENT}) 
lead to global, nonlinearly stable $C^2$ solutions in both $\R^{1+1}$ and $\R^{1+3}$ 
\cites{abbrescia2019geometric,abbrescia2020global}.

\subsection{Null frame}
\label{SS:1DNULLFRAME}
In $1D$, it is possible to explicitly write down a pair of transversal null vectorfields that dictate the geometry of sound waves. 
As we explain in Sect.\,\ref{SS:NONLINEARGEOMETRICOPTICS}, in more than one spatial dimension, it is not possible to explicitly write down
null vectorfields in terms of the fluid variables. Rather, in that context, the null vectorfields depend on the gradient
of an eikonal function $\eik$.

\begin{definition}[Null vectorfields in $1D$] \label{D:NULLVECTORFIELDSINPLANESYMMETRY}
Relative to the Minkowski-rectangular coordinates $(t,x^1)$, 
we define the vectorfields $\Lunit$ and $\uLunit$ as follows:
\begin{subequations}
\begin{align}
	\Lunit
	& :=  \partial_t
				+
				\frac{\frac{\fourvelocity^1}{\fourvelocity^0} + \speed}{1 + \frac{\fourvelocity^1}{\fourvelocity^0} \speed} \partial_1,
			\label{E:1DLUNIT} \\
	\uLunit
	& :=  \partial_t
				+
				\frac{\frac{\fourvelocity^1}{\fourvelocity^0} - \speed}{1 - \frac{\fourvelocity^1}{\fourvelocity^0} \speed} \partial_1.
				\label{E:1DULUNIT}
\end{align}
\end{subequations}
\end{definition}

Straightforward calculations 
based on \eqref{E:1DACOUSTICALMETRIC}
and \eqref{E:1DLUNIT}--\eqref{E:1DULUNIT}
yield the following identities:
\begin{align} \label{E:GFOURINNERPRODUCTOFNULLVECTORFIELDSIN1D}
	\hfour(\Lunit,\Lunit)
	= 
	\hfour(\uLunit,\uLunit)
	& = 0,
	&
\hfour(\Lunit,\uLunit) 
& = -2.
\end{align}
In particular, \eqref{E:GFOURINNERPRODUCTOFNULLVECTORFIELDSIN1D} shows that $\Lunit$ and $\uLunit$ are $\hfour$-null.
The pair $\lbrace \Lunit, \uLunit \rbrace$ is known as a $\hfour$-\emph{null frame}.

From \eqref{E:GFOURINNERPRODUCTOFNULLVECTORFIELDSIN1D} and straightforward calculations,
we deduce the following lemma.

\begin{lemma}[$\hfour^{-1}$ in terms of $\Lunit$ and $\uLunit$]
	\label{L:1DGFOURINVERSEINTERMSOFLUNITANDULUNIT}
	In $1 + 1$ spacetime dimensions,
	the inverse acoustical metric $\hfour^{-1}$ from \eqref{E:1DINVERSEACOUSTICALMETRIC} satisfies:
	\begin{align} \label{E:1DACOUSTICALMETRICINVERSEINTERMSOFNULLVECTORFIELDS}
		(\hfour^{-1})^{\alpha \beta}
		& = - 
				\frac{1}{2} \Lunit^{\alpha} \uLunit^{\beta}
				 - 
				\frac{1}{2} \uLunit^{\alpha} \Lunit^{\beta},
		&
		(\alpha, \beta 
		& = 0,1),
	\end{align}
	i.e., $\hfour^{-1} = - \frac{1}{2} \Lunit \otimes \uLunit 
		- 
		\frac{1}{2} \uLunit \otimes \Lunit
		$. 
\end{lemma}

\subsection{Riemann invariants} \label{SS:RIEMANNINVARIANTS}
In this section, we introduce the \emph{Riemann invariants}, which are a pair of fluid variables, one of which is 
constant along the integral curves of $\Lunit$ and the other along the integral curves of $\uLunit$. For 
isentropic plane-symmetric solutions to the relativistic Euler equations,
the Riemann invariants can be used as state-space variables for the fluid, i.e., the other fluid variables are determined by them.
Riemann invariants were discovered by Riemann in his proof of shock formation for the non-relativistic 
compressible Euler equations in $1D$. 
In \cite{aT1948}, Taub discovered Riemann invariants for the isentropic relativistic Euler equations in $1D$. 
Studying the flow with respect to the Riemann invariants allows one to study the $\hfour$-MGHD using
techniques based on transport equations and ODE theory. In particular, energy estimates are not needed in $1D$.
Although Riemann invariants are not available in multi-dimensions, approximate Riemann invariants \emph{are} available, and they
are convenient for studying perturbations of isentropic plane-symmetric solutions; see Sect.\,\ref{SS:ALMOSTRIEMANNINVARIANTS}.

To proceed, we let $\RIfunction = \RIfunction(\Enth)$ be the solution to
the initial value problem:
\begin{align} \label{E:ODEFORRIEMANNINVARIANT}
	\frac{\mathrm{d}}{\mathrm{d} \Enth} 
	\RIfunction(\Enth)
	& = \frac{1}{\Enth \speed},
	&
	\RIfunction(\overline{\Enth})
	= 0,
\end{align}	
where $\overline{\Enth} > 0$ is the constant
\eqref{E:FIXEDPOSITIVEENTHALPYCONSTANT},
and in \eqref{E:ODEFORRIEMANNINVARIANT}, 
we are viewing $\speed$ as a function of $\Enth$. 

Note that since $\speed(\overline{\Enth}) \neq 0$ by assumption,
it follows from \eqref{E:ODEFORRIEMANNINVARIANT} that the function $\RIfunction$ 
has a local inverse $\RIfunction^{-1}$ defined near the origin such that $\RIfunction^{-1}(0) = \overline{\Enth}$.

\begin{definition}[Riemann invariants]	
	\label{D:RIEMANNINVARIANTS}
	We define the \emph{Riemann invariants} $\RRiemannPS$ and $\LRiemannPS$ as follows:
	\begin{subequations} 
	\begin{align} \label{E:RRIEMANNPS}
		\RRiemannPS
		& := 
				\RIfunction(\Enth)
				+
				\frac{1}{2} \ln \left( \frac{1 + \frac{\fourvelocity^1}{\fourvelocity^0}}{1 - \frac{\fourvelocity^1}{\fourvelocity^0}} \right),
			\\
		\LRiemannPS
		& := 
			\RIfunction(\Enth)
			-
			\frac{1}{2} \ln \left( \frac{1 + \frac{\fourvelocity^1}{\fourvelocity^0}}{1 - \frac{\fourvelocity^1}{\fourvelocity^0}} \right).
		\label{E:LRIEMANNPS}
	\end{align}
	\end{subequations}
\end{definition}

In the next proposition, for isentropic plane-symmetric solutions, we express the relativistic Euler equations
in terms of $(\RRiemannPS,\LRiemannPS)$, and we provide explicit formulae
for various quantities in terms of $(\RRiemannPS,\LRiemannPS)$.

\begin{proposition}[Isentropic plane-symmetric flow in terms of the Riemann invariants]	
\label{P:QUASILINEARTRANSPORTSYSTEMFORRIEMANNINVARAINTSANDOTHERRESULTS}
	Let $(\Lnenth,\fourvelocity^0,\fourvelocity^1)$ be a $C^1$ solution to \eqref{E:1DENTHALPYEVOLUTION}--\eqref{E:1DFLUIDFOURVELOCITYNORMALIZED},
	and let $\Lunit$ and $\uLunit$ be the vectorfields from Def.\,\ref{D:NULLVECTORFIELDSINPLANESYMMETRY}.
	Then the Riemann invariants of Def.\,\ref{D:RIEMANNINVARIANTS} satisfy the 
	following quasilinear transport system:
	\begin{subequations}
	\begin{align} \label{E:RRIEMANNPSTRANSPORT}
		\Lunit \RRiemannPS
		& = 0,
			\\
		\uLunit \LRiemannPS
		& = 0.
		\label{E:LRIEMANNPSTRANSPORT}
	\end{align}
	\end{subequations}
	Moreover, the components of the fluid velocity 
	in the Minkowski-rectangular coordinate system $(t,x^1)$
	can be expressed in terms of the Riemann invariants as follows: 
	\begin{subequations}
	\begin{align}
		\fourvelocity^0  & = 
		\cosh\left(\frac{\RRiemannPS - \LRiemannPS}{2}\right), \qquad \fourvelocity^1  = \sinh\left(\frac{\RRiemannPS - \LRiemannPS}{2}\right). 
		\label{E:1DFLUIDVELOCITYINTERMSOFRIEMANNINVARIANTS} 
	\end{align}
	In addition, with $\RIfunction^{-1}$ denoting the inverse function of the map $\RIfunction$ from Def.\,\ref{D:RIEMANNINVARIANTS},
	we have:
	\begin{align} \label{E:1DENTHALPYINTERMSOFRIEMANNINVARIANTS}
		\Enth 
		& 
		=  
		\RIfunction^{-1}\left( \frac{1}{2}\left(\RRiemannPS + \LRiemannPS\right)\right). 
	\end{align}
	\end{subequations}
	Finally, the following identities hold:
	\begin{subequations}
		\begin{align}
			\Lunit & = \p_t + \frac{\sinh\left(\frac{\RRiemannPS - \LRiemannPS}{2}\right) + \speed \cosh\left(\frac{\RRiemannPS - \LRiemannPS}{2}\right)}{\cosh\left(\frac{\RRiemannPS - \LRiemannPS}{2}\right) + \speed \sinh\left(\frac{\RRiemannPS - \LRiemannPS}{2}\right)} \p_1, 
			\label{E:1DLUNITINTERMSOFRIEMANNINVARIANTS}
				\\
			\uLunit 
			& 
			= 
			\p_t 
			+ 
			\frac{\sinh\left(\frac{\RRiemannPS - \LRiemannPS}{2}\right) 
			- 
			\speed \cosh\left(\frac{\RRiemannPS - \LRiemannPS}{2}\right)}{
			\cosh\left(\frac{\RRiemannPS - \LRiemannPS}{2}\right) - \speed \sinh\left(\frac{\RRiemannPS - \LRiemannPS}{2}\right)} \p_1. 
			\label{E:1DULUNITINTERMSOFRIEMANNINVARIANTS} 
		\end{align}
	\end{subequations}
	\end{proposition}
	
	\begin{proof}
		Equations \eqref{E:RRIEMANNPSTRANSPORT}--\eqref{E:LRIEMANNPSTRANSPORT} follow from straightforward but tedious calculations involving
		Def.\,\ref{D:RIEMANNINVARIANTS}, equations \eqref{E:1DENTHALPYEVOLUTION}--\eqref{E:1DFLUIDFOURVELOCITYNORMALIZED}, 
		the identity $\fourvelocity^0 \p_{\alpha} \fourvelocity^0 = \fourvelocity^1 \p_{\alpha} \fourvelocity^1$ 
		(which follows from differentiating \eqref{E:1DFLUIDFOURVELOCITYNORMALIZED} with $\p_{\alpha}$), 
		and the identity 
		$\p_{\alpha} \RIfunction(\Enth) = \frac{1}{c} \p_{\alpha} \Lnenth$ 
		(which follows from \eqref{E:LOGENTHALPY} and \eqref{E:ODEFORRIEMANNINVARIANT}). 
		
		To prove the identities in \eqref{E:1DFLUIDVELOCITYINTERMSOFRIEMANNINVARIANTS}, 
		we first subtract
		\eqref{E:LRIEMANNPS}
		from
		\eqref{E:RRIEMANNPS}
		and use
		\eqref{E:1DFLUIDFOURVELOCITYNORMALIZED} to compute that 
		$\frac{1}{2} 
		\left(
			\RRiemannPS - \LRiemannPS
		\right)
		= 
		\ln(\fourvelocity^0 + \fourvelocity^1)$. 
		Rewriting~\eqref{E:1DFLUIDFOURVELOCITYNORMALIZED} as $(\fourvelocity^0 + \fourvelocity^1)(\fourvelocity^0-\fourvelocity^1) = 1$
		and carrying out straightforward calculations based on the definitions of $\cosh$ and $\sinh$,
		we conclude \eqref{E:1DFLUIDVELOCITYINTERMSOFRIEMANNINVARIANTS}. 
		\eqref{E:1DENTHALPYINTERMSOFRIEMANNINVARIANTS} follows from adding \eqref{E:RRIEMANNPS} to \eqref{E:LRIEMANNPS},
		dividing by $2$,
		and then applying $\RIfunction^{-1}$ to each side of the resulting identity.
		
		Finally, \eqref{E:1DLUNITINTERMSOFRIEMANNINVARIANTS}--\eqref{E:1DULUNITINTERMSOFRIEMANNINVARIANTS} 
		follow from definitions~\eqref{E:1DLUNIT}--\eqref{E:1DULUNIT}
		and the already proven identity \eqref{E:1DFLUIDVELOCITYINTERMSOFRIEMANNINVARIANTS}. 
	\end{proof}
	
\begin{convention}[The fluid variables as functions of the Riemann invariants] \label{CON:1DFLUIDVARIABLESASFUNCTIONSOFRIEMANNINVARIANTS}
	In view of \eqref{E:1DFLUIDVELOCITYINTERMSOFRIEMANNINVARIANTS}--\eqref{E:1DENTHALPYINTERMSOFRIEMANNINVARIANTS},
	throughout Sect.\,\ref{S:1DMAXIMALDEVELOPMENT},
	we often view all the fluid variables as smooth functions of the Riemann invariants
	(or of only $\RRiemannPS$ when we are studying simple isentropic plane-symmetric solutions), 
	e.g., $\fourvelocity^0,\fourvelocity^1,\Lnenth,\Enth = \smoothfunction(\RRiemannPS,\LRiemannPS)$,
	where throughout the paper, $\smoothfunction$ schematically denotes a smooth function that can vary from
	line to line. 
	For example, we use 
	\eqref{E:ODEFORRIEMANNINVARIANT},
	\eqref{E:1DENTHALPYINTERMSOFRIEMANNINVARIANTS},
	the inverse function theorem,
	and the chain rule
	to compute the following identity, which holds for isentropic plane-symmetric solutions:  	
	\begin{align} \label{E:1DARBITRARYDERIVATIVEOFSPEEDOFSOUNDINTERMSOFRIEMANNINVARIANTS}
		\p_{\alpha} \speed 
		& = 
		\frac{\speed \Enth}{2} 
		\speed' 
		\left\lbrace
			\p_{\alpha} \RRiemannPS - \p_{\alpha} \LRiemannPS
		\right\rbrace,
	\end{align}
where on RHS~\eqref{E:1DARBITRARYDERIVATIVEOFSPEEDOFSOUNDINTERMSOFRIEMANNINVARIANTS}, 
$\speed' := \frac{\mathrm{d} \speed}{\mathrm{d} \Enth}$.
	
\end{convention}

\subsection{Brief remarks on Riccati-type blowup and genuine nonlinearity}
\label{SS:1DRICCATIBLOWUP}
Consider a simple isentropic plane-symmetric solution $\RRiemannPS$ to \eqref{E:RRIEMANNPSTRANSPORT} (with $\LRiemannPS = 0$). 
By differentiating the equation with respect to $\partial_1$, one obtains a Riccati-type equation
for $\partial_1 \RRiemannPS$ of the form 
$\Lunit \partial_1 \RRiemannPS = - \frac{\mathrm{d} }{\mathrm{d} \RRiemannPS} \Lunit^1 \cdot (\partial_1 \RRiemannPS)^2$,
where here we are viewing $\Lunit^1$ as a function of $\RRiemannPS$, which is possible thanks to
Prop.\,\ref{P:QUASILINEARTRANSPORTSYSTEMFORRIEMANNINVARAINTSANDOTHERRESULTS} 
and Convention~\ref{CON:1DFLUIDVARIABLESASFUNCTIONSOFRIEMANNINVARIANTS}.
Hence, as long as the \emph{genuine nonlinearity condition} 
$\frac{\mathrm{d} }{\mathrm{d} \RRiemannPS} \Lunit^1 \neq 0$ holds,
then for solutions $\RRiemannPS$ launched by smooth compactly supported initial data, 
one can easily show that $\partial_1 \RRiemannPS$ typically blows up along the integral curves of $\Lunit$,
much like in the case of Burgers' equation $\partial_t \Psi + \Psi \partial_1 \Psi = 0$;
we refer to \cite{cD2010}*{Section~7.5} for a detailed discussion of the notion of genuine nonlinearity
in the context of $1D$ hyperbolic conservation laws.
One can check that the non-degeneracy assumption \eqref{E:NONDEGENARCYASSUMPTION} indeed implies that
$
\frac{\mathrm{d} }{\mathrm{d} \RRiemannPS} \Lunit^1|_{\RRiemannPS = 0}
\neq 0
$,
i.e., that equation \eqref{E:RRIEMANNPSTRANSPORT} is genuinely nonlinear
near $(\RRiemannPS,\LRiemannPS) = (0,0)$. 

Our study of the blowup in Theorem~\ref{T:MAINTHEOREM1DSINGULARBOUNDARYANDCREASE} does not rely on differentiating
the equations with $\partial_1$ to obtain a Riccati-type PDE.
Instead, it relies on a much sharper approach tied to nonlinear geometric optics, which we set up in 
Sect.\,\ref{SS:1DSIMPLEPLANESYMMETRYANDACOUSTICGEOMETRYSETUP}. 
In $3D$, such a sharpened approach seems essential, as 
there is no known simple approach for studying gradient-blowup based only on
differentiating the equations with respect to a standard, solution-independent partial derivative operator.
In particular, in our discussion of shock formation in $3D$ in Sect.\,\ref{S:SHOCKFORMATIONAWAYFROMSYMMETRY},
nonlinear geometric optics will play a fundamental role.

\subsection{Simple isentropic plane symmetry and nonlinear geometric optics} 
\label{SS:1DSIMPLEPLANESYMMETRYANDACOUSTICGEOMETRYSETUP}
Our study of the $\hfour$-MGHD in Theorem~\ref{T:MAINTHEOREM1DSINGULARBOUNDARYANDCREASE} 
intimately relies on nonlinear geometric optics, implemented via an acoustic eikonal function $\eik$,
and a careful study of the corresponding acoustic geometry. The main idea of the proof of 
Theorem~\ref{T:MAINTHEOREM1DSINGULARBOUNDARYANDCREASE} 
is to show that the solution remains smooth relative to the ``geometric coordinates'' $(t,\eik)$ all the way up to the shock
and to recover the formation of the shock as a degeneracy between the geometric coordinates
and the Minkowski-rectangular coordinates $(t,x^1)$.

In this section, we construct the eikonal function and corresponding acoustic geometry for 
the isentropic relativistic Euler equations in $1D$,
i.e., to the following initial value problem: 
\begin{align} \label{E:1DCAUCHYIVPFORRIEMANNINVARIANTS}
	& \begin{cases} 
		\Lunit \RRiemannPS = 0, 
		\\
		\RRiemannPS|_{t=0} =  \dataRRiemannPS, 
	\end{cases} & & 
	\begin{cases} 
		\uLunit \LRiemannPS = 0, 
		\\
		\LRiemannPS|_{t=0} =  \LRiemannPSdata.
	\end{cases} 
\end{align} 

More precisely, we restrict our attention to \emph{simple} isentropic plane-symmetric solutions, 
defined here to be those solutions such that $\LRiemannPS \equiv 0$. 
From \eqref{E:1DCAUCHYIVPFORRIEMANNINVARIANTS} and \eqref{E:1DULUNIT},
it follows that such solutions arise from initial data with $\LRiemannPSdata \equiv 0$. In this case, 
\eqref{E:1DFLUIDVELOCITYINTERMSOFRIEMANNINVARIANTS} yields the identity
$\fourvelocity^0 = \cosh(\frac{1}{2} \RRiemannPS)$.
As we alluded to in Convention~\ref{CON:1DFLUIDVARIABLESASFUNCTIONSOFRIEMANNINVARIANTS},
\eqref{E:1DFLUIDVELOCITYINTERMSOFRIEMANNINVARIANTS}--\eqref{E:1DENTHALPYINTERMSOFRIEMANNINVARIANTS} 
also yield similar formulas
for $\fourvelocity^1,\, \Enth$ as functions of $\RRiemannPS$.

We start by defining the acoustic eikonal function, which plays a central role in all of the forthcoming analysis. 

\begin{definition}[The acoustic eikonal function in $1D$ and the characteristics] \label{D:1DEIKONALFUNCTION} 
Let $\hfour^{-1}$ be the inverse acoustical metric defined in \eqref{E:1DACOUSTICALMETRICINVERSEINTERMSOFNULLVECTORFIELDS}. Then we define the acoustic eikonal function (eikonal function for short) to be the solution $\eik$ to the following fully nonlinear hyperbolic PDE with Cauchy data:
\begin{subequations}
	\begin{align}
		 (\hfour^{-1})^{\alpha\beta} \partial_{\alpha} \eik \p_\beta \eik & = 0, \label{E:1DEIKONALNULLGRADIENTVANISHES} \\
		 \p_t \eik & > 0, \label{E:MINKPARTIALTOF1DEIKISPOSITIVE} \\
		 \eik|_{t=0} & = - x^1. \label{E:1DEIKONALDATAFORFULLYNONLINEARIVP}
	\end{align} 
\end{subequations}

We refer to the level sets of $\eik$ as ``the characteristics.''

\end{definition}

\begin{remark}[Geometric optics is easier in $1D$] \label{E:GEOMETRICOPTICSISEASIERIN1D} Although understanding the fully nonlinear structure of \eqref{E:1DEIKONALNULLGRADIENTVANISHES} is necessary in multiple space dimensions, in $1D$,  \eqref{E:1DEIKONALNULLGRADIENTVANISHES} is equivalent to the \emph{linear} (in $\eik$) transport equation of $\eik$ derived in the following lemma. This makes the analysis of the acoustic geometry in this section considerably simpler than in multi-$D$. It is also unique to $1D$ that the $\hfour$-null vectorfield $\Lunit$ can be explicitly expressed in terms of the fluid via the formula \eqref{E:1DLUNIT}. The correct analog of \eqref{E:1DLUNIT} in higher dimensions is precisely $\Lunit = \frac{1}{\Lgeo^0} \Lgeo$, where $\Lgeo^{\alpha} = - (\hfour^{-1})^{\alpha \kappa} \p_{\kappa} \eik$. 
\end{remark}

\begin{lemma}[Linear transport equation for $\eik$] \label{E:LINEARTRANSPORTEQUATIONFOREIKONALFUNCTION}
Let $\Lunit$ be the vectorfield defined in \eqref{E:1DLUNIT}. 
For simple isentropic plane-symmetric solutions,
if $\RRiemannPS$ is sufficiently small,\footnote{The solutions from Theorem~\ref{T:MAINTHEOREM1DSINGULARBOUNDARYANDCREASE}
satisfy the needed smallness. The key point is that if $\RRiemannPS$ is large, then
$\Lunit$ and $\uLunit$ could both be right-pointing or both be left-pointing, in which case the initial
condition \eqref{E:1DEIKONALDATA} would not be sufficient for identifying the ``correct root of the eikonal equation,''
namely the ``root'' $\Lunit \eik = 0$.
\label{FN:1DSMALLNESSOFRRIEMANNFORCORRECTROOTOFEIKONALEQUATION}} 
then
the scalar function $\eik$ solves \eqref{E:1DEIKONALNULLGRADIENTVANISHES}--\eqref{E:1DEIKONALDATAFORFULLYNONLINEARIVP} 
\textbf{if and only if} 
it solves the following transport equation initial value problem:
\begin{subequations}
	\begin{align}
		\Lunit \eik & = 0, \label{E:1DEIKONALTRANSPORTEQUATION} \\
		\eik|_{t=0} & = -x^1. \label{E:1DEIKONALDATA}
	\end{align}
\end{subequations}

\end{lemma}

\begin{proof}
Suppose $\eik$ solves \eqref{E:1DEIKONALNULLGRADIENTVANISHES}--\eqref{E:MINKPARTIALTOF1DEIKISPOSITIVE}. 
Then by \eqref{E:1DACOUSTICALMETRICINVERSEINTERMSOFNULLVECTORFIELDS}, 
we find that $0 = (\Lunit \eik) \cdot \uLunit \eik$. 
If $\Lunit \eik = 0$, there is nothing to prove. Hence, we assume that $\Lunit \eik \neq 0$,
and we will show that this leads to a contradiction.
Then by \eqref{E:1DACOUSTICALMETRICINVERSEINTERMSOFNULLVECTORFIELDS}, we have $\uLunit \eik = 0$. 
By 
\eqref{E:SPEEDOFSOUNDASSUMEDBOUNDS},
\eqref{E:1DULUNITINTERMSOFRIEMANNINVARIANTS},
\eqref{E:1DEIKONALDATAFORFULLYNONLINEARIVP}, 
\eqref{E:1DEIKONALDATA},
and our smallness assumption on $\RRiemannPS$, 
we have:
	\begin{align}\label{E:CONTRADICTIONSTEPIFULUNITEIKONALVANISHES}
	\p_t \eik\big|_{t=0} 
	& = - \frac{\sinh(\tfrac{1}{2}\RRiemannPS) 
	- 
	\speed \cosh(\tfrac{1}{2} \RRiemannPS)}{\cosh(\tfrac{1}{2}\RRiemannPS) 
	- 
	\speed \sinh(\tfrac{1}{2}\RRiemannPS)} \p_1 \eik \big|_{t=0} 
	=
	\frac{\sinh(\tfrac{1}{2}\RRiemannPS) 
	- 
	\speed \cosh(\tfrac{1}{2} \RRiemannPS)}{\cosh(\tfrac{1}{2}\RRiemannPS) 
	- 
	\speed \sinh(\tfrac{1}{2}\RRiemannPS)}
	< 0.  
\end{align}
\eqref{E:CONTRADICTIONSTEPIFULUNITEIKONALVANISHES} contradicts \eqref{E:MINKPARTIALTOF1DEIKISPOSITIVE}. 

Conversely, if \eqref{E:1DEIKONALTRANSPORTEQUATION}--\eqref{E:1DEIKONALDATA} hold, 
then since \eqref{E:1DACOUSTICALMETRICINVERSEINTERMSOFNULLVECTORFIELDS} implies that
$(\hfour^{-1})^{\alpha \beta} \partial_{\alpha} \eik \p_\beta \eik = - (\Lunit \eik) \cdot \uLunit \eik$,
it immediately follows that \eqref{E:1DEIKONALNULLGRADIENTVANISHES} and \eqref{E:1DEIKONALDATAFORFULLYNONLINEARIVP} hold. 
To prove \eqref{E:MINKPARTIALTOF1DEIKISPOSITIVE}, 
we use
\eqref{E:SPEEDOFSOUNDASSUMEDBOUNDS},
\eqref{E:1DLUNITINTERMSOFRIEMANNINVARIANTS},
\eqref{E:1DEIKONALDATA},
and the smallness of $\RRiemannPS$,
arguing as in the proof of \eqref{E:CONTRADICTIONSTEPIFULUNITEIKONALVANISHES}
to deduce that $\partial_t \eik > 0$, as is desired.
\end{proof} 

Some remarks are in order.
\begin{itemize}
	\item The minus sign in \eqref{E:1DEIKONALDATA} is just a convention imposed for consistency with related results that have
		appeared in the literature, such as \cites{jSgHjLwW2016,jLjS2018,jLjS2021,lAjS2022}.
	\item The precise initial condition in \eqref{E:1DEIKONALDATA} is not important; small perturbations of this initial condition
		could also be used to study the $\hfour$-MGHD.
	\item We have adapted our eikonal function to $\Lunit$ because in 
	Theorem~\ref{T:MAINTHEOREM1DSINGULARBOUNDARYANDCREASE}, we will study initial data that lead to shock-forming solutions
	such that the singularity is tied to the infinite density of the $\Lunit$-characteristics and the blowup of
	$\partial_1 \RRiemannPS$. There is nothing special about $\Lunit$ compared to $\uLunit$, and it is only for convenience that we have
	adapted our constructions to it. That is,
	using the methods of this section, one could also construct an eikonal function $\underline{\eik}$ adapted to $\uLunit$
	and study different initial data that lead to a singularity
	tied to the infinite density of the $\uLunit$-characteristics and the blowup of
	$\partial_1 \LRiemannPS$. One could also try to study richer problems in which singularities occur
	along both families of characteristics.
\end{itemize}

We now define $\upmu$, the inverse foliation density of the characteristics. 
The name is justified by the fact that $\frac{1}{\upmu}$ is a measure of the density of the stacking of the 
characteristics relative to the level-sets $\Sigma_{t'} := \lbrace t = t' \rbrace$. 
Roughly speaking, $\upmu$ is initially positive, while shock formation corresponds to $\upmu \downarrow 0$,
which signifies the infinite density of the integral curves of $\Lunit$ 
(or equivalently, the level sets of $\eik$ in $(t,x^1)$-space);
we refer to Fig.\,\ref{F:MINKOWSKIRECTANGULAR1DPLANESYMMETRYCAUCHYHORANDSINGULARCURVE}, a context in which
$\upmu$ vanishes along the curve portion denoted by ``$\Upsilon(\singularcurve_{[-\interestingu,0]})$''
(we will explain these issues in much more detail later on).

\begin{definition}[Inverse foliation density] \label{D:MU} 
We define the \emph{inverse foliation density} of the characteristics, denoted by $\upmu$, as follows:
\begin{align} \label{E:INVERSEFOLIATIONDENSITYINPLANESYMMETRY}
	\upmu 
	& 
	:= 
	- 
	\frac{1}{\frac{\speed}{\normalizer}
	\partial_{1} \eik}. 
\end{align}

In Sect.\,\ref{S:SHOCKFORMATIONAWAYFROMSYMMETRY}, which concerns solutions in $3D$, 
we will define $\upmu$ to be $- \frac{1}{(\hfour^{-1})^{\kappa \lambda} \p_{\kappa} t \p_{\lambda} \eik}$;
see \eqref{E:INVERSEFOLIATIONDENSITY}.
Under the initial condition \eqref{E:EIKONALDATAAWAYFROMSYMMETRY} that we will choose for $\eik$, 
that definition of $\upmu$ can be shown to agree with the one for plane-symmetric solutions
stated in \eqref{E:INVERSEFOLIATIONDENSITYINPLANESYMMETRY}.

We also note the following identities, which follow from \eqref{E:1DEIKONALDATA}
and
\eqref{E:INVERSEFOLIATIONDENSITYINPLANESYMMETRY}:
\begin{align} \label{E:DATAIDENTITIESFORMU}
	\upmu|_{t=0} 
	& = \frac{\normalizer}{\speed}\big|_{t=0}, \qquad \Longrightarrow \qquad \frac{\speed}{\normalizer}\upmu \big|_{t=0}  = 1.
\end{align}
\end{definition}

We now define a collection of vectorfields associated to the eikonal function.

\begin{definition}[The vectorfields $X,\, \muX,\, \newuL$] \label{D:SEVERALACOUSTICVECTORFIELDS}
Let $\Lunit$ and $\uLunit$ be the vectorfields from Def.\,\ref{D:NULLVECTORFIELDSINPLANESYMMETRY}
(they are $\hfour$-null by \eqref{E:GFOURINNERPRODUCTOFNULLVECTORFIELDSIN1D}).
We define:
\begin{align} \label{E:SEVERALACOUSTICVECTORFIELDS}
X &:= \frac{1}{2} (\uLunit - \Lunit), & \muX & := \upmu X, & \newuL &:= \upmu \uLunit.
\end{align}
\end{definition}

Using 
\eqref{E:1DLUNIT}--\eqref{E:1DULUNIT},
\eqref{E:GFOURINNERPRODUCTOFNULLVECTORFIELDSIN1D}, 
and \eqref{E:SEVERALACOUSTICVECTORFIELDS},
it is straightforward to see that $X$ is $\Sigma_t$-tangent (i.e., $X t = 0$) and satisfies 
$\hfour(X,X) = 1$, that $\hfour(\muX,\muX) = \upmu^2$, and that $\hfour(\newuL,\newuL) = 0$.
We also note that in plane-symmetry, 
using Def.\,\ref{D:NULLVECTORFIELDSINPLANESYMMETRY},
\eqref{E:NORMAIZERFACTORSIN1D}, 
and Def.\,\ref{D:SEVERALACOUSTICVECTORFIELDS},
one can compute the following identity:
\begin{align} \label{E:EXPLICITFORMOFXIN1D}
	 X 
	& = -\frac{\speed}{\normalizer} \p_1.
\end{align}
Our main results concern solutions in which the factor $\frac{\speed}{\normalizer}$
on RHS~\eqref{E:EXPLICITFORMOFXIN1D} satisfies $\frac{\speed}{\normalizer} \approx 1$.

\begin{definition}[Geometric coordinates in $1D$ and the corresponding partial derivative vectorfields] 
\label{D:GEOMETRICCOORDINATES1D}
We define $(t,\eik)$ to be the \emph{geometric coordinates}, 
and we denote the corresponding geometric coordinate partial derivative vectorfields
by $\left\lbrace \geop{t},\geop{\eik} \right\rbrace$
(which are not to be confused with the partial derivatives $\lbrace \partial_t, \partial_1 \rbrace$ 
in the $(t,x^1)$-coordinate system). 
\end{definition}

Using \eqref{E:1DLUNIT}, \eqref{E:1DEIKONALTRANSPORTEQUATION}, \eqref{E:INVERSEFOLIATIONDENSITYINPLANESYMMETRY}, 
and \eqref{E:EXPLICITFORMOFXIN1D}, it is straightforward to compute that:
\begin{subequations}
\begin{align} \label{E:1DGEOMETRICVECTORFIELDSAPPLIEDTOGEOMETRICCOORDIANTEFUNCTIONS}
	\Lunit t 
	& = \muX  \eik = 1,
	\\
\Lunit \eik 
& = \muX t = 0.
\label{E:SECOND1DGEOMETRICVECTORFIELDSAPPLIEDTOGEOMETRICCOORDIANTEFUNCTIONS}
\end{align}
\end{subequations}
From 
\eqref{E:1DGEOMETRICVECTORFIELDSAPPLIEDTOGEOMETRICCOORDIANTEFUNCTIONS}--\eqref{E:SECOND1DGEOMETRICVECTORFIELDSAPPLIEDTOGEOMETRICCOORDIANTEFUNCTIONS},
it follows that in geometric coordinates, the following identities hold: 
\begin{align} \label{E:PLANESYMMETRYCOMMUTATORVECTORFIELDSAREGEOCOORDINATEVECTORFIELDS}
	\Lunit 
	& = \geop{t}, 
	& 
	\muX & = \geop{\eik}, 
		\\
	 [\Lunit, \muX] &
	 = 0. \label{E:LANDMUXCOMMUTEINPLANESYMMETRY}
\end{align} 
From \eqref{E:1DULUNIT},
\eqref{E:SEVERALACOUSTICVECTORFIELDS}, 
and the identities in \eqref{E:PLANESYMMETRYCOMMUTATORVECTORFIELDSAREGEOCOORDINATEVECTORFIELDS},
we deduce the following expression for $\newuL$ relative to the geometric coordinates:
\begin{align} \label{E:NEWULINGEOMETRICCOORDINATES}
	\newuL 
	& = \upmu \geop{t} + 2 \geop{\eik}.
\end{align}
We will often silently use 
\eqref{E:PLANESYMMETRYCOMMUTATORVECTORFIELDSAREGEOCOORDINATEVECTORFIELDS}--\eqref{E:LANDMUXCOMMUTEINPLANESYMMETRY}
in the rest of Sect.\,\ref{S:1DMAXIMALDEVELOPMENT}.

The following sets play a fundamental role in our analysis of solutions.

\begin{definition}[Portions of submanifolds of geometric coordinate space $\mathbb{R}_t \times \mathbb{R}_{\eik}$]
\label{D:1DSUBSETSOFGEOMETRICCOORDINATESPACE}
Given intervals $I,J \subset \mathbb{R}$,
relative to the geometric coordinates, we define:
\begin{subequations}
	\begin{align}
		\nullhyparg{\eik'}^{I} 
		& := \big\{ (t,\eik) \in \R^2 \ | \ t \in I, \, \eik = \eik' \big\}, \label{E:CHARACTERISTICSURFACES1D} \\
		\Sigma_{t'}^{J} & := \big\{ (t,\eik) \in \R^2 \ | \ t = t', \, \eik \in J \big\}. \label{E:TRUNCATEDSIGMAT}
	\end{align}
\end{subequations}
We often write $\nullhyparg{\eik}$ instead of $\nullhyparg{\eik}^{[0,\infty)}$ and
$\Sigma_t$ instead of $\Sigma_t^{(-\infty,\infty)}$.
We often refer to the $\nullhyparg{\eik}$ as \emph{the characteristics}.
\end{definition}
	
The following lemma shows that the geometric coordinates degenerate with respect to the Minkowski-rectangular coordinates precisely when 
$\upmu = 0$. 
However, Theorem~\ref{T:MAINTHEOREM1DSINGULARBOUNDARYANDCREASE} shows that for appropriate initial data, 
$\Upsilon$ will still be a \emph{homeomorphism} on the subset of the boundary of the $\hfour$-MGHD where $\upmu$ vanishes.

\begin{lemma}[The Jacobian determinant of $(t,\eik) \mapsto (t,x^1)$]
Let: 
\begin{align} \label{E:1DCHOVGEOTORECT}
\Upsilon(t,\eik) 
& := (t,x^1) 
\end{align}
denote the change of variables map from geometric coordinates to Minkowski-rectangular coordinates. 
Then the following identity holds:
\begin{align} \label{E:JACOBIANDETERMINANTOFCHANGEOFVARIABLESMAP1D}
\det \mathrm{d} \Upsilon 
& = - \frac{\speed}{\normalizer} \upmu.
\end{align}
\end{lemma}

\begin{proof}
\eqref{E:JACOBIANDETERMINANTOFCHANGEOFVARIABLESMAP1D} follows from straightforward calculations based
on the identity $\frac{\p}{\p \eik} x^1 = -\frac{\speed}{\normalizer} \upmu$, 
which is implied by \eqref{E:SEVERALACOUSTICVECTORFIELDS}--\eqref{E:EXPLICITFORMOFXIN1D} and 
\eqref{E:PLANESYMMETRYCOMMUTATORVECTORFIELDSAREGEOCOORDINATEVECTORFIELDS}. 
\end{proof}

\subsubsection{Explicit solution formula in geometric coordinates}
\label{SSS:PSEXPLICITFLUIDSOLUTION}
From \eqref{E:PLANESYMMETRYCOMMUTATORVECTORFIELDSAREGEOCOORDINATEVECTORFIELDS}, we see 
that in simple isentropic plane-symmetry,
the transport equation \eqref{E:RRIEMANNPSTRANSPORT} 
takes the following form in geometric coordinates:
\begin{align} \label{E:RRIEMANNTRIVIALTRANSPORTEQUATIONINGEOMETRICCOORDINATES}
	\geop{t} \RRiemannPS(t,\eik) 
	& = 0.
\end{align}
From \eqref{E:RRIEMANNTRIVIALTRANSPORTEQUATIONINGEOMETRICCOORDINATES} and \eqref{E:1DCAUCHYIVPFORRIEMANNINVARIANTS},
we see that in geometric coordinates, the solution to \eqref{E:RRIEMANNTRIVIALTRANSPORTEQUATIONINGEOMETRICCOORDINATES} is:
\begin{align} \label{E:SOLUIONRRIEMANNTRIVIALTRANSPORTEQUATIONINGEOMETRICCOORDINATES}
	\RRiemannPS(t,\eik) 
	& = \dataRRiemannPS(\eik).
\end{align}

\subsubsection{The evolution equation for $\upmu$, the non-degeneracy condition, and explicit expressions for various solution variables}
\label{SSS:1DEVOLUTIONEQUATIONFORMUNONDEGENANDEXPLICITFORMULAE}
In this section, we derive various identities involving $\upmu$. In particular, 
we will later use the identities to show that suitable assumptions on the initial data of $\RRiemannPS$
lead to the vanishing of $\upmu$ in finite time, i.e., the formation of a shock.

We begin with the following lemma, which provides the evolution equation for $\upmu$. 

\begin{lemma}[Transport equation satisfied by $\upmu$]
\label{L:1DPSTRANSPORTFORMU}
For simple isentropic plane-symmetric solutions, 
the inverse foliation density, defined in \eqref{E:INVERSEFOLIATIONDENSITYINPLANESYMMETRY}, 
satisfies the following transport equation,
where 
$
\speed'(\Enth) := \frac{\mathrm{d} }{\mathrm{d} \Enth} \speed(\Enth)
$:
\begin{align} \label{E:LMUPS}
 \Lunit \left( \frac{\speed}{\normalizer} \upmu\right)
	& = \PSLmusourcetermfunction,
		\\
\PSLmusourcetermfunction 
& := 
- 
 \frac{1 - \speed^2 + \frac{\speed'}{\speed \Enth}}{2(\fourvelocity^0 + \speed \fourvelocity^1)^2} \muX \RRiemannPS.
\label{E:PSLMUSOURCETERMFUNCTION}
\end{align}
Moreover, 
\begin{align}\label{E:LLUPS}
\Lunit \Lunit \upmu 
& = 
\Lunit \Lunit \left( \frac{\speed}{\normalizer} \upmu\right)
= 0.
\end{align}
\end{lemma}

\begin{proof}
First, using \eqref{E:1DLUNITINTERMSOFRIEMANNINVARIANTS}, 
\eqref{E:1DARBITRARYDERIVATIVEOFSPEEDOFSOUNDINTERMSOFRIEMANNINVARIANTS}, 
\eqref{E:EXPLICITFORMOFXIN1D},
the quotient rule, 
the identity $\cosh^2(z) - \sinh^2(z) = 1$, 
and straightforward calculations, we deduce the following commutator identity, valid for simple isentropic plane-symmetric solutions: 
\[ [\Lunit,\p_1] = \frac{\normalizer}{\speed} \left\lbrace \frac{1 - \speed^2 + \frac{\speed'}{\speed \Enth}}{2(\fourvelocity^0 + \speed \fourvelocity^1)^2} \right\rbrace X \RRiemannPS \p_1.\]
Equation~\eqref{E:LMUPS} then follows from multiplying \eqref{E:INVERSEFOLIATIONDENSITYINPLANESYMMETRY} by $\speed/\normalizer$, differentiating both sides with respect to $\Lunit$, and using the following commutator identity, 
which is a consequence the eikonal equation $\Lunit \eik = 0$:
\[ \Lunit \left( \frac{\speed}{\normalizer} \upmu\right) = \frac{1}{(\p_1 \eik)^2} \Lunit \p_1 \eik = \left( \frac{\speed}{\normalizer} \upmu \right)^2 \Lunit \p_1 \eik =  \left( \frac{\speed}{\normalizer} \upmu \right)^2  [\Lunit,\p_1 ] \eik.\] 

Equation~\eqref{E:LLUPS} then follows from differentiating \eqref{E:LMUPS} with respect to $\Lunit$ 
and using \eqref{E:RRIEMANNPSTRANSPORT},
\eqref{E:EXPLICITFORMOFXIN1D},
\eqref{E:LANDMUXCOMMUTEINPLANESYMMETRY},
and the assumption of simple isentropic plane-symmetry.

\end{proof}

We find it convenient to work with an \emph{anti-derivative} of the nonlinearity $\PSLmusourcetermfunction$ 
defined in \eqref{E:PSLMUSOURCETERMFUNCTION}. 
To this end, 
viewing the quantity
$- \frac{1 - \speed^2 + \frac{\speed'}{\speed \Enth}}{2(\fourvelocity^0 +  \speed \fourvelocity^1)^2}$ as a function of $\RRiemannPS$ 
(see Convention~\ref{CON:1DFLUIDVARIABLESASFUNCTIONSOFRIEMANNINVARIANTS}), 
we define the function $\RRiemannPS \rightarrow \antiderivativePSLmusourcetermfunction[\RRiemannPS]$ as follows:
\begin{align} \label{E:ANTIDERIVATIVESOURCETERM}
	\antiderivativePSLmusourcetermfunction[\RRiemannPS] 
	& :=
		\int_0^{\RRiemannPS} 
			\left\lbrace
				-\frac{1 - \speed^2 + \frac{\speed'}{\speed \Enth}}{2(\fourvelocity^0 +  \speed \fourvelocity^1)^2}
			\right\rbrace(z) 
		\, \mathrm{d} z.
\end{align}
From \eqref{E:PLANESYMMETRYCOMMUTATORVECTORFIELDSAREGEOCOORDINATEVECTORFIELDS},
\eqref{E:ANTIDERIVATIVESOURCETERM},
the fundamental theorem of calculus,
and the chain rule, it follows that:
\begin{subequations}
	\begin{align}
		\geop{u} \antiderivativePSLmusourcetermfunction[\RRiemannPS] 
		& 
		=
		\muX \antiderivativePSLmusourcetermfunction[\RRiemannPS]  
		= 
		\PSLmusourcetermfunction, \label{E:ANTIDERIVATIVEIDENTITYFORSOURCETERM} 
			\\
		\antiderivativePSLmusourcetermfunction[0] & = 0. \label{E:NORMALIZATIONOFANTIDERIVATIVEFORLMUSOURCETERM}
	\end{align}
\end{subequations}
We highlight the important fact that in simple isentropic plane-symmetry, 
the non-degeneracy assumption \eqref{E:NONDEGENARCYASSUMPTION}, 
the initial condition of \eqref{E:ODEFORRIEMANNINVARIANT}, 
\eqref{E:ANTIDERIVATIVESOURCETERM},
and \eqref{E:NORMALIZATIONOFANTIDERIVATIVEFORLMUSOURCETERM} 
collectively imply that: 
\begin{align} \label{E:1DNONDEGENERACYONEOSCONSEQUENCE}
\frac{\mathrm{d} }{\mathrm{d} \RRiemannPS} 
\antiderivativePSLmusourcetermfunction[\RRiemannPS]|_{\RRiemannPS = 0}
& 
=
- \frac{1}{2}
\left\lbrace
	1 
	- 
	(\overline{\speed})^2 
	+ 
	\frac{\overline{\speed}'}{\overline{\speed} \overline{\Enth}}
\right\rbrace 
\neq 0.
\end{align}
From 
\eqref{E:ANTIDERIVATIVESOURCETERM}
and
\eqref{E:1DNONDEGENERACYONEOSCONSEQUENCE}, 
it follows that the function
$\RRiemannPS \mapsto \antiderivativePSLmusourcetermfunction [\RRiemannPS]$ is 
a diffeomorphism from a neighborhood of the origin 
(i.e., near $\RRiemannPS = 0$) onto a neighborhood of the origin;
we will silently use this basic fact in the rest of Sect.\,\ref{S:1DMAXIMALDEVELOPMENT}.
As we mentioned in Sect.\,\ref{SS:1DRICCATIBLOWUP},
\eqref{E:1DNONDEGENERACYONEOSCONSEQUENCE} is fundamental for ensuring the 
formation of shocks, for it implies that near $(\RRiemannPS,\LRiemannPS) = (0,0)$,
equation \eqref{E:RRIEMANNPSTRANSPORT} is genuinely nonlinear.

Next, we note that in simple isentropic plane-symmetry, 
in view of \eqref{E:SOLUIONRRIEMANNTRIVIALTRANSPORTEQUATIONINGEOMETRICCOORDINATES},
the quantities $\PSLmusourcetermfunction$, $\speed$, and $\normalizer$ can be viewed as functions of $\eik$ alone. 
We sometimes emphasize this fact by using notation such as $\PSLmusourcetermfunction(\eik)$, $\speed(\eik)$, etc.
The following corollary is a straightforward application of this fact and the above discussion;
we omit the straightforward proof.

\begin{corollary}[Explicit expressions for $\upmu$, $\Lunit\upmu$, $\muX \upmu$, and $\partial_1 \RRiemannPS$] 
\label{C:PSEXPLICITEXPRESSIONSFORSOLUTION}
Let $\PSLmusourcetermfunction$ be the source term from \eqref{E:PSLMUSOURCETERMFUNCTION}, viewed as a function of $\eik$. 
Let $\antiderivativePSLmusourcetermfunction = \antiderivativePSLmusourcetermfunction[\RRiemannPS]$ 
be the function of $\RRiemannPS$ defined by
\eqref{E:ANTIDERIVATIVESOURCETERM}.
Then for simple isentropic plane-symmetric solutions, the following identities hold
relative to the geometric coordinates,
where
$
\antiderivativePSLmusourcetermfunction'[z]
:=
\frac{\mathrm{d} }{\mathrm{d} z} 
\antiderivativePSLmusourcetermfunction[z]
$:
\begin{align}
\Lunit 
\left\lbrace 
	\frac{\speed}{\normalizer} \upmu(t,\eik) 
\right\rbrace
& = \PSLmusourcetermfunction(\eik)
	=
	\frac{ \mathrm{d} }{ \mathrm{d}  \eik}
	\antiderivativePSLmusourcetermfunction[\dataRRiemannPS(\eik)]
	= \antiderivativePSLmusourcetermfunction'[\dataRRiemannPS(\eik)]
\frac{\mathrm{d}}{\mathrm{d} \eik} \dataRRiemannPS(\eik),
	\label{E:LSPEEDTIMESMUPLANESYMMETRYWITHEXPLICITSOURCE} 
	\\
\begin{split} \label{E:SPEEDTIMESMUPLANESYMMETRYWITHEXPLICITSOURCE} 
\frac{\speed}{\normalizer} \upmu(t,\eik) 
& = 
	1
	+ 
	t \PSLmusourcetermfunction(\eik)
= 1 
	+
	t
	\frac{ \mathrm{d} }{ \mathrm{d}  \eik}
	\antiderivativePSLmusourcetermfunction[\dataRRiemannPS(\eik)]
		\\
	& 
= 1 
	+
	t
	\antiderivativePSLmusourcetermfunction'[\dataRRiemannPS(\eik)]
	\frac{\mathrm{d}}{\mathrm{d} \eik}
	\dataRRiemannPS(\eik),
\end{split}
	\\
\begin{split} \label{E:MUXMUPLANESYMMETRYWITHEXPLICITSOURCE}
	\muX \upmu(t,\eik) 
	& = 
		t 
		\left( 
		\frac{\normalizer}{\speed}
		\PSLmusourcetermfunction'
		\right)(\eik)
		+
		\left\lbrace
			\frac{ \mathrm{d}}{ \mathrm{d}  \eik}
			\left( \frac{\normalizer}{\speed}\right)(\eik)
		\right\rbrace
		\left\lbrace	
			1
			+ 
			t \PSLmusourcetermfunction(\eik)
		\right\rbrace
			\\
	& = 
	\frac{\normalizer}{\speed} t
	\frac{ \mathrm{d}^2 }{ \mathrm{d}  \eik^2}
	\antiderivativePSLmusourcetermfunction[\dataRRiemannPS(\eik)]
	+
	\left\lbrace
		\frac{ \mathrm{d}}{ \mathrm{d}  \eik}
		\left( \frac{\normalizer}{\speed}\right)(\eik)
	\right\rbrace
	\left\lbrace
	1 
	+
	t
	\frac{\mathrm{d} }{\mathrm{d} \eik}
		\antiderivativePSLmusourcetermfunction[\dataRRiemannPS(\eik)]
	\right\rbrace,
\end{split}
	\\
\begin{split}  \label{E:CLOSEDPARTIAL1DERIVATIVEOFRPLUSPS}
	[\partial_1 \RRiemannPS](t,\eik)
	& = - \frac{1}{\frac{\speed}{\normalizer} \upmu (t,\eik)} \geop{\eik} \RRiemannPS(t,\eik)
	= 
	- \frac{\frac{ \mathrm{d} }{ \mathrm{d}  \eik} \dataRRiemannPS(\eik)}{1
	+ 
	t \PSLmusourcetermfunction(\eik)} 
		\\
	& 
	= 
	- \frac{\frac{ \mathrm{d} }{ \mathrm{d}  \eik} \dataRRiemannPS(\eik)}{
	1 
	+
	t
	\antiderivativePSLmusourcetermfunction'[\dataRRiemannPS(\eik)]
	\frac{ \mathrm{d} }{ \mathrm{d}  \eik}
	\dataRRiemannPS(\eik)}.
\end{split}
\end{align}	

\end{corollary}

\subsection{Initial data that lead to admissible shock-forming solutions}
\label{SS:EXAMPLESOFBONAFIDEDATALEADINGTOADMISSIBLE}
In this section, we construct examples of initial data that fall under the scope of
Theorem~\ref{T:MAINTHEOREM1DSINGULARBOUNDARYANDCREASE},
i.e., initial data that launch shock-forming solutions 
such that we can derive the structure of a portion of the maximal classical development that includes a neighborhood of its boundary.
Our construction could be substantially generalized to exhibit a much larger class of initial data for which the results of
Theorem~\ref{T:MAINTHEOREM1DSINGULARBOUNDARYANDCREASE} hold; here we have striven for simplicity. 

\subsubsection{Assumptions on the ``seed'' profile}
\label{SSS:SEEDPROFILE}
To start, we fix any scalar ``seed profile'' $\seed$ with the following properties
(it is straightforward to show that such functions exist):

\begin{itemize}
	\item $\seed = \seed(\eik)$ is a\footnote{Our proof of Theorem~\ref{T:MAINTHEOREM1DSINGULARBOUNDARYANDCREASE} 
	does not actually require $\seed$ to be $C^{\infty}$; $C^4$ would suffice. \label{FN:C3ISENOUGH}} 
		$C^{\infty}$ function that is
		compactly supported in
		an interval $[-\rightu,\leftu]$ of $\eik$-values, where $\rightu, \leftu > 1$.
	\item $\frac{\mathrm{d} }{\mathrm{d} \eik} \seed(\eik)$ 
		has a unique, non-degenerate minimum at $\eik=0$ (our choice $\eik = 0$ is only for convenience).
		In particular, $\frac{\mathrm{d}^3 }{\mathrm{d} \eik^3} \seed(0) > 0$.
	\item Modifying $\seed$ by multiplying it by a constant 
		and composing it with a linear map of the form $\eik \rightarrow z \eik$
		for some constant $z \in \mathbb{R}$ if necessary, we assume that 
		(for the modified $\seed$), there is a constant $\blowupdelta$ satisfying:
		\begin{align} \label{E:1DBLOWUPDELTAISPOSITIVE} 
		\blowupdelta
			& > 0
		\end{align}
		such that:
		\begin{align} \label{AE:CONVENIENTRESCALINGOFSEED}
		\frac{\mathrm{d} }{\mathrm{d} \eik} \seed(\eik)\bigg|_{\eik=0} 
		& = - \blowupdelta,
		& &
		\frac{\mathrm{d}^2 }{\mathrm{d} \eik^2} \seed(\eik)\bigg|_{\eik=0} = 0;
	\end{align}
	that there is a constant $\PSdatamuHessianTaylorcoefficient$ satisfying:
		\begin{align}  \label{AE:PSKEYMUHESSIANTAYLORCOEFFICIENT}
			\PSdatamuHessianTaylorcoefficient 
			& > 0 
		\end{align} 
		such that:
		\begin{align} \label{E:1DROLEOFPSKEYMUHESSIANTAYLORCOEFFICIENT}
		\frac{\mathrm{d}^3 }{\mathrm{d} \eik^3} \seed(0)
		& = \PSdatamuHessianTaylorcoefficient,
		\end{align}
		and such that:
		\begin{align} \label{E:1SEEDESTIMATETIEDTOPSKEYMUHESSIANTAYLORCOEFFICIENT}
			\frac{1}{2} \PSdatamuHessianTaylorcoefficient 
			& \leq \frac{d^3}{\mathrm{d} \eik^3} \seed(\eik) 
			\leq 
			2 \PSdatamuHessianTaylorcoefficient,
			&&
			\mbox{if } 
			|\eik| \leq 1;
		\end{align}
		and that there is a constant $\PSLmunottoonegativeparameter$ satisfying:
		\begin{align} \label{AE:PSLOWERBOUNDONNULLDERIVATIVEOFMUINBORINGREGION}
			0 
			& 
			\leq
			\PSLmunottoonegativeparameter 
			< 1 
		\end{align}
		such that:
		\begin{align} \label{AE:PSMUNOTTOONEGATIVEINBORINGREGION}
		\frac{\mathrm{d} }{\mathrm{d} \eik} \seed(\eik) 
		& > - \PSLmunottoonegativeparameter \blowupdelta,
		& & \mbox{if } |\eik| \geq 1.
		\end{align}
\end{itemize}
See Fig.\,\ref{F:SEEDFUNCTION} for the graph 
of a representative seed profile.

\begin{center}
	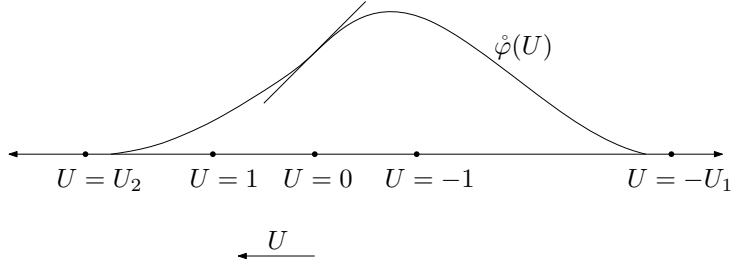
\begin{figure}  
		\begin{overpic}[scale=.6, grid = false, tics=5, trim=-.5cm -1cm -1cm -.5cm, clip]{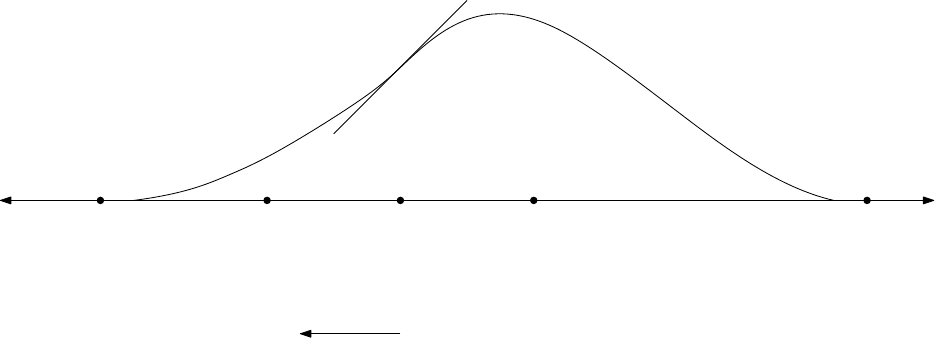}
			\put (9,15) {$\eik=\leftu$}
			\put (26,15) {$\eik=1$}
			\put (38,15) {$\eik=0$}
			\put (51,15) {$\eik=-1$}
			\put (82,15) {$\eik = - \rightu$}
			\put (36,7) {$\eik$}
			\put (65,32) {$\seed(\eik)$}
		\end{overpic}
		\caption{The graph of a representative ``seed profile,'' with $\eik$ increasing from right to left.}
		\label{F:SEEDFUNCTION}
	\end{figure}
\end{center}

\subsubsection{The initial data for $\RRiemannPS$ and amplitude smallness assumption}
\label{SSS:1DADMISSIBLEDATAFORRRIEMANN}
We now construct the initial data for $\RRiemannPS$.
Let $\seed$ be as in Sect.\,\ref{SSS:SEEDPROFILE}. 
We define the initial data function $\dataRRiemannPS := \RRiemannPS|_{t=0}$ to be the following function of
$\eik$:
\begin{align} 
	\dataRRiemannPS(\eik)
	& := \antiderivativePSLmusourcetermfunction^{-1}[\seed(\eik)],
	\label{AE:RIEMANNINVARIANTRESCALEDSEEDDATA}
\end{align}
where $\antiderivativePSLmusourcetermfunction^{-1}$ is the inverse function of the map
$\RRiemannPS \rightarrow \antiderivativePSLmusourcetermfunction[\RRiemannPS]$ 
defined by \eqref{E:ANTIDERIVATIVESOURCETERM}.

\medskip

\begin{center}
	\underline{\textbf{Amplitude smallness assumption on the initial data}}
\end{center}	
By \eqref{E:1DNONDEGENERACYONEOSCONSEQUENCE} and continuity,
we can further modify $\seed$ by multiplying it by a constant 
to shrink the amplitude of $\dataRRiemannPS$ and ensure that 
the following two conditions hold:
\begin{enumerate}
\item The isentropic plane-symmetric initial data array 
		$\mathring{\solutionarray} := (\Lnenth,\fourvelocity^0,\fourvelocity^1,0,0,0)|_{t=0}$
		determined by $\dataRRiemannPS$ via
		\eqref{E:1DFLUIDVELOCITYINTERMSOFRIEMANNINVARIANTS}--\eqref{E:1DENTHALPYINTERMSOFRIEMANNINVARIANTS}
		(with $\LRiemannPS = 0$ in those formulas)
		satisfies $\mathring{\solutionarray} \in \mbox{\upshape interior}(\mathcal{H})$,
		where $\mathcal{H}$ is the regime of hyperbolicity defined in \eqref{E:REGIMEOFHYPERBOLICITY}.
\item
The following
non-degeneracy condition holds:
\begin{align} \label{E:ANOTHER1DNONDEGENERACYONEOSCONSEQUENCE}
\frac{\mathrm{d} }{\mathrm{d} \RRiemannPS} 
\antiderivativePSLmusourcetermfunction[\RRiemannPS]\bigg|_{\RRiemannPS = \dataRRiemannPS(\eik)} 
& 
\neq 0,
& & \eik \in [-\rightu,\leftu].
\end{align}
\end{enumerate}
In the rest of Sect.\,\ref{S:1DMAXIMALDEVELOPMENT}, we will use the above two assumptions without always explicitly mentioning them.

Taking into account \eqref{E:NONDEGENARCYASSUMPTION} and \eqref{E:NORMALIZATIONOFANTIDERIVATIVEFORLMUSOURCETERM},
we deduce from Taylor expansions and standard calculus 
that the following conclusions
hold, where $C > 0$ depends on $\seed$,
and for functions $f = f(\eik)$ on $\Sigma_0 := \lbrace t = 0 \rbrace$, 
$\| f \|_{C^0(\Sigma_0)} := \max_{\eik \in \mathbb{R}} |f(\eik)|$:
\begin{itemize}
	\item
		$\dataRRiemannPS$ is compactly supported in $[-\rightu,\leftu]$
		and satisfies the following bounds:
		\begin{align} \label{AE:PSRESCALEDDATALINFINITYBOUND} 
			\left\| \frac{\mathrm{d}^M}{\mathrm{d} \eik^M} \dataRRiemannPS \right\|_{C^0(\Sigma_0)} 
			& \leq C,
			& 
			M = 0,1,2,3,4.
		\end{align}
	\item $\frac{\mathrm{d}}{\mathrm{d} \eik} \antiderivativePSLmusourcetermfunction[\dataRRiemannPS(\eik)]$ 
		has a unique, negative, non-degenerate minimum at $\eik=0$.
		In particular, 
		$\frac{\mathrm{d}^3}{\mathrm{d} \eik^3} \antiderivativePSLmusourcetermfunction[\dataRRiemannPS(\eik)] |_{\eik=0} > 0$. 
	\item There is a differentiable function $\PSthirdorderTaylorremaindercoefficientfunction: [-\rightu,\leftu] \rightarrow \mathbb{R}$ 	
		such that
		$
		\left\| 
			\PSthirdorderTaylorremaindercoefficientfunction
		\right\|_{C^1(\Sigma_0^{[-\rightu,\leftu]})} 
		\leq C
		$ 
		and such that:
		\begin{align} \label{E:TAYLOREXPANSIONOFKEYMUTERMABOUTCREASE}
			\frac{\mathrm{d}}{\mathrm{d} \eik} \antiderivativePSLmusourcetermfunction[\dataRRiemannPS(\eik)]
			& = - 
					\blowupdelta
					+
					\frac{1}{2}  \PSdatamuHessianTaylorcoefficient \eik^2
					+
					\PSthirdorderTaylorremaindercoefficientfunction(\eik) \eik^3.
		\end{align}
		\item For $j=1,2,3$, there exists a continuous function 
		$\PSinitialkeymufunctioncoefficients_j : [-\rightu,\leftu] \rightarrow (0,\infty)$
		such that:
		\begin{align} \label{E:INITIALMUFUNCTIONSKEYEXPANSION}
			\PSinitialkeymufunctioncoefficients_j(\eik)
			& \approx 1,
			&& \eik \in [-\rightu,\leftu],
		\end{align}
		and such that for $M=0,1,2$, we have:
		\begin{subequations}
		\begin{align}
			\frac{\mathrm{d}^M}{\mathrm{d} \eik^M} 
			\left(
				\frac{\mathrm{d}}{\mathrm{d} \eik} 
				\antiderivativePSLmusourcetermfunction[\dataRRiemannPS(\eik)]
				+
				\blowupdelta
			\right)
			& = \PSinitialkeymufunctioncoefficients_{M+1}(\eik)  \eik^{2-M},
				\label{E:RESCALEDMUFUNCTIONANDDERIVATIVESEXPANSIONSNEAR0} 
			& \mbox{if } |\eik | \leq 1,
				\\
			\frac{\mathrm{d}}{\mathrm{d} \eik} \antiderivativePSLmusourcetermfunction[\dataRRiemannPS(\eik)] 
			& \geq - \PSLmunottoonegativeparameter \blowupdelta,
			& \mbox{if } |\eik | \geq 1,
				\label{E:RESCALEDFIRSTDERIVATIVEOFMUFUNCTIONUNIFORMLYPOSITIVEAWAYFROMINTERESTINGREGION} 
		\end{align}
		\end{subequations}
		where $\PSLmunottoonegativeparameter \in [0,1)$ is the non-negative constant in \eqref{AE:PSMUNOTTOONEGATIVEINBORINGREGION},
		and in \eqref{E:INITIALMUFUNCTIONSKEYEXPANSION} and throughout,
		$A \approx B$ means that there is a $C \geq 1$ such that $\frac{1}{C} B \leq A \leq C B$.
		\item The constant $\blowupdelta$ appearing in \eqref{AE:CONVENIENTRESCALINGOFSEED} satisfies:
			\begin{align}  \label{E:1DUSEFULCHARACTERIZATIONOFBLOWUPDELTA}
				\blowupdelta 
				& 
				= 
				\max_{\eik \in [- \rightu,\leftu]} 
				[\PSLmusourcetermfunction(\eik)]_-	
				=  
				\max_{\eik \in [- \rightu,\leftu]}	
				\left[\frac{\mathrm{d}}{\mathrm{d} \eik} \antiderivativePSLmusourcetermfunction[\dataRRiemannPS(\eik)]\right]_-,
			\end{align}
			where $[z]_- := \max\lbrace -z, 0 \rbrace$ is the negative part of $z$.
\end{itemize}

\subsubsection{The time of first blowup}
\label{SSS:TIMEOFFIRSTBLOWUP}

\begin{definition}[Time of first blowup] \label{D:BLOWUPTIME} 
Given any initial data function $\dataRRiemannPS(\eik)$ constructed above, 
with $\blowupdelta > 0$ as in \eqref{AE:CONVENIENTRESCALINGOFSEED},
we define:
\begin{align} \label{E:BLOWUPTIME}
	\blowuptime
	& := \frac{1}{\blowupdelta}.
\end{align}
\end{definition}

Note that 
\eqref{E:SPEEDTIMESMUPLANESYMMETRYWITHEXPLICITSOURCE}
and
\eqref{E:1DUSEFULCHARACTERIZATIONOFBLOWUPDELTA} imply that $\blowuptime$ is the Minkowskian time 
at which $\upmu$ first vanishes, which by \eqref{E:CLOSEDPARTIAL1DERIVATIVEOFRPLUSPS} is
the positive time of first blowup for $|\partial_1 \RRiemannPS|$.

\subsubsection{Sharp estimates for the inverse foliation density}
\label{SSS:1DSHARPESTIMATESFORINVERSEFOLIATIONDENSITY}
In Lemma~\ref{L:SHARPESTIMATESFORMUNEARCREASE}, 
we provide various estimates for $\upmu$ and its derivatives
in a neighborhood of $[0,\blowuptime] \times \{0\} \subset \mathbb{R}_t \times \mathbb{R}_{\eik}$. 
These estimates will be useful in the description of the boundary of the maximal development. 

\begin{lemma}[Sharp estimates for $\upmu$ and $\PSLmusourcetermfunction$] 
\label{L:SHARPESTIMATESFORMUNEARCREASE} 
Let $\dataRRiemannPS = \dataRRiemannPS(\eik)$ 
be compactly supported initial data constructed in 
and satisfying the assumptions stated in
Sect.\,\ref{SSS:1DADMISSIBLEDATAFORRRIEMANN}. 
Recall that for the corresponding simple isentropic plane-symmetric solution, 
relative to the geometric coordinates $(t,\eik)$,
$\RRiemannPS$, $\speed$, and $\normalizer$ are functions of $\eik$.
With this in mind, we define:
\begin{align} \label{E:UPPERANDLOWERBOUNDSONSPEEDOVERNORMALIZER}
	\minimumspeed 
	& := \min_{\eik \in \mathbb{R}} \frac{\speed(\eik)}{\normalizer(\eik)}, 
	& \maximumspeed & := \max_{\eik \in \mathbb{R}} \frac{\speed(\eik)}{\normalizer(\eik)}.
\end{align}
Note that by \eqref{E:SPEEDOFSOUNDASSUMEDBOUNDS} and \eqref{E:ACOUSTICALMETRICNORMALIZER}, we have:
\begin{align} \label{E:EASYSPEEDOVERNORMALIZERBOUNDS}
	0 & < \minimumspeed \leq \maximumspeed < 1.
\end{align}

Then there exists\footnote{With minor additional effort, one could explicitly compute such a constant $\interestingu$ in terms of the constants associated to $\seed$ such as $\blowupdelta, \, \PSdatamuHessianTaylorcoefficient$, etc.} a 
constant $\interestingu$ depending only on the seed profile $\seed$ 
and satisfying $0 < \interestingu  < 1$ 
such that the following hold. 

\noindent \underline{\textbf{Behavior of $\upmu$ on $[-\interestingu,\interestingu]$}}.
Let $\blowuptime > 0$ be as defined in \eqref{E:BLOWUPTIME}.
Then the following estimates hold for $t \in [0,\blowuptime]$:
\begin{align} \label{AE:BOUNDSONLMUINTERESTINGREGION} 
	- 
	\minimumspeed^{-1}
	\blowupdelta
	\leq
	\min_{\Sigma_t^{[-\interestingu,\interestingu]}} \Lunit \upmu
	\leq 
	\max_{\Sigma_t^{[-\interestingu,\interestingu]}} \Lunit \upmu
	\leq 
	- \maximumspeed^{-1} \blowupdelta.
\end{align}

Moreover, for
$
(t,\eik) 
\in
[0,\blowuptime] 
			\times [-\interestingu,\interestingu]
$,
the following estimates hold, where here and throughout, 
$A = \positivebigO(B)$ means that there exist constants $C_1, C_2$
(depending on $\seed$)\footnote{For example, the implicit constants present on RHS~\eqref{AE:PSLUNTMUTAYLOREXPANSIONININTERESTINGREGION} can be taken to be 
$C_1 = \maximumspeed^{-1}$ and $C_2 = \minimumspeed^{-1}$; see \eqref{AE:BOUNDSONLMUINTERESTINGREGION}.} such that 
$0 < C_1 \leq C_2$ and $C_1 B \leq A \leq C_2 B$,
and 
$A = \mathcal{O}(B)$ means that there exists a $C > 0$ (depending on $\seed$)
such that $|A| \leq C |B|$:
\begin{align}
			 \upmu(t,\eik) 
				& 
				= 
					\positivebigO(1)
					\frac{\PSdatamuHessianTaylorcoefficient}{2} \eik^2
					+
					\positivebigO(1) \blowupdelta 
					(\blowuptime - t), 
				\label{AE:PSMUTAYLOREXPANSIONININTERESTINGREGION} 
					\\
				\Lunit \upmu(t,\eik) 
				& 
				= - 
					\positivebigO(1)
					\blowupdelta, 
				\label{AE:PSLUNTMUTAYLOREXPANSIONININTERESTINGREGION} 
					\\
		\muX \upmu(t,\eik) 
		& 
			=
			\positivebigO(1)
			\PSdatamuHessianTaylorcoefficient \eik + \mathcal{O}(1) (\blowuptime - t),
			\label{AE:PSMUFIRSTDERIVATIVETAYLOREXPANSIONININTERESTINGREGION}
				\\
			\Lunit \muX \upmu(t,\eik)
				& 
				= 
					\mathcal{O}(1),
			\label{AE:PSLUNITMUXMUTAYLOREXPANSIONININTERESTINGREGION} 
					\\
		\muX \muX \upmu(t,\eik) 
		& = \positivebigO(1)
			\PSdatamuHessianTaylorcoefficient
			+
			\mathcal{O}(1)
			(\blowuptime - t).
			\label{AE:PSMUSECONDDERIVATIVETAYLOREXPANSIONININTERESTINGREGION}
\end{align}

In addition, we have:
	\begin{align} \label{AE:PSLOCATIONOFXMUEQUALSMINUSKAPPA}
			\lbrace \muX \upmu = 0 \rbrace
			\cap
			\Sigma_{\blowuptime}^{[-\interestingu,\interestingu]}
			& = \lbrace (\blowuptime,0) \rbrace
			\subset
			\Sigma_{\blowuptime}^{[-\frac{1}{4}\interestingu,\frac{1}{4}\interestingu]},
					\\
	\min_{\Sigma_{\blowuptime}^{[-\interestingu,\interestingu]}
			\backslash
			\Sigma_{\blowuptime}^{[-\frac{1}{2} \interestingu,\frac{1}{2} \interestingu]}}
			|\muX \upmu | 
			& \geq \frac{\PSdatamuHessianTaylorcoefficient \interestingu}{8},
			\label{AE:PSREGIONWHEREMUXMUKAPPALEVELSETISNOTLOCATED}
	\end{align}

\begin{align}  \label{AE:PSMUTRANSVERSALCONVEXITY}
\begin{split}
			\frac{1}{2} \maximumspeed^{-1} \blowuptime \PSdatamuHessianTaylorcoefficient
			& \leq 
			\min_{\Sigma_{\blowuptime}^{[-\interestingu,\interestingu]}}
				\muX \muX \upmu,
			\leq 
			\max_{\Sigma_{\blowuptime}^{[-\interestingu,\interestingu]}} 
				\muX \muX \upmu,
			\leq 
			2 \minimumspeed^{-1} \blowuptime \PSdatamuHessianTaylorcoefficient.
\end{split}
\end{align}

\medskip
\noindent \underline{\textbf{$\muPS$ is uniformly positive away from the interesting region}}.
With $\PSLmunottoonegativeparameter \in [0,1)$ denoting the non-negative constant on 
RHS~\eqref{E:RESCALEDFIRSTDERIVATIVEOFMUFUNCTIONUNIFORMLYPOSITIVEAWAYFROMINTERESTINGREGION},
we have the following estimate:
\begin{align} \label{AE:PSMUISLARGEAWAYFROMINTERESTINGREGION}
	\min_{\lbrace (t,\eik) \ | \ t \in [0,\blowuptime], \,  |u| \geq \interestingu \rbrace} \upmu(t,\eik)
	& 
	\geq \positivebigO(1) (1 - \PSLmunottoonegativeparameter).
\end{align}

\medskip
\noindent \underline{\textbf{Properties of $\PSLmusourcetermfunction$ in the interesting region}}.
Let $\PSLmusourcetermfunction$ be the source term from \eqref{E:PSLMUSOURCETERMFUNCTION}, 
viewed as a function of $\eik$, and let $\PSLmusourcetermfunction' := \frac{\mathrm{d}}{\mathrm{d} \eik} \PSLmusourcetermfunction$.
Then for $\eik \in [-\interestingu,\interestingu]$, the following estimates hold:
	\begin{subequations}
	\begin{align}  \label{E:ALTERNATERESCALEDMUFUNCTIONEXPANSIONSNEAR0}
			\PSLmusourcetermfunction(\eik)
			& = 
			- 
			\blowupdelta
			+
			\positivebigO(1) \eik^2
			= 
			-\positivebigO(1) \blowupdelta,
				\\
			\frac{1}
			{\PSLmusourcetermfunction(\eik)}
			& = 
			- 
			\frac{1}{\blowupdelta}
			-
			\positivebigO(1) \eik^2,	
				\label{E:ALTERNATERERECIPROCALSCALEDMUFUNCTIONEXPANSIONSNEAR0} 
					\\
			\PSLmusourcetermfunction'(\eik)
			& = 
			\positivebigO(1) \eik.
	\label{E:ALTERNATERESCALEDMUFUNCTIONANDDERIVATIVESEXPANSIONSNEAR0}		
	\end{align}
	\end{subequations}
		In particular:
		\begin{itemize}
			\item $\PSLmusourcetermfunction(\eik) < 0$ for $\eik \in [-\interestingu,\interestingu]$
			\item $\PSLmusourcetermfunction'(0) = 0$ 
			\item $\PSLmusourcetermfunction'(\eik) > 0$ for $\eik \in [0,\interestingu]$ 
			\item $\PSLmusourcetermfunction'(\eik) < 0$ for $\eik \in [-\interestingu,0]$.
		\end{itemize}
\end{lemma}

\begin{proof}
All of the results except for those concerning $\PSLmusourcetermfunction$
are straightforward consequences of Cor.\,\ref{C:PSEXPLICITEXPRESSIONSFORSOLUTION}, 
the properties of the initial data $\dataRRiemannPS(\eik)$ from Sect.\,\ref{SSS:1DADMISSIBLEDATAFORRRIEMANN}, 
and Taylor expansions; we omit the details.
	
	To prove the results for $\PSLmusourcetermfunction$,
	we first use \eqref{E:SOLUIONRRIEMANNTRIVIALTRANSPORTEQUATIONINGEOMETRICCOORDINATES}
	and \eqref{E:ANTIDERIVATIVEIDENTITYFORSOURCETERM}
	to deduce that
	$\PSLmusourcetermfunction(\eik)
	=
	\frac{\mathrm{d}}{\mathrm{d} \eik}
	\antiderivativePSLmusourcetermfunction[\dataRRiemannPS(\eik)] 
	$.
	Using this identity and
	\eqref{E:TAYLOREXPANSIONOFKEYMUTERMABOUTCREASE}, we conclude
	\eqref{E:ALTERNATERESCALEDMUFUNCTIONANDDERIVATIVESEXPANSIONSNEAR0},
	provided $\interestingu > 0$ is taken to be sufficiently small.
\end{proof}

\subsection{The crease, singular curve, and Cauchy horizon in $1D$ simple isentropic plane-symmetry} \label{SS:1DBOUNDARYOFMAXIMALDEVELOPMENT}
In this section, 
we provide a description of a localized portion of the boundary of the $\hfour$-MGHD. 
We emphasize the following key point, which follows from
\eqref{E:SOLUIONRRIEMANNTRIVIALTRANSPORTEQUATIONINGEOMETRICCOORDINATES}, 
\eqref{E:SPEEDTIMESMUPLANESYMMETRYWITHEXPLICITSOURCE},
and \eqref{E:CLOSEDPARTIAL1DERIVATIVEOFRPLUSPS}:

\begin{quote} For the initial data from Sect.\,\ref{SSS:1DADMISSIBLEDATAFORRRIEMANN},
the corresponding fluid solution $\RRiemannPS$ and the inverse foliation density 
$\upmu$ are \emph{explicit smooth functions} on all\footnote{In multi-dimensions, the best known estimates allow for the possibility of high order energy-blowup; see Sect.\,\ref{SSS:WAVEENERGYESTIMATEHIERARCHY}. 
However, even in multi-dimensions, the solution remains smooth at the low-to-mid order derivative levels.}  
of $\R^2$ \emph{in the $(t,\eik)$ differential structure}, even though 
the Minkowskian partial derivative $\partial_1 \RRiemannPS$ blows up in finite time.
\end{quote}

Therefore, when we speak of the $\hfour$-MGHD, we are referring to the standard differential structure corresponding to the $(t,x^1)$ coordinates. 
Nonetheless, \emph{we prefer to carry out most our analysis of the $\hfour$-MGHD in $(t,\eik)$ coordinates},
thanks to the explicit structure of the equations and solutions in geometric coordinates revealed in 
Cor.\,\ref{C:PSEXPLICITEXPRESSIONSFORSOLUTION}. Hence, to reach conclusions about the behavior of solution in the $(t,x^1)$-differential structure, we have to control the change of variables map $\Upsilon(t,\eik) := (t,x^1)$; indeed, 
in Theorem~\ref{T:MAINTHEOREM1DSINGULARBOUNDARYANDCREASE}, we reveal various diffeomorphism and homeomorphism properties of $\Upsilon$.

For the solutions and localized region under study, 
when described in geometric coordinates,
the boundary of the $\hfour$-MGHD consists of two pieces, depicted in Fig.\,\ref{F:1DPLANESYMMETRYCAUCHYHORANDSINGULARCURVE}: 
\begin{enumerate}
\item A singular curve portion 
$\singularcurve_{[-\interestingu,0]}$, emanating from the first shock point -- called the \emph{crease} and denoted by $\crease$ --
such that $\upmu$ vanishes on $\singularcurve_{[-\interestingu,0]}$.
\item A Cauchy horizon portion, denoted by $\cauchyhor_{[0,\interestingu]}$, which is a portion of 
the integral curve of $\uLunit$ that 
that emanates from $\crease$, and along which $\upmu$ is positive, 
except at $\crease$. 
\end{enumerate}
In particular, the image $\Upsilon(\singularcurve_{[-\interestingu,0]})$, 
depicted in Fig.\,\ref{F:MINKOWSKIRECTANGULAR1DPLANESYMMETRYCAUCHYHORANDSINGULARCURVE},
is a boundary of the $\hfour$-MGHD in the differential structure of the Minkowski-rectangular coordinates $(t,x^1)$ in the sense that $|\partial_1 \RRiemannPS|$ blows up when $\upmu$ vanishes.
In contrast, even though $\RRiemannPS$ remains smooth all the way up to 
$\Upsilon(\cauchyhor_{[0,\interestingu]} \backslash \crease)$, 
$\Upsilon(\cauchyhor_{[0,\interestingu]})$ is a boundary of the $\hfour$-MGHD  
because the crease $\Upsilon(\crease)$ -- a point at which $|\partial_1 \RRiemannPS|$ blows up -- 
lies in the $\hfour$-causal past of any point in $\Upsilon(\cauchyhor_{[0,\interestingu]})$. 
We highlight again that the singular behavior is only visible relative to the differential structure of Minkowski space in $(t,x^1)$ coordinates. The discrepancy between the behavior of the solution in the two coordinate systems
is tied to the breakdown in the diffeomorphism property of $\Upsilon$ 
when $\upmu$ vanishes; see \eqref{E:JACOBIANDETERMINANTOFCHANGEOFVARIABLESMAP1D}.

\subsubsection{Singular curve, crease, Cauchy horizon, and developments} \label{SSS:1DMAXIMALDEVBOUNDARYDEFINITIONS}
We now define some subsets of geometric coordinate space $\mathbb{R}_t \times \mathbb{R}_{\eik}$ that, 
in Theorem~\ref{T:MAINTHEOREM1DSINGULARBOUNDARYANDCREASE}, 
will play an important role in describing the shape of the $\hfour$-MGHD.
Many of our definitions are localized to the subset $\lbrace |\eik| \leq \interestingu \rbrace$,
where $\interestingu > 0$ is the constant from Lemma~\ref{L:SHARPESTIMATESFORMUNEARCREASE}.
This is mainly for convenience, since Lemma~\ref{L:SHARPESTIMATESFORMUNEARCREASE} provides sharp
control over the solution when $|\eik| \leq \interestingu$.

\begin{definition}[Truncated portions of the singular curve and the crease] \label{D:CREASEANDTRUNCATEDB}
For the solution launched by the compactly supported initial data $\dataRRiemannPS= \dataRRiemannPS(\eik)$ 
constructed in Sect.\,\ref{SSS:1DADMISSIBLEDATAFORRRIEMANN}, 
let $\PSLmusourcetermfunction$ be the corresponding source term from \eqref{E:PSLMUSOURCETERMFUNCTION},
viewed as a function of $\eik$. 
We define: 
\begin{align} \label{E:TISAFUNCTIONALONGSINGULARCURVE}
	\timeisafunctionofeikonalinsingularcurve(\eik) 
	& : = 
	- \frac{1}{\PSLmusourcetermfunction(\eik)}.
\end{align}

Let $I \subset [-\interestingu,\interestingu]$ be a subset of $\eik$-values,
where $\interestingu > 0$ is the constant from Lemma~\ref{L:SHARPESTIMATESFORMUNEARCREASE}.
We define the \emph{truncated singular curve} to be the
following subset of geometric coordinate space 
$\mathbb{R}_t \times \mathbb{R}_{\eik}$:
\begin{align} \label{E:TRUNCATEDPORTIONOFSINGULARCURVE} 
\begin{split} 
	\singularcurve_I 
	& 
	:= \big\{ (t,\eik) \in [0,\infty) \times I \ | \ \upmu(t,\eik) = 0 \big\} 
			\cap
		 \big\{ (t,\eik) \in [0,\infty) \times I \ | \ \muX \upmu(t,\eik) \leq 0 \big\}
			\\
	& = \big\{ 
				\left( \timeisafunctionofeikonalinsingularcurve(\eik),\eik \right) 
					\ | \ 
				\eik \in I
				\big\} 
			\cap
			[0,\infty) \times \mathbb{R}
			\cap
		 \big\{ (t,\eik) \in [0,\infty) \times I \ | \ \muX \upmu(t,\eik) \leq 0 \big\},		
\end{split}
\end{align}
where $\muX$ is the vectorfield defined in \eqref{E:SEVERALACOUSTICVECTORFIELDS},
and the second identity in \eqref{E:TRUNCATEDPORTIONOFSINGULARCURVE} follows from
\eqref{E:SPEEDTIMESMUPLANESYMMETRYWITHEXPLICITSOURCE},
\eqref{E:EASYSPEEDOVERNORMALIZERBOUNDS},
and
\eqref{E:TISAFUNCTIONALONGSINGULARCURVE}.

We define the \emph{crease}, which we denote by $\crease$, 
to be the following subset of geometric coordinate space $\mathbb{R}_t \times \mathbb{R}_{\eik}$
(Theorem~\ref{T:MAINTHEOREM1DSINGULARBOUNDARYANDCREASE} implies that $\crease$ is a single point):
\begin{align} \label{E:1DCREASE}
	\crease 
	& 
	:= \big\{ (t,\eik) \in [0,\infty) \times [-\interestingu,\interestingu] \ | \ \upmu(t,\eik) = 0 \big\} 
		\cap 
		\big\{ (t,\eik) \in [0,\infty) \times [-\interestingu,\interestingu] \ | \ \muX \upmu(t,\eik) = 0 \big\}.
\end{align}
\end{definition}

From
\eqref{E:SPEEDTIMESMUPLANESYMMETRYWITHEXPLICITSOURCE},
\eqref{E:MUXMUPLANESYMMETRYWITHEXPLICITSOURCE},
\eqref{E:EASYSPEEDOVERNORMALIZERBOUNDS},
\eqref{E:ALTERNATERERECIPROCALSCALEDMUFUNCTIONEXPANSIONSNEAR0},
\eqref{E:ALTERNATERESCALEDMUFUNCTIONANDDERIVATIVESEXPANSIONSNEAR0},
and \eqref{E:TISAFUNCTIONALONGSINGULARCURVE},
it follows that if 
$(t',\eik') \in \big\{ (t,\eik) \in  [0,\infty) \times [-\interestingu,\interestingu] \ | \ \upmu(t,\eik) = 0 \big\}$,
then $\muX \upmu(t',\eik') = \positivebigO(1) \eik'$ and thus $\muX \upmu(t',\eik') \leq 0 \iff \eik' \in [-\interestingu,0]$.
We therefore have the following alternate characterization of $\singularcurve_I$:
\begin{align} \label{E:ALTERNATETRUNCATEDPORTIONOFSINGULARCURVE} 
\begin{split} 
	\singularcurve_I 
	& = \big\{ 
				\left( \timeisafunctionofeikonalinsingularcurve(\eik),\eik \right) 
					\ | \ 
				\eik \in I \cap [-\interestingu,0]
				\big\} 
				\cap
				[0,\infty) \times \mathbb{R}.		
\end{split}
\end{align}

We make the following remarks.
\begin{itemize}
\item One might wonder why in the singular curve definition~\eqref{E:TRUNCATEDPORTIONOFSINGULARCURVE}, 
	there is a restriction to regions in which $\muX \upmu \leq 0$.
	This is because the portion of $\lbrace \upmu = 0 \rbrace$ on which we formally have $\muX \upmu > 0$
	never actually dynamically develops in the $\hfour$-MGHD; 
	that ``fictitious portion'' is cut off by the Cauchy horizon before it has a chance to dynamically form; see also
	Remark~\ref{R:NOTALLOFMUEQUALSZEROBELONGSTOMGHD}.
\item Our terminology ``crease'' is motivated by picture of the $\hfour$-MGHD in multi-dimensions, where the
	analog of $\crease$ is a co-dimension $2$ hypersurface given by the intersection of the singular
	boundary and the Cauchy horizon; see Fig.\,\ref{F:CONJECTUREDMAXDEVELOPMENTINMINKRECT}.
\item 
From \eqref{E:CLOSEDPARTIAL1DERIVATIVEOFRPLUSPS},
\eqref{E:TRUNCATEDPORTIONOFSINGULARCURVE}, 
and \eqref{E:ALTERNATETRUNCATEDPORTIONOFSINGULARCURVE},
we deduce the following key result: \emph{in regions where we can \textbf{justify} the use of geometric coordinates 
for reaching conclusions about the behavior of the solution in $(t,x^1)$ space, 
$\singularcurve_I$ is the subset of points $(t,\eik) \in [0,\infty) \times \left(I \cap [-\interestingu,0] \right)$ 
such that $|\p_1 \RRiemannPS(t,\eik)|$ blows up at $\Upsilon(t,\eik)$.}
In Theorem~\ref{T:MAINTHEOREM1DSINGULARBOUNDARYANDCREASE}, we will indeed justify the use of the geometric coordinates
when $I = [-\interestingu,0]$.
\item 
\eqref{E:SPEEDTIMESMUPLANESYMMETRYWITHEXPLICITSOURCE},
\eqref{E:1DUSEFULCHARACTERIZATIONOFBLOWUPDELTA},
and
\eqref{E:EASYSPEEDOVERNORMALIZERBOUNDS} 
imply that $\blowuptime$ is the smallest value of $t$ along $\singularcurve_{[-\interestingu,0]}$.
\item To further flesh out the previous point, we note that for the solutions 
	we study in Theorem~\ref{T:MAINTHEOREM1DSINGULARBOUNDARYANDCREASE},
	the smallest value of $t$ on $\singularcurve_{[-\interestingu,0]}$ will occur precisely at $\eik = 0$, and
	we will have $\timeisafunctionofeikonalinsingularcurve(0) = \blowuptime$.
\end{itemize}

\begin{definition}[Truncated Cauchy horizon] \label{D:CAUCHYHORIZON}
Assume that $\crease = (\blowuptime,0)$, where $\blowuptime$ is defined by \eqref{E:BLOWUPTIME}
(Theorem~\ref{T:MAINTHEOREM1DSINGULARBOUNDARYANDCREASE} shows that this assumption is satisfied for the solutions under study here).
Let $\newuL$ be the vectorfield defined in \eqref{E:SEVERALACOUSTICVECTORFIELDS}.
Let $\intcurvenewuL$ be the integral curve of $\frac{1}{2} \newuL$ in geometric coordinate space $\mathbb{R}_t \times \mathbb{R}_{\eik}$
parametrized by $\eik$ (note that $\frac{1}{2} \newuL \eik = 1$ by \eqref{E:NEWULINGEOMETRICCOORDINATES}),
with the initial condition $\intcurvenewuL(0) = \crease$.
That is, $\intcurvenewuL$ solves the following initial value problem:
\begin{subequations}
	\begin{align}  
		\frac{\mathrm{d}}{\mathrm{d} \eik} \intcurvenewuL(\eik) 
		& 
		= 
		\frac{1}{2}
		\newuL \circ \intcurvenewuL(\eik), 
		\label{E:INTEGRALCURVEOFNEWULODE} 
		\\
		\intcurvenewuL(0) & = \left(\blowuptime,0 \right). \label{E:INTEGRALCURVEOFNEWULINITIALCONDITION}
	\end{align}
\end{subequations}
We use the notation $\timeisafunctionofeikonalcauchyhorizon$ to denote the $t$-component of the curve $\intcurvenewuL$. 
That is, by \eqref{E:NEWULINGEOMETRICCOORDINATES}, we have:
\begin{align}  \label{E:1DUPARAMETERIZEDINTEGRALCURVEOFMULUNIT}
	\intcurvenewuL(\eik) 
	& = \left(\timeisafunctionofeikonalcauchyhorizon(\eik),\eik \right),
\end{align}
where $\timeisafunctionofeikonalcauchyhorizon$ solves the following initial value problem:
\begin{align} \label{E:ODEFORTCOMPONENTOFMUULUNITINTEGRALCURVE}
	\frac{\mathrm{d}}{\mathrm{d} \eik} \timeisafunctionofeikonalcauchyhorizon(\eik)
	& = \frac{1}{2} \upmu\left(\timeisafunctionofeikonalcauchyhorizon(\eik),\eik \right),
	& 
	\timeisafunctionofeikonalcauchyhorizon(0)
	& = \blowuptime.
\end{align}

Then for an interval $I \subset [0,\infty)$ of non-negative\footnote{In Theorem~\ref{T:MAINTHEOREM1DSINGULARBOUNDARYANDCREASE}, the 
Cauchy horizon will correspond to non-negative values of $\eik$.} 
$\eik$-values, 
we define the truncated Cauchy horizon to be: 
\begin{align} \label{E:CAUCHYHORIZON}
	\cauchyhor_I 
	& := \intcurvenewuL(I). 
\end{align}
\end{definition}

We now define $\classicaldev$, which is a subset of $\mathbb{R}_t \times \mathbb{R}_{\eik}$ 
that is tailored to the shape of the $\hfour$-MGHD.

\begin{definition}[Classical development regions]
\label{D:CLASSICALDEVELOPMENTNEARVANISHINGOFINVERSEFOLIATIONDENSITY}
Let $\singularcurve_{[-\interestingu,0]}$ and $\cauchyhor_{[0,\interestingu]}$ denote the truncated singular curve and Cauchy horizon from Defs.\,\ref{D:CREASEANDTRUNCATEDB}--\ref{D:CAUCHYHORIZON}, 
where $0 < \interestingu < 1$ is the constant from Lemma~\ref{L:SHARPESTIMATESFORMUNEARCREASE}. 
Let $\timeisafunctionofeikonalinsingularcurve$ and $\timeisafunctionofeikonalcauchyhorizon$
be the functions of $\eik$ from \eqref{E:TISAFUNCTIONALONGSINGULARCURVE} and \eqref{E:1DUPARAMETERIZEDINTEGRALCURVEOFMULUNIT}.
We define the following subsets of geometric coordinate space $\mathbb{R}_t \times \mathbb{R}_{\eik}$:
\begin{align}
	\classicaldevsingular 
	& 
	:= 
	\left\{ 
		(t,\eik) \in \mathbb{R}^2 
		\ | \  
		0 \leq t  < \timeisafunctionofeikonalinsingularcurve(\eik), \, \eik \in [-\interestingu,0]
	\right\}, 
	\label{E:SINGULARPORTIONOFCLASSICALDEVBOUNDARYININTERESTINGREGION} 
		\\
	\classicaldevregular 
	& := 
	\left\{ 
		(t,\eik) \in \mathbb{R}^2 
		\ | \  0 \leq t \leq \timeisafunctionofeikonalcauchyhorizon(\eik), \, \eik \in (0, \interestingu]
	\right\}, 
		\label{E:REGULARPORTIONOFCLASSICALDEVBOUNDARYININTERESTINGREGION} \\
	\classicaldev & :=\classicaldevsingular \cup\classicaldevregular. \label{E:INTERESTINGREGION} 
\end{align} 

\end{definition}

The set $\classicaldev$ defined by \eqref{E:INTERESTINGREGION}
is the region of geometric coordinate space 
lying between $\Sigma_0^{[-\interestingu,\interestingu]}$ and $\cauchyhor_{(0,\interestingu]} \cup \singularcurve_{[-\interestingu,0]}$, 
where the boundary portions 
$\cauchyhor_{(0,\interestingu]}$, 
$\nullhyparg{-\interestingu}^{\left[0,\timeisafunctionofeikonalinsingularcurve(-\interestingu)\right)}$,
and
$\nullhyparg{\interestingu}^{\left[0,\timeisafunctionofeikonalcauchyhorizon(\interestingu)\right]}$ are included, 
but $\singularcurve_{[-\interestingu,0]}$ is not; see Fig.\,\ref{F:1DPLANESYMMETRYCAUCHYHORANDSINGULARCURVE}. 
Theorem~\ref{T:MAINTHEOREM1DSINGULARBOUNDARYANDCREASE} shows that $\Upsilon(\classicaldev) \subset \mathbb{R}_t \times \mathbb{R}_{x^1}$ 
is a subset of Minkowski-rectangular coordinate space where we can \emph{justify} 
the use of $(t,\eik)$ coordinates to reach conclusions about the behavior of the solution in $(t,x^1)$ coordinates.\footnote{With some additional effort, given the results of Theorem\,\ref{T:MAINTHEOREM1DSINGULARBOUNDARYANDCREASE}, we could extend the region of spacetime on which $\Upsilon(t,\eik) = (t,x^1)$ is a homeomorphism to include the infinite region $\classicaldev^{\textnormal{Large}} 
:= 
\classicaldev  
\bigcup 
\left(
[0,\blowuptime] \times (-\infty,-\rightu) \cup (\leftu,\infty)
\right)$. 
We refrain from  
giving the proof because it would require some technical, but uninteresting, adjustments of the arguments presented in the proof of the theorem. The details are uninteresting in the sense that $\upmu$ is uniformly positive on 
$\classicaldev^{\textnormal{Large}} \setminus \classicaldev$ 
(see \eqref{AE:PSMUISLARGEAWAYFROMINTERESTINGREGION})
and thus no singularity occurs in that region.}

\subsubsection{The main theorem} \label{SSS:1DMAXIMALDEVMAINTHM}
We now prove Theorem~\ref{T:MAINTHEOREM1DSINGULARBOUNDARYANDCREASE}, which is the main result of 
Sect.\,\ref{S:1DMAXIMALDEVELOPMENT}. On the region $\classicaldev$, the theorem provides a sharp description of the solution in the differential structure of the geometric coordinates $(t,\eik)$ as well as the behavior of the solution on
$\Upsilon(\classicaldev)$ in differential structure of the Minkowski-rectangular coordinates $(t,x^1)$,
where we recall that $\Upsilon(t,\eik) = (t,x^1)$ is the change of variables map.
In particular, the theorem reveals various diffeomorphism and homeomorphism properties of $\Upsilon$ and exhibits the blowup 
of $|\partial_1 \RRiemannPS|$ as $\Upsilon(\singularcurve_{[-\interestingu,0]})$ is approached by points in $\Upsilon(\classicaldev)$.
In Fig.\,\ref{F:1DPLANESYMMETRYCAUCHYHORANDSINGULARCURVE}, we have depicted the region in geometric coordinate space 
$\mathbb{R} \times \mathbb{R}_{\eik}$.
In Fig.\,\ref{F:MINKOWSKIRECTANGULAR1DPLANESYMMETRYCAUCHYHORANDSINGULARCURVE}, we have depicted the 
region in Minkowski-rectangular coordinate space $\mathbb{R} \times \mathbb{R}_{x^1}$, i.e., the image
under $\Upsilon$ of the region from geometric coordinate space.

\begin{center}
\begin{figure}  
\begin{overpic}[scale=.7, grid = false, tics=5, trim=-.5cm -1cm -1cm -.5cm, clip]{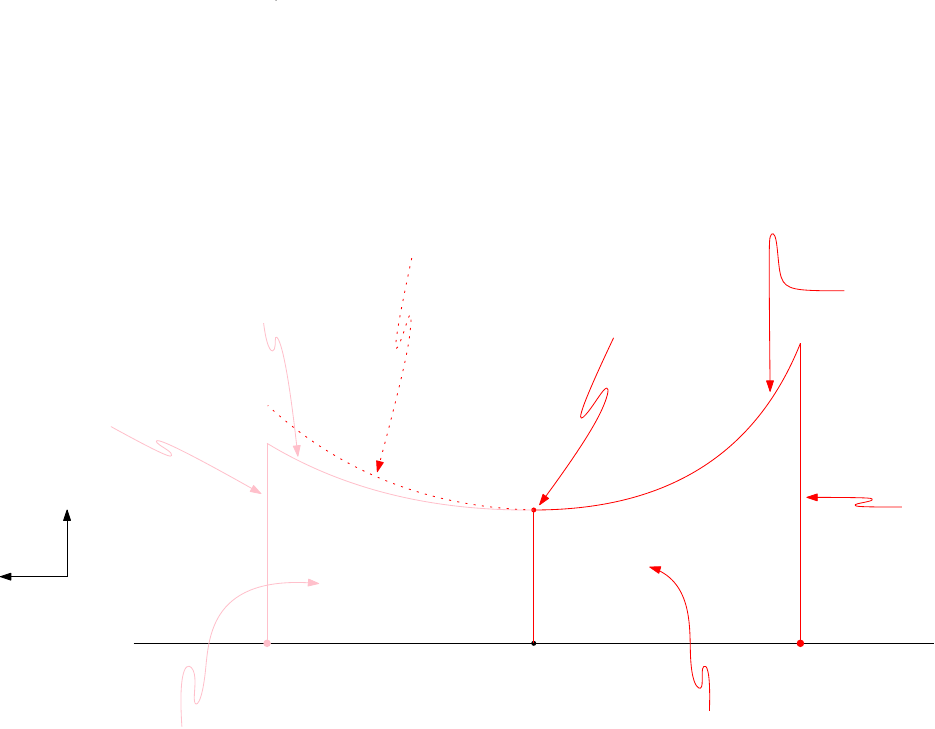}
\put (52,11) {$(0,0)$}
\put (26,11) {$(0,\interestingu)$}
\put (77,11) {$(0,-\interestingu)$}
\put (60,51) {$\crease$}
\put (62,48) {\rotatebox{90}{$=$}}
\put (58,45) {$(\blowuptime,0)$}
\put (28,46.5) {$\cauchyhor_{[0,\interestingu]}$}
\put (18,3.5) {$\classicaldevregular$}
\put (69,5) {$\classicaldevsingular$}
\put (86,47.5) {$\singularcurve_{[-\interestingu,0]}$}
\put (92,43) {\rotatebox{90}{$=$}}
\put (86,40.5) {$\lbrace \upmu = 0 \rbrace \cap \lbrace \muX \upmu \leq 0 \rbrace$}
\put (97.5,37.5) {$\cap$}
\put (90,34.5) {$\lbrace \eik \in [-\interestingu,0] \rbrace$}
\put (33,53) {$\lbrace \upmu = 0 \rbrace \cap \lbrace \muX \upmu > 0 \rbrace$}
\put (4,21) {$\eik$}
\put (8,25.5) {$t$}
\put (95,13) {$\Sigma_0$}  
\put (92,26.5) {$\nullhyparg{-\interestingu}^{\left[0,\timeisafunctionofeikonalinsingularcurve(-\interestingu)\right)}$}
\put (6,35.5) {$\nullhyparg{\interestingu}^{\left[0,\timeisafunctionofeikonalcauchyhorizon(\interestingu)\right]}$}                  
\end{overpic}
\caption{$\classicaldev$ in geometric coordinate space, with $\eik$ increasing from right to left}
\label{F:1DPLANESYMMETRYCAUCHYHORANDSINGULARCURVE}
\end{figure}
\end{center}

\begin{remark}[Only a subset of $\lbrace \upmu = 0 \rbrace$ is relevant for the $\hfour$-MGHD]
	\label{R:NOTALLOFMUEQUALSZEROBELONGSTOMGHD}
	The dotted curve in Fig.\,\ref{F:1DPLANESYMMETRYCAUCHYHORANDSINGULARCURVE}
	labeled ``$\lbrace \upmu = 0 \rbrace \cap \lbrace \muX \upmu > 0 \rbrace$''
	is not part of the $\hfour$-MGHD or its boundary. Formally, $\lbrace \upmu = 0 \rbrace \cap \lbrace \muX \upmu > 0 \rbrace$
	is the portion of the level set $\lbrace \upmu = 0 \rbrace$
	that never has a chance to dynamically emerge because it is cut off by the Cauchy horizon portion
	$\cauchyhor_{(0,\interestingu]}$, \emph{which lies below it}
	(as the proof of Theorem~\ref{T:MAINTHEOREM1DSINGULARBOUNDARYANDCREASE} shows).
	For this reason, we do not display $\Upsilon\left(\lbrace \upmu = 0 \rbrace \cap \lbrace \muX \upmu > 0 \rbrace \right)$
	in Fig.\,\ref{F:MINKOWSKIRECTANGULAR1DPLANESYMMETRYCAUCHYHORANDSINGULARCURVE}.
	One can use Taylor expansions to show that formally, 
	$\Upsilon\left(\lbrace \upmu = 0 \rbrace \cap \lbrace \muX \upmu > 0 \rbrace \right)$ would 
	be a curve in $\Upsilon(\classicaldevsingular)$ that lies \emph{below} $\Upsilon(\singularcurve_{[-\interestingu,0]})$
	and thus formally, $\Upsilon$ would not even be injective if its domain extended 
	up to $\lbrace \upmu = 0 \rbrace \cap \lbrace \muX \upmu > 0 \rbrace$.
\end{remark}

\begin{center}
\begin{figure}  
\begin{overpic}[scale=.7, grid = false, tics=5, trim=-.5cm -1cm -1cm -.5cm, clip]{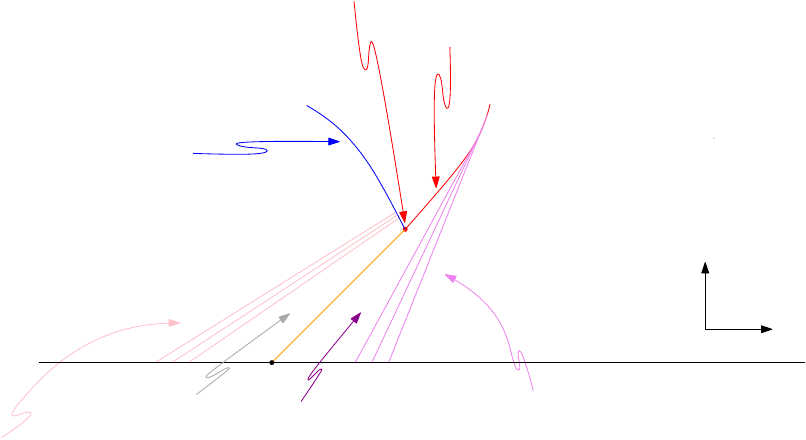}
\put (87,20) {$x^1$}
\put (83,24.5) {$t$}
\put (53,9) {\small $\mbox{Shock-forming}$}
\put (50,5) {\small $\mbox{characteristic family}$}
\put (62,1) {\rotatebox{90}{$=$}}
\put (49,-2) {\small $\mbox{Integral curves of } \Lunit$}
\put (9,39.5) {$\Upsilon(\cauchyhor_{[0,\interestingu]})$}    
\put (17,35.5) {\rotatebox{90}{$=$}}
\put (8,32.5) {\small $\mbox{Cauchy horizon}$}  
\put (17,28.5) {\rotatebox{90}{$=$}}
\put (7,25.5) {\small $\mbox{An integral curve}$}     
\put (14,21.5) {\small $\mbox{of } \newuL$}                                          
\put (36.5,56.5) {$\Upsilon(\crease)$}
\put (48,52) {$\Upsilon(\singularcurve_{[-\interestingu,0]})$}
\put (-8,4)  {\small $\mbox{Characteristic family}$}           
\put (-8,0)  {\small $\mbox{``blocked'' by } \Upsilon(\cauchyhor_{[0,\interestingu]})$}
\put (94.5,14) {$\Sigma_0$}
\put (18,8.5) {$\Upsilon(\classicaldevregular)$}
\put (33,7.5) {$\Upsilon(\classicaldevsingular)$}
\end{overpic}                                                
\caption{$\Upsilon(\classicaldev)$ in Minkowski-rectangular coordinate space}
\label{F:MINKOWSKIRECTANGULAR1DPLANESYMMETRYCAUCHYHORANDSINGULARCURVE}                                         
\end{figure}
\end{center}

\begin{theorem}[A large portion of the $\hfour$-MGHD, including a localized portion of its boundary] 
\label{T:MAINTHEOREM1DSINGULARBOUNDARYANDCREASE}
Let $\dataRRiemannPS$ be initial data satisfying the assumptions of Sect.\,\ref{SSS:1DADMISSIBLEDATAFORRRIEMANN},
and let $\RRiemannPS$ be the corresponding solution of the quasilinear initial value problem \eqref{E:1DCAUCHYIVPFORRIEMANNINVARIANTS}
(with $\LRiemannPSdata \equiv 0$).
Under the assumptions and conclusions of Lemma~\ref{L:SHARPESTIMATESFORMUNEARCREASE},
perhaps shrinking the constant $\interestingu > 0$ if necessary, the following results hold,
where the notation $\positivebigO(1)$ is defined in the statement of
Lemma~\ref{L:SHARPESTIMATESFORMUNEARCREASE}.

\medskip
\noindent \underline{\textbf{Classical existence with respect to the geometric coordinates.}} 
With respect to the geometric coordinates $(t,\eik)$, the following results hold.

\begin{itemize}	
	\item $\RRiemannPS$ is a smooth function of $\eik$ alone, i.e., $\RRiemannPS(t,\eik) = \dataRRiemannPS(\eik)$. 
	\item $\RRiemannPS$, the null vectorfields $\Lunit$ and $\newuL$, 
	the vectorfield $\muX$,
	and the inverse foliation density $\upmu$ exist classically on 
	$\mathbb{R}_t \times \mathbb{R}_{\eik}$.
	\item  On $\{(t,\eik) \in \R^2 \ | \ \eik \notin [-\rightu,\leftu]\}$,
			$\RRiemannPS$ vanishes,
			$\Lunit^1 := \Lunit x^1 = \overline{c}$, 
			$\newuL^1 := \newuL x^1 = - \overline{c}$, 
			and $\upmu = 1/\overline{c}$,
			where the constant $\overline{c} > 0$ denotes the speed of sound when the enthalpy is constantly $\overline{\Enth}$,
			where $\overline{\Enth} > 0$ is the constant fixed in \eqref{E:FIXEDPOSITIVEENTHALPYCONSTANT}.
		\item The following estimate holds:
			\begin{align} \label{E:QUANTITATIVENEGATIVITYOFSPEEDTIMESLMU}
				\Lunit \upmu(t,\eik) 
				& = - \positivebigO(1)\blowupdelta,
				&
				& \mbox{for } (t,\eik) \in \R^2 \times [-\interestingu,\interestingu].
			\end{align}
			In particular, with $\overline{\classicaldev}$ denoting the closure of 
			the region $\classicaldev$ defined in \eqref{E:INTERESTINGREGION},
			since Def.\,\ref{D:CLASSICALDEVELOPMENTNEARVANISHINGOFINVERSEFOLIATIONDENSITY} implies that
			$\overline{\classicaldev} \subset \R^2 \times [-\interestingu,\interestingu]$,
			it follows that the estimate \eqref{E:QUANTITATIVENEGATIVITYOFSPEEDTIMESLMU} holds on
			$\overline{\classicaldev}$.
\end{itemize}

\medskip
\noindent \underline{\textbf{Description of the $\hfour$-MGHD in geometric coordinates.}} 
	\begin{itemize}
		\item Let $\singularcurve_{[-\interestingu,0]}$ be the singular curve portion
		from \eqref{E:ALTERNATETRUNCATEDPORTIONOFSINGULARCURVE},
		and let $\classicaldev$ be the subset of geometric coordinate space defined in \eqref{E:INTERESTINGREGION}.
		Then the following disjoint union holds,
		where $\overline{\classicaldev}$ is the closure of $\classicaldev$: 
\begin{align} \label{E:CLOSUREOFCLASSICALDEVREGION}
	\overline{\classicaldev} 
	& = \classicaldev \sqcup \singularcurve_{[-\interestingu,0]}.
\end{align}
\item The crease, which is defined in \eqref{E:1DCREASE}, 
			is a single point:
			\begin{align} \label{E:MAINTHMCREASEISSINGLEPOINT1D}
				\crease 
				& = \{ \left(\blowuptime, 0 \right) \}. 
			\end{align}
		\item The singular curve function $\timeisafunctionofeikonalinsingularcurve = \timeisafunctionofeikonalinsingularcurve(\eik)$ 
			from \eqref{E:TISAFUNCTIONALONGSINGULARCURVE} 
			and
			\eqref{E:ALTERNATETRUNCATEDPORTIONOFSINGULARCURVE}
			satisfies the following estimates:
			\begin{subequations}
			\begin{align}  \label{E:TAYLOREXPANSIONSUDERIVATIVEOFSINGULARCURVE}
				\frac{\mathrm{d} }{\mathrm{d} \eik} \timeisafunctionofeikonalinsingularcurve(\eik) 
				& = 
				\positivebigO(1) \eik, 
				& & \eik \in [-\interestingu,0], 
					\\
				\timeisafunctionofeikonalinsingularcurve(\eik) 
				& = 
				\blowuptime 
				+ 
				\positivebigO(1) \eik^2, 
				& & \eik \in [-\interestingu,0].
				\label{E:TAYLOREXPANSIONSOFSINGULARCURVE}
			\end{align}
			\end{subequations}
		\item The rectangular spatial coordinate $x^1$
		satisfies the following estimate on $\singularcurve_{[-\interestingu,0]}$:
		\begin{align} \label{E:TAYLOREXPANSIONSOFDERIVATIVERECTANGULARX1ALONGSINGULARCURVE}
			\frac{\mathrm{d} }{\mathrm{d} \eik} x^1\left(\timeisafunctionofeikonalcauchyhorizon(\eik),\eik \right) 
			& 
			= 
			\positivebigO(1) \eik, 
			& & \eik \in [-\interestingu,0]. 
		\end{align}
		\item The integral curve initial value problem for $\intcurvenewuL(\eik)$ given by 
		\eqref{E:INTEGRALCURVEOFNEWULODE}--\eqref{E:INTEGRALCURVEOFNEWULINITIALCONDITION} 
		has a unique solution on the interval $\eik \in [0,\interestingu]$. 
		Hence, the truncated Cauchy horizon $\cauchyhor_{[0,\interestingu]}$, 
		defined in \eqref{E:CAUCHYHORIZON}, 
		can be parameterized as follows:
		\begin{align} \label{E:MAIN1DTHEOREMPARAMETERIZATIONOFTRUNCATEDCAUCHYHORIZON}
			\cauchyhor_{[0,\interestingu]} 
			& = \{ \left(\timeisafunctionofeikonalcauchyhorizon(\eik),\eik \right) \ | \ \eik \in [0,\interestingu]\},
		\end{align}
		as in \eqref{E:CAUCHYHORIZON}.
		Moreover, $\timeisafunctionofeikonalcauchyhorizon(\eik)$ satisfies the following estimates:
			\begin{subequations}
			\begin{align} \label{E:TAYLOREXPANSIONUDERIVATIVEOFCAUCHYHORIZON}
				\frac{\mathrm{d} }{\mathrm{d} \eik} \timeisafunctionofeikonalcauchyhorizon(\eik) 
				& 
				= 
				\positivebigO(1) \eik^2, 
				& & \eik \in [0,\interestingu],
					\\
				\timeisafunctionofeikonalcauchyhorizon(\eik) 
				& 
				= 
				\blowuptime 
				+ 
				\positivebigO(1) \eik^3, 
				& & \eik \in [0,\interestingu].
				\label{E:TAYLOREXPANSIONOFCAUCHYHORIZON}
			\end{align}
			\end{subequations}
		\item The inverse foliation density $\upmu$, 
		defined in \eqref{E:INVERSEFOLIATIONDENSITYINPLANESYMMETRY},
		satisfies the following estimate along the Cauchy horizon portion $\cauchyhor_{[0,\interestingu]}$:
		\begin{align} \label{E:TAYLOREXPANSIONOFMUALONGCAUCHYHORIZON}
			\upmu\left(\timeisafunctionofeikonalcauchyhorizon(\eik),\eik \right) 
			& 
			= 
			\positivebigO(1) \eik^2, 
			& & \eik \in [0,\interestingu]. 
		\end{align}
	\item The rectangular spatial coordinate $x^1$
		satisfies the following estimate along the Cauchy horizon portion $\cauchyhor_{[0,\interestingu]}$:
		\begin{align} \label{E:TAYLOREXPANSIONDERIVATIVEOFRECTANGULARX1ALONGCAUCHYHORIZON}
			\frac{\mathrm{d} }{\mathrm{d} \eik} x^1\left(\timeisafunctionofeikonalcauchyhorizon(\eik),\eik \right) 
			& 
			= 
			-
			\positivebigO(1) \eik^2, 
			& & \eik \in [0,\interestingu]. 
		\end{align} 
	\item $\upmu$ is positive on $\classicaldev$ but vanishes on $\singularcurve_{[-\interestingu,0]}$.
		In particular, \eqref{E:TAYLOREXPANSIONOFMUALONGCAUCHYHORIZON} implies
		that $\upmu$ is positive on $\cauchyhor_{(0,\interestingu]}$
		but vanishes on the crease $\crease \subset \cauchyhor_{[0,\interestingu]}$, which corresponds to $\eik = 0$ in 
		\eqref{E:TAYLOREXPANSIONOFMUALONGCAUCHYHORIZON}.
	\end{itemize}

\medskip
\noindent \underline{\textbf{Description of the singularity formation and the $\hfour$-MGHD in Minkowski-rectangular coordinates.}}
	\begin{itemize}
		\item Let $\Upsilon(t,\eik) := (t,x^1)$ be the change of variables map from geometric coordinates to Minkowski-rectangular coordinates.
			Then $\Upsilon$ is a diffeomorphism from $\classicaldev$ onto its image. Moreover,
			$\Upsilon$ is a \textbf{homeomorphism} from $\overline{\classicaldev}$ onto its image. 
		\item The solution $\RRiemannPS$ exists classically relative to the Minkowski-rectangular coordinates $(t,x^1)$ 
			on the subset $\Upsilon(\classicaldev)$ of $\mathbb{R}_t \times \mathbb{R}_{x^1}$.
		\item The following estimate holds on $\Upsilon\left(\classicaldev \right)$:
			\begin{align}\label{E:BLOWUPOFCARTESIANCOORDINATES}
				\left| [\p_1 \RRiemannPS] \circ \Upsilon^{-1}(t,x^1) \right| 
				& = 
				\frac{1}{\upmu} 
				\left|  
					\frac{\normalizer}{\speed}  \muX \RRiemannPS 
				\right| \circ \Upsilon^{-1}(t,x^1) 
				= 
				\frac{1}{\upmu} \positivebigO(1), & & (t,x^1) \in \Upsilon\left(\classicaldev \right). 
			\end{align}
			In particular, since $\upmu$ vanishes along $\Upsilon(\singularcurve_{[0,\interestingu]})$,
			given any point $(t_*,x_*^1) \in \Upsilon(\singularcurve_{[0,\interestingu]})$
			and any sequence of points $\{ (t_m, x_m^1) \}_{m \in \mathbb{N}} \subset \Upsilon\left(\classicaldev \right)$ 
			with $\lim_{m \to \infty} (t_m, x_m^1) = (t_*,x_*^1)$, 
			we have that $|\p_1 \RRiemannPS| \circ \Upsilon^{-1} \left((t_m, x_m^1) \right) \to \infty$ as $m \to \infty$. 
		\item Recall that the vectorfield $\Lunit$ is defined relative to the $(t,x^1)$-coordinates in \eqref{E:1DLUNIT}
		(see also \eqref{E:1DLUNITINTERMSOFRIEMANNINVARIANTS}).
		With respect to the $(t,x^1)$-coordinates,
		the solution $\RRiemannPS$ and all of its $\Lunit$-derivatives of any order, i.e., $\Lunit^M \RRiemannPS$ for any $M$, 
		are continuous on $\Upsilon(\overline{\classicaldev})$, 
		including the singular curve portion $\Upsilon(\singularcurve_{[0,\interestingu]})$. 
		Moreover,
		on $\Upsilon(\overline{\classicaldev})$,
		relative to the $(t,x^1)$-coordinates,
		the $1D$ acoustical metric $\hfour_{\alpha \beta}$ from \eqref{E:1DACOUSTICALMETRIC}
		is a continuous, non-degenerate, $2 \times 2$ matrix of signature $(-,+)$.
	\end{itemize}

\medskip
\noindent \underline{\textbf{The causal structure of the Cauchy horizon.}} 
		\begin{itemize}
			\item $\cauchyhor_{[0,\interestingu]}$ 
				is a $\hfour$-null curve portion, i.e., its tangent vector
				is null with respect to $\hfour$.
		\end{itemize}
	
\medskip
\noindent \underline{\textbf{The causal structure of the singular curve.}} 
	\begin{itemize}
		\item Relative to the geometric coordinates $(t,\eik)$, 
		the vectorfield $\singcurvevectorfield$ defined by:
		\begin{align} \label{E:SINGCURVEVECTORFIELD}
			\singcurvevectorfield 
			& : = \Lunit -  \frac{\Lunit(\frac{\speed}{\normalizer} \upmu)}{\muX (\frac{\speed}{\normalizer} \upmu)} \muX
		\end{align}
		is well-defined on the singular curve portion $\singularcurve_{[-\interestingu,0)}$, and $\singcurvevectorfield$ 
		is tangent to $\singularcurve_{[-\interestingu,0)}$.
		\item On $\singularcurve_{[-\interestingu,0)}$, we have 
		$\singcurvevectorfield t > 0$ and $\singcurvevectorfield \eik < 0$. In particular,
		$\singcurvevectorfield$ is future-directed and transversal to the characteristics 
		$\nullhyparg{\eik}$.
		\item For every $q \in \singularcurve_{[-\interestingu,0)}$, 
		the pushforward vectorfield $[\mathrm{d} \Upsilon(q)] \cdot \singcurvevectorfield(q)$, which is tangent to the curve
		$\Upsilon(\singularcurve_{[-\interestingu,0)}) \subset \mathbb{R}_t \times \mathbb{R}_{x^1}$, 
		is equal to $\Lunit|_{\Upsilon(q)} = [\Lunit^{\kappa} \partial_{\kappa}]_{\Upsilon(q)}$.
		\item For every $q \in \singularcurve_{[-\interestingu,0)}$, $\Lunit|_{\Upsilon(q)}$ is $\hfour$-orthogonal to the tangent space of  
		$\Upsilon(\singularcurve_{[-\interestingu,0)})$.
	\end{itemize}
	Since there exist integral curves of the $\hfour$-null vectorfield $\Lunit^{\kappa} \partial_{\kappa}$ that foliate and are $\hfour$-orthogonal to $\Upsilon(\singularcurve_{[-\interestingu,0)})$, it follows that $\Upsilon(\singularcurve_{[-\interestingu,0)})$ is a $\hfour$-null curve in the Minkowskian coordinate differential structure.

\end{theorem} 

Before proving the theorem, we make some remarks.

\begin{itemize}
	\item All the important results of the theorem are stable under perturbations of the initial data of
			$\RRiemannPS$. That is, even though the theorem was stated only for initial data that satisfy 
			the assumptions of Sect.\,\ref{SSS:1DADMISSIBLEDATAFORRRIEMANN}, the proof shows that
			if one perturbs any of those data by an arbitrary (small) $C^4$ function, then the corresponding
			perturbed simple isentropic plane-wave solution has an $\hfour$-MGHD/exhibits singularity formation 
			that is quantitatively and qualitatively close to the un-perturbed solution.
	\item In a similar vein as the previous point, with only modest additional effort, the results of
		Theorem~\ref{T:MAINTHEOREM1DSINGULARBOUNDARYANDCREASE} could be extended to 
		a large, open set of (non-simple) isentropic plane-symmetric solutions
		in which $\LRiemannPS$ is sufficiently small compared to $\RRiemannPS$
		(or, alternatively, in which $\RRiemannPS$ is sufficiently small compared to $\LRiemannPS$).
	\item The emergence of the Cauchy horizon from the crease is fundamentally tied to the 
		fact that in $1D$, the isentropic compressible Euler equations have two characteristic directions,
		$\Lunit$ and $\uLunit$, as is evident from equations \eqref{E:RRIEMANNPSTRANSPORT}--\eqref{E:LRIEMANNPSTRANSPORT}.
		In particular, even within the class of simple isentropic plane-symmetric solutions
		in which $\LRiemannPS$ (which is transported by $\uLunit$) vanishes, 
		when investigating the structure of the $\hfour$-MGHD,
		one cannot ``ignore'' the $\uLunit$ direction.
		This is in contrast to Burgers' equation, in which there is only one characteristic direction.
		Relatedly, for Burgers' equation, there does not exist a notion of ``globally hyperbolic''
		that enjoys the same fundamental significance as the notion of ``globally hyperbolic''
		in the context of Lorentzian geometry, relativistic Euler flow, or wave equations.
	\item From the point of view of determinism, the Cauchy horizon is an ally in that it ``blocks'' the non-uniqueness
		of classical solutions (in $\hfour$-globally hyperbolic regions). 
		In particular, in Fig.\,\ref{F:MINKOWSKIRECTANGULAR1DPLANESYMMETRYCAUCHYHORANDSINGULARCURVE}, 
		one can see that the Cauchy horizon prevents the characteristics in the left of the figure from entering into the region on the right,
		where the fluid's gradient blows up. That is, the Cauchy horizon prevents crossing of the characteristics
		$\Upsilon(\classicaldev)$, which in turn prevents multi-valuedness/non-uniqueness of classical solutions in this region.
		This should be contrasted against Burgers' equation, 
		which exhibits non-uniqueness of classical solutions due to multi-valuedness stemming from the crossing of
		characteristics emanating from widely separated regions in the initial data hypersurface.
	\item In Fig.\,\ref{F:MINKOWSKIRECTANGULAR1DPLANESYMMETRYCAUCHYHORANDSINGULARCURVE}, 
		along $\Upsilon(\singularcurve_{[-\interestingu,0]})$,
		one sees the interesting phenomenon of
		\emph{non-uniqueness} of integral curves of $\Lunit$. In particular, $\Lunit$ is tangent to
		$\Upsilon(\singularcurve_{[-\interestingu,0]})$
		\underline{and} the characteristics, which are the straight lines in the figure
		that intersect $\Upsilon(\singularcurve_{[-\interestingu,0]})$ tangentially.
		The non-uniqueness is tied to the blowup of $|\partial_1 \Lunit^1|$
		along $\Upsilon(\singularcurve_{[-\interestingu,0]})$, which in particular 
		shows that $\Lunit^{\alpha}$ does not enjoy the Lipschitz regularity that lies
		behind standard ODE uniqueness theorems for integral curves.
	\item One can show that in $1D$, the acoustical metric takes the following form in geometric coordinates: 
		$\hfour 
		= 
		- 
		\upmu \mathrm{d} t \otimes \mathrm{d} \eik 
		- 
		\upmu \mathrm{d} \eik \otimes \mathrm{d} t
		+ 
		\upmu^2 \mathrm{d} \eik \otimes \mathrm{d} \eik$. 
		In particular, in geometric coordinates, $\hfour$ \emph{vanishes} at points such that
		$\upmu  = 0$. Hence, $\hfour$ is not capable of ``measuring'' causality along the singular curve portion
		$\singularcurve_{[-\interestingu,0]}$, along which $\upmu$ vanishes.
		We note, however, that in more than one space dimension, even when $\upmu = 0$,
		there are some remaining non-degenerate components of $\hfour$ in geometric coordinates.
		Relatedly, in our $3D$ paper \cite{lAjS2022} on the non-relativistic compressible Euler equations, 
		we prove that the $\crease$ is a co-dimension $2$ $\hfour$-spacelike submanifold of geometric coordinate space.
\end{itemize}

\begin{proof}[Proof of Theorem~\ref{T:MAINTHEOREM1DSINGULARBOUNDARYANDCREASE}] 
Throughout the proof, we will silently shrink the size of $\interestingu$ as necessary; 
this will in particular guarantee that some important terms arising in various 
Taylor expansions have a sign.
\ \\

\noindent \textbf{Classical existence and properties of the solution relative to the geometric coordinates:} The fact that 
$\RRiemannPS(t,\eik) = \dataRRiemannPS(\eik)$ is smooth on all of $\mathbb{R}_t \times \mathbb{R}_{\eik}$
was proved in \eqref{E:SOLUIONRRIEMANNTRIVIALTRANSPORTEQUATIONINGEOMETRICCOORDINATES}.
In particular, in geometric coordinates, $\RRiemannPS$ 
smoothly extends to the closure $\overline{\classicaldev}$.
The fact that $\RRiemannPS$ vanishes on 
the complement of $\{ (t,\eik) \in \R^2 \ | \ \eik \in [-\rightu,\leftu]\}$ then follows from 
our assumptions on the initial data, which in particular imply that
$\dataRRiemannPS$ is supported in $[-\rightu,\leftu]$.

Next, using the identity \eqref{E:SPEEDTIMESMUPLANESYMMETRYWITHEXPLICITSOURCE},
the estimate \eqref{E:EASYSPEEDOVERNORMALIZERBOUNDS},
and the fact that $\RRiemannPS$ is supported in $\{ (t,\eik) \in \R^2 \ | \ \eik \in [-\rightu,\leftu]\}$,
we conclude the asserted properties of $\upmu$ in geometric coordinates,
except for the positivity of $\upmu$ on $\classicaldev$, which we prove below.
The asserted properties of $\Lunit$, $\newuL$, $\Lunit x^1$, and $\uLunit x^1$ in geometric coordinates then follow from 
\eqref{E:1DLUNITINTERMSOFRIEMANNINVARIANTS}--\eqref{E:1DULUNITINTERMSOFRIEMANNINVARIANTS},
\eqref{E:PLANESYMMETRYCOMMUTATORVECTORFIELDSAREGEOCOORDINATEVECTORFIELDS},
\eqref{E:LANDMUXCOMMUTEINPLANESYMMETRY},
and the fact that $\RRiemannPS$ is supported in $\{ (t,\eik) \in \R^2 \ | \ \eik \in [-\rightu,\leftu]\}$.

\medskip
\noindent \textbf{Proof of the properties of $\crease$ and $\singularcurve_{[-\interestingu,0]}$ relative to the geometric coordinates:} 
The fact that the crease is the single point $(\blowuptime,0)$ is a consequence of
definition~\eqref{E:1DCREASE},
\eqref{AE:PSMUTAYLOREXPANSIONININTERESTINGREGION}, 
and \eqref{AE:PSLOCATIONOFXMUEQUALSMINUSKAPPA}.
We have therefore proved \eqref{E:MAINTHMCREASEISSINGLEPOINT1D}.

From 
\eqref{E:SPEEDTIMESMUPLANESYMMETRYWITHEXPLICITSOURCE},
\eqref{E:EASYSPEEDOVERNORMALIZERBOUNDS},
\eqref{E:TISAFUNCTIONALONGSINGULARCURVE},
\eqref{E:TRUNCATEDPORTIONOFSINGULARCURVE},
\eqref{E:ALTERNATETRUNCATEDPORTIONOFSINGULARCURVE},
and definition~\eqref{E:SINGULARPORTIONOFCLASSICALDEVBOUNDARYININTERESTINGREGION},
we see that $\upmu > 0$ on $\classicaldevsingular$ while $\upmu$ vanishes on 
$\singularcurve_{[-\interestingu,0]}$. 

The estimate \eqref{E:QUANTITATIVENEGATIVITYOFSPEEDTIMESLMU} follows from the first equality in
\eqref{E:LSPEEDTIMESMUPLANESYMMETRYWITHEXPLICITSOURCE} and the estimate \eqref{E:ALTERNATERESCALEDMUFUNCTIONEXPANSIONSNEAR0}. 

The estimates \eqref{E:TAYLOREXPANSIONSUDERIVATIVEOFSINGULARCURVE}--\eqref{E:TAYLOREXPANSIONSOFSINGULARCURVE}
follow from definitions~\eqref{E:BLOWUPTIME} and \eqref{E:TISAFUNCTIONALONGSINGULARCURVE}
and the estimates
\eqref{E:ALTERNATERERECIPROCALSCALEDMUFUNCTIONEXPANSIONSNEAR0}--\eqref{E:ALTERNATERESCALEDMUFUNCTIONANDDERIVATIVESEXPANSIONSNEAR0}.

\medskip
\noindent \textbf{Proof of the properties of $\cauchyhor_{[0,\interestingu]}$ relative to the geometric coordinates:}
We first prove that the integral curve initial value problem for $\intcurvenewuL(\eik)$ given by 
\eqref{E:INTEGRALCURVEOFNEWULODE}--\eqref{E:INTEGRALCURVEOFNEWULINITIALCONDITION} (or equivalently, \eqref{E:ODEFORTCOMPONENTOFMUULUNITINTEGRALCURVE})
has a unique solution on the interval $\eik \in [0,\interestingu]$.  
Since $\upmu$ is a smooth function on all of $(t,\eik) \in \R^2$, the standard existence and uniqueness theorem for ODEs
yields a unique solution to \eqref{E:ODEFORTCOMPONENTOFMUULUNITINTEGRALCURVE}
for $\eik \in [0,\interestingu]$, provided we shrink the size of $\interestingu$ if necessary. 

We now prove \eqref{E:TAYLOREXPANSIONUDERIVATIVEOFCAUCHYHORIZON}--\eqref{E:TAYLOREXPANSIONOFCAUCHYHORIZON}. 
Differentiating \eqref{E:ODEFORTCOMPONENTOFMUULUNITINTEGRALCURVE} and using 
\eqref{E:PLANESYMMETRYCOMMUTATORVECTORFIELDSAREGEOCOORDINATEVECTORFIELDS} and the chain rule,
we deduce that:
\begin{align} \label{E:SECONDUDERIVATIVEFORTCOMPONENTOFMUULUNITINTEGRALCURVE}
	\frac{\mathrm{d}^2}{\mathrm{d} \eik^2} \timeisafunctionofeikonalcauchyhorizon(\eik)
	& 
	= 
	\frac{1}{4} [\upmu (\Lunit \upmu)]|_{(\timeisafunctionofeikonalcauchyhorizon(\eik),\eik )}
	+
	\frac{1}{2} \muX \upmu|_{(\timeisafunctionofeikonalcauchyhorizon(\eik),\eik )},
			\\
\begin{split} 	\label{E:THIRDUDERIVATIVEFORTCOMPONENTOFMUULUNITINTEGRALCURVE}
\frac{\mathrm{d}^3}{\mathrm{d} \eik^3} \timeisafunctionofeikonalcauchyhorizon(\eik)
	& 
	= 
	\frac{1}{8} [\upmu^2 (\Lunit \Lunit \upmu)]_{\left(\timeisafunctionofeikonalcauchyhorizon(\eik),\eik \right)}
	+
	\frac{1}{8} [\upmu (\Lunit \upmu)^2]_{\left(\timeisafunctionofeikonalcauchyhorizon(\eik),\eik \right)}
	+
	\frac{1}{4} [(\muX \upmu) (\Lunit \upmu)]_{\left(\timeisafunctionofeikonalcauchyhorizon(\eik),\eik \right)}
		\\
	& 
	\ \
	+
	\frac{1}{2} [\upmu (\muX \Lunit \upmu)]_{\left(\timeisafunctionofeikonalcauchyhorizon(\eik),\eik \right)}	
	+
	\frac{1}{2} \muX \muX \upmu|_{\left(\timeisafunctionofeikonalcauchyhorizon(\eik),\eik \right)}.
\end{split}
\end{align}
Using
\eqref{E:PLANESYMMETRYCOMMUTATORVECTORFIELDSAREGEOCOORDINATEVECTORFIELDS},
\eqref{E:ODEFORTCOMPONENTOFMUULUNITINTEGRALCURVE},
\eqref{E:SECONDUDERIVATIVEFORTCOMPONENTOFMUULUNITINTEGRALCURVE}--\eqref{E:THIRDUDERIVATIVEFORTCOMPONENTOFMUULUNITINTEGRALCURVE},
and the Taylor expansions 
\eqref{AE:PSMUTAYLOREXPANSIONININTERESTINGREGION}--\eqref{AE:PSMUSECONDDERIVATIVETAYLOREXPANSIONININTERESTINGREGION}, 
we can Taylor expand 
the solution $\timeisafunctionofeikonalcauchyhorizon(\eik)$ to \eqref{E:ODEFORTCOMPONENTOFMUULUNITINTEGRALCURVE} around $\eik = 0$, 
thereby concluding \eqref{E:TAYLOREXPANSIONUDERIVATIVEOFCAUCHYHORIZON}--\eqref{E:TAYLOREXPANSIONOFCAUCHYHORIZON}.

Since \eqref{E:ODEFORTCOMPONENTOFMUULUNITINTEGRALCURVE} implies that 
$
\upmu\left(\timeisafunctionofeikonalcauchyhorizon(\eik),\eik \right)
= 2 \frac{\mathrm{d}}{\mathrm{d} \eik} \timeisafunctionofeikonalcauchyhorizon(\eik)$
along $\cauchyhor_{(0,\interestingu]}$, we can use the Taylor expansions 
described in the previous paragraph to conclude the estimate \eqref{E:TAYLOREXPANSIONOFMUALONGCAUCHYHORIZON}
for $\upmu$. 

To prove \eqref{E:TAYLOREXPANSIONDERIVATIVEOFRECTANGULARX1ALONGCAUCHYHORIZON}, 
we first use 
\eqref{E:1DULUNITINTERMSOFRIEMANNINVARIANTS},
\eqref{E:SEVERALACOUSTICVECTORFIELDS},
\eqref{E:INTEGRALCURVEOFNEWULODE},
our assumed smallness on the amplitude of the initial data
(which in particular implies that the ratio
$
\frac{\sinh\left(\frac{\RRiemannPS}{2}\right) 
			- 
			\speed \cosh\left(\frac{\RRiemannPS}{2}\right)}{
			\cosh\left(\frac{\RRiemannPS}{2}\right) - \speed \sinh\left(\frac{\RRiemannPS}{2}\right)}
$
from RHS~\eqref{E:1DULUNIT} 
(with $\LRiemann \equiv 0$)
is $-\positivebigO(1)$),
and the chain rule 
to deduce that
$
\frac{\mathrm{d} }{\mathrm{d} \eik} x^1\left(\timeisafunctionofeikonalcauchyhorizon(\eik),\eik \right) 
= \frac{1}{2} \newuL^1\left(\timeisafunctionofeikonalcauchyhorizon(\eik),\eik \right) 
= 
\frac{1}{2}
[\upmu \uLunit^1]\left(\timeisafunctionofeikonalcauchyhorizon(\eik),\eik \right)
=
- 
\upmu
\positivebigO(1)
$. From this relation and \eqref{E:TAYLOREXPANSIONOFMUALONGCAUCHYHORIZON}, we conclude 
\eqref{E:TAYLOREXPANSIONDERIVATIVEOFRECTANGULARX1ALONGCAUCHYHORIZON}.

Finally, $\cauchyhor_{[0,\interestingu]}$ is a $\hfour$-null curve portion because its tangent vector is everywhere proportional to
the vectorfield $\newuL$, which is $\hfour$-null by \eqref{E:GFOURINNERPRODUCTOFNULLVECTORFIELDSIN1D}
and \eqref{E:SEVERALACOUSTICVECTORFIELDS}.

\medskip
\noindent \textbf{Proof of \eqref{E:CLOSUREOFCLASSICALDEVREGION}:}
We have already shown that $\timeisafunctionofeikonalinsingularcurve(\eik)$
is smooth on $[-\interestingu,0]$,
that $\timeisafunctionofeikonalcauchyhorizon(\eik)$ is smooth on
$[0, \interestingu]$, and that 
$\timeisafunctionofeikonalinsingularcurve(0) = \timeisafunctionofeikonalcauchyhorizon(0) = \blowuptime$.
From these facts,
\eqref{E:ALTERNATETRUNCATEDPORTIONOFSINGULARCURVE},
and
\eqref{E:SINGULARPORTIONOFCLASSICALDEVBOUNDARYININTERESTINGREGION}--\eqref{E:INTERESTINGREGION},
we conclude \eqref{E:CLOSUREOFCLASSICALDEVREGION}.

\medskip 
\noindent \textbf{Proof that $\upmu > 0$ on $\classicaldev$:}
We have already shown that $\upmu > 0$ on $\classicaldevsingular$.
Hence, in view of definition~\ref{E:INTERESTINGREGION}, 
we see that to prove $\upmu > 0$ on $\classicaldev$, it suffices to prove
$\upmu > 0$ on $\classicaldevregular$.
To this end, we first note that the same arguments given above imply that 
the curve $\eik \rightarrow \left( \timeisafunctionofeikonalinsingularcurve(\eik),\eik \right)$
is smooth for $\eik \in [-\interestingu,\interestingu]$
and that $\timeisafunctionofeikonalinsingularcurve(\eik)$ satisfies the estimate stated in \eqref{E:TAYLOREXPANSIONSOFSINGULARCURVE}
for $\eik \in [-\interestingu,\interestingu]$
(not just for $\eik \in [-\interestingu,0]$). The portion of this curve corresponding to $\eik \in (0,\interestingu]$
is shown as a dotted curve in Fig.\,\ref{F:1DPLANESYMMETRYCAUCHYHORANDSINGULARCURVE};
for purposes of this proof,
we denote this dotted curve, which by definition does not contain the crease, 
by ``$\mathcal{B}_{\textnormal{Fictitious}}$''
(``fictitious'' because the image of $\mathcal{B}_{\textnormal{Fictitious}}$ under $\Upsilon$ in $(t,x^1)$-space is not part of the $\hfour$-MGHD).
The same arguments that we used to prove that $\upmu > 0$ on $\classicaldevsingular$
(which relied on \eqref{E:SPEEDTIMESMUPLANESYMMETRYWITHEXPLICITSOURCE},
\eqref{E:EASYSPEEDOVERNORMALIZERBOUNDS},
and
\eqref{E:TISAFUNCTIONALONGSINGULARCURVE})
imply that $\upmu > 0$ on 
$
\mathcal{D}
:=
\left\{ 
		(t,\eik) \in \mathbb{R}^2 
		\ | \  
		0 \leq t  < \timeisafunctionofeikonalinsingularcurve(\eik), \, \eik \in (0,\interestingu]
	\right\}
$,
where we note that the top boundary of this region is $\mathcal{B}_{\textnormal{Fictitious}}$.
The key point is that 
\eqref{E:1DUPARAMETERIZEDINTEGRALCURVEOFMULUNIT},
\eqref{E:CAUCHYHORIZON},
and the estimates
\eqref{E:TAYLOREXPANSIONSOFSINGULARCURVE} (valid also for $\eik \in (0,\interestingu]$)
and 
\eqref{E:TAYLOREXPANSIONOFCAUCHYHORIZON}
imply that $\timeisafunctionofeikonalinsingularcurve(\eik) < \timeisafunctionofeikonalcauchyhorizon(\eik)$
for $\eik \in (0,\interestingu]$ and thus $\cauchyhor_{(0,\interestingu]}$ lies \emph{below} $\mathcal{B}_{\textnormal{Fictitious}}$.
In view of definition \eqref{E:REGULARPORTIONOFCLASSICALDEVBOUNDARYININTERESTINGREGION},
we see that $\classicaldev \subset \mathcal{D}$ 
and hence $\upmu > 0$ on $\classicaldevregular$, as desired.

\medskip
\noindent \textbf{Proof of the diffeomorphism and homeomorphism properties of $\Upsilon$:}
From Def.\,\ref{D:CLASSICALDEVELOPMENTNEARVANISHINGOFINVERSEFOLIATIONDENSITY},  
we see that the top boundary of $\classicaldev$
is the following $\eik$-parameterized curve
(note that the curve is well-defined since $\timeisafunctionofeikonalinsingularcurve(0) = \timeisafunctionofeikonalcauchyhorizon(0) = \blowuptime$):
\begin{align} \label{E:TOPBOUNDARYOFCLASSDEVELOPMENT}
	\upgamma_{\textnormal{top}}(\eik)
	& = (\mathfrak{t}_{\textnormal{top}}(\eik),\eik),
	&&
	\eik \in [-\interestingu,\interestingu],
		\\
	\mathfrak{t}_{\textnormal{top}}(\eik)
	& 
	:=
	\begin{cases}
			(\timeisafunctionofeikonalinsingularcurve(\eik),\eik),
			& \eik \in [-\interestingu,0],
				\\
		\left(\timeisafunctionofeikonalcauchyhorizon(\eik),\eik \right),
		& \eik \in [0,\interestingu].
		\label{E:TALONGTOPBOUNDARYOFCLASSDEVELOPMENT}
	\end{cases}
\end{align}
Using 
\eqref{E:TAYLOREXPANSIONSOFDERIVATIVERECTANGULARX1ALONGSINGULARCURVE},
\eqref{E:TAYLOREXPANSIONDERIVATIVEOFRECTANGULARX1ALONGCAUCHYHORIZON},
and the mean value theorem,
we see that $x^1(\mathfrak{t}_{\textnormal{top}}(\eik),\eik)$ is strictly decreasing
for $\eik \in [-\interestingu,\interestingu]$, i.e., the $x^1$
coordinate function strictly decreases as we move from right to left along
the top boundary of $\classicaldev$ in Fig.\,\ref{F:1DPLANESYMMETRYCAUCHYHORANDSINGULARCURVE}.
Next, we note that \eqref{E:SEVERALACOUSTICVECTORFIELDS} and \eqref{E:EXPLICITFORMOFXIN1D}
imply that $\geop{\eik} x^1 = - \frac{\speed}{\normalizer} \upmu$.
Using this identity, the fact that $\frac{\speed}{\normalizer} > 0$, and the fact that we have already shown that $\upmu > 0$ on $\classicaldev$
and $\upmu$ vanishes precisely along the curve portion $\singularcurve_{[-\interestingu,0]}$ in $\overline{\classicaldev}$,
we see that for any $(t,\eik) \in \overline{\classicaldev}$,
the map $\eik \rightarrow x^1(t,\eik)$ is strictly decreasing.
From these two monotonicity properties of $x^1$, we conclude, given the shape of 
$\overline{\classicaldev}$ (see Fig.\,\ref{F:1DPLANESYMMETRYCAUCHYHORANDSINGULARCURVE}), 
that the map $\Upsilon(t,\eik) = (t,x^1(t,\eik))$ is
injective on the compact set $\overline{\classicaldev}$ and thus is a homeomorphism
from $\overline{\classicaldev}$ onto $\Upsilon(\overline{\classicaldev})$.
Also considering the identity
$\det \mathrm{d} \Upsilon = - \frac{\speed}{\normalizer} \upmu$ proved in
\eqref{E:JACOBIANDETERMINANTOFCHANGEOFVARIABLESMAP1D}, we further conclude that
$\Upsilon$ is a diffeomorphism on $\classicaldev$, where $\upmu > 0$.

\medskip
\noindent \textbf{Proof of \eqref{E:BLOWUPOFCARTESIANCOORDINATES} and the structure of the $\hfour$-MGHD in Minkowski-rectangular coordinates $(t,x^1)$:} 
We deduce from 
\eqref{E:SEVERALACOUSTICVECTORFIELDS},
\eqref{E:EXPLICITFORMOFXIN1D}, 
\eqref{E:SOLUIONRRIEMANNTRIVIALTRANSPORTEQUATIONINGEOMETRICCOORDINATES},
\eqref{E:LSPEEDTIMESMUPLANESYMMETRYWITHEXPLICITSOURCE},
the chain rule,
the non-degeneracy assumption \eqref{E:ANOTHER1DNONDEGENERACYONEOSCONSEQUENCE},
and \eqref{E:QUANTITATIVENEGATIVITYOFSPEEDTIMESLMU}
that for $(t,\eik) \in \classicaldev$,
the following estimate holds:
 \begin{align} 
	\begin{split} \label{E:1DPROOFKEYBOUNDFORCARTESIANBLOWUP}
 	\big| \upmu \p_1 \RRiemannPS \big| (t,\eik)  
	& = 
	\frac{\normalizer(\eik)}{\speed(\eik)}
	\left| \muX \RRiemannPS(\eik) \right| 
	=
	\frac{\normalizer(\eik)}{\speed(\eik)} 
	\left| \frac{\mathrm{d} }{\mathrm{d} \eik} \dataRRiemannPS(\eik) \right| 
	=
	\frac{\normalizer(\eik)}{\speed(\eik)} 
	\left|
	\frac{\frac{\mathrm{d} }{\mathrm{d} \eik} 
	\antiderivativePSLmusourcetermfunction[\dataRRiemannPS(\eik)]}{\frac{\mathrm{d} }{\mathrm{d} \RRiemannPS} 
	\antiderivativePSLmusourcetermfunction[\dataRRiemannPS(\eik)]}
	\right|
		\\
	& 
	=
	\left|
	\frac{\Lunit \upmu(\eik)}{\frac{\mathrm{d} }{\mathrm{d} \RRiemannPS} 
	\antiderivativePSLmusourcetermfunction[\dataRRiemannPS(\eik)]}
	\right|
	=
	\positivebigO(1).
\end{split}
\end{align}
From \eqref{E:1DPROOFKEYBOUNDFORCARTESIANBLOWUP}, we conclude \eqref{E:BLOWUPOFCARTESIANCOORDINATES}.

The continuity of $\RRiemannPS, \, \Lunit^M \RRiemannPS$ and 
$\hfour_{\alpha \beta} = \hfour_{\alpha \beta}(\RRiemannPS)$
all the way up to $\Upsilon(\overline{\classicaldev})$ in Minkowski-rectangular coordinates 
$(t,x^1)$ follows from the fact that $\RRiemannPS(t,\eik)$ is smooth on $\mathbb{R}_t \times \mathbb{R}_{\eik}$
and the already proven fact that $\Upsilon$ is a homeomorphism from $\overline{\classicaldev}$
onto $\Upsilon(\overline{\classicaldev})$. 
The non-degeneracy of the $2 \times 2$ matrix $\hfour_{\alpha \beta}$ relative to the $(t,x^1)$-coordinates
follows from the expression \eqref{E:1DACOUSTICALMETRIC} for $\hfour_{\alpha \beta} = \hfour_{\alpha \beta}(\RRiemannPS)$,
our assumed smallness on the amplitude of the initial data (which in particular implies that $0 < \speed \leq 1$),
and the identities stated in Prop.\,\ref{P:QUASILINEARTRANSPORTSYSTEMFORRIEMANNINVARAINTSANDOTHERRESULTS}.

\medskip
\noindent \textbf{Proof of the causal structure of the singular curve in Minkowski-rectangular coordinates $(t,x^1)$:} 
Using \eqref{E:PLANESYMMETRYCOMMUTATORVECTORFIELDSAREGEOCOORDINATEVECTORFIELDS},
\eqref{E:LSPEEDTIMESMUPLANESYMMETRYWITHEXPLICITSOURCE}--\eqref{E:SPEEDTIMESMUPLANESYMMETRYWITHEXPLICITSOURCE}, 
and
\eqref{E:TRUNCATEDPORTIONOFSINGULARCURVE},
we deduce that along $\singularcurve_{[-\interestingu,0)}$, 
the vectorfield $\singcurvevectorfield$ 
defined by \eqref{E:SINGCURVEVECTORFIELD}
can be expressed as
follows, where 
 $\PSLmusourcetermfunction$ is the source term from \eqref{E:PSLMUSOURCETERMFUNCTION}, viewed as a function of $\eik$,
and $\PSLmusourcetermfunction'$ is its derivative:
\begin{align} \label{E:TANGENTVECTORFIELDTOBINGEOMETRICCOORDINATES}
\singcurvevectorfield 
	& = 
	\geop{t} 
	- 
	\frac{\PSLmusourcetermfunction(\eik)}{t \PSLmusourcetermfunction'(\eik)} \geop{\eik}
	= 
		\geop{t} 
		+ 
		\frac{\PSLmusourcetermfunction^2(\eik)}{\PSLmusourcetermfunction'(\eik)} \geop{\eik},
	& &
	\mbox{along } \singularcurve_{[-\interestingu,0)}.
\end{align}
Hence, using \eqref{E:ALTERNATERESCALEDMUFUNCTIONEXPANSIONSNEAR0} and \eqref{E:ALTERNATERESCALEDMUFUNCTIONANDDERIVATIVESEXPANSIONSNEAR0},
we see that $\singcurvevectorfield$ is well-defined on $\singularcurve_{[-\interestingu,0)}$,
i.e., the denominator term $\PSLmusourcetermfunction'(\eik)$ on RHS~\eqref{E:TANGENTVECTORFIELDTOBINGEOMETRICCOORDINATES} 
is non-zero for $\eik \in [-\interestingu,0)$.
Next, we deduce from \eqref{E:SINGCURVEVECTORFIELD} that 
$
\singcurvevectorfield\left( \frac{\speed}{\normalizer} \upmu \right) = 0
$, i.e., that $\singcurvevectorfield$ is tangent to the level sets of $\frac{\speed}{\normalizer} \upmu$.
Using the fact that $\frac{\speed}{\normalizer} > 0$ (see \eqref{E:EASYSPEEDOVERNORMALIZERBOUNDS}), 
we see that the zero level set of $\frac{\speed}{\normalizer}\upmu$ 
coincides with the zero level set of $\upmu$. Hence, in view of definition \eqref{E:TRUNCATEDPORTIONOFSINGULARCURVE}, we see that
$\singcurvevectorfield$ is tangent to $\singularcurve_{[-\interestingu,0)}$. 
The inequalities $\singcurvevectorfield t > 0$ and $\singcurvevectorfield \eik < 0$ 
follow easily from 
\eqref{E:TANGENTVECTORFIELDTOBINGEOMETRICCOORDINATES},
\eqref{E:ALTERNATERESCALEDMUFUNCTIONEXPANSIONSNEAR0}, 
and \eqref{E:ALTERNATERESCALEDMUFUNCTIONANDDERIVATIVESEXPANSIONSNEAR0}.

Next, using \eqref{E:SEVERALACOUSTICVECTORFIELDS},
\eqref{E:EXPLICITFORMOFXIN1D}, 
and \eqref{E:PLANESYMMETRYCOMMUTATORVECTORFIELDSAREGEOCOORDINATEVECTORFIELDS},
we see that the pushforward of $\geop{\eik}$ by $\Upsilon$, i.e., $\mathrm{d} \Upsilon \cdot \geop{\eik}$ 
is equal to $-\frac{\speed}{\normalizer} \upmu \p_1$. 
Hence, $\mathrm{d} \Upsilon \cdot \geop{\eik}$ 
\emph{vanishes along all of $\Upsilon(\singularcurve_{[-\interestingu,0)})$}, where $\upmu$ is $0$. 
Moreover, from $\eqref{E:PLANESYMMETRYCOMMUTATORVECTORFIELDSAREGEOCOORDINATEVECTORFIELDS}$, we deduce that
 $\mathrm{d} \Upsilon \cdot \geop{t} = \Lunit^{\kappa} \partial_{\kappa}$.
From these facts and \eqref{E:TANGENTVECTORFIELDTOBINGEOMETRICCOORDINATES}, we deduce 
that for every $q \in \singularcurve_{[-\interestingu,0)}$, 
we have $[\mathrm{d} \Upsilon(q)] \cdot \singcurvevectorfield(q) = \Lunit|_{\Upsilon(q)}$. 
In particular, 
since $\Lunit|_{\Upsilon(q)}$ spans the tangent space of $\Upsilon(\singularcurve_{[-\interestingu,0)})$ and is $\hfour$-null 
(see \eqref{E:GFOURINNERPRODUCTOFNULLVECTORFIELDSIN1D}), 
it follows that $\Upsilon(\singularcurve_{[-\interestingu,0)})$ is a $\hfour$-null curve.

\end{proof}

\section{The new formulation of the flow in three spatial dimensions}
\label{S:NEWFORMULATIONOFFLOW}
In Sect.\,\ref{S:SHOCKFORMATIONAWAYFROMSYMMETRY}, we provide an overview of how to extend some aspects of 
the simple isentropic plane-symmetric shock formation results yielded Theorem~\ref{T:MAINTHEOREM1DSINGULARBOUNDARYANDCREASE} to the 
much harder case of three spatial dimensions in the presence of vorticity and entropy. 
The strategy we present in Sect.\,\ref{S:SHOCKFORMATIONAWAYFROMSYMMETRY} is based on the approach we used in studying shock formation
for the non-relativistic multi-dimensional compressible Euler equations \cites{jLjS2018,jLjS2021,lAjS2022}.
Fundamental to that strategy is the availability of a new formulation of the flow as 
as system of geometric wave-transport-div-curl equations that exhibit remarkable regularity properties 
and geometric null structures. As we explain in Sect.\,\ref{S:SHOCKFORMATIONAWAYFROMSYMMETRY},
such a formulation allows one to implement multi-dimensional nonlinear geometric optics,
which is important for the study of shocks 
(as is already evident from Theorem~\ref{T:MAINTHEOREM1DSINGULARBOUNDARYANDCREASE}).
In the non-relativistic case, such a formulation was derived in \cites{jLjS2020a,jS2019c}, 
while in the relativistic case, such a formulation was derived in \cite{mDjS2019}. 
In Sect.\,\ref{S:NEWFORMULATIONOFFLOW}, in the relativistic case,
we set up the machinery needed to state the formulation from \cite{mDjS2019} and present it in a simplified, 
schematic form as Theorem~\ref{T:GEOMETRICWAVETRANSPORTDIVCURLFORMULATION}. 
In Sect.\,\ref{SS:FIRSTCOMMENTSONMOREGENERALSPACETIMES}, we make some brief comments on how the new formulation can be extended to more general spacetimes $(\spacetimemanifold,\gfour)$.

\subsection{Connection to quasilinear wave equations}
\label{SS:CONNECTIONWITHWAVEEQUATIONS}
Although our study of isentropic plane-symmetric solutions in Sect.\,\ref{S:NEWFORMULATIONOFFLOW} was based on
analyzing the first-order Riemann invariant system \eqref{E:RRIEMANNPSTRANSPORT}--\eqref{E:LRIEMANNPSTRANSPORT},
as we will see in Sect.\,\ref{S:SHOCKFORMATIONAWAYFROMSYMMETRY}, in multi-dimensions, it is advantageous
to work with the formulation of the flow provided by Theorem~\ref{T:GEOMETRICWAVETRANSPORTDIVCURLFORMULATION},
which features geometric wave equations. 
The main advantage is that for geometric wave equations, there is an advanced geo-analytic machinery available,
tied to the vectorfield method, for studying the global behavior of solutions.
To motivate the relevance of wave equations for the study of relativistic fluids,
we now establish a simple connection between quasilinear wave equations and the Riemann invariant formulation
of isentropic plane-symmetric solutions that we provided in Sect.\,\ref{S:NEWFORMULATIONOFFLOW}.
More precisely, upon differentiating \eqref{E:RRIEMANNPSTRANSPORT}--\eqref{E:LRIEMANNPSTRANSPORT},
we deduce that for isentropic plane-symmetric solutions, the Riemann invariants
$(\RRiemannPS,\LRiemannPS)$ satisfy the following coupled system of quasilinear wave equations:
\begin{subequations}
\begin{align}
	\uLunit \Lunit \RRiemannPS
	& 
	= 0,
			\label{E:RRIEMANNPSWAVE} 
				\\
	\Lunit \uLunit \LRiemannPS
	& = 0.
	\label{E:LRIEMANNPSWAVE} 
\end{align}
\end{subequations}
Equations \eqref{E:RRIEMANNPSWAVE}--\eqref{E:LRIEMANNPSWAVE} are quasilinear wave equations
because by Lemma~\ref{L:1DGFOURINVERSEINTERMSOFLUNITANDULUNIT}, the principal operator on the LHSs
is proportional to $(\hfour^{-1})^{\alpha \beta} \partial_{\alpha} \partial_{\beta}$,
where $\hfour$ is the acoustical metric.

\subsection{Additional fluid variables}
\label{SS:ADDITIONALFLUIDVARIABLES}
The remainder of Sect.\,\ref{S:NEWFORMULATIONOFFLOW} concerns solutions 
$\left(\Lnenth,\fourvelocity^0,\fourvelocity^1,\fourvelocity^2,\fourvelocity^3,\Ent \right)$
to the relativistic Euler equations \eqref{E:ENTHALPYEVOLUTION}--\eqref{E:AGAINFLUIDFOURVELOCITYNORMALIZED}
in three spatial dimensions without any symmetry, isentropicity, or irrotationality assumptions.

In Sect.\,\ref{SS:ADDITIONALFLUIDVARIABLES}, 
we introduce some additional fluid variables that complement the ones
we introduced in Sect.\,\ref{S:EQUATIONSANDLOCALWELLPOSEDNESS}. The new variables play a fundamental role
in the new formulation of relativistic Euler flow provided by Theorem~\ref{T:GEOMETRICWAVETRANSPORTDIVCURLFORMULATION}.

\begin{definition}[The $\fourvelocity$-orthogonal vorticity of a one-form]
	\label{D:VORTICITYOFAONEFORM}
	Let $\bm{\upxi}$ be a one-form.
	We define its $\fourvelocity$-orthogonal vorticity, denoted by
	$\uperpvort^{\alpha}(\bm{\upxi})$, to be the following vectorfield, where
	$\upepsilon$ is the fully antisymmetric symbol normalized by $\upepsilon_{0123} = 1$:
	\begin{align} \label{E:VORTICITYOFAONEFORM}
		\uperpvort^{\alpha}(\bm{\upxi})
		& := - \upepsilon^{\alpha \beta \gamma \delta} \fourvelocity_{\beta} \partial_{\gamma} \bm{\upxi}_{\delta},
	\end{align}
	and we recall that throughout the paper, 
	we use the Minkowski metric $\mink$ -- as opposed to $\hfour$ -- to lower and raise indices.
	In particular, $\upepsilon^{0123} = - 1$.
\end{definition}

Note that $\fourvelocity_{\alpha} \uperpvort^{\alpha}(\bm{\upxi}) = 0$, which explains our
terminology ``$\fourvelocity$-orthogonal vorticity'' of $\bm{\upxi}$.

\begin{definition}[Vorticity vectorfield]
We define the vorticity $\vort^{\alpha}$ to be the following vectorfield:
\begin{align} \label{E:VORTICITYVECTORFIELD}
		\vort^{\alpha}
		& 
		:=
		\uperpvort^{\alpha}(\Enth \fourvelocity)
		= - \upepsilon^{\alpha \beta \gamma \delta} \fourvelocity_{\beta} \partial_{\gamma} (\Enth \fourvelocity_{\delta}).
\end{align}
\end{definition}

\begin{definition} [Entropy gradient one-form]
\label{D:ENTROPYGRADIENTONEFORM}
We define $\GradEnt_{\alpha}$ to be the following one-form:
\begin{align} \label{E:ENTROPYGRADIENTONEFORM}
		\GradEnt_{\alpha}
		& := \partial_{\alpha} \Ent.
	\end{align}
\end{definition}

The fluid variables
$\modVortVort$ and $\modDivGradEnt$ from the next definition play a fundamental 
role in the setup of the new formulation of the flow. In particular,
Theorem~\ref{T:GEOMETRICWAVETRANSPORTDIVCURLFORMULATION}
shows that they satisfy transport equations with source terms that
enjoy remarkable regularity and null structures.

\begin{definition}[Modified fluid variables]
	\label{D:MODIFIEDFLUIDVARIABLES}
	We define $\modVortVort^{\alpha}$
	and
	$\modDivGradEnt$
	to respectively be the following vectorfield and scalar function:
	\begin{subequations}
	\begin{align} \label{E:MODIFIEDVORTVORT} 
	\begin{split}
		\modVortVort^{\alpha}
		& := \uperpvort^{\alpha}(\vort)
			+
		\speed^{-2}
		\upepsilon^{\alpha \beta \gamma \delta} 
		\fourvelocity_{\beta}
		(\partial_{\gamma} \Lnenth) \vort_{\delta}
			\\
		& \ \
			+
			(\Temp - \Temp_{;\Lnenth}) 
			\GradEnt^{\alpha} 
			(\partial_{\kappa} u^{\kappa})
		+
		(\Temp - \Temp_{;\Lnenth}) u^{\alpha} (\GradEnt^{\kappa} \partial_{\kappa} \Lnenth)
		\\
		& 
		\ \
		+
		(\Temp_{;\Lnenth} - \Temp) \GradEnt^{\kappa} 
		((\mink)^{-1})^{\alpha \lambda} \partial_{\lambda} \fourvelocity_{\kappa}),
		\end{split}	
			\\
		\modDivGradEnt
		& := \frac{1}{n} (\partial_{\kappa} \GradEnt^{\kappa})
			+
			\frac{1}{n} (\GradEnt^{\kappa} \partial_{\kappa} \Lnenth)
			-
			\frac{1}{n} \speed^{-2} (\GradEnt^{\kappa} \partial_{\kappa} \Lnenth).
			 \label{E:MODIFIEDDIVGRADENT}
\end{align}
\end{subequations}
\end{definition}

\subsubsection{Covariant wave operator}
\label{SSS:COVARIANTWAVEOPERATOR}
Covariant wave equations play a central role in Theorem~\ref{T:GEOMETRICWAVETRANSPORTDIVCURLFORMULATION}.
In Def.\,\ref{D:COVWAVEOP}, we recall the standard definition of a covariant wave operator acting on scalar functions;
in Theorem~\ref{T:GEOMETRICWAVETRANSPORTDIVCURLFORMULATION}, we apply the covariant wave operator \emph{only} 
to quantities that we view to be scalar functions.

\begin{definition}[Covariant wave operator acting on a scalar function]
\label{D:COVWAVEOP}
The covariant wave operator $\square_{\hfour}$ acts on scalar functions $\varphi$ as follows:
\begin{align} \label{E:COVWAVEOP}
\square_{\hfour} \varphi
& := \frac{1}{\sqrt{|\mbox{\upshape det} \mbox{$\hfour$}|}}
	\partial_{\alpha}
	\left\lbrace
			\sqrt{|\mbox{\upshape det} \hfour|} (\hfour^{-1})^{\alpha \beta}
			\partial_{\beta} \varphi
	\right\rbrace.
\end{align}	
\end{definition}

\subsubsection{Null forms}
\label{SSS:NULLFORMS}
In Def.\,\ref{D:NULLFORMS}, we recall the definition of
standard null forms relative to $\hfour$.
The key point is that in Theorem~\ref{T:GEOMETRICWAVETRANSPORTDIVCURLFORMULATION}, all derivative-quadratic terms
on the RHS of the equations are null forms relative to the acoustical metric.
In Sect.\,\ref{SSS:COMMUTATORMETHOD}, we will provide a detailed explanation as to why, in the context of shock formation,
such null forms are harmless error terms. However, at this point in the paper, 
we merely state that $\hfour$-null forms are derivative-quadratic nonlinearities 
that, when a shock forms, are \emph{strictly weaker than the 
Riccati-type terms that drive the blowup} (see Sect.\,\ref{SS:1DRICCATIBLOWUP}).
As we will further explain in Remark~\ref{R:SHOCKDRIVINGTERMSHIDDEN}, 
the dangerous Riccati-type terms are ``hiding'' in the definition of the covariant wave operator on LHS~\eqref{E:WAVE}.

\begin{definition}[Standard null forms relative to $\hfour$]
\label{D:NULLFORMS}
We define the standard null forms relative to $\hfour$ (which we refer to as ``standard $\hfour$-null forms'' for short)
to be the following derivative-quadratic terms, 
where $\phi$ and $\psi$ are scalar functions and $0 \leq \mu < \nu \leq 3$:
\begin{subequations} 
\begin{align} 
	\nullform^{(\hfour)}(\pmb{\partial} \phi, \pmb{\partial} \psi)
	& := (\hfour^{-1})^{\kappa \lambda} \partial_{\kappa} \phi \partial_{\lambda} \psi,
	\label{E:GNULLFORMS} 
		\\
\nullform_{\mu \nu}(\pmb{\partial} \phi, \pmb{\partial} \psi)
	& := \partial_{\mu} \phi 
				\partial_{\nu} \psi 
			- 
			\partial_{\nu} \phi
			\partial_{\mu} \psi.
				\label{E:ANTISYMMETRICFORMS}
\end{align}
\end{subequations}
\end{definition}

\subsubsection{Two vectorfields}
\label{SSS:TWOVECTORFIELDS}
The two vectorfields featured in the following definition will play a role in our forthcoming discussion.

\begin{definition}[$\Transport$ and $\Sigmatnormal$]
	\label{D:NORMALIZEDMATERIALDERIVATIVEANDSIGMATNORMAL}
	We define $\Transport$ and $\Sigmatnormal$ to respectively be the vectorfields
	with the following components relative to the Minkowski-rectangular coordinates $(t,x^1,x^2,x^3)$:
	\begin{subequations}
	\begin{align} \label{E:NORMALIZEDMATERIALDERIVATIVE}
		\Transport^{\alpha}
		& := \frac{1}{\fourvelocity^0}\fourvelocity^{\alpha},
			\\
		\Sigmatnormal^{\alpha}
		& := 
			-
			(\hfour^{-1})^{\alpha \beta} \partial_{\beta} t
			= 
			-
			(\hfour^{-1})^{\alpha 0}.
			\label{E:SIGMATNORMAL}
	\end{align}
	\end{subequations}
\end{definition}

Note that $\Transport t = 1$, while by \eqref{E:GINVERSE00ISMINUSONE},
$\Sigmatnormal$ is the future-directed $\hfour$-unit normal to $\Sigma_t$.
In particular,
\begin{align} \label{E:SIGMATNORMALHASUNITLENGTH}
	\hfour(\Sigmatnormal,\Sigmatnormal)
	& = - 1.
\end{align}
We also compute that:
\begin{align} \label{E:GINNERPRODUCTOFTRANSPORTANDSIGMATNORMALISMINUSONE}
	\hfour(\Transport,\Sigmatnormal)
	& = - \Transport^0
		= - 1
\end{align}
and, with the help of \eqref{E:AGAINFLUIDFOURVELOCITYNORMALIZED} and \eqref{E:ACOUSTICALMETRIC}, that:
\begin{align} \label{E:GINNERPRODUCTOFTRANSPORTANDLUNIT}
	\hfour(\Transport,\Lunit)
	& = \normalizer \mink(\Transport,\Lunit).
\end{align}
Note that $\mbox{RHS~\eqref{E:GINNERPRODUCTOFTRANSPORTANDLUNIT}} < 0$
because $\normalizer > 0$,
$\Transport$ is future-directed and $\hfour$-timelike,
$\Lunit$ is future-directed and $\hfour$-null, 
and the $\hfour$-null cones are inside the $\mink$-null cones (see Remark~\ref{R:SUBLUMINAL}).

\subsubsection{The geometric wave-transport-div-curl formulation of the flow}
\label{SSS:GEOMETRICWAVETRANSPORTDIVCURLFORMULATION}
We now state -- in a somewhat abbreviated form -- the formulation of relativistic Euler flow derived in \cite{mDjS2019}. 
The proof of theorem is based on differentiating the equations
\eqref{E:ENTHALPYEVOLUTION}--\eqref{E:AGAINFLUIDFOURVELOCITYNORMALIZED} in well-chosen directions
and finding a myriad of cancellations and special structures.

\begin{theorem}\cite{mDjS2019}*{The geometric wave-transport-div-curl formulation of the flow}
\label{T:GEOMETRICWAVETRANSPORTDIVCURLFORMULATION}
Let $\wavearray = \left(\Lnenth,\fourvelocity^0,\fourvelocity^1,\fourvelocity^2,\fourvelocity^3,\Ent \right)$ be a
$C^3$ solution to the relativistic Euler equations \eqref{E:ENTHALPYEVOLUTION}--\eqref{E:AGAINFLUIDFOURVELOCITYNORMALIZED}
on the Minkowski spacetime background. Then for $\Psi \in \wavearray$, the following equations hold, where 
$\nullform$ denotes a linear combination of standard $\hfour$-null forms (see Def.\,\ref{D:NULLFORMS})
and $\smoothfunction$ denotes a smooth function, both of which are 
free to vary from line to line:
\begin{subequations}
\begin{align} \label{E:WAVE}
 \square_{\hfour(\wavearray)} \Psi
	& 
	= 
	\smoothfunction(\wavearray)
	\cdot
	(\modVortVort, \modDivGradEnt)
	+
	\smoothfunction(\wavearray)
	\cdot
	\nullform(\pmb{\partial} \wavearray, \pmb{\partial} \wavearray)
	+
	\smoothfunction(\wavearray,\GradEnt)
	\cdot
	(\vort,\GradEnt)
	\cdot
	\pmb{\partial} \wavearray, 
		\\
	\Transport (\vort^i,\GradEnt^i)
	& =  
	\smoothfunction(\wavearray,\GradEnt)
	\cdot
	(\vort,\GradEnt)
	\cdot
	\pmb{\partial} \wavearray, 
		\label{E:TRANSPORT} 
		\\
\begin{split} 	\label{E:TRANSPORTFORMODIFIED}
\Transport (\modVortVort^i,\modDivGradEnt)
& = 
	\smoothfunction(\wavearray)
	\cdot
  \nullform(\pmb{\partial} \vort,\pmb{\partial} \wavearray)
	+
	\smoothfunction(\wavearray)
	\cdot
	\nullform(\pmb{\partial} \GradEnt,\pmb{\partial} \wavearray)
		\\
	& \ \
	+
	\smoothfunction(\wavearray)
	\cdot
	(\vort,\GradEnt)
	\cdot
  \nullform(\pmb{\partial} \wavearray,\pmb{\partial} \wavearray)
		\\
	& \ \
	+
	\smoothfunction(\wavearray)
	\cdot
	\pmb{\partial} \wavearray
	\cdot
	(\modVortVort,\modDivGradEnt)
	+
	\smoothfunction(\wavearray)
	\cdot
	\GradEnt
	\cdot
	\pmb{\partial} \GradEnt
	+
	\smoothfunction(\wavearray,\vort,\GradEnt)
	\cdot
	(\vort,\GradEnt)
	\cdot
	\pmb{\partial} \wavearray,
\end{split} 
			\\
\partial_{\kappa} \vort^{\kappa}
& = 
\smoothfunction(\wavearray)
\cdot
\vort
\cdot
\pmb{\partial} \wavearray,
\qquad
\uperpvort^{\alpha}(\GradEnt)
= 0.
	\label{E:DIVCURLFORMODIFIED}
\end{align}
\end{subequations}

\end{theorem}

\subsubsection{Comments on the case of more general ambient spacetimes $(\spacetimemanifold,\gfour)$} \label{SS:FIRSTCOMMENTSONMOREGENERALSPACETIMES}
Here we make some brief comments on some of the adjustments that would be needed to extend the results of 
Sect.\,\ref{S:NEWFORMULATIONOFFLOW} to general $\gfour$-globally hyperbolic spacetimes $(\spacetimemanifold,\gfour)$. 
By $\gfour$-globally hyperbolic, we mean that there exists a $\gfour$-Cauchy hypersurface 
$\Sigma_0 \subset \spacetimemanifold$. 
By a ``$\gfour$-Cauchy hypersurface,'' 
we mean a co-dimension one submanifold $\Sigma_0$ such that every inextendible 
$\gfour$-causal curve in $\spacetimemanifold$ intersects $\Sigma_0$ exactly once. It is well known \cites{rG1970,bernal2006further} that 
$\gfour$-globally hyperbolic spacetimes admit a smooth time function $t \colon \spacetimemanifold \to \R$. 
By a ``time function,'' we mean that 
$t^{-1}(0) = \Sigma_0$, 
that the level sets $\Sigma_{t'} := \{ t \equiv t' \}$ are $\gfour$-spacelike Cauchy hypersurfaces, 
that $\spacetimemanifold = \cup_t \Sigma_t$, and 
that $t$ has a past-directed timelike gradient $\nabla^{\#} t := \gfour^{-1} \cdot \mathrm{d} t$ (which is everywhere normal to $\Sigma_t$).
Here, $\nabla^{\#} t$ denotes the vectorfield equal to the dual of $\nabla t$ with respect to $\gfour$.

First, we note that 
the extra terms described in Remark~\ref{R:ADDITIONALTERMS} would
lead to the presence of extra terms in the system \eqref{E:WAVE}--\eqref{E:DIVCURLFORMODIFIED}.
However, as we described in Remark~\ref{R:ADDITIONALTERMS}, in the context of
the study of shock formation, in which fluid gradients are large, the extra terms
would be small relative to the shock-driving Riccati-type fluid terms
(which are ``hidden'' in $\square_{\hfour(\wavearray)} \Psi$; see Remark~\ref{R:SHOCKDRIVINGTERMSHIDDEN}).

Next, we note that under the above assumptions, we can define a lapse function $\Phi$ by $\Phi : = (-\gfour(\nabla t, \nabla t))^{-1/2}$.
We can then define the vectorfield $T := - \Phi^2 \nabla^{\#} t$, which is $\gfour$-normal to $\Sigma_t$ and satisfies $T t = 1$. 
Hence, given a coordinate system $(x^1,x^2,x^3)$ on $\Sigma_0$, 
it can be propagated to any $\Sigma_t$ by the flow of $T$, 
thereby yielding a coordinate system $(t,x^1,x^2,x^3)$ on spacetime. 
In this setup, the ambient spacetime metric can be decomposed as follows:
\begin{align} \label{E:AMBIENTSPACETIMEMETRICDECOMPOSITION}
	\gfour & = - \Phi^2 \mathrm{d}t \otimes \mathrm{d}t + g, 
\end{align}
where $g = g(t)$ is the Riemannian metric on $\Sigma_t$ induced by $\gfour$, i.e.,
the first fundamental form of $\Sigma_t$ with respect to $\gfour$.
Given $\gfour$ and the relativistic fluid, the corresponding acoustical metric $\hfour$ is defined by
\eqref{E:ACOUSTICALMETRICGENERALAMBINETSPACETIME}.

The upshot of introducing a time function $t$ is that it allows for many of the constructions from this section to easily be generalized. 
For example, using \eqref{E:FLUIDFOURVELOCITYNORMALIZED} and \eqref{E:AMBIENTSPACETIMEMETRICDECOMPOSITION}, one sees that 
$\fourvelocity^0 := \fourvelocity \cdot \nabla t := \fourvelocity^{\alpha} \partial_{\alpha} t  > 0$ 
and hence one could naturally define the analog of \eqref{E:NORMALIZEDMATERIALDERIVATIVE} 
by $\Transport := \frac{1}{\fourvelocity^0} \fourvelocity$ and 
the analog of \eqref{E:SIGMATNORMAL} by $\Sigmatnormal := - \hfour^{-1} \cdot \mathrm{d} t$.

\section{Some prior works on shocks}
\label{S:PRIORWORKSSHOCKFORMATION}
In this section, we discuss some prior works on shocks. 
There is a vast literature, and we cannot hope to mention all of it here.
We have aimed to discuss works that provide further context and motivation for
Sects.\,\ref{S:EQUATIONSANDLOCALWELLPOSEDNESS} and \ref{S:NEWFORMULATIONOFFLOW}
and to help prepare the reader for Sects.\,\ref{S:SHOCKFORMATIONAWAYFROMSYMMETRY} 
and \ref{S:OPENPROBLEMS}.

\subsection{Comments on the $1D$ theory}
\label{SS:1DSHOCKFORMATION}
In Riemann's foundational paper \cite{bR1860}, he developed the method of Riemann invariants and combined it with the method of
characteristics to prove that for the $1D$ non-relativistic compressible Euler equations, there exist large sets of initial data such that a shock forms in the sense that the fluid's gradient blows up in finite time, though the fluid variables themselves remain bounded.
Once one has formulated the flow in terms of Riemann invariants, the proof of the blowup is based on differentiating the equations 
to obtain a coupled system of Riccati-type equations along the characteristics. This can be viewed as a
multi-directional (i.e., there are two characteristic directions in $1D$ isentropic compressible Euler flow) 
version of the blowup of $\partial_x \Psi$
that often occurs in solutions to Burgers' equation $\partial_t \Psi(t,x) + \Psi \partial_x \Psi(t,x) = 0$;
in Sect.\,\ref{SS:1DRICCATIBLOWUP}, we discussed this approach in more detail in the case of 
simple isentropic plane-symmetric solutions to the relativistic Euler equations.
Lax later generalized \cite{pL1964} Riemann's results to $2 \times 2$ genuinely nonlinear hyperbolic systems in $1D$.
John later extended \cite{fJ1974} Lax's blowup-results to apply to some systems in $1D$ with more than two unknowns.
Christodoulou and Raoul-Perez used a sharpened, more geometric version 
(in the style of our proof of Theorem~\ref{T:MAINTHEOREM1DSINGULARBOUNDARYANDCREASE})
of Lax's approach to give a detailed proof of shock formation in $1D$ for electromagnetic waves in nonlinear crystals \cite{dCdRP2016}.
Related results have recently been proved for plane-symmetric solutions to the equations of elasticity 
\cite{xAhCsY2020} and the equations of non-relativistic magnetohydrodynamics \cite{xAhCsY2021},
and in these works, the plane-symmetric blowup was also used to prove ill-posedness results in $3D$;
see also the survey article \cite{xAhCsY2022}.
For the relativistic Euler equations in $1D$, stable blowup-results 
have been proved for large sets of initial data 
(which are allowed to be large) for isentropic flows \cite{nAsZ2021} as well as flows with dynamic entropy 
\cite{nAtBRzT2023}.
Theorem~\ref{T:MAINTHEOREM1DSINGULARBOUNDARYANDCREASE} provides a sharpened version of these kind of results for 
simple isentropic plane-symmetric solutions to the relativistic Euler equations; by ``sharpened,'' we mean that,
different from the works mentioned above, 
Theorem~\ref{T:MAINTHEOREM1DSINGULARBOUNDARYANDCREASE} follows the solution all the way up to the singular boundary
and Cauchy horizon, rather than just up to the time of first blowup.

Even though shocks can form for large classes of hyperbolic PDEs, signifying the end of the classical evolution,
ideally, one would like to develop a theory of weak
solutions that can accommodate the formation of shocks and their subsequent interactions. In $1D$, 
a robust such theory exists, at least for strictly hyperbolic systems and initial data that are small in a suitable
bounded variation (BV) function space. We refer to \cite{cD2010} for a comprehensive discussion of the $1D$ theory.
In multi-dimensions, much less is known. A key obstacle is that it is known,
thanks to the important paper \cite{jR1986} by Rauch,
that quasilinear hyperbolic PDEs are typically ill-posed in BV spaces. In fact, for most 
multi-dimensional quasilinear hyperbolic systems
of interest, even local well-posedness is known \emph{only} in $L^2$-type Sobolev spaces. In practice, this means that to prove
shock formation in multi-dimensions, one must commute the equations and derive energy estimates up the singularity;
this turns out to be a difficult task, for reasons we discuss in detail in Sect.\,\ref{S:SHOCKFORMATIONAWAYFROMSYMMETRY}.

\subsection{Blowup-results in multi-dimensions without symmetry: proofs by contradiction}
\label{SS:SIDERISETC}
In light of the previous paragraph, 
it is perhaps not surprising that the first blowup-results without symmetry assumptions
in fluid mechanics were non-constructive.
Specifically, in \cite{tS1985}, Sideris studied the $3D$ compressible Euler equations under equations of state that satisfy a convexity assumption. For open sets of initial data (including near-constant-state-data) without symmetry or irrotationality assumptions, 
he used arguments based on integrated quantities to give a proof by contradiction that a smooth solution cannot exist for all time.
For the $3D$ relativistic Euler equations and relativistic Euler--Maxwell equations, Guo and Tahvildar-Zadeh proved similar results 
\cite{yGsTZ1998} for open sets of initial data that are \emph{large} perturbations of constant-states with positive density.
In contrast to the present work, \cites{tS1985,yGsTZ1998} relied on convexity assumptions on the equation of state.

\subsection{Proofs of shock formation in multi-dimensions via nonlinear geometric optics}
\label{SS:CONSTRUCTIVEMULTIDIMENSIONAL}
Alinhac \cites{sA1999a,sA1999b,sA2001b,sA2002} 
was the first to prove stable shock formation without symmetry assumptions
for scalar quasilinear wave equations in $2D$ and $3D$ 
of the form:
\begin{align} \label{E:ALINHACQUASLINEARWAVEEQUATION}
	(\hfour^{-1})^{\kappa \lambda}(\pmb{\partial} \Phi)
	\partial_{\kappa} \partial_{\lambda} \Phi
	& = 0
\end{align}
whenever the nonlinearities in \eqref{E:ALINHACQUASLINEARWAVEEQUATION} fail to satisfy the null condition.
In \eqref{E:ALINHACQUASLINEARWAVEEQUATION}, $\hfour$ is a Lorentzian metric whose components 
$\hfour_{\alpha \beta}$ in the standard coordinate system are prescribed functions of $\pmb{\partial} \Phi$,
i.e., functions of the spacetime-gradient of $\Phi$ in the standard coordinate system.
Alinhac's proof applied to open sets of initial data that satisfy a non-degeneracy assumption that guaranteed
that within the constant-time hypersurface of first blowup, the singularity is an isolated point.
His non-degeneracy assumption guaranteed that the singular boundary $\mathcal{B}$ and its past boundary $\crease$
are strictly convex, as in Fig.\,\ref{F:STRICTLYCONVEX}, and in the context of the figure, his approach allowed
him to follow the solution up to the flat constant-time hypersurface that contains the point $b_{\textnormal{lowest}}$;
in Sect.\,\ref{SS:MAXIMALDEVELOPMENT}, we discuss $\mathcal{B}$ and $\crease$ in more detail.
Alinhac's main results showed that at $b_{\textnormal{lowest}}$, a shock forms in the sense that
$\pmb{\partial}^2 \Phi$ blows up while $\pmb{\partial}^{\leq 1} \Phi$ remains bounded.
His proof relied on nonlinear geometric optics 
(i.e., eikonal functions and the method of characteristics, 
as in Sects.\,\ref{SS:1DSIMPLEPLANESYMMETRYANDACOUSTICGEOMETRYSETUP} and \ref{SS:NONLINEARGEOMETRICOPTICS}),
and his proof showed that relative to a \emph{geometric coordinate system} constructed out of the eikonal function,
$\Phi$ and its partial derivatives remain bounded, except possibly at the high derivative levels.
To close the energy estimates in the geometric coordinate system without derivative loss,
he relied on a Nash--Moser iteration scheme that necessarily terminated at the time of first blowup.

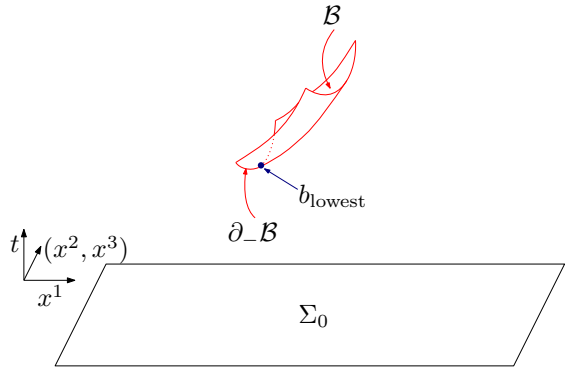
\begin{figure}[H]
 \centering
  	\begin{overpic}[scale=.4, grid = false, tics=5, trim= 0cm 0cm 0cm -1.5cm, clip]{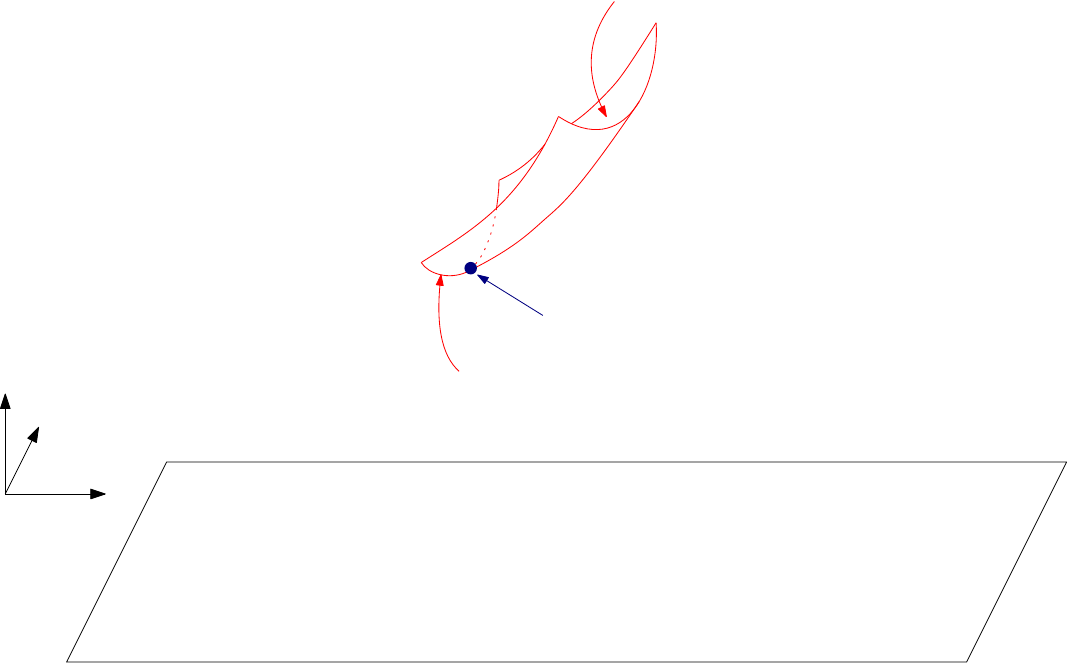}
			\put (51,8) {$\Sigma_0$}
			\put (4,20) {$(x^2,x^3)$}
			\put (-2,21) {$t$}
			\put (3,11) {$x^1$}
			\put (38,23) {$\crease $}
			\put (55.5,63) {$\mathcal{B}$}
			\put (51,30) {$b_{\textnormal{lowest}}$}
		\end{overpic}
		\caption{Strictly convex crease and singular boundary}
\label{F:STRICTLYCONVEX}
\end{figure}

In Christodoulou's breakthrough monograph \cite{dC2007},
he proved a substantially sharper result for the $3D$ relativistic Euler equations 
in regions where the solution is irrotational and isentropic. The dynamics in such regions
is described by a scalar quasilinear wave equation for a potential function $\Phi$ 
(i.e., $\pmb{\partial} \Phi$ determines the four-velocity and proper energy density) 
that is of the form \eqref{E:ALINHACQUASLINEARWAVEEQUATION}
and that is invariant under the Poincar\'{e} group.
Like Alinhac, Christodoulou used an eikonal function in his proof, that is, a solution $\eik$
to the eikonal equation:
\begin{align} \label{E:WAVEEIKONAL}
	(\hfour^{-1})^{\kappa \lambda}(\pmb{\partial} \Phi)
	\partial_{\kappa} \eik \partial_{\lambda} \eik
	& = 0.
\end{align}
The level sets of $\eik$ are characteristic for equation \eqref{E:ALINHACQUASLINEARWAVEEQUATION}, i.e., 
hypersurfaces that are null with respect to $\hfour$.
Also like Alinhac, Christodoulou 
constructed a geometric coordinate system tied to $\eik$ and proved that relative to the geometric coordinates,
the solution remains smooth, except possibly at the high derivative levels. 
He showed that for open sets of initial data, the shock singularity 
--
i.e., the blowing-up of the second-order Minkowski-rectangular derivatives
$\pmb{\partial}^2 \Phi$ while $\pmb{\partial}^{\leq 1} \Phi$ remains bounded 
--
is tied to the vanishing of the \emph{inverse foliation density} $\upmu$,
defined by:
\begin{align} \label{E:WAVEINVERSEFOLIATIONDENSITY}
	\upmu
	& := - \frac{1}{(\hfour^{-1})^{\kappa \lambda}(\pmb{\partial} \Phi) \partial_{\kappa} \eik \partial_{\lambda} t} = 
		- \frac{1}{(\hfour^{-1})^{\kappa 0}(\pmb{\partial} \Phi) \partial_{\kappa} \eik},
\end{align}
where $t$ is the standard time coordinate on Minkowski spacetime.
The vanishing of $\upmu$ signifies the infinite density of the level-sets of $\eik$ in Minkowski-spacetime, i.e.,
the infinite density of the characteristics and the blowup of $\pmb{\partial} \eik$,
as in Theorem~\ref{T:MAINTHEOREM1DSINGULARBOUNDARYANDCREASE}.
Motivated by techniques used in his joint proof of the stability of Minkowski spacetime with Klainerman \cite{dCsK1993},
Christodoulou was able to avoid Nash--Moser estimates. Rather, to avoid derivative loss in the eikonal function 
and to close the proof, he used a \emph{geometric energy method} for wave equations and $\eik$;
see Sect.\,\ref{S:SHOCKFORMATIONAWAYFROMSYMMETRY}
for further discussion of a related geometric energy method
in the context of the $3D$ relativistic Euler equations with vorticity and entropy.

The results of \cite{dC2007} go beyond Alinhac's in several crucial ways.
First, for smooth compactly supported perturbations of non-vacuum constant-state initial data, Christodoulou proved that the solution
remains global unless $\upmu$ vanishes, i.e., he proved a conditional global existence result. Second, his geometric energy method allowed
for commutator and multiplier vectorfields with time and radial weights, which allowed him to simultaneously study the 
interaction between the dispersive tendency of waves and the fact that nonlinearities can focus waves and cause singularities.
Third, for open sets of initial data that do not have to satisfy Alinhac's non-degeneracy condition, Christodoulou proved that
$\upmu$ vanishes in finite time, leading to the blowup of $\pmb{\partial}^2 \Phi$. In particular, even \emph{without} 
strict convexity of the type shown in Fig.\,\ref{F:STRICTLYCONVEX}, Christodoulou was able show that there is at least one singular point.
Finally, for open sets of initial data that
satisfy a non-degeneracy assumption that is weaker than Alinhac's, Christodoulou was able to follow the solution beyond the time
of first blowup and to reveal a portion of the $\hfour$-MGHD, 
up to the boundary. The portion of the boundary of the $\hfour$-MGHD revealed by \cite{dC2007}
was not described explicitly and thus we cannot make straightforward comparisons
with Theorem~\ref{T:MAINTHEOREM1DSINGULARBOUNDARYANDCREASE}.
The weaker notion of non-degeneracy used by Christodoulou coincides with our notion of 
transversal convexity, which we describe in Sect.\,\ref{SS:MAXIMALDEVELOPMENT}.

Christodoulou's results \cite{dC2007} have been extended and applied to other equations, including:
\begin{itemize}
	\item Under the assumption of an irrotational and isentropic flow,
		the (non-relativistic) $3D$ compressible Euler equations 
		were handled in \cite{dCsM2014}.
	\item A larger class of wave equations (e.g., without the assumption of Poincar\'{e} invariance)
		was treated in \cites{sMpY2017,jS2016b}. See also the survey article \cite{gHsKjSwW2016}, which is centered around
		the works \cites{dC2007,jS2016b}.
	\item Solution regimes that are different than the small, compactly supported data regime were treated in
		\cites{gHsKjSwW2016,sMpY2017}.
	\item Solutions that exist classically precisely on a past-infinite half-slab were studied in \cite{sM2018}.
	\item Some systems in $1D$ involving multiple speeds of propagation were treated in \cites{dCdRP2016,xAhCsY2020,xAhCsY2021}.
	\item Other multiple speed systems in multi-dimensions were treated in \cites{jS2018b,jS2019a}.
	\item $2D$ compressible Euler solutions with vorticity were handled in \cite{jLjS2018} with the help of
		the new formulation of the flow derived in \cite{jLjS2020a},
		which yields a non-relativistic analog of Theorem~\ref{T:GEOMETRICWAVETRANSPORTDIVCURLFORMULATION} for barotropic equations
		of state, which by definition are such that the pressure is a function of the density.
	\item $3D$ compressible Euler solutions with vorticity and entropy were treated in \cite{jLjS2021}
		with the help of the new formulation of the flow derived in \cite{jS2019c}, 
		which extends the results of \cite{jLjS2020a} to general equations of state
		in which the pressure is a function of the density and entropy.
	\item For the $3D$ compressible Euler equations with vorticity and entropy,
		a complete description of open sets of solutions up to a neighborhood of the boundary of the $\hfour$-MGHD
		is provided by our series \cites{lAjS2020,lAjS2022,lAjS20XX}, which collectively yield a
		multi-dimensional version of Theorem~\ref{T:MAINTHEOREM1DSINGULARBOUNDARYANDCREASE} for the non-relativistic equations.
		See also Sect.\,\ref{SS:MAXIMALDEVELOPMENT}.
\end{itemize}

\subsection{A different approach to multi-dimensional singularity formation}
\label{SS:DIFFERENTAPPROACH}
In the works \cites{tBsSvV2019a,tBsSvV2020}, the authors developed a new approach for proving gradient-blowup for
for solutions the $3D$ compressible Euler equations under adiabatic equations of state with vorticity 
(and entropy in \cite{tBsSvV2020}). Instead of using nonlinear geometric optics, the authors used 
modulation parameters to show that for open sets of initial data with large gradients, the 
solution forms a gradient singularity in finite time that is a perturbation of a self-similar
Burgers-type shock. The approach
allows one to follow the solution to the time of first blowup, but not further. 
It applies to non-degenerate initial data such that within the constant-time hypersurface of first blowup, the singularity occurs
at an isolated point. That is, the approach applies when the singular boundary $\mathcal{B}$ and its past boundary $\crease$
are strictly convex, as in Fig.\,\ref{F:STRICTLYCONVEX}, and in the context of the figure, it allows one
to follow the solution up to the flat constant-time hypersurface that contains the point 
$b_{\textnormal{lowest}}$, where the singularity occurs.
Such singularities can be viewed as fluid analogs of the ones studied by Alinhac in his
aforementioned works \cites{sA1999a,sA1999b,sA2001b,sA2002} on quasilinear wave equations.
See also the precursor work \cite{tBsSvV2022}, in which the authors studied the same problem 
for the $2D$ compressible Euler equations in azimuthal symmetry, and the work \cite{tBsI2022},
which, in the same symmetry class, constructs unstable self-similar solutions whose cusp-like
behavior at the singularity is non-generic.

Self-similar blowup is a phenomenon that occurs in many other PDEs besides those of fluid mechanics.
In particular, singularity formation modeled on a self-similar
Burgers'-type shock has been proved for various
\emph{non-hyperbolic PDEs} \cites{qYlZ2021,cCteGnM2022,rY2021,krCrcMVgP2021,sjOfP2021,cCteGnM2021,cCteGsInM2022}.

\subsection{Maximal classical $\hfour$-globally hyperbolic developments}
\label{SS:MAXIMALDEVELOPMENT}
It is of mathematical and physical interest to study shock-forming solutions in a region of classical existence
that goes beyond the constant-time hypersurface of first blowup, e.g., in a region larger than the one studied by
Alinhac in \cites{sA1999a,sA1999b,sA2001b,sA2002},
as we described in Sect.\,\ref{SS:CONSTRUCTIVEMULTIDIMENSIONAL}.
The holy grail object in this vein is the
maximal (classical) $\hfour$-globally hyperbolic development
($\hfour$-MGHD), which is what we studied in Theorem~\ref{T:MAINTHEOREM1DSINGULARBOUNDARYANDCREASE}
within the class of simple isentropic plane-symmetric solutions to the relativistic Euler equations.
Recall that, roughly speaking, the $\hfour$-MGHD the largest possible classical solution + $\hfour$-globally hyperbolic region
that is uniquely determined by the initial data.
In particular, the $\hfour$-MGHD can have a complicated boundary 
that includes points lying to the future of the constant-time hypersurface of first blowup,
as in Figs.\,\ref{F:MINKOWSKIRECTANGULAR1DPLANESYMMETRYCAUCHYHORANDSINGULARCURVE}
and \ref{F:CONJECTUREDMAXDEVELOPMENTINMINKRECT}.
Here, ``$\hfour$-globally hyperbolic'' means that the development contains a Cauchy hypersurface, 
i.e., a surface such that every inextendible $\hfour$-causal curve 
(where $\hfour$ is the acoustical metric) in the development intersects it.
In the important work \cite{fEhRjS2019}, the authors showed that for general quasilinear hyperbolic PDEs, 
\emph{one cannot ensure uniqueness of the MGHD}
until one constructs it and shows that it enjoys some crucial structural properties,
which are global in nature.

In our series \cites{lAjS2020,lAjS2022,lAjS20XX}, for open sets of initial data for the $3D$ compressible Euler equations
leading to shock-forming solutions, we constructed a large (though bounded) portion of the $\hfour$-MGHD, 
up to the boundary, i.e., in the non-relativistic case without symmetry assumptions,
we proved an analog of Theorem~\ref{T:MAINTHEOREM1DSINGULARBOUNDARYANDCREASE}.
In particular, these are the first results to fully justify
Fig.\,\ref{F:CONJECTUREDMAXDEVELOPMENTINMINKRECT} for open sets of solutions to the $3D$ compressible Euler equations
with vorticity and dynamic entropy. 
Although these works are of mathematical interest in themselves, they are also fundamental for setting up
the shock development problem, as we describe in Sect.\,\ref{SS:SHOCKDEVELOPMENT}.

Our results \cites{lAjS2020,lAjS2022,lAjS20XX} exhibit a localized version of the crucial property that the $\hfour$-MGHD 
``lies on one side of its boundary,'' as 
is shown in Fig.\,\ref{F:CONJECTUREDMAXDEVELOPMENTINMINKRECT}.
In \cite{fEhRjS2019}, the authors showed, roughly, 
that if this ``one-sidedness'' property holds globally, then the $\hfour$-MGHD is unique.
They also showed, for a specific hyperbolic PDE and specific initial data, that \emph{uniqueness of the MGHD fails} and the property fails!
As in Theorem~\ref{T:MAINTHEOREM1DSINGULARBOUNDARYANDCREASE}, 
the boundary portion we construct in \cites{lAjS2020,lAjS2022,lAjS20XX} has two kinds of components:
\begin{itemize}
	\item A singular boundary $\mathcal{B}$, constructed in \cite{lAjS2020},
	along which the fluid's gradient blows up, the past boundary of which 
	is a $2$-dimensional acoustically spacelike\footnote{Acoustically spacelike means spacelike with respect to
	the acoustical metric of the $3D$ compressible Euler equations. Acoustically null means null with respect to the
	acoustical metric. \label{FN:ACOUSTICALLYSPACELIKEANDNULL}} 
	submanifold known as \emph{the crease} and denoted by ``$\partial_- \mathcal{B}$'' in Fig.\,\ref{F:CONJECTUREDMAXDEVELOPMENTINMINKRECT}.
\item A Cauchy horizon, constructed in \cite{lAjS20XX}
	and denoted by ``$\cauchyhor$'' in Fig.\,\ref{F:CONJECTUREDMAXDEVELOPMENTINMINKRECT},
	which is an acoustically null
	hypersurface that emanates from the crease and along which the solution extends smoothly (except at the crease).
	The crease plays the role of the \emph{true initial singularity} in shock-forming solutions.
\end{itemize}

Our analysis in \cites{lAjS2022,lAjS20XX} relies on an assumption of \emph{transversal convexity},
which is a weak, stable form of convexity 
that can be read off the initial data and propagated by the flow.
Roughly, transversal convexity means that when
the inverse foliation density 
$\upmu$ 
(which in the context of \cites{lAjS2022,lAjS20XX}
is an analog of \eqref{E:WAVEINVERSEFOLIATIONDENSITY} or \eqref{E:INVERSEFOLIATIONDENSITY} for the $3D$ compressible Euler equations)
is small, $\upmu$ has a positive second derivative in a direction transversal to the level sets of the eikonal function $\eik$.
For the simple isentropic plane-symmetric solutions 
we studied in Sect.\,\ref{S:1DMAXIMALDEVELOPMENT}, transversal convexity is ensured by the estimate 
\eqref{AE:PSMUSECONDDERIVATIVETAYLOREXPANSIONININTERESTINGREGION} and is manifested in 
Fig.\,\ref{F:MINKOWSKIRECTANGULAR1DPLANESYMMETRYCAUCHYHORANDSINGULARCURVE} by the upwards-bending
nature of $\Upsilon(\singularcurve_{[-\interestingu,0]})$. We highlight the importance of transversal
convexity with the following remarks.

\begin{quote}
		Even for simple isentropic plane-symmetric solutions, in the absence of transversal convexity, the
		qualitative structure of the singularity can be radically altered.
		For example, if instead of the estimate \eqref{AE:PSMUSECONDDERIVATIVETAYLOREXPANSIONININTERESTINGREGION}
		there held $\muX \muX \upmu \equiv 0$ at the first singularity, 
		then in the $(t,x^1)$-coordinates picture, 
		an entire continuum of characteristics would fold into a single point,
		and the change of variables map $\Upsilon$ from Theorem~\ref{T:MAINTHEOREM1DSINGULARBOUNDARYANDCREASE}
		would dramatically fail to be an injection. This would be a serious obstacle to even the local
		well-posedness of the shock development problem, which we describe in Sect.\,\ref{SS:SHOCKDEVELOPMENT}.
		This is in contrast to the ``favorable'' situation shown in
		Fig.\,\ref{F:MINKOWSKIRECTANGULAR1DPLANESYMMETRYCAUCHYHORANDSINGULARCURVE}, in which, thanks to transversal convexity,
		the characteristics \emph{graze} the singular boundary $\Upsilon(\singularcurve_{[-\interestingu,0]})$,
		but distinct characteristics do not actually intersect along $\Upsilon(\singularcurve_{[-\interestingu,0]})$.
\end{quote}

In the context of the $3D$ solutions depicted in
Fig.\,\ref{F:CONJECTUREDMAXDEVELOPMENTINMINKRECT}, the transversal convexity is manifested by
$\mathcal{B}$ being ``convex in the $x^1$-direction'' (i.e., upwards bending in the $x^1$ direction)
but \emph{not necessarily convex in the $(x^2,x^3)$-directions}. Note that if we view plane-symmetric solutions as solutions in $3D$
with symmetry, then the corresponding singular boundary $\mathcal{B}$ is not strictly convex because of the ``symmetric directions''
$(x^2,x^3)$. 

Of the works on shock formation cited in Sect.\,\ref{SS:CONSTRUCTIVEMULTIDIMENSIONAL}, Christodoulou's
monograph \cite{dC2007} on $3D$ irrotational and isentropic relativistic fluids
is the only one that follows the solution beyond the time of first blowup.
He used eikonal functions and foliations of spacetime by $\hfour$-spacelike planes $\Sigma$ 
(more precisely, the $\Sigma$ were coordinate planes with respect to the standard Minkowski-rectangular coordinates)
to reveal an implicit portion of the boundary of the $\hfour$-MGHD, where, for example,
the portion of the crease $\crease$ that was revealed is such that the 
crease had to lie to the future of any $\Sigma$ that is tangent to it.
For initial data such that the crease is a strictly convex subset of Minkowski-rectangular coordinate space 
(as in Fig.\,\ref{F:STRICTLYCONVEX}),
his approach can be used to study the entire crease and a neighborhood of the singular boundary
that emerges from it. However, strict convexity does not hold for all shock-forming solutions. In particular,
as we noted in the previous paragraph,
strict convexity fails for symmetric solutions (roughly, strict convexity will fail in the directions of symmetry) 
and general small perturbations of them.
In \cite{lAjS2022}, under the weaker assumption of transversal convexity, we construct the entire crease
and a neighborhood of the singular boundary that contains it for solutions with vorticity and dynamic entropy.
To handle the lack of strict convexity, we dynamically construct a new foliation of spacetime
by ``rough acoustically spacelike hypersurfaces,'' 
depicted in Fig.\,\ref{F:INTERESTINGREGIONMAINRESULTSCARTESIAN},
that are precisely adapted to the shape of the crease and singular boundary.
We refer to the foliations as ``rough'' because they are less differentiable than the fluid solution, 
which is a key difficulty that has to be overcome in the PDE analysis.
To handle the presence of vorticity and dynamic entropy, we combine the special structure of the 
geometric wave-transport-div-curl formulation of the flow derived in \cite{jS2019c} with upgraded versions of the integral
identities derived in \cite{lAjS2020}, which in total yield a regularity theory for the vorticity and entropy that
is adapted to the rough foliations 
(and hence the shape of the singular boundary)
and that allows one to treat the fluid equations as a perturbation of a wave equation system.
In particular, as in \cites{jLjS2018,jLjS2021}, the framework yields (and requires) one extra degree of
differentiability for the vorticity and entropy compared to standard estimates.
See also Sect.\,\ref{SSS:TOPORDERELLIPTICVORTICITYANDENTROPY}, in which we discuss
the extra differentiability of the vorticity and entropy in the context of the
$3D$ relativistic Euler equations.

\begin{center}
	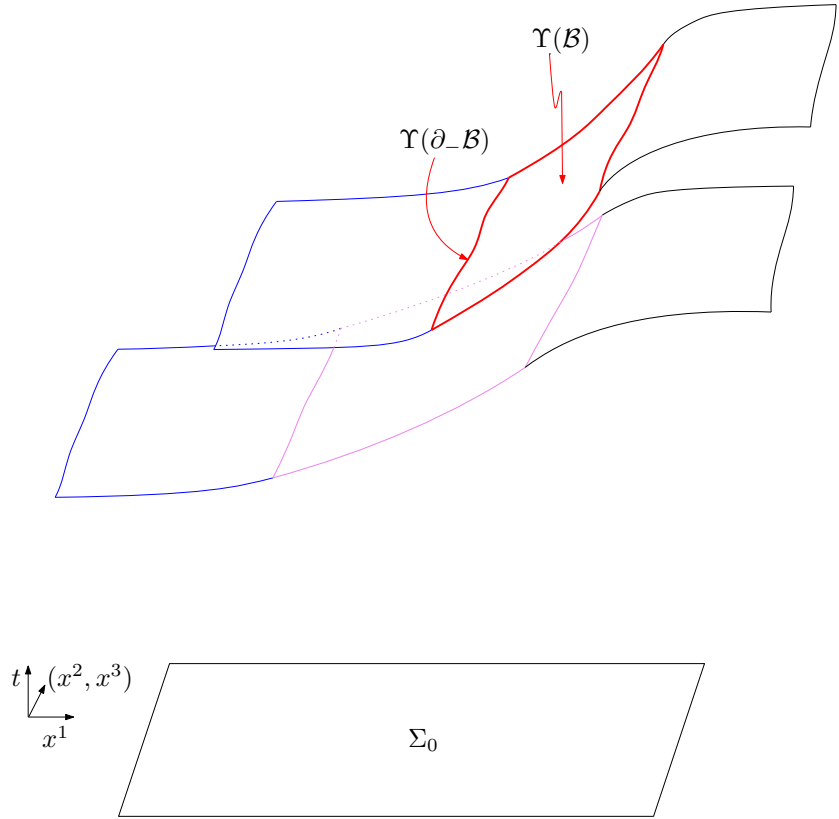
\begin{figure}[ht]  
		\begin{overpic}[scale=.6, grid = false, tics=5, trim=-.5cm -1cm -1cm -.5cm, clip]{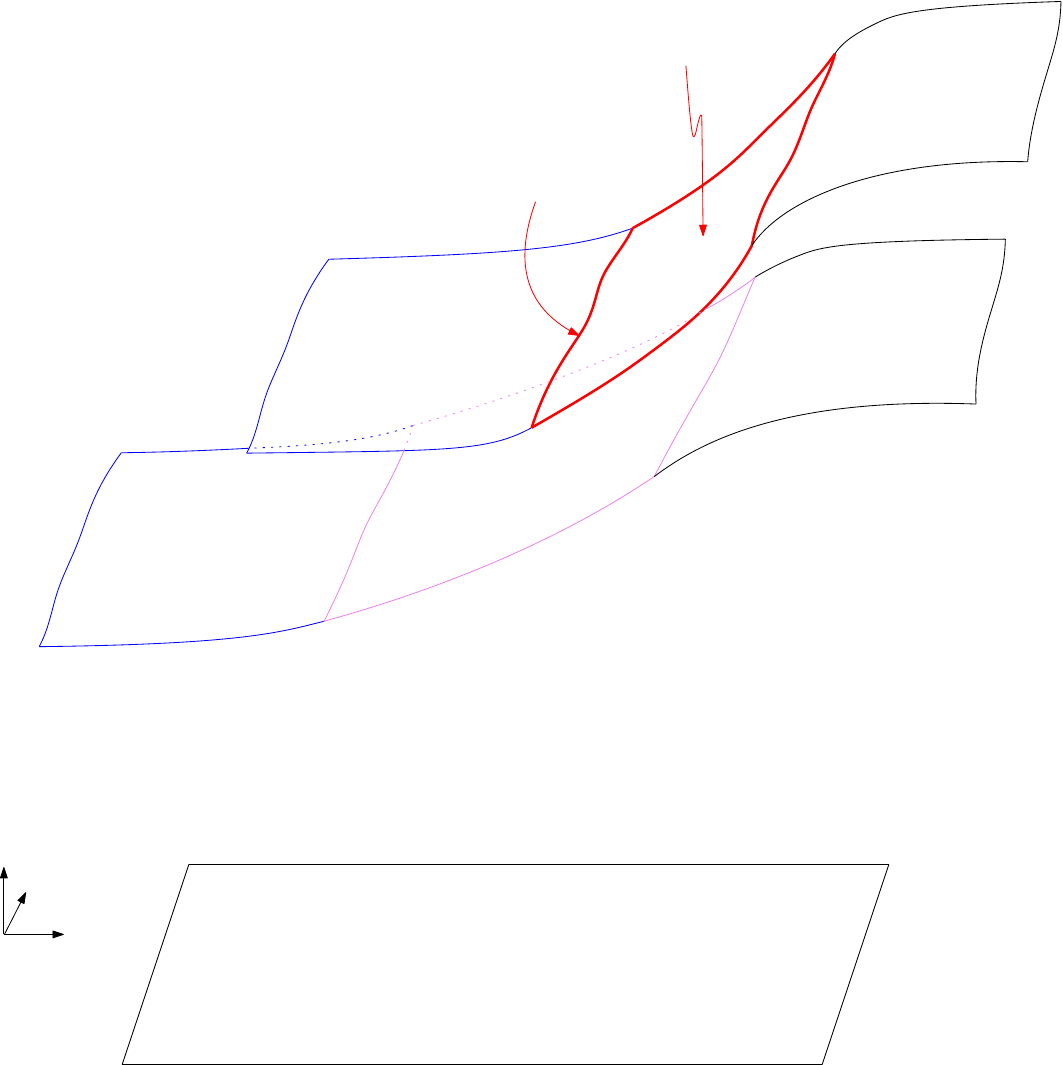}
			\put (46,13) {$\Sigma_0$}
			\put (5,20) {$(x^2,x^3)$}
			\put (1,20) {$t$}
			\put (4.5,13) {$x^1$}
			\put (45,81) {$\Upsilon(\crease)$}
			\put (60,92.5) {$\Upsilon(\mathcal{B})$}
	\end{overpic}
		\caption{Rough foliations and the singular boundary from \cite{lAjS2022}, depicted in Cartesian coordinate space}
	\label{F:INTERESTINGREGIONMAINRESULTSCARTESIAN}
	\end{figure}
\end{center}

\subsection{Shock development problem}
\label{SS:SHOCKDEVELOPMENT}
In Majda's celebrated works \cites{aM1981,aM1983a}, he solved the \emph{shock front problem} for
open sets of solutions to the non-relativistic $3D$ compressible Euler equations. Roughly, this means that he considered
piecewise smooth initial data that have a jump discontinuity across a smooth hypersurface, and he then constructed a local weak solution
to the equations and a shock hypersurface such that the solution is piecewise smooth on either side of 
shock hypersurface but has a jump discontinuity across it and satisfies the Rankine--Hugoniot condition
and an entropy-type condition.

The \emph{shock development problem} is closely related to the shock front problem
but is laden with additional technical difficulties. 
It is the problem of describing how initially smooth solutions form a ``first singularity''
(which is the crease, as described in Sect.\,\ref{SS:MAXIMALDEVELOPMENT}), 
from which emerges a shock hypersurface 
(which is not known in advance, as in the shock front problem),
and describing the transition of the solution from classical to weak 
such that in the ``weak solution region,'' the solution jumps across the shock hypersurface
and satisfies the Rankine--Hugoniot jump condition and an entropy-type condition. In particular,
in the shock development problem, there is no jump discontinuity in the initial data; rather, the jump
emerges dynamically, starting from the crease. 
In problems that are effectively one-dimensional
(and thus energy estimates can be avoided), the shock development
problem has been solved for various hyperbolic PDEs
\cites{mL1994,sClD2001,yHzL2022a,dCaL2016,yHzL2022b,tBtdDsSvV2022}.

The only solution to a multi-dimensional shock development problem without symmetry assumptions was provided by
Christodoulou's recent breakthrough monograph \cite{dC2019},
in which he used nonlinear geometric optics 
(i.e., a pair of acoustic eikonal functions, one ingoing and one outgoing)
to solve the \emph{restricted shock development problem} 
in an arbitrary number of spatial dimensions for the compressible Euler equations and the
relativistic Euler equations. Roughly, the word ``restricted'' means that he solved an idealized problem in which he
ignored the jump in entropy and vorticity across the shock hypersurface,
thereby producing a weak solution to a hyperbolic PDE system that approximates the true one;
in the real problem, which remains unsolved, the vorticity and entropy must jump across the shock hypersurface.
Compared to the $1D$ case, the main new difficulty in multiple spatial dimensions is that one must derive
energy estimates, which are degenerate near the first singularity for reasons related to the energy degeneracies
that occur in the shock formation problem, as we discuss in Sect.\,\ref{SSS:WAVEENERGYESTIMATEHIERARCHY}.

We highlight the following key issue, which is a primary motivating factor for our construction
of large portions of the $\hfour$-MGHD in \cites{lAjS2020,lAjS2022,lAjS20XX}:

\begin{quote}
The setup of the shock development problem 
(see \cites{dCaL2016,dC2019}) is such that the data for
it are an $\hfour$-MGHD\footnote{The shock development problem is local, so to locally study the solution, one only needs 
as initial data a portion of the $\hfour$-MGHD, 
specifically a portion that contains a crease and a portion of the singular boundary and Cauchy horizon that emerge from
it, which we constructed in \cites{lAjS2020,lAjS2022,lAjS20XX}. \label{FN:SHOCKDEVELOPMENTPROBLEMISLOCAL}} 
of a shock-forming solution launched by smooth initial data, with a precise description of the 
gradient-blowup along a singular boundary, a precise description of the classical solution's regular behavior 
along a Cauchy horizon $\cauchyhor$, and a sharp description of the structure of the ``first singularity,'' which
in Sect.\,\ref{SS:MAXIMALDEVELOPMENT} (also in Theorem~\ref{T:MAINTHEOREM1DSINGULARBOUNDARYANDCREASE}),
we referred to as ``the crease'' and which we denote by ``$\crease$'' in Fig.\,\ref{F:CONJECTUREDMAXDEVELOPMENTINMINKRECT}.
Crucially,
under the
``transversal convexity'' assumption described in Sect.\,\ref{SS:MAXIMALDEVELOPMENT},
$\crease$ is a co-dimension $2$, 
$\hfour$-spacelike submanifold equal to the intersection
of the Cauchy horizon and the singular boundary; without this structure, it is not even clear whether the
shock development problem is well-posed. These structures, 
as well as the full state of the fluid up to $\crease, \, \cauchyhor$,  
were \emph{posited} as data in \cites{dCaL2016,dC2019}. 
In $1D$, Theorem~\ref{T:MAINTHEOREM1DSINGULARBOUNDARYANDCREASE} provides the first rigorous \emph{construction} of these qualitative and quantitative aspects of the data for the shock development problem, starting from smooth initial conditions on a spacelike
Cauchy hypersurface. 
For the non-relativistic Euler equations in $3D$ and without symmetry assumptions, 
our works \cites{lAjS2020, lAjS2022, lAjS20XX} collectively provide an analog of Theorem~\ref{T:MAINTHEOREM1DSINGULARBOUNDARYANDCREASE}. 
\end{quote}

\subsection{Rarefaction waves}
\label{SS:RAREFACTION}
In his foundational papers \cites{sA1989a,sA1989b}, for a large class of 
multi-dimensional hyperbolic systems that includes scalar
conservation laws and the compressible Euler equations as special cases, 
Alinhac proved local existence and uniqueness for rarefaction wave solutions.
Multi-dimensional rarefaction waves are analogs of a class of solutions to the well-known
Riemann problem in $1D$, 
in which the initial data are piecewise smooth and discontinuous, 
and the initial discontinuity is immediately smoothed out by the flow.
His approach relied on Nash--Moser estimates to overcome derivative loss
in linearized versions of the equations.
Alinhac proved \cites{sA1989a,sA1989b} that,
as in the $1D$ case, in the corresponding multi-dimensional 
region + solution, the initial discontinuity is immediately smoothed out by the flow.
In the important recent works \cites{tWLpY2023a,tWLpY2023b}, 
the authors study the isentropic $2D$ compressible 
Euler equations with a family of irrotational discontinuous initial data that are (asymmetric) perturbations of plane-symmetric
data for a corresponding $1D$ Riemann problem. 
For irrotational data, 
their main result provides a sharpened version of the local existence and uniqueness results of Alinhac \cites{sA1989a,sA1989b}.
Their work shows in particular that the $2D$ irrotational rarefaction solution is a perturbation of the standard $1D$ rarefaction wave solution to the Riemann problem. Compared to \cites{sA1989a,sA1989b}, the works \cites{tWLpY2023a,tWLpY2023b} yield two important advances. 
First, the techniques of \cites{tWLpY2023a,tWLpY2023b} avoid a loss of derivatives in a corresponding linearized problem and consequently, 
the authors were able to close the energy estimates without Nash--Moser estimates. 
Second, \cites{tWLpY2023a,tWLpY2023b} provide a complete description of the characteristic geometry in the problem.

\section{Some ideas behind the proof of shock formation in \texorpdfstring{$3D$}{3D} without symmetry}
\label{S:SHOCKFORMATIONAWAYFROMSYMMETRY}
In this section, we provide an outline (without complete proofs) of how to extend some aspects of 
Theorem~\ref{T:MAINTHEOREM1DSINGULARBOUNDARYANDCREASE} to apply to open sets of
$3D$ relativistic Euler solutions without symmetry, isentropicity, or irrotationality assumptions.
We anticipate that \emph{all} aspects of Theorem~\ref{T:MAINTHEOREM1DSINGULARBOUNDARYANDCREASE}
can be extended. 
A key reason behind our expectation is the availability of the new formulation of the flow provided by 
Theorem~\ref{T:GEOMETRICWAVETRANSPORTDIVCURLFORMULATION},
which is qualitatively similar to the formulation of the non-relativistic flow from \cite{jS2019c}, 
which was used in prior related works  
\cites{jLjS2018,jLjS2021,lAjS2020,lAjS2022,lAjS20XX}
on shocks for the (non-relativistic) $3D$ compressible Euler equations.
To keep the discussion short, we only 
sketch some key ideas behind the proof of Conjecture~\ref{CON:SHOCKWITHVORTICITYANDENTROPY},
which essentially concerns studying $3D$ solutions until their time of first blowup. 
Our discussion here mirrors the strategy we used in our work \cite{jLjS2021},
in which we proved an analog of Conjecture~\ref{CON:SHOCKWITHVORTICITYANDENTROPY}
for the non-relativistic $3D$ compressible Euler equations.

\subsection{Conjectures for the $3D$ relativistic Euler equations}
\label{SS:CONJECTUREDSHOCKFORMATIONWITHOUTSYMMETRY}
To set the stage, we state three conjectures for the $3D$ relativistic Euler equations that are tied to the following figure:

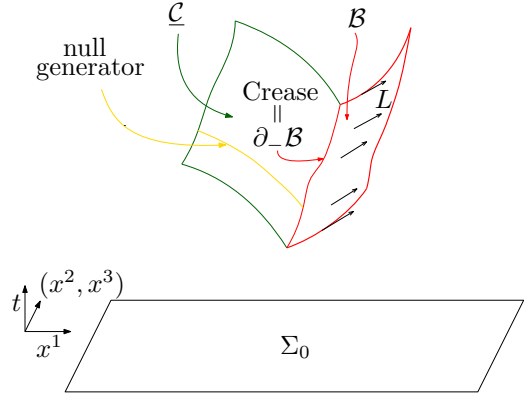
\begin{figure}[H] 
\centering
\begin{overpic}[scale=.36, grid = false, tics=5, trim=-.5cm 0cm -1cm -.5cm, clip]{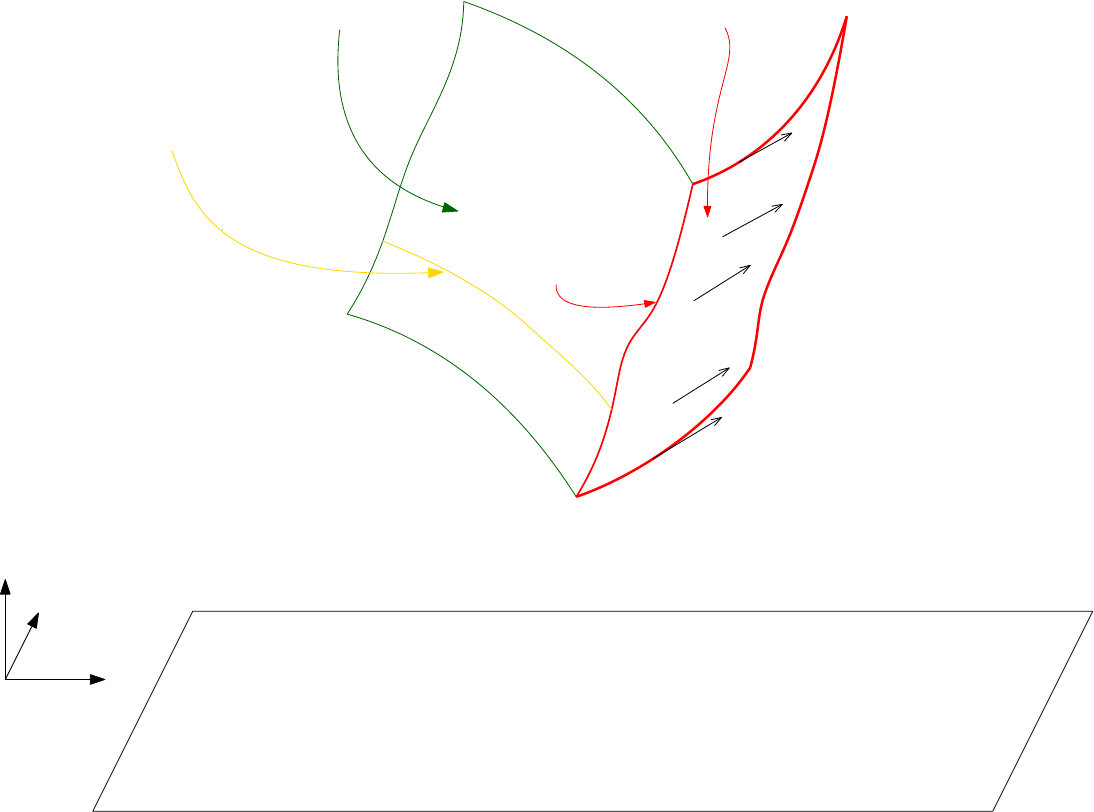}
			\put (50,7) {$\Sigma_0$}
			\put (5.5,18.5) {$(x^2,x^3)$}
			\put (.5,15) {$t$}
			\put (5,6.5) {$x^1$}
			\put (43,54) {$\mbox{\upshape Crease}$}
			\put (49,49.5) {\rotatebox{90}{$=$}}
			\put (45,45.5) {$\partial_- \mathcal{B}$}
			\put (67.5,52.5) {$\Lunit$}
			\put (62.5,67.5) {$\mathcal{B}$}
			\put (29.5,68.5) {$\cauchyhor$}
			\put (10,62) {$\mbox{\upshape null}$}
			\put (5,58.5) {$\mbox{\upshape generator}$}
	\end{overpic}
	\caption{Conjectured local picture of the $\hfour$-MGHD for general small perturbations of simple isentropic plane-symmetric solutions 
		to the $3D$ relativistic Euler equations in Minkowski-rectangular coordinate space,
		with the map $\Upsilon(t,x^1,x^2,x^3) = (t,\eik,x^2,x^3)$ suppressed}
	\label{F:CONJECTUREDMAXDEVELOPMENTINMINKRECT}
\end{figure}

For the $3D$ compressible Euler equations, Fig.\,\ref{F:CONJECTUREDMAXDEVELOPMENTINMINKRECT} has been justified 
\cites{lAjS2020,lAjS2022,lAjS20XX} for open sets of initial data.

\begin{conjecture}[Time of first blowup for the $3D$ relativistic Euler equations with vorticity and entropy]
	\label{CON:SHOCKWITHVORTICITYANDENTROPY}
	Consider the simple isentropic plane-symmetric solutions of Theorem~\ref{T:MAINTHEOREM1DSINGULARBOUNDARYANDCREASE} 
	as ``background solutions'' to the $3D$ relativistic Euler equations with symmetry. Consider smooth
	initial data in $3D$ -- without symmetry, irrotationality, or isentropicity assumptions --
	that are close (in a sufficiently high-order Sobolev space) to the data of one of the background solutions.
	Then the perturbed solution forms a shock at a time that is a perturbation of $\blowuptime$, i.e.,
	a perturbation of the shock-formation-time of the background solution;
	the conjectured perturbed first blowup-time is the smallest value of $t$ along the crease $\crease$
	depicted in Fig.\,\ref{F:CONJECTUREDMAXDEVELOPMENTINMINKRECT}.
\end{conjecture}

\begin{conjecture}[The structure of the crease and singular boundary for the $3D$ relativistic Euler equations with vorticity and entropy]
	\label{CON:SINGULARBOUNDARY}
	Under the assumptions of Conjecture~\ref{CON:SHOCKWITHVORTICITYANDENTROPY}, 
	the perturbed solution has a crease and singular boundary that are 
	perturbations of the crease singular boundary of the 
	background solution; the background crease and singular boundary are depicted in
	Fig.\,\ref{F:MINKOWSKIRECTANGULAR1DPLANESYMMETRYCAUCHYHORANDSINGULARCURVE},
	while the conjectured perturbed singular boundary and crease are depicted in
	Fig.\,\ref{F:CONJECTUREDMAXDEVELOPMENTINMINKRECT}.
\end{conjecture}

\begin{conjecture}[The structure of the Cauchy horizon with vorticity and entropy]
	\label{CON:CAUCHYHORIZON}
	Under the assumptions of Conjecture~\ref{CON:SHOCKWITHVORTICITYANDENTROPY},  
	the perturbed solution has a Cauchy horizon that is a perturbation of the Cauchy horizon of the 
	background solution; the background Cauchy horizon is depicted in
	Fig.\,\ref{F:MINKOWSKIRECTANGULAR1DPLANESYMMETRYCAUCHYHORANDSINGULARCURVE},
	while the conjectured perturbed Cauchy horizon is depicted in
	Fig.\,\ref{F:CONJECTUREDMAXDEVELOPMENTINMINKRECT}.
\end{conjecture}

Conjecture~\ref{CON:SINGULARBOUNDARY} is strictly harder than
Conjecture~\ref{CON:SHOCKWITHVORTICITYANDENTROPY},
while Conjecture~\ref{CON:CAUCHYHORIZON} is independent (though closely related).

\subsection{Almost Riemann invariants}
\label{SS:ALMOSTRIEMANNINVARIANTS}
To study general (asymmetric) perturbations of simple isentropic plane-symmetric solutions, it is convenient
to use analogs of the Riemann invariants from Def.\,\ref{D:RIEMANNINVARIANTS}.

\begin{definition}[Almost Riemann invariants]	
	\label{D:ALMOSTRIEMANNINVARIANTS}
	Let $\RIfunction = \RIfunction(\Enth,\Ent)$ be the solution to
	the following transport equation initial value problem:
\begin{align} \label{E:NOSYMMETRYODEFORRIEMANNINVARIANT}
	\frac{\partial}{\partial \Enth}
	\RIfunction(\Enth,\Ent)
	& = \frac{1}{\Enth \speed(\Enth,\Ent)},
	&
	\RIfunction(\overline{\Enth},\Ent)
	= 0,
\end{align}	
where $\overline{\Enth} > 0$ is the constant
\eqref{E:FIXEDPOSITIVEENTHALPYCONSTANT},
and in \eqref{E:NOSYMMETRYODEFORRIEMANNINVARIANT}, we are viewing the speed of
sound $\speed$ as a function of $\Enth$ and $\Ent$.
	
	We define the \emph{almost Riemann invariants} $\RRiemann$ and $\LRiemann$ as follows:
	\begin{subequations} 
	\begin{align} \label{E:ALMOSTRRIEMANNPS}
		\RRiemann
		& := 
				\RIfunction(\Enth,\Ent)
				+
				\frac{1}{2} \ln \left( \frac{1 + \frac{\fourvelocity^1}{\sqrt{1 + (\fourvelocity^1)^2}}}{1 - \frac{\fourvelocity^1}{\sqrt{1 + (\fourvelocity^1)^2}}} \right),
			\\
		\LRiemann
		& := 
			\RIfunction(\Enth,\Ent)
			-
			\frac{1}{2} \ln \left(\frac{1 + \frac{\fourvelocity^1}{\sqrt{1 + (\fourvelocity^1)^2}}}{1 - \frac{\fourvelocity^1}{\sqrt{1 + (\fourvelocity^1)^2}}} \right).
		\label{E:ALMOSTLRIEMANNPS}
	\end{align}
	\end{subequations}
\end{definition}
Note that for plane-symmetric solutions with $\Ent \equiv 0$, the quantities 
defined by \eqref{E:ALMOSTRRIEMANNPS}--\eqref{E:ALMOSTLRIEMANNPS}
coincide with the ones defined in \eqref{E:RRIEMANNPS}--\eqref{E:LRIEMANNPS}.

We view $\RRiemann$ and $\LRiemann$ as replacements for $\Enth$ and $\fourvelocity^1$ 
that are convenient for studying perturbations of simple isentropic plane-symmetric solutions.
In particular, for values of
$
(\Enth,\Ent)
$
near
$
(\overline{\Enth},0)
$,
the factor $\frac{1}{\Enth \speed(\Enth,\Ent)}$ in
\eqref{E:NOSYMMETRYODEFORRIEMANNINVARIANT} is positive. 
The implicit function theorem then allows one to solve for $\Enth$ as a smooth function of 
$F$ and $\Ent$.
We can then use the relations \eqref{E:ALMOSTRRIEMANNPS}--\eqref{E:ALMOSTLRIEMANNPS}
to express $(\Enth,\fourvelocity^1)$ as smooth functions of $(\RRiemann,\LRiemann,\Ent)$.
Then, under the algebraic relation \eqref{E:AGAINFLUIDFOURVELOCITYNORMALIZED}, 
a complete set of state-space variables for the $3D$ relativistic Euler equations is given by:
\begin{align} \label{E:WAVEARRAYWITHALMOSTRIEMANNINVARIANTS}
	\wavearray
	& := (\RRiemann,\LRiemann,\fourvelocity^2,\fourvelocity^3,\Ent).
\end{align}
Although we have previously used the symbol $\wavearray$ to denote the array
$\left(\Lnenth,\fourvelocity^0,\fourvelocity^1,\fourvelocity^2,\fourvelocity^3,\Ent \right)$, 
in the rest of Sect.\,\ref{S:SHOCKFORMATIONAWAYFROMSYMMETRY},
we use $\wavearray$ to denote the array in \eqref{E:WAVEARRAYWITHALMOSTRIEMANNINVARIANTS}.

Away from isentropic plane-symmetry, 
$
\RRiemann
$
and
$
\LRiemann
$
no longer solve the homogeneous transport equations 
\eqref{E:RRIEMANNPSTRANSPORT}--\eqref{E:LRIEMANNPSTRANSPORT}.
However, from \eqref{E:WAVE},
the chain rule, 
and Def.\,\ref{D:ALMOSTRIEMANNINVARIANTS},
it follows that for general solutions,
$
\RRiemann
$
and
$
\LRiemann
$
satisfy geometric wave equations that have the same schematic form
as \eqref{E:WAVE}. In the rest of Sect.\,\ref{S:SHOCKFORMATIONAWAYFROMSYMMETRY}, we use this fact without
always explicitly mentioning it.

\subsection{Nonlinear geometric optics}
\label{SS:NONLINEARGEOMETRICOPTICS}
As in the $1D$ case treated in Sect.\,\ref{S:1DMAXIMALDEVELOPMENT},
the study of shock formation in $3D$ relies on nonlinear geometric optics,
i.e., multi-dimensional 
analogs of the constructions from Sect.\,\ref{SS:1DSIMPLEPLANESYMMETRYANDACOUSTICGEOMETRYSETUP}.
The main object behind the construction is an eikonal function.
The use of eikonal functions to study the \emph{global} properties of multi-dimensional quasilinear hyperbolic
PDE solutions originated Christodoulou--Klainerman's proof
\cite{dCsK1993} of the stability of Minkowski spacetime.
See also Sect.\,\ref{SS:CONSTRUCTIVEMULTIDIMENSIONAL} for a discussion of nonlinear geometric optics in the context
of proofs of multi-dimensional shock formation.

\begin{quote}
As we discussed at the beginning of Sect.\,\ref{SS:1DSIMPLEPLANESYMMETRYANDACOUSTICGEOMETRYSETUP}, 
for simple isentropic plane-symmetric solutions, the role of nonlinear geometric optics/geometric coordinates 
is to yield a framework/differential structure in which the solution remains smooth all the way up to the shock, 
thereby allowing one to reduce the problem of shock formation
to a more traditional problem in which one studies long-time existence. 
The same remarks hold in the present context of three spatial dimensions,
although many additional geometric constructions are required. In particular,
all of the geometric quantities we introduce in Sects.\,\ref{SS:NONLINEARGEOMETRICOPTICS} and \ref{SS:MULTIPLIERMETHOD}, 
including the acoustic
eikonal function, the geometric coordinates, various geometric vectorfields, 
and various geometric energies
are introduced in order 
regularize the problem. However, unlike in the simple isentropic plane-symmetric case, 
in the present multi-dimensional context, one needs to derive energy estimates, and it turns out that
a remnant of the shock singularity survives, even in the ``good'' geometric coordinates. The remnant manifests
as singular estimates for the high-order geometric energies, though it is crucial for our approach 
that the mid-order-and-below geometric energies remain bounded; see Sect.\,\ref{SSS:WAVEENERGYESTIMATEHIERARCHY}
for further discussion.
\end{quote}

\begin{definition}[Acoustic eikonal function]
\label{D:GENERALEIKONALFUNCTION}
An acoustic eikonal function (eikonal function for short) is a solution
to the eikonal equation, which is the following fully nonlinear hyperbolic PDE, where
$\hfour$ is the acoustical metric from Def.\,\ref{D:ACOUSTICALMETRIC}:
\begin{align} \label{E:EIKONALEQUATION}
	(\hfour^{-1})^{\kappa \lambda}(\wavearray) \partial_{\kappa} \eik \partial_{\lambda} \eik 
	& = 0,
	&
	\partial_t & \eik > 0.
\end{align}
\end{definition}

When studying shocks close to plane-symmetry, it is convenient to assume the following
initial condition for $\eik$, which is the exact same initial condition \eqref{E:1DEIKONALDATA}
that we used in our study of isentropic plane-symmetric solutions:
\begin{align} \label{E:EIKONALDATAAWAYFROMSYMMETRY}
	\eik|_{t=0} 
	& = - x^1.
\end{align}
We refer to the level sets of $\eik$ as \emph{the characteristics}.
One can show that for isentropic plane-symmetric solutions, the solution ``$\eik$'' to
\eqref{E:1DEIKONALTRANSPORTEQUATION}--\eqref{E:1DEIKONALDATA}
coincides with the solution to \eqref{E:EIKONALEQUATION}--\eqref{E:EIKONALDATAAWAYFROMSYMMETRY}.

\begin{definition}[Inverse foliation density]
\label{D:INVERSEFOLIATIONDENSITY}
We define the inverse foliation density of the characteristics to be the following scalar function:
\begin{align} \label{E:INVERSEFOLIATIONDENSITY}
	\upmu
	& := - \frac{1}{(\hfour^{-1})^{\kappa \lambda}(\wavearray) \partial_{\kappa} t \partial_{\lambda} \eik}
		=
		\frac{1}{\Sigmatnormal \eik},
\end{align}
where the second equality in \eqref{E:INVERSEFOLIATIONDENSITY} follows from \eqref{E:SIGMATNORMAL}.
\end{definition}

One can show that for isentropic plane-symmetric solutions, the 
quantity ``$\upmu$'' defined by \eqref{E:INVERSEFOLIATIONDENSITYINPLANESYMMETRY}
coincides with the one defined in \eqref{E:INVERSEFOLIATIONDENSITY}.
As in Sect.\,\ref{S:1DMAXIMALDEVELOPMENT}, the formation of a shock occurs when the density of the characteristics becomes
infinite, that is, when $\upmu \downarrow 0$.

\subsection{Geometric coordinates and vectorfield frames tied the eikonal function}
\label{SS:GOODVECTORFIELDFRAME}
In this section, we introduce geometric coordinates and related vectorfields
that are analogs of quantities we used in our study isentropic plane-symmetric solutions.

\subsubsection{Geometric coordinates}
\label{SSS:GEOMETRICCOORDIANTES}

\begin{definition}[Geometric coordinates and the corresponding partial derivative vectorfields]
\label{D:3DGEOMETRICCOORDINATES}
We refer to $\lbrace t, \eik, x^2, x^3 \rbrace$ as the \emph{geometric coordinates} on spacetime,
where $\eik$ is the eikonal function and
$t$, $x^2$, and $x^3$ are the Minkowski-rectangular coordinates, with $t := x^0$.
We denote the corresponding partial derivatives by $\lbrace \geop{t}, \geop{\eik}, \geop{x^2}, \geop{x^3} \rbrace$
(which are not to be confused with the partial derivatives $\partial_{\alpha}$ in the $(t,x^1,x^2,x^3)$-coordinate system).
\end{definition}

\begin{definition}[Some important submanifolds of geometric coordinate space]
\label{D:3DSUBMANIFOLDSSOFGEOMETRICCOORDINATESPACE}
\ \\

\begin{itemize}
	\item We denote the level set $\lbrace \eik = \eik' \rbrace$ by $\nullhyparg{\eik'}$, and we refer to these
		level sets as \emph{the characteristics}; see Fig.\,\ref{F:3DFRAME}.
	\item We define $\Sigma_{t'}$ to be the hypersurface
	$\lbrace t = t'\rbrace$, where $t$ is the Minkowski time function.
	\item We define the two-dimensional surfaces $\ell_{t',\eik'} := \Sigma_{t'} \cap \nullhyparg{\eik'}$; 
		see Fig.\,\ref{F:3DFRAME}.
\end{itemize}
\end{definition}

The surfaces $\nullhyparg{\eik}$ are characteristic for the wave operator
$\square_{\hfour}$ on LHS~\eqref{E:WAVE}, or equivalently,
null with respect to the acoustical metric $\hfour$ ($\hfour$-null for short).

As we mentioned in Sect.\,\ref{S:1DMAXIMALDEVELOPMENT} and just above Def.\,\ref{D:GENERALEIKONALFUNCTION},
a key strategy behind the proof of shock formation is to show that the solution remains
smooth with respect to the geometric coordinates. The blowup of the solution's gradient in $(t,x^1,x^2,x^3)$-coordinates,
i.e., the blowup of $\pmb{\partial} \wavearray$, occurs when $\upmu$ vanishes.  
One can prove (see e.g.\ \cite{jLjS2021}*{Lemmas~2.22 and 2.23})
the following relationship between the partial derivatives in 
the two coordinate systems, where the $a_{\alpha}^{\beta}$ are smooth (solution-dependent) functions of the geometric coordinates:
\begin{align} \label{E:SCHEMATICALLYDEGENERACYBETWEENCOORDINATES}
	\partial_{\alpha}
	& = 
			a_{\alpha}^0
			\frac{1}{\upmu} \geop{\eik} 
			+
			a_{\alpha}^1
			\geop{t}
			+
			a_{\alpha}^2
			\geop{x^2}
			+
			a_{\alpha}^3
			\geop{x^3}.
\end{align}
The first product on RHS~\eqref{E:SCHEMATICALLYDEGENERACYBETWEENCOORDINATES} shows in particular why
$\pmb{\partial} \wavearray$ can blow up as $\upmu \downarrow 0$, even when the derivatives of $\wavearray$ 
with respect to elements of $\lbrace \geop{t}, \geop{\eik}, \geop{x^2}, \geop{x^3} \rbrace$
remain bounded. The use of geometric coordinates allows one to study the problem of shock formation 
by using ideas that have traditionally been used in long-time existence problems.
However, different from the $1D$ case, in $3D$, some difficult degeneracies can occur
in the energy high-order energy estimates; we discuss this difficulty in more detail in
Sect.\,\ref{SSS:WAVEENERGYESTIMATEHIERARCHY}.

\subsubsection{Vectorfield frames}
\label{SSS:VECTORFIELDFRAMES}
Experience has shown (e.g., \cite{dC2007} and the related works cited in Sect.\,\ref{SS:CONSTRUCTIVEMULTIDIMENSIONAL}) 
that in the study of shocks,
to avoid loss of derivatives in commutator estimates and other difficulties,
it is advantageous to use vectorfield frames adapted to the eikonal function.
In particular, once one has constructed $\eik$,
one can construct the following vectorfield frames, depicted in Fig.\,\ref{F:3DFRAME}:
\begin{align} \label{E:VECTORFIELDFRAME}
	\Fullset
	& :=
	\lbrace
		\Lunit, \muX, \Yvf{2}, \Yvf{3}
	\rbrace,
	&
	\Tanset
	& :=
	\lbrace
		\Lunit, \Yvf{2}, \Yvf{3}
	\rbrace,
	&
	\Angularset
	& := 
	\lbrace
		\Yvf{2}, \Yvf{3}
	\rbrace,
\end{align}
which we describe below.
The frame $\Fullset$ spans the tangent space of geometric coordinate space,
the subset $\Tanset$ spans the tangent space of the characteristics $\nullhyparg{\eik}$,
and the subset $\Angularset$ spans the tangent space of the two-dimensional surfaces $\ell_{t,\eik}$,
i.e., the span of $\Angularset$ is equal to that of $\lbrace \geop{x^2}, \geop{x^3} \rbrace$.

\begin{center}
\begin{overpic}[scale=.25,grid = false, tics=5, trim=0 0 0 -1cm, clip]{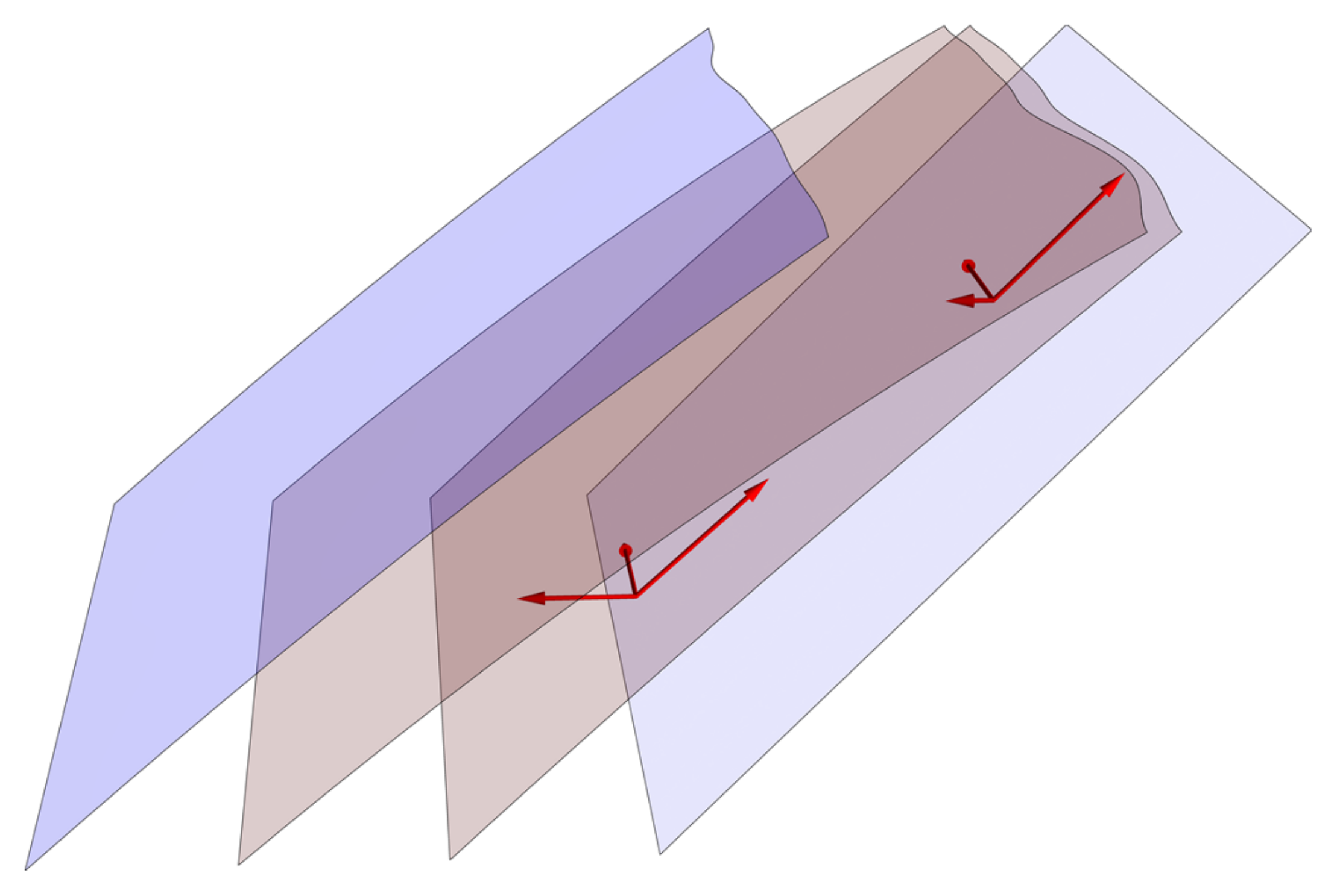} 
\put (82.7,49.7) {\large$\displaystyle \Lunit$}
\put (72,40.2) {\large$\displaystyle \muX$}
\put (70.5,48) {\large$\displaystyle \Yvf{2}$}
\put (54,25) {\large$\displaystyle \Lunit$}
\put (38.5,17) {\large$\displaystyle \muX$}
\put (45,27) {\large$\displaystyle \Yvf{2}$}
\put (49.5,10) {\large$\nullhyparg{-\rightu}$}
\put (35,10) {\large$\nullhyparg{\eik}$}
\put (5,10) {\large$\nullhyparg{\leftu}$}
\put (22,26) {\large$\displaystyle \upmu \approx 1$}
\put (66,67) {\large$\displaystyle \upmu \ \mbox{\upshape is small}$}
\put (55,60.5) {\large$\displaystyle \ell_{t,\leftu}$}
\put (-4,15) {\large$\displaystyle \ell_{0,\leftu}$}
\end{overpic}
\captionof{figure}{The frame $\Fullset$ at two distinct points on $\nullhyparg{\eik}$
with the $x^3$-direction suppressed}
\label{F:3DFRAME}
\end{center}

We now briefly describe how to construct the vectorfields. First, $\Lunit$ is defined by
$\Lunit^{\alpha} := - \upmu (\hfour^{-1})^{\alpha \kappa}(\wavearray) \partial_{\kappa} \eik$.
$\Lunit$ is clearly $\hfour$-orthogonal to the characteristics $\nullhyparg{\eik}$, 
and by \eqref{E:EIKONALEQUATION} and \eqref{E:INVERSEFOLIATIONDENSITY},
we see that $\hfour(\Lunit,\Lunit) = 0$ and $\Lunit t = 1$.
In particular, since $\Lunit$ is $\hfour$-orthogonal to $\nullhyparg{\eik}$ and $\hfour$-null,
$\nullhyparg{\eik}$ is a $\hfour$-null hypersurface.
One can show that for isentropic plane-symmetric solutions, the 
quantity ``$\Lunit$'' defined in this fashion
coincides with the one defined in \eqref{E:1DLUNIT}.
However, for general solutions in $3D$, a big difference with the plane-symmetric case is that
there does not exist any explicit formula for $\Lunit$ in the spirit of \eqref{E:1DLUNIT}. That is, in general,
$\Lunit$ depends on the gradient of $\eik$, which in turns depends implicitly on the fluid solution
via the coefficients in the eikonal equation \eqref{E:EIKONALEQUATION}.

Next, we define $X$ to be the $\Sigma_t$-tangent vectorfield that is 
$\hfour$-orthogonal to the $\ell_{t,\eik}$, normalized by $\hfour(X,X) = 1$,
and is such that $\eik$ increases along the integral curves of $X$, i.e., $X \eik > 0$. We then define
$\muX : = \upmu X$, where $\upmu$ is defined in \eqref{E:INVERSEFOLIATIONDENSITY}. 
It follows that $\hfour(\muX,\muX) =  \upmu^2$. With the help of
\eqref{E:GINVERSE00ISMINUSONE}, one can show that $\muX \eik = 1$.
That is, $\muX$ is a geometric replacement\footnote{\eqref{E:LANDMUXCOMMUTEINPLANESYMMETRY}
shows that in plane-symmetry, $\muX = \geop{\eik}$. However,
in the general $3D$ case, $\muX$ is equal to $\geop{\eik}$ plus corrections belonging to 
$\mbox{\upshape span} \left\lbrace \geop{x^2}, \geop{x^3} \right\rbrace$. \label{FN:MUXISNOTDDUIN3D}} 
for $\geop{\eik}$, and among the elements of
$\Fullset$, it is the one vectorfield that is transversal to the characteristics.
One can show that for isentropic plane-symmetric solutions, the 
quantities ``$X$'' and ``$\muX$'' defined in this fashion
coincide with the ones defined in \eqref{E:EXPLICITFORMOFXIN1D}.

Next, for $A=2,3$, we define $\Yvf{A}$ to be the $\hfour$-orthogonal projection of
$\partial_A$ onto $\ell_{t,\eik}$, where $\partial_A$ is the standard Minkowski-rectangular 
partial derivative vectorfield. In our study of isentropic plane-symmetric solutions,
we did not rely on any analogs of the $\Yvf{A}$.

\subsubsection{The first fundamental form of $\ell_{t,\eik}$ and related decompositions}
\label{SSS:FIRSTFUNDFORMOFTORIANDRELATEDDECOMPOSITIONS}
Standard calculations (see, e.g., \cite{jSgHjLwW2016}*{Equation~(2.40b)})
imply that the inverse acoustical metric from \eqref{E:INVERSEACOUSTICALMETRIC} can be decomposed as follows:
\begin{align} \label{E:INVERSEACOUSTICALMETRICDECOMPOSEDINTOLPARTXPARTELLTUPART}
	(\hfour^{-1})^{\alpha \beta}
	& = 
		- 
		\Lunit^{\alpha} \Lunit^{\beta}
		-
		\Lunit^{\alpha} X^{\beta}
		-
		X^{\alpha} \Lunit^{\beta}
		+
		(\gtorus^{-1})^{\alpha \beta},
\end{align}
where $\gtorus$ is the Riemannian metric on $\ell_{t,\eik}$ induced by $\hfour$, i.e., the first fundamental form of $\ell_{t,\eik}$.
We view $\gtorus$ to be a spacetime tensor that vanishes on contraction with $\Lunit$ or $X$ and that agrees with $\hfour$ on 
$\ell_{t,\eik}$-tangent vectors. Moreover, in \eqref{E:INVERSEACOUSTICALMETRICDECOMPOSEDINTOLPARTXPARTELLTUPART},
$(\gtorus^{-1})^{\alpha \beta} := (\hfour^{-1})^{\alpha \gamma} (\hfour^{-1})^{\beta \delta} \gtorus_{\gamma \delta}$
is the corresponding inverse first fundamental form.

In what follows, $\angD$ denotes the Levi-Civita connection of $\gtorus$ and $|\cdot|_{\gtorus}$ denotes the pointwise norms of tensors
with respect to $\gtorus$, e.g., for scalar functions $f$, we have 
$|\angD f|_{\gtorus}^2 = (\gtorus^{-1})^{\alpha \beta} \angDarg{\alpha} f \angDarg{\beta} f$.

For future use, we note that for scalar functions $f$, we have the following decomposition 
of $\upmu \square_{\hfour(\wavearray)} f$ (see, e.g., \cite{jSgHjLwW2016}*{Proposition~2.16}),
where $\angLap$ denotes the covariant Laplacian associated to $\gtorus$:
\begin{align} \label{E:WAVEOPERATORDECOMPLOUTSIDE} 
\upmu \square_{\hfour(\wavearray)} f 
	& = 
	- 
	\Lunit (\upmu \Lunit f + 2 \muX f) 
	+ 
	\upmu \angLap f 
	+ 
	\cdots.
\end{align}
On RHS~\eqref{E:WAVEOPERATORDECOMPLOUTSIDE}, 
$\cdots$ denotes terms that depend on the first derivatives of $f$
and that we will not discuss in detail here.

\begin{remark}[Comments on the case of more general ambient spacetimes $(\spacetimemanifold,\gfour)$] \label{R:SECONDCOMMENTONMOREGENERALSPACETIMES}
Under the setup described in Sect.\,\ref{SS:FIRSTCOMMENTSONMOREGENERALSPACETIMES}, 
almost all the constructions from this section have natural analogs 
for the study of the relativistic Euler equations
on a $\gfour$-globally hyperbolic spacetime $(\spacetimemanifold,\gfour)$. 
That is, consider the time function $t \colon \spacetimemanifold \to \R$ and the foliations $\spacetimemanifold = \cup_t \Sigma_t$. 
Here are three concrete examples: 
\renewcommand{\theenumi}{\textbf{\Roman{enumi})}}
\renewcommand{\labelenumi}{\theenumi}
\begin{enumerate}
	\itemsep1pt 
	\item The definition of the eikonal function \eqref{E:ACOUSTICALMETRIC} would remain unchanged
	\item With $\Sigmatnormal := - \hfour^{-1} \cdot \mathrm{d} t$, 
	\eqref{E:INVERSEFOLIATIONDENSITY} would again yield a suitable definition of the 
	inverse foliation density (of the characteristics, with respect to $\Sigma_t$).
	\item $X$ could be defined to be a $\Sigma_t$-tangent vectorfield that is $\hfour$-orthogonal to 
	$\ell_{t,\eik}$ and normalized by $\hfour(X,X) = 1$ (this would define $X$ up to an overall minus sign, 
	where in Sect.\,\ref{S:1DMAXIMALDEVELOPMENT}, we defined $X$ such that is is ``left-pointing'').
\end{enumerate}
We do, however, highlight the following point: general spacetime manifolds have no symmetries and thus
do not generally allow for the existence of the simple plane-symmetric solutions constructed in Sect.\,\ref{S:1DMAXIMALDEVELOPMENT}. 
Nonetheless, we expect the robust geometric framework described above and many of the key ideas presented in the subsequent sections to be useful for proving shock formation results for initial data with large gradients. 
We also note that the initial data \eqref{E:EIKONALDATAAWAYFROMSYMMETRY} 
for the eikonal function generally needs to be
adjusted to fit the regime one is studying.
\end{remark}

\subsection{Description of the compactly supported initial data of perturbations of simple isentropic plane-symmetric solutions}
\label{SS:3DINITIALDATA}
We now describe the initial data in Conjecture~\ref{CON:SHOCKWITHVORTICITYANDENTROPY} in more detail.
One simply takes any of the ``background'' simple isentropic plane-symmetric shock-forming solutions from 
Theorem~\ref{T:MAINTHEOREM1DSINGULARBOUNDARYANDCREASE}
and then considers a general small (asymmetric) perturbation of its initial data on $\Sigma_0$.
For the proof to close, the perturbed initial data need to belong to a sufficiently high order Sobolev space;
we will discuss the issue of regularity in more detail later on.
For convenience, it is easiest to consider perturbed initial data that, like the background data, are
compactly supported in $\Sigma_0^{[-\rightu,\leftu]} := \Sigma_0 \cap \lbrace (t,\eik,x^2,x^3) \ | \ -\rightu \leq \eik \leq \leftu \rbrace$.
Then by finite speed of propagation, the corresponding perturbed solution remains spatially compactly supported for all time
and in particular vanishes along the characteristic $\nullhyparg{-\rightu}$
(which therefore appears flat in Fig.\,\ref{F:CONJECTUREDMAXDEVELOPMENTINMINKRECT}).
The compact support allows one to avoid difficult boundary terms in the elliptic estimates for the vorticity and entropy; see
Sect.\,\ref{SSS:TOPORDERELLIPTICVORTICITYANDENTROPY}.

For such initial data, all of the wave variables in the array $\wavearray$ defined in \eqref{E:WAVEARRAYWITHALMOSTRIEMANNINVARIANTS}
-- except for $\RRiemann$ -- will initially have a negligible effect on the dynamics. 
Similarly, the fluid variables $\vort^{\alpha}$, $\GradEnt_{\alpha}$, 
$\modVortVort^{\alpha}$, and $\modDivGradEnt$ from Sect.\,\ref{SS:ADDITIONALFLUIDVARIABLES}
will also be small initially. The challenge is to propagate suitable versions of this smallness 
all the way up to the shock.

\subsection{Bootstrap assumptions and the regions $\mathcal{M}_{[0,t'),[-\rightu,\eik']}$}
\label{SS:BOOTSTRAP}
To prove that a shock forms, one commutes the equations up to $\Ntop$ times
and derives energy estimates up to top-order and $L^{\infty}$ estimates at the lower orders.
For reasons described later on, the proof requires $\Ntop$ to be rather large 
(i.e., substantially larger than a proof of local well-posedness requires)
and thus the initial data need to belong to a high order Sobolev space.
To close the proof, it is convenient to make $L^{\infty}$ bootstrap assumptions 
that capture the expectation that at the lower derivative levels, 
the solution behaves like a perturbation of a simple
isentropic plane-symmetric solution (in which only $\RRiemann$ is non-vanishing).
That is, one aims to propagate various aspects of the smallness enjoyed by the
initial data described in Sect.\,\ref{SS:3DINITIALDATA}.
The bootstrap assumptions are convenient because explicit formulae in the spirit of Cor.\,\ref{C:PSEXPLICITEXPRESSIONSFORSOLUTION}
are not available in $3D$. They capture the expectation that nonetheless, the $3D$ solution should obey
estimates that are similar to the ones
proved in Lemma~\ref{L:SHARPESTIMATESFORMUNEARCREASE} and
Theorem~\ref{T:MAINTHEOREM1DSINGULARBOUNDARYANDCREASE} in the simple isentropic plane-symmetric case.

Before introducing the bootstrap assumptions, we first introduce some notation.
In what follows, $\Tanset^N$ denotes an arbitrary
order $N$ string of elements of the $\nullhyparg{\eik}$-tangent vectorfields $\Tanset$ defined in \eqref{E:VECTORFIELDFRAME}.
We refer to such vectorfields as ``tangential'' (to the characteristics $\nullhyparg{\eik}$).
In contrast, the vectorfield $\muX$ from \eqref{E:VECTORFIELDFRAME} is transversal
(because it satisfies $\muX \eik = 1$). 

The bootstrap assumptions are made on a region
$\mathcal{M}_{[0,\Tboot),[-\rightu,\compactsupportu]}$
for some $\Tboot > 0$ and $\compactsupportu > \leftu$
(discussed further below, where $\rightu$ and $\leftu$ are as in Sect.\,\ref{SS:3DINITIALDATA}),
where in the rest of Sect.\,\ref{S:SHOCKFORMATIONAWAYFROMSYMMETRY}, 
given any $t' > 0$ and $\eik' \in [-\rightu,\leftu]$,
$\mathcal{M}_{[0,t'),[-\rightu,\eik']}$ denotes the following open-at-the-top
slab in geometric coordinate space:
\begin{align} \label{E:3DSPACETIMEREGIONGEOMETRICCOORDINATESPACE}
	\mathcal{M}_{[0,t'),[-\rightu,\eik']}
	& := 
		\lbrace 
		(t,\eik,x^2,x^3)
		\ | \
		t \in [0,t') \rbrace 
		\cap 
		\lbrace 
		(t,\eik,x^2,x^3)
			\ | \
		- \rightu \leq \eik \leq \eik' 
		\rbrace.
\end{align}
The proof will show that $\upmu$ first vanishes at a time that is a perturbation of
the plane-symmetric blowup-time $\blowuptime$ from Theorem~\ref{T:MAINTHEOREM1DSINGULARBOUNDARYANDCREASE}.
That is, in all estimates and arguments, one can safely assume that $0 < \Tboot \leq 2 \blowuptime$.
From this assumption, the assumption that the data are supported in $\Sigma_0^{[-\rightu,\leftu]}$,
and finite speed of propagation, one can compute/choose $\compactsupportu > \leftu$
such that for $t \in [0,\Tboot)$, the solution is trivial\footnote{For general perturbations of the background solutions,
the perturbed solution will not be spatially supported in $\Sigma_t \cap \lbrace \eik \in [-\rightu, \leftu] \rbrace$; this is
why we have introduced the parameter $\compactsupportu$. \label{FN:SUPPORTCANGROW}} 
at points in $\Sigma_t$ 
belonging to the complement of $\Sigma_t \cap \lbrace \eik \in [-\rightu, \compactsupportu] \rbrace$.

Specifically, for $N$ up to mid-order, i.e., roughly, for $N \leq \frac{\Ntop}{2}$, 
one assumes $L^{\infty}$ \emph{smallness} on $\mathcal{M}_{[0,\Tboot),[-\rightu,\compactsupportu]}$ for the following quantities,
which, aside from $\Lunit \upmu$, vanish in the case of simple isentropic plane-symmetric solutions:
\begin{itemize}
	\item The pure $\nullhyparg{u}$-tangential derivatives 
		$\Tanset^N \RRiemann$, $\Tanset^N \vort^{\alpha}$, $\Tanset^N \GradEnt_{\alpha}$, $\Tanset^N \modVortVort^{\alpha}$, 
		$\Tanset^N \modDivGradEnt$,
		$\Tanset^N \Lunit^i$, and $\Tanset^N \upmu$, except $\Lunit \upmu$ does not have to be small 
			(cf.\ \eqref{E:QUANTITATIVENEGATIVITYOFSPEEDTIMESLMU}).
	\item Mixed tangential-transversal derivatives of all these quantities except $\upmu$,
		e.g., $\Lunit \muX \RRiemann$.
	\item Pure transversal derivatives of all fluid variables excluding $\RRiemann$, e.g., 
		for $\muX \LRiemann$, $\muX v^2$, $\muX \vort^{\alpha}$, $\muX \modVortVort^{\alpha}$.
\end{itemize}
One also assumes $L^{\infty}$ \emph{boundedness} (not smallness) on $\mathcal{M}_{[0,\Tboot),[-\rightu,\compactsupportu]}$ 
for the following quantities, which are non-zero for the simple isentropic plane-symmetric solutions
we studied in Theorem~\ref{T:MAINTHEOREM1DSINGULARBOUNDARYANDCREASE}:
\begin{itemize}
	\item $\muX \RRiemann$
	\item $\muX \Lunit^i$, $\Lunit \upmu$.
\end{itemize}
One also assumes that:
\begin{align} \label{E:BANOSHOCK}
	\upmu & > 0, 
	&
	& \mbox{on } \mathcal{M}_{[0,\Tboot),[-\rightu,\compactsupportu]}.
\end{align}
The assumption \eqref{E:BANOSHOCK} captures that no shock has occurred on $\mathcal{M}_{[0,\Tboot),[-\rightu,\compactsupportu]}$, though
it leaves open the possibility that the shock happens exactly at time $\Tboot$.

The difficult part of the proof is proving energy estimates that allow one to derive, through a combination of Sobolev embedding
and transport estimates, improvements of the bootstrap assumptions. 
Then through a standard continuation argument, one can extend the solution
all the way up until the first time that $\upmu$ vanishes.

\subsection{The behavior of $\upmu$ and the formation of a shock}
\label{SS:BEHAVOROFMUANDFORMATIONOFSHOCK}
Given the framework established above and the bootstrap assumptions,
it is not difficult to prove that shocks can form for open sets of initial data that
are perturbations of simple isentropic plane-symmetric data. 
We now sketch the argument.
Using the eikonal equation \eqref{E:EIKONALEQUATION}
and the bootstrap assumptions of Sect.\,\ref{SS:BOOTSTRAP}
one can show that for perturbations of simple isentropic plane-symmetric solutions,
the inverse foliation density defined in \eqref{E:INVERSEFOLIATIONDENSITY}
satisfies an evolution equation of the following schematic form
(cf.\ Lemma~\ref{L:1DPSTRANSPORTFORMU}) on $\mathcal{M}_{[0,\Tboot),[-\rightu,\compactsupportu]}$:
\begin{align} \label{E:MUSCHEMATICEVOLUTION}
	\Lunit \upmu(t,\eik,x^2,x^3)
	& = \muX \RRiemann(t,\eik,x^2,x^3)
			+
			\cdots,
\end{align}
where $\cdots$ denotes small error terms.
In addition, using that $\Lunit \muX \RRiemann$ is small by the bootstrap assumptions,
one can show that on $\mathcal{M}_{[0,\Tboot),[-\rightu,\compactsupportu]}$, we have:
\begin{align} \label{E:RRIEMANNSTAYSCLOSETODATA}
	\muX \RRiemann(t,\eik,x^2,x^3)
	& = \muX \RRiemann(0,\eik,x^2,x^3)
			+
			\cdots,
\end{align}
that is, $\muX \RRiemann$ stays close to its initial condition
(cf.\ \eqref{E:SOLUIONRRIEMANNTRIVIALTRANSPORTEQUATIONINGEOMETRICCOORDINATES}, which shows that in simple isentropic plane-symmetry
$\RRiemann$ depends only on $\eik$).
Combining \eqref{E:MUSCHEMATICEVOLUTION}--\eqref{E:RRIEMANNSTAYSCLOSETODATA}, 
we see that on $\mathcal{M}_{[0,\Tboot),[-\rightu,\compactsupportu]}$, we have:
\begin{align} \label{E:SECONDMUSCHEMATICEVOLUTION}
	\Lunit \upmu(t,\eik,x^2,x^3)
	& = \muX \RRiemann(0,\eik,x^2,x^3)
			+
			\cdots.
\end{align}
Since $\Lunit t = 1$ (i.e., $\Lunit = \frac{\mathrm{d}}{\mathrm{d} t}$ along the integral curves of $\Lunit$),
and since the $\ell_{t,\eik}$-tangential derivative $\geop{x^A} \upmu$ is small by the bootstrap assumptions,
we deduce from \eqref{E:SECONDMUSCHEMATICEVOLUTION} and the fundamental theorem of calculus that
on $\mathcal{M}_{[0,\Tboot),[-\rightu,\compactsupportu]}$, we have:
\begin{align} \label{E:MUSCHEMATICINTEGRATED}
	\upmu(t,\eik,x^2,x^3)
	& = \upmu(0,\eik,x^2,x^3)
			+
			t 
			\left
			\lbrace
				\muX \RRiemann(0,\eik,x^2,x^3)
			+
			\cdots
			\right\rbrace.
\end{align}
Thus, for initial data such that there are points where
$\muX \RRiemann(0,\eik,x^2,x^3)$ is negative and large enough to dominate the terms $\cdots$
on RHS~\eqref{E:MUSCHEMATICINTEGRATED}, we infer from \eqref{E:MUSCHEMATICINTEGRATED} that 
$\min_{\Sigma_t} \upmu$ will vanish in finite time, and that the time of first vanishing can be controlled
by the initial data.
Moreover, since this argument implies that $|\muX \RRiemann| \neq 0$ at the points where $\upmu$ vanishes,
and since:
\begin{align} \label{E:SIMPLEXRRIEMANNRELATION}
	|X \RRiemann|
	& = \frac{1}{\upmu} |\muX \RRiemann|,
\end{align}
it follows that $|X \RRiemann|$ blows up like $\frac{C}{\upmu}$ at the points where $\upmu$ vanishes
(cf.\ \ref{E:BLOWUPOFCARTESIANCOORDINATES} and \eqref{E:SCHEMATICALLYDEGENERACYBETWEENCOORDINATES}). 
In total, these arguments show why $\upmu$ can vanish in finite time and reveal how the vanishing is connected to
the blowup of $|X \RRiemann|$.

We also note that the above arguments yield the following:
\begin{align} \label{E:NEGATIVITYOFLUNITMU}
	\Lunit \upmu(t,\eik,x^2,x^3) < 0,
	&& \mbox{at points where } \upmu(t,\eik,x^2,x^3) \mbox{ is close to } 0,
\end{align}
which is an analog of the estimate \eqref{E:QUANTITATIVENEGATIVITYOFSPEEDTIMESLMU} in simple
isentropic plane-symmetry. In particular, at fixed $(\eik,x^2,x^3)$,
$\upmu$ vanishes linearly in $t$. These facts have important implications for the energy estimates, described below;
see in particular Remark~\ref{R:SPACETIMEINTEGRAL}.

\subsection{Vectorfield multiplier method and energies}
\label{SS:MULTIPLIERMETHOD}
In this section, we describe some basic ingredients that are needed to set up the energy estimates.
To control the solution up to the shock, one needs geometric machinery that is much more advanced than
that of Sect.\,\ref{SSS:ENERGYCURRENTSANDIDENTITY}
and that takes advantage of the structure of the equations provided by Theorem~\ref{T:GEOMETRICWAVETRANSPORTDIVCURLFORMULATION}.

\subsubsection{Energies and null fluxes for wave equations}
\label{SSS:ENERGIESFORWAVEEQUATIONS}
In this section, we set up the vectorfield multiplier method for the wave equations \eqref{E:WAVE}. 

Given a scalar function $f$, we define the energy-momentum tensor associated to it to be the
following symmetric type $\binom{0}{2}$ tensor:
\begin{align} \label{E:DEFOFENERGYMOMENTUM} 
\enmomem_{\alpha \beta} 
&	
= \enmomem_{\alpha \beta}[f] 
:=
\Dfour_\alpha f
\Dfour_\beta f 
-
\frac{1}{2} 
\hfour_{\alpha \beta} 
(\hfour^{-1})^{\kappa \lambda}
\Dfour_{\kappa} f
\Dfour_{\lambda} f.
\end{align}
In \eqref{E:DEFOFENERGYMOMENTUM} and throughout, $\Dfour$ denotes the Levi-Civita connection of $\hfour$.

Straightforward calculations, based on the symmetry property 
$\Dfour_{\alpha} \Dfour_{\beta} f
=
\Dfour_{\beta} \Dfour_{\alpha} f
$
and the Leibniz rule,
yield that given any \emph{multiplier vectorfield} $Z$, 
we have the following identity:
\begin{align} \label{E:DIVERGENCEOFWAVEENERGYCURRENT}
(\hfour^{-1})^{\alpha \beta}
(\Dfour_\alpha \enmomem_{\beta \kappa}[f] Z^{\kappa})
& =  
(\square_{\hfour(\wavearray)} f) Z f 
+  
\frac{1}{2} 
(\hfour^{-1})^{\alpha \beta}
(\hfour^{-1})^{\kappa \lambda}
\enmomem_{\alpha \kappa} 
\deformarg{Z}{\beta}{\lambda},
\end{align}
where:
\begin{align} \label{E:DEFORMATIONTENSOR}
\deformarg{Z}{\alpha}{\beta} 
&
:= 
\hfour_{\beta \kappa} \Dfour_{\alpha} Z^{\kappa} 
+ 
\hfour_{\alpha \kappa} \Dfour_{\beta} Z^{\kappa}
\end{align}
is the deformation tensor of $Z$ (with respect to $\hfour$).
The identity \eqref{E:DIVERGENCEOFWAVEENERGYCURRENT} is convenient for bookkeeping in the divergence theorem
when one integrates by parts.

To derive energy-null flux identities for wave equations, we will use \eqref{E:DIVERGENCEOFWAVEENERGYCURRENT}
with the multiplier vectorfield $Z$ equal to $\multipliervectorfield$, which is defined as
follows, where $\Lunit$ and $\muX$ are as in Sect.\,\ref{SSS:VECTORFIELDFRAMES}:
\begin{align} \label{E:MULTIPLIERVECTORFIELD}
\multipliervectorfield 
& 
:= (1 + 2 \upmu) \Lunit 
+ 
2 \muX.
\end{align}
One can compute that $\hfour(\multipliervectorfield,\multipliervectorfield) = -4\upmu(1 + \upmu)$ 
and thus $\multipliervectorfield$ is $\hfour$-timelike whenever $\upmu > 0$. The $\hfour$-timelike
nature of $\multipliervectorfield$ is crucial for generating coercive energies.
The $\upmu$-weights in \eqref{E:MULTIPLIERVECTORFIELD} have been carefully placed.
Recall that $\muX = \upmu X$, where $\hfour(X,X) = 1$. Thus, $\muX$ effectively contains a $\upmu$-weight.

To close the $L^2$ estimates, one needs energies on: 
\begin{align} \label{E:CONSTANTPORTIONS}
	\Sigma_{t'}^{[-\rightu,\eik']}
	& := 
	\lbrace 
		(t,\eik,x^2,x^3)
		\ | \
		t = t' \rbrace 
		\cap 
		\lbrace 
		(t,\eik,x^2,x^3)
			\ | \
		- \rightu \leq \eik \leq \eik' 
		\rbrace
\end{align} 
and null fluxes on: 
\begin{align} \label{E:CHARACTERISTICPORTIONS}
	\nullhyptwoarg{\eik'}{[0,t')}
	& := 
	\lbrace 
		(t,\eik,x^2,x^3)
		\ | \
		0 \leq t < t' 
		\rbrace 
		\cap 
		\lbrace 
		(t,\eik,x^2,x^3)
			\ | \
			\eik = \eik' 
		\rbrace.
\end{align} 

Recall that $\Sigmatnormal$ is the future-directed $\hfour$-unit normal to $\Sigma_t$ defined in
\eqref{E:SIGMATNORMAL}. The strength of our energies on the $\Sigma_t^{[-\rightu,\eik]}$ is determined by the following 
identity (see, e.g., \cite{jSgHjLwW2016}*{Lemma~3.4}):
\begin{align} \label{E:SIGMATENERGYDENSITYIDENTITY}
\enmomem[f](\multipliervectorfield,\Sigmatnormal)
& 
= 
\frac{1}{2}(1 + 2 \upmu) \upmu (\Lunit f)^2
		+
		2 \upmu (\muX f) \Lunit f
		+
		2 (\muX f)^2
		+
		\frac{1}{2}(1 + 2 \upmu) \upmu |\angD f|_{\gtorus}^2.
\end{align}
Similarly, the strength of our null fluxes on the $\nullhyptwoarg{\eik}{[0,t)}$ is determined by the following identity
(again, see, e.g., \cite{jSgHjLwW2016}*{Lemma~3.4}):
\begin{align} \label{E:NULLHYPERSURFACEENERGYDENSITYIDENTITY}
\enmomem[f](\multipliervectorfield,\Lunit)
& :=
	(1 + \upmu)
	(\Lunit f)^2
	+
	\upmu
	|\angD f|_{\gtorus}^2.
\end{align}

These identities \eqref{E:SIGMATENERGYDENSITYIDENTITY}--\eqref{E:NULLHYPERSURFACEENERGYDENSITYIDENTITY}
motivate the following definition.

\begin{definition} [Wave equation energies and null-fluxes]
\label{D:WAVEENERGIES}
Given a scalar function $f$, we define the associated energy $\mathbb{E}_{(\textnormal{Wave})}[f]$
and null flux $\mathbb{E}_{(\textnormal{Wave})}[f]$ as follows:
\begin{subequations} 
\begin{align} \label{E:WAVEENERGYDEF} 
\begin{split}
&
\mathbb{E}_{(\textnormal{Wave})}[f](t,\eik)
	\\
& := 
\int_{\Sigma_t^{[-\rightu,\eik]}}  
	\left\lbrace
		\frac{1}{2}(1 + 2 \upmu) \upmu (\Lunit f)^2
		+
		2 \upmu (\muX f) \Lunit f
		+
		2 (\muX f)^2
		+
		\frac{1}{2}(1 + 2 \upmu) \upmu |\angD f|_{\gtorus}^2
	\right\rbrace
\, \voltorus \rmd \eik',
\end{split}
		\\
\mathbb{F}_{(\textnormal{Wave})}[f] (t,\eik) 
& 
:= \int_{\nullhyptwoarg{\eik}{[0,t)}} 
		\left
		\lbrace
			(1 + \upmu)
			(\Lunit f)^2
			+
			\upmu
			|\angD f|_{\gtorus}^2
		\right\rbrace
\, \voltorus \rmd t',
\label{E:WAVENULLFLUXDEF}
\end{align}
\end{subequations}
\end{definition}
where $\gtorus$ and $|\angD f|_{\gtorus}^2$ are as in Sect.\,\ref{SSS:FIRSTFUNDFORMOFTORIANDRELATEDDECOMPOSITIONS},
and on each $\ell_{t',\eik'}$,
\begin{align} \label{E:TORUSVOLUMEFORM}
	 \voltorus
	& := \sqrt{\mbox{\upshape det} \gtorus(t',\eik',x^2,x^3)} \, \rmd x^2 \rmd x^3,
\end{align}
is the canonical volume form induced by $\gtorus$.
In particular, the determinant in \eqref{E:TORUSVOLUMEFORM} is taken relative to the coordinates
$(x^2,x^3)$ on $\ell_{t',\eik'}$.

\begin{remark}[$\upmu$ weights and error terms]
\label{R:MUWEIGHTSINENERGIES}
Note that one $\Lunit$-derivative involving term on RHS~\eqref{E:WAVENULLFLUXDEF} \emph{lacks} a $\upmu$-weight. This means that
$\mathbb{F}_{(\textnormal{Wave})}[f] (t,\eik)$ is strongly coercive in the $\Lunit$-derivatives even in regions where
$\upmu$ is small. This is crucial for controlling $\Lunit$-derivative-involving energy estimate error terms that lack $\upmu$-weights.
In contrast, all terms on RHS~\eqref{E:WAVEENERGYDEF}--\eqref{E:WAVENULLFLUXDEF} involving 
the $\angD$-derivatives of $f$ have a $\upmu$-weight.
Hence, to control dangerous 
$\angD$-derivative-involving energy estimate error terms that lack $\upmu$-weights,
one must rely on the special structure afforded by the coercive spacetime integral described in
Remark~\ref{R:SPACETIMEINTEGRAL}.
\end{remark}

Using \eqref{E:DIVERGENCEOFWAVEENERGYCURRENT} with $Z := \multipliervectorfield$, 
one can prove the following lemma (see, e.g., \cite{jSgHjLwW2016}*{Proposition~3.5}), which forms
the starting point for the $L^2$-analysis of solutions to the wave equations \eqref{E:WAVE}.

\begin{lemma}[Energy-null flux identity for wave equations]
\label{L:ENERGYIDENTITYFORWAVES}
For scalar functions $f$ on $\mathcal{M}_{[0,\Tboot),[-\rightu,\compactsupportu]}$ 
that vanish along\footnote{This vanishing assumption holds under the compact support assumptions on the 
initial data stated in Sect.\,\ref{SS:3DINITIALDATA}. \label{FN:VANISHINGASSUMPTIONISSATISFIED}} 
$\nullhyparg{-\rightu}$, 
the following energy-null flux identity holds for $(t,\eik) \in [0,\Tboot) \times [-\rightu,\compactsupportu]$:
\begin{align}
\begin{split} \label{E:FUNDAMENTALENERGYINTEGRALIDENTITYWAVE}
\mathbb{E}_{(\textnormal{Wave})}[f](t,\eik) 
+ 
\mathbb{F}_{(\textnormal{Wave})}[f](t,\eik) 
&
= 
\mathbb{E}_{(\textnormal{Wave})}[f](0,\eik) 
	\\
& \ \
- 
\frac{1}{2}
\int_{\mathcal{M}_{[0,t),[-\rightu,\eik]}} 
	\upmu
	(\hfour^{-1})^{\alpha \beta}
	(\hfour^{-1})^{\kappa \lambda}
	\enmomem_{\alpha \kappa} 
	\deformarg{\multipliervectorfield}{\beta}{\lambda}
\, \voltorus \rmd \eik' \rmd t' 
	\\
& \ \
- 
\int_{\mathcal{M}_{[0,t),[-\rightu,\eik]}} 
	(\multipliervectorfield f)
	(\upmu \square_{\hfour(\wavearray)} f)
\, \voltorus \rmd \eik' \rmd t'.
\end{split}
\end{align}
\end{lemma}

\begin{remark}[The coercive spacetime integral {$\spacetimeintegralcontrolwave[f](t,\eik)$}] 
	\label{R:SPACETIMEINTEGRAL}
	Upon decomposing the $\enmomem$-involving integral on RHS~\eqref{E:FUNDAMENTALENERGYINTEGRALIDENTITYWAVE}
	relative to the vectorfields in $\Fullset$, 
	one finds that it contains the following term, which, after moving it to LHS~\eqref{E:FUNDAMENTALENERGYINTEGRALIDENTITYWAVE},
	takes the following form:
	\begin{align} \label{E:SPACETIMEINTEGRAL}
		\spacetimeintegralcontrolwave[f](t,\eik)
		& = 
			\frac{1}{2}
			\int_{\mathcal{M}_{[0,t),[-\rightu,\eik]}} 
				[\Lunit \upmu]_-
				|\angD f|_{\gtorus}^2
			\, \voltorus \rmd \eik' \rmd t'.
	\end{align}
	In \eqref{E:SPACETIMEINTEGRAL}, $[z]_- := \max \lbrace 0, - z \rbrace$ denotes the negative part of $z$.
	Importantly, the discussion in Sect.\,\ref{SS:BEHAVOROFMUANDFORMATIONOFSHOCK} shows that
	$\Lunit \upmu$ is negative in regions close to the shock (i.e., in regions where $\upmu$ is small)
	and thus in such regions, $\spacetimeintegralcontrolwave[f](t,\eik)$,
	which we again emphasize has the favorable sign \eqref{E:SPACETIMEINTEGRAL} 
	when moved to LHS~\eqref{E:FUNDAMENTALENERGYINTEGRALIDENTITYWAVE}, 
	yields \emph{non-$\upmu$-weighted spacetime $L^2$ control}
	of $|\angD f|_{\gtorus}^2$. As we mentioned above, 
	the control afforded by this term is crucial for controlling error terms in the energy estimates
	that depend on $\angD f$ but lack $\upmu$-weights.
\end{remark}

\subsubsection{Energies and null fluxes for transport equations}
\label{SSS:ENERGIESFORTRANSPORTEQUATIONS}
We now define energies and null fluxes for transport equations of the form $\Transport f = \cdots$,
where the vectorfield $\Transport$ is defined in \eqref{E:NORMALIZEDMATERIALDERIVATIVE}.
The energies and null fluxes are relevant for controlling solutions to the transport
equations from Theorem~\ref{T:GEOMETRICWAVETRANSPORTDIVCURLFORMULATION}.

Specifically, motivated by the identities 
\eqref{E:GINNERPRODUCTOFTRANSPORTANDSIGMATNORMALISMINUSONE}--\eqref{E:GINNERPRODUCTOFTRANSPORTANDLUNIT},
we define the following energies and null fluxes.

\begin{definition}[Transport equation energies and null-fluxes]
\label{D:TRANSPORTEQUATIONENERGIESANDNULLFLUXES}
\begin{subequations}
\begin{align}
 \mathbb{E}_{(\textnormal{Transport})}[f](t,\eik) 
& 
:= 
\int_{\Sigma_t^{\eik}} 
	\upmu f^2 
\, \voltorus \rmd \eik', \label{E:TRANSPORTENERGYDEF} 
	\\
 \mathbb{F}_{(\textnormal{Transport})}[f](t,\eik) 
	& 
	:= 	
	\int_{\nullhyptwoarg{\eik}{[0,t)}} 
		[- \normalizer \mink(\Transport,\Lunit)]
		f^2 
	\, \voltorus \rmd t'. 
\label{E:TRANSPORTNULLFLUXDEF}
 \end{align}
\end{subequations}
\end{definition}
Note that the discussion below \eqref{E:GINNERPRODUCTOFTRANSPORTANDLUNIT} implies that
the factor $- \normalizer \mink(\Transport,\Lunit)$
on RHS~\eqref{E:TRANSPORTNULLFLUXDEF} is positive, which of course is important
for the coercivity of $\mathbb{F}_{(\textnormal{Transport})}[f](t,\eik)$.
For the same reasons given in Sect.\,\ref{SSS:ENERGIESFORWAVEEQUATIONS}, 
it is crucially important that RHS~\eqref{E:TRANSPORTNULLFLUXDEF} lacks a $\upmu$-weight.

By applying the divergence theorem to the vectorfield
$f^2 \Transport$ on $\mathcal{M}_{[0,t),[-\rightu,\eik]}$,
using \eqref{E:GINNERPRODUCTOFTRANSPORTANDSIGMATNORMALISMINUSONE}--\eqref{E:GINNERPRODUCTOFTRANSPORTANDLUNIT},
and using that in geometric coordinates we have (see, e.g., \cite{jS2016b}*{Corollary~3.47}):
\begin{align} \label{E:CANONICALVOLUMEFORMINGEOMETRICCOORDINATES}
	\sqrt{|\mbox{\upshape det} \hfour|}
	\rmd x^2 \rmd x^3 \rmd \eik' \rmd t' 
	& =
	\upmu \voltorus \rmd \eik' \rmd t',
\end{align}
one can prove the following lemma, which forms
the starting point for the $L^2$-analysis of solutions to 
the transport equations \eqref{E:TRANSPORT}--\eqref{E:TRANSPORTFORMODIFIED}.

\begin{lemma}[Energy-null flux identity for transport equations]
\label{L:ENERGYIDENTITYFORTRANSPORT}
For scalar functions $f$ on $\mathcal{M}_{[0,\Tboot),[-\rightu,\compactsupportu]}$ 
that vanish along $\nullhyparg{-\rightu}$ (Footnote~\ref{FN:VANISHINGASSUMPTIONISSATISFIED} applies here as well), 
the following energy-null flux identity holds for $(t,\eik) \in [0,\Tboot) \times [-\rightu,\compactsupportu]$:
\begin{align}\label{E:TRANSPORTEQUATIONENERGYIDENTITY}
\begin{split}
\mathbb{E}_{(\textnormal{Transport})}[f](t,\eik)
+ 
\mathbb{F}_{(\textnormal{Transport})}[f](t,\eik) 
& 
= 
\mathbb{E}_{(\textnormal{Transport})}[f](0,\eik)	
	\\
 & \ \ 
+ \int_{\mathcal{M}_{[0,t),[-\rightu,\eik]}}  
		\upmu (\Dfour_{\alpha} \Transport^{\alpha}) f^2 
	\, \voltorus \rmd \eik' \rmd t'
		\\
& \ \
	+ 
	\int_{\mathcal{M}_{[0,t),[-\rightu,\eik]}}  
		2 f (\upmu \Transport f)
	\, \voltorus \rmd \eik' \rmd t'.
\end{split}
\end{align}

\end{lemma}

\subsubsection{Commutator method and the shock-driving terms}
\label{SSS:COMMUTATORMETHOD}
To derive suitable energy estimates,
one commutes $\upmu$-weighted\footnote{The $\upmu$-weights allow one to avoid problematic error terms,
and they are compatible with the $\upmu$-weights in Lemmas~\ref{L:ENERGYIDENTITYFORWAVES} and \ref{L:ENERGYIDENTITYFORTRANSPORT}.
 \label{FN:WHYMUWEIGHTS}} 
versions 
of the equations of Theorem~\ref{T:GEOMETRICWAVETRANSPORTDIVCURLFORMULATION} with
the elements of the $\nullhyparg{\eik}$-tangent subset $\Tanset$ defined in \eqref{E:VECTORFIELDFRAME}
and then applies the energy-null flux identities of Lemmas~\ref{L:ENERGYIDENTITYFORWAVES} and \ref{L:ENERGYIDENTITYFORTRANSPORT}
and the elliptic estimates described in Sect.\,\ref{SSS:TOPORDERELLIPTICVORTICITYANDENTROPY}.
In particular, it has been understood since \cite{jLjS2021} that to close the energy estimates in a shock formation problem,
one does \underline{not} need to commute the equations
with the vectorfield $\muX$, which is transversal to the characteristics.
We stress that it does not seem possible to close the estimates by commuting the equations with
the Minkowski-rectangular partial derivatives $\partial_{\alpha}$; 
such an approach would generate error terms that we do not know how to control.

Although one must commute with all the elements of $\Tanset$ to close the energy estimates, 
the most difficult terms arise from commuting the wave equations \eqref{E:WAVE}
with a string of elements of the $\ell_{t,\eik}$ subset $\Angularset$ defined in \eqref{E:VECTORFIELDFRAME}.
Schematically, for $\Psi \in \wavearray$, 
with $\Angularset^N$ denoting an order $N$ string of elements of $\Angularset$,
we have:
\begin{align} \label{E:SCHEMTICCOMMUTEDWAVE}
	\upmu \square_{\hfour(\wavearray)} \Angularset^N \Psi
	& = 
		(\muX \Psi) \cdot \Angularset^N \mytr_{\gtorus} \upchi
		+
		\upmu 
		\smoothfunction(\wavearray)
		\Angularset^N \modVortVort
		+
		\upmu 
		\smoothfunction(\wavearray)
		\Angularset^N \modDivGradEnt
		+
		\cdots,
\end{align}
where 
$\upchi$ is the null second fundamental form of $\nullhyparg{\eik}$,
$\mytr_{\gtorus} \upchi$ denotes its trace with respect to $\gtorus$, and
$\cdots$ denotes easier error terms.

The $\Angularset^N \modVortVort$- and $\Angularset^N \modDivGradEnt$-involving terms
on RHS~\eqref{E:SCHEMTICCOMMUTEDWAVE} clearly arise from the first term on RHS~\eqref{E:WAVE}, 
and we will discuss how to control them in Sect.\,\ref{SSS:TOPORDERELLIPTICVORTICITYANDENTROPY}.
We refer to the proofs of \cite{jSgHjLwW2016}*{Lemmas~2.18 and 4.2} 
for details on the origin of the commutator term $(\muX \Psi) \cdot \Angularset^N \mytr_{\gtorus} \upchi$,
which we describe how to control just below; this term would lead to the loss of a derivative at the top order
if handled in a naive fashion. 

Before proceeding, we first highlight that we have relegated all null-form-involving terms
on RHS~\eqref{E:WAVE} to the terms $\cdots$ on RHS~\eqref{E:SCHEMTICCOMMUTEDWAVE}. 
The reason is that one can show (see \cite{jLjS2018}*{Lemma~2.56})
that $\hfour$-null forms (see Def.\,\ref{D:NULLFORMS}) enjoy the following good decomposition,
expressed schematically, relative to the frame
$\Fullset$ \eqref{E:VECTORFIELDFRAME}, where $\Singletan$ schematically denotes elements of $\Tanset$:
\begin{align} \label{E:NULLFORMGOODSTRUCTURESCHEMATICA}
	\upmu \nullform(\pmb{\partial} \phi, \pmb{\partial} \psi)
	& = \upmu \Singletan \phi \cdot \Singletan \psi
			+
			\muX \phi \cdot \Singletan \psi
			+
			\muX \psi \cdot \Singletan \phi. 
\end{align}
The key point is that there are not any terms on RHS~\eqref{E:NULLFORMGOODSTRUCTURESCHEMATICA}
that are proportional to $\muX \psi \cdot \muX \phi$; signature considerations imply that
such terms, if present, would be multiplied by a dangerous factor of $\frac{1}{\upmu}$.
Such a factor would become singular as $\upmu \downarrow 0$, which would spoil the philosophy 
of proving that the solution remains quite regular in geometric coordinates, even as $\upmu$ vanishes.
A decomposition similar to \eqref{E:NULLFORMGOODSTRUCTURESCHEMATICA} also holds 
for the up-to-top-order derivatives of the $\upmu$-weighted $\hfour$-null forms.
We stress that the absence of terms proportional to $\muX \psi \cdot \muX \phi$
on RHS~\eqref{E:NULLFORMGOODSTRUCTURESCHEMATICA} is a statement about the
\emph{full nonlinear structure} of the $\hfour$-null forms from Def.\,\ref{D:NULLFORMS}.
In particular, this notion of null form is tied to the acoustical metric $\hfour$
and is stronger than the ``classic null condition'' established by Klainerman 
in \cite{sK1984}, which is tied to the Minkowski metric and is indifferent to the
structure of most cubic nonlinearities.

\begin{remark}[The shock-driving terms are hidden in $\square_{\hfour(\wavearray)} \Psi$]
	\label{R:SHOCKDRIVINGTERMSHIDDEN}
	The upshot is that all terms in \eqref{E:WAVE} that drive the formation of the shock
	are hidden in the covariant wave operator term $\square_{\hfour(\wavearray)} \Psi$ on the LHS.
	More precisely, the shock-driving terms are semilinear Riccati-type terms of type $X \Psi \cdot X \Psi$
	(cf.\ Sect.\,\ref{SS:1DRICCATIBLOWUP})
	that become visible if one first expands LHS~\eqref{E:WAVE} relative
	to the Minkowski-rectangular coordinates
	and then decomposes the semilinear terms with respect to the frame $\Fullset$ from \eqref{E:VECTORFIELDFRAME}.
\end{remark}

In view of the above discussion, it follows that the goal is to show that
all terms on RHS~\eqref{E:SCHEMTICCOMMUTEDWAVE} are error terms that
do not interfere with the shock formation processes. 
As we have alluded to earlier, this is indeed possible except at the high derivative levels,
where the difficult regularity theory of the eikonal function 
(in particular, the difficult regularity theory of $\upchi$)
leads to singular high-order estimates;
see Sect.\,\ref{SSS:WAVEENERGYESTIMATEHIERARCHY}.

We now explain how to control the first product on RHS~\eqref{E:SCHEMTICCOMMUTEDWAVE}.
One starts with the Raychaudhuri equation \cite{aR1955}, 
which plays an important role in mathematical General Relativity,
and which takes the following form in the present context:
\begin{align}  \label{E:RAYCHAUDHURITRANSPORTCHI}
\Lunit \mytr_{\gtorus} \upchi 
& = 
\frac{\Lunit \upmu}{\upmu}
\mytr_{\gtorus} \upchi 
- 
\Ricfour_{\Lunit \Lunit} 
- 
|\upchi|_{\gtorus}^2.
\end{align}
In \eqref{E:RAYCHAUDHURITRANSPORTCHI}, $\Ricfour$ is the Ricci curvature of 
the acoustical metric $\hfour$ and $\Ricfour_{\Lunit \Lunit} : = \Ricfour_{\alpha \beta} \Lunit^{\alpha} \Lunit^{\beta}$. 
The difficulty is that since $\hfour = \hfour(\wavearray)$,
$\Ricfour$ depends on the second derivatives of $\wavearray$. This suggests that
control of the term $\Angularset^N \mytr_{\gtorus} \upchi$ on
RHS~\eqref{E:SCHEMTICCOMMUTEDWAVE} requires control over $N+2$ derivatives of $\wavearray$,
which is one too many derivatives to be compatible with the regularity for $\wavearray$ 
afforded by energy estimates for the wave equation \eqref{E:SCHEMTICCOMMUTEDWAVE},
i.e, one has to worry about losing a derivative at the top-order.

Fortunately, based on ideas originating in \cites{dCsK1993,sKiR2003},
there is a way to avoid the loss of derivatives by combining the special structure of
$\Ricfour_{\Lunit \Lunit}$
with the special structures of the equations of Theorem~\ref{T:GEOMETRICWAVETRANSPORTDIVCURLFORMULATION}.
We will briefly explain the main ideas of the argument.
First, one decomposes $\Ricfour_{\Lunit \Lunit}$ to deduce an
algebraic identity that can schematically be expressed as:
\begin{align} \label{E:RICLLALGEBRAICDECOMPOSED}
	\upmu \Ricfour_{\Lunit \Lunit}
	& = \Lunit 	
			\left\lbrace
				\upmu
				\smoothfunction(\wavearray)
				\cdot
				\Lunit \wavearray
				+\upmu
				\smoothfunction(\wavearray)
				\cdot
				\angD \wavearray
			\right\rbrace
			+
			\upmu
			\smoothfunction(\wavearray)
			\cdot
			\angLap \wavearray
			+
			\cdots,
\end{align}
where $\smoothfunction$ schematically denotes a smooth function that is allowed to vary from line to line
and $\cdots$ denotes lower order terms that depend on the first derivatives of $\wavearray$.
Next, one uses \eqref{E:WAVEOPERATORDECOMPLOUTSIDE} and \eqref{E:MUSCHEMATICEVOLUTION} 
to replace the term
$\upmu
			\smoothfunction(\wavearray)
			\cdot
			\angLap \wavearray$ 
on RHS~\eqref{E:RICLLALGEBRAICDECOMPOSED} with
a perfect $\Lunit$-derivative term plus $\upmu \square_{\hfour(\wavearray)} \wavearray$ plus lower-order terms.
One then uses the relativistic Euler wave equation \eqref{E:WAVE} to replace 
$\upmu \square_{\hfour(\wavearray)} \wavearray$ with
the inhomogeneous terms on RHS~\eqref{E:WAVE}.
Combining with \eqref{E:RAYCHAUDHURITRANSPORTCHI} and bringing all perfect $\Lunit$-derivative terms to the left,
we deduce a transport equation that can schematically be expressed as:
\begin{align} \label{E:MODIFIEDTRCHISCHEMATIC}
	\Lunit 
	\left\lbrace
		\upmu
		\mytr_{\gtorus} \upchi 
		+
		\smoothfunction(\wavearray)
		\muX \wavearray
		+
		\upmu \smoothfunction(\wavearray)
		\Lunit \Psi
		+
		\upmu
		\smoothfunction(\wavearray)
		\angD \Psi
	\right\rbrace
	& =
		- 
		\upmu
		|\upchi|_{\gtorus}^2
		+
		\upmu 
		\smoothfunction(\wavearray)
		\modVortVort
		+
		\upmu 
		\smoothfunction(\wavearray)
		\modDivGradEnt
		+
		\cdots,
\end{align}
where $\cdots$ denotes easier error terms that involve 
$\vort$, $\GradEnt$, and the first derivatives of $\wavearray$.
In particular, \emph{RHS~\eqref{E:MODIFIEDTRCHISCHEMATIC} does not involve the second derivatives of 
$\wavearray$}, which is its main advantage compared to equation~\eqref{E:RAYCHAUDHURITRANSPORTCHI}.

Similar results hold for the commuted equations,
i.e., given any order $N$ string of $\ell_{t,\eik}$-tangent vectorfields $\Angularset^N$,
with $\fullymodquant{\Angularset^N}$ denoting the ``fully modified'' order $N$
version of $\mytr_{\gtorus} \upchi$ defined by:
\begin{align} \label{E:FULLYMODIFIEDORDERN}
	\fullymodquant{\Angularset^N}
	& :=
		\upmu
		\Angularset^N \mytr_{\gtorus} \upchi 
		+
		\Angularset^N
		\left\lbrace
			\smoothfunction(\wavearray)
			\muX \wavearray
			+
			\upmu \smoothfunction(\wavearray)
			\Lunit \Psi
			+
			\upmu
			\smoothfunction(\wavearray)
			\angD \Psi
		\right\rbrace,
\end{align}
one can derive a transport equation of the following schematic form
(see, e.g., the proof of \cite{jS2016b}*{Proposition~11.10} for the main ideas behind the proof):
\begin{align} \label{E:COMMUTEDMODIFIEDTRCHISCHEMATIC}
	\Lunit 
	\fullymodquant{\Angularset^N}
	& =
		- 
		\upmu
		\Angularset^N
		(
		|\upchi|_{\gtorus}^2)
		+
		\upmu 
		\smoothfunction(\wavearray)
		\Angularset^N \modVortVort
		+
		\upmu 
		\smoothfunction(\wavearray)
		\Angularset^N \modDivGradEnt
		+
		\cdots.
\end{align}

Equation \eqref{E:COMMUTEDMODIFIEDTRCHISCHEMATIC} is an order $N$ version of equation \eqref{E:MODIFIEDTRCHISCHEMATIC}
that allows one to control $\fullymodquant{\Angularset^N}$ without derivative loss,
assuming that one can adequately control RHS~\eqref{E:COMMUTEDMODIFIEDTRCHISCHEMATIC}.
In Sect.\,\ref{SSS:TOPORDERELLIPTICVORTICITYANDENTROPY}, we will explain how to control the
$\Angularset^N \modVortVort$- and $\Angularset^N \modDivGradEnt$-involving terms
on RHS~\eqref{E:COMMUTEDMODIFIEDTRCHISCHEMATIC}.
One remaining difficulty is that the term 
$
\upmu
\Angularset^N
		(|\upchi|_{\gtorus}^2)$
 on RHS~\eqref{E:COMMUTEDMODIFIEDTRCHISCHEMATIC}
involves the full tensor $\upchi$, rather than just the trace part featured on LHS~\eqref{E:COMMUTEDMODIFIEDTRCHISCHEMATIC}. 
However, one can control this term using
a strategy that goes back to \cites{dCsK1993,sKiR2003}. Specifically, one can use elliptic estimates
on the two-dimensional surfaces $\ell_{t,\eik}$, based on the Codazzi equations from geometry, 
to control the $\ell_{t,\eik}$-tangent derivatives of
$|\upchi|_{\gtorus}^2$ in terms of the $\ell_{t,\eik}$-tangent derivatives of
$\mytr_{\gtorus} \upchi$ plus error terms with an admissible amount of regularity;
similar results hold up to top-order, and we refer, for example, 
to the proof \cite{jS2016b}*{Lemma~20.20} for the main ideas behind the analysis.
In total, the approach described above allows one to control the product $(\muX \Psi) \cdot \Angularset^N \mytr_{\gtorus} \upchi$
on RHS~\eqref{E:SCHEMTICCOMMUTEDWAVE} in appropriate $L^2$-based Sobolev spaces. 
However, the argument introduces a difficult factor of $\frac{1}{\upmu}$ into the top-order estimates,
as we explain in Sect.\,\ref{SS:ENERGYESTIMATES}.

\subsection{Energy estimates}
\label{SS:ENERGYESTIMATES}
In this section, we will describe some features of the proof of the energy estimates,
which are the most technical and difficult aspect of studying shock formation.

\subsubsection{Transport variable energy estimates}
\label{SSS:TRANSPORTENERGYESTIMATEHIERARCHY}
We now explain how to derive energy estimates for 
$(\vort,\GradEnt)$, which solve the transport equation
\eqref{E:TRANSPORT}. These estimates are relatively easy to derive.
Let $\Tanset^N$ denote an arbitrary
order $N$ string of elements of the $\nullhyparg{\eik}$-tangent vectorfields $\Tanset$ defined in \eqref{E:VECTORFIELDFRAME}.
One multiplies equation \eqref{E:TRANSPORT} by $\upmu$, commutes with
$\Tanset^N$ for $N \leq \Ntop$, and uses the bootstrap assumptions and the transport equation energy
identity \eqref{E:TRANSPORTEQUATIONENERGYIDENTITY} to deduce
that for $(t,\eik) \in [0,\Tboot) \times [-\rightu,\compactsupportu]$, we have:
\begin{align} \label{E:TRANSPORTENERGYINEQUALITY}
\begin{split}
&
\mathbb{E}_{(\textnormal{Transport})}[\Tanset^{\leq N}(\vort,\GradEnt)](t,\eik)
+ 
\mathbb{F}_{(\textnormal{Transport})}[\Tanset^{\leq N}(\vort,\GradEnt)](t,\eik) 
	\\
& 
\lesssim
\mbox{\upshape data}
+ 
\int_{\mathcal{M}_{[0,t),[-\rightu,\eik]}}  
	\mbox{\upshape $\upmu$-regular $\wavearray$ terms}
\, \voltorus \rmd \eik' \rmd t'
+
\int_{-\rightu}^{\eik}
	\mathbb{F}_{(\textnormal{Transport})}[\Tanset^{\leq N}(\vort,\GradEnt)](t,\eik')
\rmd \eik'.
\end{split}
\end{align}
From \eqref{E:TRANSPORTENERGYINEQUALITY} and Gr\"{o}nwall's inequality in $\eik$, we deduce that:
\begin{align} \label{E:SECONDTRANSPORTENERGYINEQUALITY}
\begin{split}
&
\mathbb{E}_{(\textnormal{Transport})}[\Tanset^{\leq \Ntop}(\vort,\GradEnt)](t,\eik)
+ 
\mathbb{F}_{(\textnormal{Transport})}[\Tanset^{\leq \Ntop}(\vort,\GradEnt)](t,\eik) 
	\\
& 
\lesssim
\mbox{\upshape data}
\cdot
\int_{\mathcal{M}_{[0,t),[-\rightu,\eik]}}  
	\mbox{\upshape $\upmu$-regular $\wavearray$ terms}
\, \voltorus \rmd \eik' \rmd t',
\end{split}
\end{align}
where here and throughout, ``$\mbox{\upshape data}$'' denotes a small term that depends on the initial data and that measures the perturbation
of the data away from simple isentropic plane-symmetric data.
The terms ``$\mbox{\upshape $\upmu$-regular $\wavearray$ terms}$'' in \eqref{E:SECONDTRANSPORTENERGYINEQUALITY}
are easily controllable by the wave energies and null fluxes, and in particular, 
no singular factor of $\frac{1}{\upmu}$ is present in these terms.
Hence, the estimate \eqref{E:SECONDTRANSPORTENERGYINEQUALITY} shows that the behavior of 
$\Tanset^{\leq \Ntop}(\vort,\GradEnt)$ is effectively determined by the behavior of 
the wave variables $\wavearray$ (defined in \eqref{E:WAVEARRAYWITHALMOSTRIEMANNINVARIANTS}).
We clarify that even though there is no singular factor
of $\frac{1}{\upmu}$ present in the $\mbox{\upshape $\upmu$-regular $\wavearray$ terms}$, these terms
can still blow up at the high derivative levels as $\upmu \downarrow 0$, for reasons we 
describe in Sect.\,\ref{SSS:WAVEENERGYESTIMATEHIERARCHY}.
Hence, in view of the coupling between the transport and wave energies shown by
\eqref{E:SECONDTRANSPORTENERGYINEQUALITY}, we see that the blowup of the $\wavearray$-energies can cause
the blowup of the $(\vort,\GradEnt)$-energies.

\subsubsection{Elliptic-hyperbolic estimates for the vorticity and entropy}
\label{SSS:TOPORDERELLIPTICVORTICITYANDENTROPY}
Recall that the terms $\Angularset^N \modVortVort$ and $\Angularset^N \modDivGradEnt$
appear on RHSs~\eqref{E:SCHEMTICCOMMUTEDWAVE} and \eqref{E:COMMUTEDMODIFIEDTRCHISCHEMATIC}.
To close the energy estimates, we need to control $\Tanset^{\leq \Ntop} (\modVortVort,\modDivGradEnt)$ in $L^2$.
These terms would lead to a loss of derivatives at the top-order if handled in a naive fashion.
There are several reasons why one cannot control these terms 
by using only the same transport equation energy methods
we used to derive \eqref{E:SECONDTRANSPORTENERGYINEQUALITY}. 
One reason is that Def.\,\ref{D:MODIFIEDFLUIDVARIABLES} 
shows that from the point of view of regularity, we have
$(\modVortVort,\modDivGradEnt) \sim \pmb{\partial} (\vort,\GradEnt) + \cdots$;
since the inhomogeneous terms in the transport equation \eqref{E:TRANSPORT}
satisfied by $(\vort,\GradEnt)$ suggest (incorrectly, as it fortunately turns out) 
that $\pmb{\partial} (\vort,\GradEnt)$
have at best the same regularity as $\pmb{\partial}^2 \wavearray$,
this is formally inconsistent (from the point of view of regularity) 
with having $\pmb{\partial} (\vort,\GradEnt)$ as a source
term in the wave equations \eqref{E:WAVE} for $\wavearray$.
A second reason is that
the transport equations \eqref{E:TRANSPORTFORMODIFIED} satisfied by
$(\modVortVort,\modDivGradEnt)$ feature the inhomogeneous terms
$\nullform(\pmb{\partial} \vort,\pmb{\partial} \wavearray)$
and
$\nullform(\pmb{\partial} \GradEnt,\pmb{\partial} \wavearray)$,
which depend on the \underline{general} first derivatives
of $(\vort,\GradEnt)$, rather than the special combinations of first derivatives of
$(\vort,\GradEnt)$ present in the definitions
\eqref{E:MODIFIEDVORTVORT}--\eqref{E:MODIFIEDDIVGRADENT} 
of $(\modVortVort,\modDivGradEnt)$.
That is, due to these inhomogeneous terms,
\eqref{E:TRANSPORTFORMODIFIED} cannot be treated as a pure transport equation in
$(\modVortVort,\modDivGradEnt)$.

To overcome the difficulties highlighted in the previous paragraph, 
one can treat \eqref{E:TRANSPORT}--\eqref{E:DIVCURLFORMODIFIED} as a coupled transport-div-curl system that yields
$L^2$-control of 
$(\modVortVort,\modDivGradEnt)$
and
$(\pmb{\partial} \vort,\pmb{\partial} \GradEnt)$,
where $(\modVortVort,\modDivGradEnt)$ are controlled with hyperbolic transport energy estimates in the spirit of
\eqref{E:SECONDTRANSPORTENERGYINEQUALITY}, and the elliptic estimates are used only
to handle the inhomogeneous terms
$\nullform(\pmb{\partial} \vort,\pmb{\partial} \wavearray)$
and
$\nullform(\pmb{\partial} \GradEnt,\pmb{\partial} \wavearray)$
on RHS~\eqref{E:TRANSPORTFORMODIFIED}.
Similar results hold up to top-order and yield $L^2$ control over
$\pmb{\partial} \Tanset^{\leq \Ntop} (\vort,\GradEnt)$
and
$\Tanset^{\leq \Ntop} (\modVortVort,\modDivGradEnt)$.
One key difficulty in this argument is that to obtain elliptic estimates on the spacelike hypersurfaces
$\Sigma_t$, one needs to extract a \emph{spatial} div-curl subsystem from
\eqref{E:TRANSPORT}--\eqref{E:DIVCURLFORMODIFIED}; the difficulty is that
equations~\eqref{E:TRANSPORT}--\eqref{E:DIVCURLFORMODIFIED}
appear to involve \emph{spacetime} div-curl equations.
Nevertheless, by splitting various derivative operators into a $\Transport$-parallel part
and a $\Sigma_t$-tangent part, one can extract the desired
spatial div-curl subsystem; we refer to the proof of \cite{mDjS2019}*{Lemma~9.21} for details on this extraction
in the context of a local well-posedness argument.
We also refer to \cite{jLjS2021}*{Sections~11.2 and 11.3} for similar but simpler elliptic estimates
in the context of shock formation for the non-relativistic $3D$ compressible Euler equations.

The argument described in the previous paragraph applies when 
the initial data of the vorticity and entropy are compactly supported, which is the case for the initial data described in
Sect.\,\ref{SS:3DINITIALDATA}.
The key point is that under compact support,
one can avoid \emph{boundary terms} in the elliptic estimates (the boundary terms vanish thanks to the compact support). 
To treat solutions that are not (spatially) compactly supported,
one would need to handle the boundary terms that arise in the elliptic estimates. 
In the case of non-relativistic $3D$ compressible Euler equations,
based on the special structures found in the formulation of the flow derived in \cite{jS2019c},
localized spacetime ``elliptic-hyperbolic'' integral identities for the vorticity and entropy were derived 
in \cite{lAjS2020}. The identities of \cite{lAjS2020} allow one, 
in the non-relativistic case, to handle the boundary terms. In particular, in \cites{lAjS2022,lAjS20XX}, 
we used specialized versions of those identities to study the structure of $\hfour$-MGHD for the
non-relativistic $3D$ compressible Euler equations.
To extend these results to $3D$ relativistic Euler solutions without compact support,
one would need to derive relativistic analogs of the integral identities from \cite{lAjS2020}.

\subsubsection{The wave energy estimate hierarchy}
\label{SSS:WAVEENERGYESTIMATEHIERARCHY}
We now discuss the energy estimates for
the wave variables $\wavearray$, defined in \eqref{E:WAVEARRAYWITHALMOSTRIEMANNINVARIANTS}.
As before, let $\Tanset^N$ denote an arbitrary
order $N$ string of elements of the $\nullhyparg{\eik}$-tangent vectorfields $\Tanset$ defined in \eqref{E:VECTORFIELDFRAME}.
The main difficulty in closing the energy estimates in multi-dimensions is that the best estimates we know
how to derive allow for the possibility that the top-order energies
blow up as $\upmu \downarrow 0$. 
Before proceeding, for notational convenience, we define:
\begin{align} \label{E:WAVEENERGIESMAXOVERPSI}
\begin{split}
\mathbb{E}_{(\textnormal{Wave})}[\Tanset^N \wavearray](t,\eik) 
& := 
\max_{\Psi \in \wavearray} 
\mathbb{E}_{(\textnormal{Wave})}[\Tanset^N \Psi](t,\eik),
	\\
\mathbb{F}_{(\textnormal{Wave})}[\Tanset^N \wavearray](t,\eik) 
& := 
	\max_{\Psi \in \wavearray} 
	\mathbb{F}_{(\textnormal{Wave})}[\Tanset^N \Psi](t,\eik),
	\\
\spacetimeintegralcontrolwave[\Tanset^N \wavearray](t,\eik)
& :=
\max_{\Psi \in \wavearray} 
\spacetimeintegralcontrolwave[\Tanset^N \Psi](t,\eik),
\end{split}
\end{align}
where $\mathbb{E}_{(\textnormal{Wave})}$ and $\mathbb{F}_{(\textnormal{Wave})}$ are as in Def.\,\ref{D:WAVEENERGIES}
and $\spacetimeintegralcontrolwave$ is as in Remark~\ref{R:SPACETIMEINTEGRAL}.

The estimates for the wave energies and null fluxes take the following hierarchical form
on $(t,\eik) \in [0,\Tboot) \times [-\rightu,\compactsupportu]$
(see, e.g., \cite{jLjS2018}*{Proposition~14.1} for complete proofs in the case of
the $2D$ non-relativistic compressible Euler equations with vorticity):
\begin{subequations}
\begin{align}
\begin{split} \label{E:TOPORDERENERGYESTIMATE}
& \mathbb{E}_{(\textnormal{Wave})}[\Tanset^{\Ntop} \wavearray](t,\eik) 
+ 
\mathbb{F}_{(\textnormal{Wave})}[\Tanset^{\Ntop} \wavearray](t,\eik)
+
\spacetimeintegralcontrolwave[\Tanset^{\Ntop} \wavearray](t,\eik) 
	\\
&
\leq
\mbox{\upshape data}
\cdot
\upmu_{\star}^{-A}(t,\eik),
\end{split}
	\\
\begin{split} \label{E:JUSTBELOWTOPORDERENERGYESTIMATE}
& \mathbb{E}_{(\textnormal{Wave})}[\Tanset^{\Ntop-1} \wavearray](t,\eik) 
+ 
\mathbb{F}_{(\textnormal{Wave})}[\Tanset^{\Ntop-1} \wavearray](t,\eik)
+
\spacetimeintegralcontrolwave[\Tanset^{\Ntop-1} \wavearray](t,\eik) 
	\\
&
\leq
\mbox{\upshape data}
\cdot
\upmu_{\star}^{-(A-2)}(t,\eik),
\end{split}
	\\
\begin{split} \label{E:TWOBELOWTOPORDERENERGYESTIMATE}
& \mathbb{E}_{(\textnormal{Wave})}[\Tanset^{\Ntop-2} \wavearray](t,\eik) 
+ 
\mathbb{F}_{(\textnormal{Wave})}[\Tanset^{\Ntop-2} \wavearray](t,\eik)
+
\spacetimeintegralcontrolwave[\Tanset^{\Ntop-2} \wavearray](t,\eik) 
	\\
&
\leq
\mbox{\upshape data}
\cdot
\upmu_{\star}^{-(A-4)}(t,\eik),
\end{split}
	\\
& \vdots
	\notag
	\\
\begin{split} \label{E:BOUNDEDENERGYESTIMATE}
& \mathbb{E}_{(\textnormal{Wave})}[\Tanset^{[1,\frac{A}{2}]} \wavearray](t,\eik) 
+ 
\mathbb{F}_{(\textnormal{Wave})}[\Tanset^{[1,\frac{A}{2}]} \wavearray](t,\eik)
+
\spacetimeintegralcontrolwave[\Tanset^{[1,\frac{A}{2}]} \wavearray](t,\eik) 
	\\
&
\leq
\mbox{\upshape data},
\end{split}
\end{align}
\end{subequations}
where as before,
``$\mbox{\upshape data}$'' denotes a small term that depends on the initial data and that measures the perturbation
of the data away from simple isentropic plane-symmetric data,
\begin{align} \label{E:MUSTARDEF}
	\upmu_{\star}(t,\eik)
	& := 
		\min
		\lbrace
			1,
			\min_{\Sigma_t^{[-\rightu,\eik]}} \upmu
		\rbrace,
\end{align}
$\Tanset^{[1,\frac{A}{2}]}$ denotes 
an arbitrary string of elements of the $\nullhyparg{\eik}$-tangent vectorfields $\Tanset$
of order in between\footnote{One can close the proof without deriving energy estimates in the case $N=0$. This is convenient
because the order $0$ wave energies can initially be large, stemming from the initial largeness of the data for $\muX \RRiemann$
(largeness that is present even for 
the simple isentropic plane-symmetric solutions treated in Theorem~\ref{T:MAINTHEOREM1DSINGULARBOUNDARYANDCREASE}). 
In contrast, the energies in \eqref{E:TOPORDERENERGYESTIMATE}--\eqref{E:BOUNDEDENERGYESTIMATE}
are initially small (in fact, they vanish for the solutions in Theorem~\ref{T:MAINTHEOREM1DSINGULARBOUNDARYANDCREASE}).
\label{FN:NORDER0WAVEENERGIES}} 
$1$ and $\frac{A}{2}$,
and for reasons described below,
$A \gg 1$ is a universal constant (independent of $\Ntop$, the initial data, and the equation of state).

Energy estimates in the spirit of
\eqref{E:TOPORDERENERGYESTIMATE}--\eqref{E:BOUNDEDENERGYESTIMATE},
which are singular at the high derivative levels, are the only kinds of
energy estimates that are known in the context of multi-dimensional shock formation. 
One might be concerned that the high-order energies are allowed to blow up when $\upmu$ vanishes, 
as this seems to be in conflict with the philosophy
described in Sect.\,\ref{SS:NONLINEARGEOMETRICOPTICS}, namely that the solution should 
look regular in geometric coordinates, all the way up to the shock. However, the full hierarchy  
\eqref{E:TOPORDERENERGYESTIMATE}--\eqref{E:BOUNDEDENERGYESTIMATE}
shows that the wave energies become less singular by two powers of $\upmu_{\star}^{-1}$ with each 
level of descent below the top, until one reaches the level \eqref{E:BOUNDEDENERGYESTIMATE}
at which the energies remain bounded. In particular, the boundedness of the mid-order and
below geometric energies, as shown by the estimate \eqref{E:BOUNDEDENERGYESTIMATE},
capture the sense in which the solution remains regular relative to the geometric coordinates.
Due to coupling (see, e.g., \eqref{E:SECONDTRANSPORTENERGYINEQUALITY}), 
these estimates imply that the fluid variables $\modVortVort, \modDivGradEnt, \vort, \GradEnt$
obey a similar energy hierarchy 
(omitted here for brevity, but see \cite{jLjS2021} for a detailed statement and proof of 
the hierarchy in the case of the non-relativistic $3D$ compressible Euler equations) 
featuring related but distinct blowup-rates, which are nonetheless compatible
with proving \eqref{E:TOPORDERENERGYESTIMATE}--\eqref{E:BOUNDEDENERGYESTIMATE}.
We highlight the following key point:

\begin{quote}
The non-singular estimates \eqref{E:BOUNDEDENERGYESTIMATE} are what allow one to improve,
through Sobolev embedding, a smallness assumption on ``\mbox{\upshape data}'', 
and derivative-losing transport-equation-type estimates,
the $L^{\infty}$ bootstrap assumptions described in Sect.\,\ref{SS:BOOTSTRAP}.
\end{quote}

The singular top-order wave equation energy estimate \eqref{E:TOPORDERENERGYESTIMATE} 
stems from the following integral inequality, 
whose proof we describe in Sect.\,\ref{SSS:PROOFOFSINGULARHIGHORDERENERGYINEQUALITY}:
\begin{align}
\begin{split} \label{E:HIGHORDERENERGYINTEGRALIDENTITYWAVEWITHDIFFICULTTERM}
& \mathbb{E}_{(\textnormal{Wave})}[\Tanset^{\Ntop} \wavearray](t,\eik) 
+ 
\mathbb{F}_{(\textnormal{Wave})}[\Tanset^{\Ntop} \wavearray](t,\eik)
+
\spacetimeintegralcontrolwave[\Tanset^{\Ntop} \wavearray](t,\eik) 
	\\
&
\leq
\mbox{\upshape data}
+
A
\int_{t'=0}^t  
	\left\|
		\frac{\Lunit \upmu}{\upmu} 
	\right\|_{L^{\infty}(\Sigma_{t'}^{[-\rightu,\eik]})}
	\mathbb{E}_{(\textnormal{Wave})}[\Tanset^{\Ntop} \wavearray](t',\eik)
\, \rmd t'
+
\cdots.
\end{split}
\end{align}
In \eqref{E:HIGHORDERENERGYINTEGRALIDENTITYWAVEWITHDIFFICULTTERM},
$A \gg 1$ is the universal positive constant mentioned above
and $\cdots$ denotes similar or easier\footnote{In reality, some of the ``easier'' error terms 
also require substantial effort to treat. For example, to handle some of the error terms
coming from the $\deformarg{\multipliervectorfield}{\beta}{\lambda}$-involving
integral on RHS~\eqref{E:FUNDAMENTALENERGYINTEGRALIDENTITYWAVE},
one needs sharp information about the behavior of $\upmu$ and its derivatives in regions where $\upmu$ is small;
see, for example, \cite{jLjS2018}*{Lemma~14.10}.
\label{FN:EASIERERRORTERMSARESTILLDIFFICULT}} 
error terms.
From \eqref{E:HIGHORDERENERGYINTEGRALIDENTITYWAVEWITHDIFFICULTTERM} and Gr\"{o}nwall's inequality,
we deduce that:
\begin{align}
\begin{split} \label{E:HIGHORDERENERGYESTIMATE}
& \mathbb{E}_{(\textnormal{Wave})}[\Tanset^{\Ntop} \Psi](t,\eik) 
+ 
\mathbb{F}_{(\textnormal{Wave})}[\Tanset^{\Ntop} \Psi](t,\eik)
+
\spacetimeintegralcontrolwave[\Tanset^{\Ntop} \Psi](t,\eik) 
	\\
& \leq
\mbox{\upshape data}
\cdot
\exp \left( 
A
\int_{t'=0}^t  
	\left\|
		\frac{\Lunit \upmu}{\upmu} 
	\right\|_{L^{\infty}(\Sigma_{t'}^{[-\rightu,\eik]})}
\, \rmd t'
\right)
+
\cdots
\end{split}
\end{align}
An argument based on precise, refined versions of
\eqref{E:SECONDMUSCHEMATICEVOLUTION}--\eqref{E:MUSCHEMATICINTEGRATED}
yields the following crucial bound, which can be proved using the same arguments given in
\cite{jSgHjLwW2016}*{Section~10} 
(the proof is trivial if one ignores the terms $\cdots$ in \eqref{E:SECONDMUSCHEMATICEVOLUTION}--\eqref{E:MUSCHEMATICINTEGRATED}):
\begin{align} \label{E:ESTIMATEFORDIFFICULTGRONWALLFACTOR}
	\exp \left( 
	\int_{t'=0}^t  
	\left\|
		\frac{\Lunit \upmu}{\upmu} 
	\right\|_{L^{\infty}(\Sigma_{t'}^{[-\rightu,\eik]})}
\, \rmd t'
\right)
& \leq
	C
	\upmu_{\star}^{-1}(t,\eik)	
	+
	C,
\end{align}
where $\upmu_{\star}$ is defined in \eqref{E:MUSTARDEF}.
Combining 
\eqref{E:HIGHORDERENERGYESTIMATE}
and \eqref{E:ESTIMATEFORDIFFICULTGRONWALLFACTOR},
we conclude \eqref{E:TOPORDERENERGYESTIMATE}.

The less singular below-top-order estimate \eqref{E:JUSTBELOWTOPORDERENERGYESTIMATE}
stems from the following integral inequality, 
whose proof we describe in Sect.\,\ref{SSS:PROOFOFJUSTBELOWTOPSINGULARHIGHORDERENERGYINEQUALITY}:
\begin{align}
\begin{split} \label{E:JUSTBELOWTOPORDERENERGYINTEGRALIDENTITYWAVEWITHDIFFICULTTERM}
& \mathbb{E}_{(\textnormal{Wave})}[\Tanset^{\Ntop-1} \wavearray](t,\eik) 
+ 
\mathbb{F}_{(\textnormal{Wave})}[\Tanset^{\Ntop-1} \wavearray](t,\eik)
+
\spacetimeintegralcontrolwave[\Tanset^{\Ntop-1} \wavearray](t,\eik) 
	\\
&
\leq
\mbox{\upshape data}
+
C
\int_{t'=0}^t  
	\frac{1}{\upmu_{\star}^{1/2}(t',\eik)}
	\mathbb{E}_{(\textnormal{Wave})}^{1/2}[\Tanset^{\Ntop-1} \wavearray](t',\eik)
	\int_{t''=0}^{t'}
		\frac{1}{\upmu_{\star}^{1/2}(t'',\eik)}
		\mathbb{E}_{(\textnormal{Wave})}^{1/2}[\Tanset^{\Ntop} \wavearray](t'',\eik)
	\, \rmd t''
\, \rmd t'
	\\
& 
+
\cdots.
\end{split}
\end{align}
Note that RHS~\eqref{E:JUSTBELOWTOPORDERENERGYINTEGRALIDENTITYWAVEWITHDIFFICULTTERM} 
involves a coupling between the top-order energies and the just-below-top-order energies.
The actual proof of \eqref{E:JUSTBELOWTOPORDERENERGYESTIMATE} involves a delicate Gr\"{o}nwall argument
that is coupled to the proof of the top-order estimate \eqref{E:TOPORDERENERGYESTIMATE}
as well as the following estimates, which hold for numbers $B > 1$ and which
can be proved by using arguments similar to the ones needed to prove \eqref{E:ESTIMATEFORDIFFICULTGRONWALLFACTOR}:
\begin{align} \label{E:ESTIMATEFOR1OVERMUSTARTOPOWER}
	\int_{t'=0}^t  
		\upmu_{\star}^{-B}(t',\eik)
\, \rmd t'
& \leq
	C
	\upmu_{\star}^{1-B}(t,\eik),
			\\
\int_{t'=0}^t  
		\upmu_{\star}^{-\frac{9}{10}}(t',\eik)
\, \rmd t'
& \leq
	C.
	\label{E:REGULARESTIMATEFOR1OVERMUSTARTOPOWER}
\end{align}
We stress that \eqref{E:ESTIMATEFOR1OVERMUSTARTOPOWER} is a ``quasilinear version'' of the following simple estimate,
which is valid for $t \in [0,1)$:
$
	\int_{t'=0}^t  
		(1-t')^{-B}
\, \rmd t'
\leq
	C
	(1-t)^{1-B}
$.
Similarly,
\eqref{E:REGULARESTIMATEFOR1OVERMUSTARTOPOWER} is
is a quasilinear version of the estimate
$
	\int_{t'=0}^t  
		(1-t')^{-\frac{9}{10}}
\, \rmd t'
\leq
	C
$.
The proof of \eqref{E:ESTIMATEFOR1OVERMUSTARTOPOWER}--\eqref{E:REGULARESTIMATEFOR1OVERMUSTARTOPOWER}
fundamentally relies on the fact that $\upmu$ vanishes linearly in $t$, as we described at
the end of Sect.\,\ref{SS:BEHAVOROFMUANDFORMATIONOFSHOCK}.
Here, we will
just use \eqref{E:JUSTBELOWTOPORDERENERGYINTEGRALIDENTITYWAVEWITHDIFFICULTTERM} 
and \eqref{E:ESTIMATEFOR1OVERMUSTARTOPOWER}
to explain why the desired estimates
\eqref{E:TOPORDERENERGYESTIMATE} and \eqref{E:JUSTBELOWTOPORDERENERGYESTIMATE} are \emph{consistent}
with respect to powers of $\upmu_{\star}^{-1}(t,\eik)$.
To confirm the consistency, we plug the estimates \eqref{E:TOPORDERENERGYESTIMATE} and \eqref{E:JUSTBELOWTOPORDERENERGYESTIMATE}
into the integrals on RHS~\eqref{E:JUSTBELOWTOPORDERENERGYINTEGRALIDENTITYWAVEWITHDIFFICULTTERM}
to deduce:
\begin{align}
\begin{split} \label{E:PROOFSTEPJUSTBELOWTOPORDERENERGYINTEGRALIDENTITYWAVEWITHDIFFICULTTERM}
& \mathbb{E}_{(\textnormal{Wave})}[\Tanset^{\Ntop-1} \wavearray](t,\eik) 
+ 
\mathbb{F}_{(\textnormal{Wave})}[\Tanset^{\Ntop-1} \wavearray](t,\eik)
+
\spacetimeintegralcontrolwave[\Tanset^{\Ntop-1} \wavearray](t,\eik) 
	\\
&
\leq
\mbox{\upshape data}
+
\mbox{\upshape data}
\cdot
\int_{t'=0}^t  
	\upmu_{\star}^{- \frac{A}{2} + \frac{1}{2}}(t',\eik)
	\int_{t''=0}^{t'}
		\upmu_{\star}^{-\frac{A}{2} - \frac{1}{2}}(t'',\eik)
	\, \rmd t''
\, \rmd t'
+
\cdots.
\end{split}
\end{align}
Using \eqref{E:PROOFSTEPJUSTBELOWTOPORDERENERGYINTEGRALIDENTITYWAVEWITHDIFFICULTTERM}
and twice using \eqref{E:ESTIMATEFOR1OVERMUSTARTOPOWER},
we deduce that:
\begin{align}
\begin{split} \label{E:SECONDPROOFSTEPJUSTBELOWTOPORDERENERGYINTEGRALIDENTITYWAVEWITHDIFFICULTTERM}
& \mathbb{E}_{(\textnormal{Wave})}[\Tanset^{\Ntop-1} \wavearray](t,\eik) 
+ 
\mathbb{F}_{(\textnormal{Wave})}[\Tanset^{\Ntop-1} \wavearray](t,\eik)
+
\spacetimeintegralcontrolwave[\Tanset^{\Ntop-1} \wavearray](t,\eik) 
	\\
&
\leq
\mbox{\upshape data}
\cdot
\upmu_{\star}^{-(A-2)}(t,\eik),
\end{split}
\end{align}
which is indeed consistent with the desired estimate \eqref{E:JUSTBELOWTOPORDERENERGYESTIMATE}.
This concludes our proof sketch of \eqref{E:JUSTBELOWTOPORDERENERGYESTIMATE}.

Using the same strategy, one can continue the descent in the energy estimate hierarchy, gaining two powers
of $\upmu_{\star}$ with each level of descent until one finally
one can use \eqref{E:REGULARESTIMATEFOR1OVERMUSTARTOPOWER} to reach the level \eqref{E:BOUNDEDENERGYESTIMATE},
where the energies remain uniformly bounded.

\subsubsection{Discussion of the proof of the top-order inequality \eqref{E:HIGHORDERENERGYINTEGRALIDENTITYWAVEWITHDIFFICULTTERM}}
\label{SSS:PROOFOFSINGULARHIGHORDERENERGYINEQUALITY}
We now sketch some key ideas behind the proof of \eqref{E:HIGHORDERENERGYINTEGRALIDENTITYWAVEWITHDIFFICULTTERM}.
We focus on the most difficult case
in which the operator $\Tanset^{\Ntop}$ in \eqref{E:HIGHORDERENERGYINTEGRALIDENTITYWAVEWITHDIFFICULTTERM}
is of the form\footnote{The case in which $\Tanset^{\Ntop}$ contains an $\Lunit$-differentiation is much easier because
in that case, one can directly use equation \eqref{E:RAYCHAUDHURITRANSPORTCHI} to estimate the $\Lunit$-derivatives
of $\mytr_{\gtorus} \upchi$. \label{FN:LUNITDERIVATIVESAREEASIER}} 
$\Tanset^{\Ntop} = \Angularset^{\Ntop}$,
where $\Angularset^{\Ntop}$ is a string of elements of the $\ell_{t,\eik}$-tangent subset $\Angularset$ defined in \eqref{E:VECTORFIELDFRAME}.
Using the wave equation energy-null flux identity \eqref{E:FUNDAMENTALENERGYINTEGRALIDENTITYWAVE},
the commuted wave equation \eqref{E:SCHEMTICCOMMUTEDWAVE},
the transport/elliptic estimates for the vorticity and entropy 
described in Sects.\,\ref{SSS:TRANSPORTENERGYESTIMATEHIERARCHY}--\ref{SSS:TOPORDERELLIPTICVORTICITYANDENTROPY},
and the coerciveness of the spacetime integral from \eqref{E:SPACETIMEINTEGRAL},
one can show that for $\Psi \in \wavearray$ (which is defined in \eqref{E:WAVEARRAYWITHALMOSTRIEMANNINVARIANTS}), 
the following identity holds for $(t,\eik) \in [0,\Tboot) \times [-\rightu,\compactsupportu]$:
\begin{align}
\begin{split} \label{E:HIGHORDERENERGYINTEGRALIDENTITYWAVE}
& \mathbb{E}_{(\textnormal{Wave})}[\Angularset^{\Ntop} \Psi](t,\eik) 
+ 
\mathbb{F}_{(\textnormal{Wave})}[\Angularset^{\Ntop} \Psi](t,\eik)
+
\spacetimeintegralcontrolwave[\Angularset^{\Ntop} \Psi](t,\eik) 
	\\
&
= 
\mbox{\upshape data}
+
\int_{\mathcal{M}_{[0,t),[-\rightu,\eik]}}  
	(\multipliervectorfield \Angularset^{\Ntop} \Psi)
	\cdot
	(\muX \Psi) 
	\cdot 
	\Angularset^{\Ntop} \mytr_{\gtorus} \upchi
\, \voltorus \rmd \eik' \rmd t'
	\\
& \ \
+
\int_{\mathcal{M}_{[0,t),[-\rightu,\eik]}}  
	\upmu
	(\multipliervectorfield \Angularset^{\Ntop} \Psi)
	\cdot
	\Angularset^{\Ntop} \modVortVort
\, \voltorus \rmd \eik' \rmd t'
	\\
& \ \
+
\int_{\mathcal{M}_{[0,t),[-\rightu,\eik]}}  
	\upmu
	(\multipliervectorfield \Angularset^{\Ntop} \Psi)
	\cdot
	\Angularset^{\Ntop} \modDivGradEnt
\, \voltorus \rmd \eik' \rmd t'
+
\cdots,
\end{split}
\end{align}
where $\cdots$ denotes easier error terms.

We first consider the difficult case $\Psi = \RRiemann$.
We recall the definition \eqref{E:MULTIPLIERVECTORFIELD} of $\multipliervectorfield$.
We consider the part of the integral 
$
\int_{\mathcal{M}_{[0,t),[-\rightu,\eik]}}  
	(\multipliervectorfield \Angularset^{\Ntop} \Psi)
	\cdot
	(\muX \Psi) 
	\cdot 
	\Angularset^{\Ntop} \mytr_{\gtorus} \upchi
\, \voltorus \rmd \eik' \rmd t'
$
on RHS~\eqref{E:HIGHORDERENERGYINTEGRALIDENTITYWAVE}
that is generated by the $2 \muX$ term in \eqref{E:MULTIPLIERVECTORFIELD}.
Using \eqref{E:FULLYMODIFIEDORDERN}, 
and using \eqref{E:MUSCHEMATICEVOLUTION} to replace $\muX \RRiemann$ with $\Lunit \upmu$
up to error terms,
we can express the $2 \muX$-involving
portion of the integral under consideration as follows, where we highlight the 
crucial factors of $\frac{1}{\upmu}$ on RHS~\eqref{E:REEXPRESSINGDIFFICULTERRORINTEGRAL},
and for simplicity, we have ignored the factors of $\smoothfunction(\wavearray)$ on RHS~\eqref{E:FULLYMODIFIEDORDERN}:
\begin{align} \label{E:REEXPRESSINGDIFFICULTERRORINTEGRAL}
\begin{split}
& 
	2
	\int_{\mathcal{M}_{[0,t),[-\rightu,\eik]}}  
	(\muX \Angularset^{\Ntop} \RRiemann)
	\cdot
	(\muX \RRiemann) 
	\cdot 
	\Angularset^{\Ntop} \mytr_{\gtorus} \upchi
\, \voltorus \rmd \eik' \rmd t'
	\\
& = 
-
2
\int_{\mathcal{M}_{[0,t),[-\rightu,\eik]}} 
	\frac{\Lunit \upmu}{\upmu}
	\cdot
	(\muX \Angularset^{\Ntop} \RRiemann)
	\cdot 
	(\muX \Angularset^{\Ntop} \RRiemann)
\, \voltorus \rmd \eik' \rmd t'
	\\
&
	\ \
+
2 
\int_{\mathcal{M}_{[0,t),[-\rightu,\eik]}} 
	\frac{\Lunit \upmu}{\upmu}
	\cdot
	(\muX \Angularset^{\Ntop} \RRiemann)
	 \cdot 
	\fullymodquant{\Angularset^{\Ntop}}
\, \voltorus \rmd \eik' \rmd t'
+
\cdots
\end{split}
\end{align}
From \eqref{E:WAVEENERGYDEF}, it follows that 
the first integral on RHS~\eqref{E:REEXPRESSINGDIFFICULTERRORINTEGRAL}
is bounded by the $A$-multiplied integral on RHS~\eqref{E:HIGHORDERENERGYINTEGRALIDENTITYWAVEWITHDIFFICULTTERM}
(this integral contributes a ``portion of'' the ``$A$'').
The $\fullymodquant{\Angularset^{\Ntop}}$-involving integral on RHS~\eqref{E:REEXPRESSINGDIFFICULTERRORINTEGRAL}
can be bounded by a similar error term involving a double time-integral, where one of the time integrations
comes from integrating \eqref{E:COMMUTEDMODIFIEDTRCHISCHEMATIC}; 
we refer to the proof of \cite{jSgHjLwW2016}*{Proposition~14.2} for the details.
The part of the first integral on RHS~\eqref{E:HIGHORDERENERGYINTEGRALIDENTITYWAVE}
that is generated by the $(1 + 2 \upmu) \Lunit$ term in \eqref{E:MULTIPLIERVECTORFIELD}
can be handled through an argument involving integration by parts in $\Lunit$, which leads to difficult
critical-strength boundary terms that make a similar contribution to the blowup of the top-order energies 
as the $A$-multiplied integral on RHS~\eqref{E:HIGHORDERENERGYINTEGRALIDENTITYWAVEWITHDIFFICULTTERM}; the same arguments
given in the proof of \cite{jSgHjLwW2016}*{Proposition~14.2} can be used to handle these terms.
The $\Angularset^{\Ntop} \modVortVort$- and $\Angularset^{\Ntop} \modDivGradEnt$-involving integrals
on RHS~\eqref{E:HIGHORDERENERGYINTEGRALIDENTITYWAVE} are easy to handle, thanks to the
helpful factor of $\upmu$ in the integrand and the
strategy for bounding $\Angularset^{\Ntop} \modVortVort$ and $\Angularset^{\Ntop} \modDivGradEnt$ via
transport-div-curl estimates that we described in Sect.\,\ref{SSS:TOPORDERELLIPTICVORTICITYANDENTROPY}.

We now consider the remaining cases, namely $\Psi \in \lbrace \LRiemann, \Ent, \fourvelocity^2, \fourvelocity^3 \rbrace$.
The identity \eqref{E:HIGHORDERENERGYINTEGRALIDENTITYWAVE} holds in these cases too. However, the factor of
$\muX \Psi$ in the identity is now very small in $L^{\infty}$ in these cases, since we are considering perturbations of simple
isentropic plane symmetric solutions (for which these quantities all vanish). Hence, due to the smallness,
all error terms on RHS~\eqref{E:HIGHORDERENERGYINTEGRALIDENTITYWAVE} are much easier to treat and can in fact
be relegated to the error terms ``$\cdots$'' on RHS~\eqref{E:HIGHORDERENERGYINTEGRALIDENTITYWAVEWITHDIFFICULTTERM};
see \cite{jLjS2018}*{page~154} for further details.

\subsubsection{Discussion of the proof of the just-below-top-order inequality \eqref{E:JUSTBELOWTOPORDERENERGYINTEGRALIDENTITYWAVEWITHDIFFICULTTERM}}
\label{SSS:PROOFOFJUSTBELOWTOPSINGULARHIGHORDERENERGYINEQUALITY}
We now sketch some key ideas behind the proof of
the just-below-top-order inequality \eqref{E:JUSTBELOWTOPORDERENERGYINTEGRALIDENTITYWAVEWITHDIFFICULTTERM}.
First, one uses the identity \eqref{E:HIGHORDERENERGYINTEGRALIDENTITYWAVE},
which, for any $\Psi \in \wavearray$, holds with $\Tanset^{\Ntop-1}$ in the role of $\Angularset^{\Ntop}$. 
The key step, which is the one that is different compared to the proof of
\eqref{E:HIGHORDERENERGYINTEGRALIDENTITYWAVEWITHDIFFICULTTERM}, is that 
one bounds the error integral on LHS~\eqref{E:REEXPRESSINGDIFFICULTERRORINTEGRAL}, namely
\begin{align} \label{E:BELOWTOPORDERDIFFICULTERRORINTEGRAL}
2
\int_{\mathcal{M}_{[0,t),[-\rightu,\eik]}}  
	(\multipliervectorfield \Tanset^{\Ntop-1} \Psi)
	\cdot
	(\muX \Psi) 
	\cdot 
	\Tanset^{\Ntop-1} \mytr_{\gtorus} \upchi
\, \voltorus \rmd \eik' \rmd t',
\end{align}
in a different way. First, one uses the bootstrap assumptions to bound the factor
$|\muX \Psi|$ in \eqref{E:BELOWTOPORDERDIFFICULTERRORINTEGRAL} by $\lesssim 1$. 
Next, one uses \eqref{E:MULTIPLIERVECTORFIELD},
the Cauchy--Schwarz inequality, and \eqref{E:WAVEENERGYDEF} 
to bound the integral in \eqref{E:BELOWTOPORDERDIFFICULTERRORINTEGRAL} 
by:
\begin{align} \label{E:FIRSTPROOFSTEPJUSTBELOWTOPORDERENERGYINTEGRALIDENTITYWAVEWITHDIFFICULTTERM}
	&
	\lesssim 
	\int_0^t
		\frac{1}{\upmu_{\star}^{1/2}(t',\eik)}
		\mathbb{E}_{(\textnormal{Wave})}^{1/2}[\Tanset^{\Ntop-1} \Psi](t',\eik) 
		\left\| \Tanset^{\Ntop-1} \mytr_{\gtorus} \upchi \right\|_{L^2(\Sigma_{t'}^{[-\rightu,\eik]})}
	\, \rmd t'
	+
	\cdots,
\end{align}
where we have used that the $\nullhyparg{\eik}$-tangent derivative terms in the energy \eqref{E:WAVEENERGYDEF} contain a $\upmu$-weight,
while there is an $\Lunit$ factor in the definition of $\multipliervectorfield$ that does \emph{not} contain a $\upmu$ weight.
Next, to bound the factor 
$\| \Angularset^{\Ntop-1} \mytr_{\gtorus} \upchi \|_{L^2(\Sigma_{t'}^{[-\rightu,\eik]})}$,
one commutes the Raychaudhuri equation \eqref{E:RAYCHAUDHURITRANSPORTCHI} with
$\Tanset^{\Ntop-1}$ and uses the schematic relation 
$\Ricfour_{\Lunit \Lunit} = \Tanset^2 \wavearray + \cdots$
to deduce the following evolution equation, schematically depicted: 
$\Lunit \Angularset^{\Ntop-1} \mytr_{\gtorus} \upchi = \Tanset^{\Ntop+1} \wavearray + \cdots$,
where $\cdots$ denotes lower order error terms. 
This equation represents a loss of one derivative in the estimates and leads to the coupling between 
the different order energies described below \eqref{E:JUSTBELOWTOPORDERENERGYINTEGRALIDENTITYWAVEWITHDIFFICULTTERM}.
However, such a loss is permissible below the top-order.
Since $\Lunit t = 1$ and $\Lunit \eik = 0$, we can integrate this evolution equation 
and use \eqref{E:WAVEENERGYDEF} to deduce that:
\begin{align}  
\begin{split} \label{E:THIRDPROOFSTEPJUSTBELOWTOPORDERENERGYINTEGRALIDENTITYWAVEWITHDIFFICULTTERM}
\left\| \Tanset^{\Ntop-1} \mytr_{\gtorus} \upchi \right\|_{L^2(\Sigma_{t'}^{[-\rightu,\eik]})}
& \lesssim 
\mbox{\upshape data}
+
\int_0^{t'}
	\| \Tanset^{\Ntop+1} \wavearray \|_{L^2(\Sigma_{t''}^{[-\rightu,\eik]})}
\, \rmd t''
+
\cdots
	\\
&
\lesssim 
\mbox{\upshape data}
+
\int_0^{t'}
	\frac{1}{\upmu_{\star}^{1/2}(t'',\eik)}
	\mathbb{E}_{(\textnormal{Wave})}^{1/2}[\Tanset^{\Ntop} \Psi](t'',\eik) 
\, \rmd t''
+
\cdots,
\end{split}
\end{align}
where to obtain the last line of \eqref{E:THIRDPROOFSTEPJUSTBELOWTOPORDERENERGYINTEGRALIDENTITYWAVEWITHDIFFICULTTERM},
we have again used that the $\nullhyparg{\eik}$-tangent derivative terms in the energy \eqref{E:WAVEENERGYDEF} contain a $\upmu$-weight.
Inserting \eqref{E:THIRDPROOFSTEPJUSTBELOWTOPORDERENERGYINTEGRALIDENTITYWAVEWITHDIFFICULTTERM}
into \eqref{E:FIRSTPROOFSTEPJUSTBELOWTOPORDERENERGYINTEGRALIDENTITYWAVEWITHDIFFICULTTERM},
we conclude \eqref{E:JUSTBELOWTOPORDERENERGYINTEGRALIDENTITYWAVEWITHDIFFICULTTERM}.

\section{Open problems}
\label{S:OPENPROBLEMS}
In this section, we describe various open problems tied to shocks.

\begin{enumerate}
	\item \textbf{Prove Conjecture~\ref{CON:SHOCKWITHVORTICITYANDENTROPY}}. 
	In Sect.\,\ref{S:SHOCKFORMATIONAWAYFROMSYMMETRY}, we outlined how to achieve this. 
	We are confident that the conjecture can be proved, especially since
	the formulation of relativistic Euler flow
	provided by Theorem~\ref{T:GEOMETRICWAVETRANSPORTDIVCURLFORMULATION} 
	is qualitatively similar to the formulation 
	of $3D$ non-relativistic compressible Euler flow derived in \cite{jS2019c},
	a system for which shock formation results have been derived
	(see Sect.\,\ref{S:PRIORWORKSSHOCKFORMATION}).
	Nonetheless, there are many non-trivial details that have to be checked,
	and this would be a good project for someone who wants to learn the field.
	\item \textbf{Prove Conjecture~\ref{CON:SINGULARBOUNDARY}}. 
	This problem is much more technically demanding than proving Conjecture~\ref{CON:SHOCKWITHVORTICITYANDENTROPY}.
	However, for the same reasons
	mentioned above, we expect that
	this can be achieved by adapting the methods used in the non-relativistic work \cite{lAjS2022}
	to the equations of Theorem~\ref{T:GEOMETRICWAVETRANSPORTDIVCURLFORMULATION}.
	\item \textbf{Prove Conjecture~\ref{CON:CAUCHYHORIZON}}. 
	This problem is also much more technically demanding than proving Conjecture~\ref{CON:SHOCKWITHVORTICITYANDENTROPY}.
	For the same reasons
	mentioned above, we expect that
	this can be achieved by adapting the methods that we are using in our forthcoming 
	non-relativistic work \cite{lAjS20XX} to the equations of Theorem~\ref{T:GEOMETRICWAVETRANSPORTDIVCURLFORMULATION}.
	\item \textbf{The shock development problem}. 
		Recall that in \cite{dC2019}, Christodoulou solved  
		the restricted shock development problem for the relativistic Euler equations
		and the non-relativistic compressible Euler equations
		in an arbitrary number of spatial dimensions; see Sect.\,\ref{SS:SHOCKDEVELOPMENT}. 
		The general problem, i.e., the shock development problem
		with vorticity and entropy, remains an outstanding open problem for both systems in two or more spatial dimensions.
	\item \textbf{The global behavior of solutions with shocks}. 
		After the shock development problem (which is local-in-time) is solved,
		a particularly compelling problem will be to understand the \emph{global}-in-time-and-space 
		behavior of the corresponding weak solutions,
		at least in a perturbative regime. Among the many outstanding 
		challenges in this problem is that of understanding the long-time behavior 
		of the vorticity and entropy; even in the context of smooth incompressible non-relativistic flows, 
		the long-time behavior of the vorticity is not well understood. We do, however, mention 
		that in the breakthrough work \cite{jCtH2022}, for axisymmetric \emph{incompressible} non-relativistic flows in $3D$
		with a boundary, smooth initial conditions were identified such that the 
		solution (including the vorticity) blows up in finite time in an approximately self-similar fashion;
		such vorticity blowup, if present in multi-dimensional relativistic Euler flow, 
		would be a serious obstacle to finding a meaningful way to continue the solution weakly past the singularity.
	\item \textbf{Extending the results to the coupled Einstein--Euler system}.
	It is of great interest to extend the above results to the Einstein--Euler system in multiple spatial dimensions.
	The expectation is that while the fluid will exhibit shock wave phenomena (as in the uncoupled problem),
	the gravitational metric will exhibit less singular behavior
	As of present, the only known singularity formation result for the coupled system 
	is in the plane symmetric case \cite{aRfS2008}, in which the dynamics are described by $1+1$-dimensional hyperbolic PDEs.
	In \cite{aRfS2008}, it was shown that for many equations of state and a large class of plane-symmetric initial data, the 
	Einstein-Euler solution breaks down in finite time. Although \cite{aRfS2008} provided heuristic arguments
	suggesting that the singularity is of shock-type,
	a precise description of the singularity was not given. 
	The multi-dimensional problem is a difficult PDE problem because surfaces that are
	null or ``barely spacelike'' with respect to the acoustical metric are in fact timelike with respect to the gravitational
	metric; the reason is that speed of sound is slower than the speed of propagation of gravitational waves
	(at least when the speed of sound is less than unity). A corresponding 
	key difficulty that arises in the context of the PDE estimates is that
	one must control the spacetime metric on various gravitationally timelike surfaces on which the fluid is singular; this is
	difficult at the top derivative level because generally, energy estimates for the spacetime metric 
	are not available on surfaces that are timelike with respect to the spacetime metric.
	In contrast, in multiple wave speed systems such that the shock forms in the \emph{fastest} wave,
	it is possible to close the top-order energy estimates and thus prove stable shock formation 
	for the coupled problem \cite{jS2018b}.
	\item \textbf{Extending the results to more complicated multiple speed systems}. It is also of great physical and mathematical
	interest to extend the above results to other multiple speed systems, such as the equations of 
	compressible magnetohydrodynamics, the GRMHD equations (i.e., the coupling of Einstein's equations to the equations 
	of relativistic magnetohydrodynamics), the equations of elasticity, the equations of nonlinear electromagnetism,
	and the equations of crystal optics. A key difficulty of these systems is that their principal symbols
	are more complicated than wave operators and transport operators. Relatedly, their corresponding geometry is more
	complicated than Lorentzian geometry, and the characteristics comprise multiple sheets, 
	which can be singular even at the tangent space level. 
	These difficulties are a serious obstacle to implementing the kind of sharp version of
	nonlinear geometric optics (i.e., eikonal functions) that has been so successfully employed 
	to study shocks for wave equations and fluid equations.
	\item \textbf{Implosion singularities}. In the breakthrough works \cites{fMpRiRjS2022a,fMpRiRjS2022b}, 
	the authors proved implosion
		singularity formation for some initially $C^{\infty}$ spherically symmetric solutions to the compressible Euler equations
		and Navier--Stokes equations under adiabatic equations of state $p = \uprho^{\upgamma}$ for $\upgamma > 1$, outside of a 
		countable set of $\upgamma$-values. 
		In \cite{tBlcGjGS2022}, the results were extended to allow for all $\upgamma > 1$.
		Implosion singularities are much more severe than shocks in the sense that the density and velocity themselves
		blow up (in shock singularities, it is their gradient that blows up)
		in finite time at the center of symmetry. The proof depends on the detailed structure of the equations
		and in particular relies on a careful analysis of a phase portrait for self-similar solutions.
		Hence, it is of interest to decide whether a similar result holds for the relativistic Euler equations.
		We highlight the important related work \cite{yGmHjJ2023}, in which the authors studied
		spherically symmetric solutions
		the Einstein--Euler equations under adiabatic equations of state $p = \varepsilon \uprho$ for $0 < \varepsilon \ll 1$ sufficiently small
		and proved the existence of self-similar solutions that form a naked singularity.
		See also \cite{yGmHjJ2022}, in which the authors studied solutions
		to the non-relativistic Euler--Poisson system under adiabatic equations of state $p = \uprho^{\upgamma}$ for $\upgamma \in (1,\frac{4}{3})$
		and showed the existence of spherically symmetric 
		self-similar imploding solutions modeling gravitational collapse, i.e., 
		initially smooth solutions such that the density blows up in finite time.
	\item \textbf{Inviscid limits}. In \cite{sCcG2023}, the authors studied the inviscid limit of solutions to the $1D$ 
		viscous Burgers' equation all the way up to the time of first shock formation in the inviscid solution.
		They decomposed the viscous solution into a singular piece and a smoother piece and proved that the viscous solution converges
		to the singular piece in $L^{\infty}$ as the viscosity vanishes, where the $L^{\infty}$ norm is taken over the entire
		slab of classical existence of the inviscid solution. This is the first result of its type that extends all the way to the time
		of first singularity formation. Important open problems include extending this result to the $1D$ compressible Euler equations
		(where the corresponding viscous equations are the Navier--Stokes equations) and, after that, to multi-dimensions.
		It would also be of interest to extend the result to the relativistic Euler equations, but as of present, 
		it is not clear if there exist any well-posed relativistic viscous fluid models that suppress the singularity formation while
		retaining physically desirable features such as causality.
	\item \textbf{Rarefaction waves}. In Sect.\,\ref{SS:RAREFACTION}, we described the recent works \cites{tWLpY2023a,tWLpY2023b}	
		on irrotational rarefaction wave solutions to the $2D$ compressible Euler equations. It is of interest to 
			eliminate the irrotationality assumption and to
			extend these results to the relativistic Euler equations, and perhaps even the Einstein--Euler equations.
\end{enumerate}

\bibliographystyle{amsalpha}
\bibliography{JBib}

\end{document}